\documentclass[11pt,a4paper, reqno]{amsart}
\usepackage{geometry}           % See geometry.pdf to learn the layout options. There are lots.
\usepackage{a4wide}
\usepackage{tikz}
\usetikzlibrary{matrix}
\usepackage{hyperref}
\usepackage{amssymb,accents}
\usepackage{amsthm}
\usepackage{enumitem}
\usepackage{tikz-cd}
\usepackage[mathscr]{euscript}
\usepackage{mathtools}
\usepackage{caption}
\usepackage{subfigure}
\usepackage{comment}
\usepackage{algorithm}
\usepackage{algorithmic}

\usepackage[bbgreekl]{mathbbol}
\DeclareSymbolFontAlphabet{\mathbbm}{bbold}
\DeclareSymbolFontAlphabet{\mathbb}{AMSb}%

\newcommand{\grapegraph}{\mathscr{G}\mathsf{rape}}

\newcommand{\cRI}{\mathcal{RI}}
\newcommand{\cI}{\mathcal{I}}

\newcommand{\Lk}{\mathrm{Lk}}

\newcommand{\bfv}{\mathbf{v}}
\newcommand{\barbfv}{\bar{\mathbf{v}}}

\newcommand{\barbfw}{\bar{\mathbf{w}}}

\newcommand{\barbfu}{\bar{\mathbf{u}}}
\newcommand{\interior}[1]{\accentset{\smash{\raisebox{-0.5ex}{$\scriptstyle\circ$}}}{#1}\rule{0pt}{1ex}}
\tikzset{% from the manual
math to/.tip={Glyph[glyph math command=rightarrow]},
loop/.tip={Glyph[glyph math command=looparrowleft, swap]},
loop'/.tip={Glyph[glyph math command=looparrowleft]},
 weird/.tip={Glyph[glyph math command=Rrightarrow, glyph length=1.5ex]},
  pi/.tip={Glyph[glyph math command=pi, glyph length=1.5ex, glyph axis=0pt]},
}

\newcommand{\grape}[2][0]{
\draw[thick, fill] (#2.center) circle(2pt) -- +({30+#1}:0.5) circle(2pt) +(0,0) -- +({-30+#1}:0.5) circle(2pt);
\draw[thick] (#2.center) +({30+#1}:0.5) arc ({120+#1}:{-120+#1}:{0.5 / sqrt(3)});
}

\newcommand{\grapeloop}[2][0]{
\draw[thick, fill] (#2.center) circle(0.1pt) -- +({30+#1}:0.5) circle(0.1pt) +(0,0) -- +({-30+#1}:0.5) circle(0.1pt);
\draw[thick] (#2.center) +({30+#1}:0.5) arc ({120+#1}:{-120+#1}:{0.5 / sqrt(3)});
}

\DeclareMathOperator{\CAT}{CAT}

\DeclareMathOperator{\link}{Lk}
\DeclareMathOperator{\lk}{\mathsf{Lk}}
\DeclareMathOperator{\st}{\mathsf{St}}
\DeclareMathOperator{\val}{val}

\newcommand\loops\ell
\newcommand\itimes{\mathbin{\interior{\times}}}

\usepgflibrary{arrows}

\numberwithin{equation}{section}

\theoremstyle{plain}
\newtheorem{theorem}[equation]{Theorem}
\newtheorem{maintheorem}{Main Theorem}

\newtheorem{lemma}[equation]{Lemma}
\newtheorem{corollary}[equation]{Corollary}
\newtheorem{proposition}[equation]{Proposition}
\newtheorem{claim}[equation]{Claim}

\theoremstyle{definition}
\newtheorem{definition}[equation]{Definition}
\newtheorem{assumption}[equation]{Assumption}
\newtheorem{question}{Question}
\newtheorem*{convention}{Convention}
\newtheorem{remark}[equation]{Remark}
\newtheorem*{observation}{Observation}
\newtheorem{Ex}[equation]{Example}

\renewcommand\bar\overline

\begin{document}
\title[Q.I. classification of certain graph $2$-braid groups and its applications]{Quasi-isometry classification of certain graph $2$-braid groups and its applications}
\author{Byung Hee An}
\address{Department of Mathematics Education, Kyungpook National University, Daegu, Korea}
\email{anbyhee@knu.ac.kr}
\author{Sangrok Oh}
\address{Department of Mathematics, University of the Basque Country, Spain}
\email{SangrokOh.math@gmail.com}
\begin{abstract}
In \cite{Oh22}, the second author defined a complex of groups decomposition of the fundamental group of a finitely generated 2-dimensional special group, called an \emph{intersection complex}, which is a quasi-isometry invariant. In this paper, using the theory of intersection complexes, we classify the class of 2-braid groups over graphs with circumference $\leq 1$ up to quasi-isometry. Moreover, we find a sufficient condition when such a graph 2-braid group is quasi-isometric to a right-angled Artin group or not.
Finally, by applying the same method, we also find that there is an algorithm to determine whether two 4-braid groups over trees are quasi-isometric or not. 
\end{abstract}

\maketitle
\tableofcontents

\section{Introduction}\label{Sec:Intro}
In geometric group theory, a central objective is the classification of finitely generated groups up to quasi-isometry. As part of this endeavor, attention often focuses on specific subclasses of \emph{special groups}, which are defined as the fundamental groups of \emph{special cube complexes} in the sense of Haglund and Wise \cite{HW08} (see Definition~\ref{Def:SpecialCubeComplex}). Notable examples include right-angled Artin groups (RAAGs) and right-angled Coxeter groups (RACGs), both of which have been extensively studied and remain active research areas. For RAAGs, see \cite{BN, BKS(a), Hua(a), Mar20}, and for RACGs, see \cite{DT17, NT19, BX20, E24}.

This paper focuses on another subclass of special groups, namely \emph{graph braid groups} (or braid groups over graphs), which, while less well-known than RAAGs and RACGs, have garnered sustained interest from researchers. Since their introduction in \cite{Abrams00}, graph braid groups have been examined for their relationships to RAAGs, as some graph braid groups are isomorphic to RAAGs; embeddability \cite{Sab07, CW} as well as conditions for isomorphism between graph braid groups and RAAGs \cite{KKP12, FS08, CD14, KLP16}. Additionally, studies have investigated their algebraic properties, including group presentations, homology, and cohomology rings (see \cite{FS05, FS12, KP12, AM21}).

The large-scale geometry of graph braid groups, however, is a relatively recent field of study compared to the more developed investigations of RAAGs and RACGs, or the algebraic properties of graph braid groups. Fernandes introduced a quasi-isometry invariant for graph $2$-braid groups called the \emph{intersection complex} \cite{Fer12}. Building on this, the second author of this paper refined the definition and used it to determine when graph $2$-braid groups are quasi-isometric to RAAGs \cite{Oh22}. Genevois \cite{Gen21} and Berlyne \cite{Ber21} independently investigated various hyperbolic properties of graph braid groups, including Gromov hyperbolicity, acylindrical hyperbolicity, and toral relative hyperbolicity.

Motivated by \cite{Oh22}, this paper seeks to further the quasi-isometric classification of graph $2$-braid groups $\mathbb{B}_2(\Gamma)$, leveraging their status as special groups. It is well-known among experts that graph $2$-braid groups are isomorphic to free groups when the defining graphs are trees. Genevois \cite{Gen21} demonstrated that $\mathbb{B}_2(\Gamma)$ is hyperbolic if and only if $\Gamma$ has no pair of disjoint induced cycles, and that $\mathbb{B}_2(\Gamma)$ is hyperbolic relative to abelian subgroups if and only if $\Gamma$ does not contain an induced cycle disjoint from two other induced cycles. Thus, it is natural to begin by examining the geometry of $2$-braid groups over graphs that are almost trees but contain pairs of disjoint cycles.

\subsection{Bunches of grapes and their \texorpdfstring{$2$}{2}-braid groups}\label{Sec:IntroMainResults}
The \emph{circumference} of a graph is the length of any longest cycle in the graph. As a graph with circumference zero is a tree, a natural candidate of a tree-like graph with induced cycles might be a planar graph with circumference at most one, which is called a \emph{bunch of grapes}. See Figure~\ref{Figure:Examples} for example and Definition~\ref{Def:BunchesofGrapes} for details.

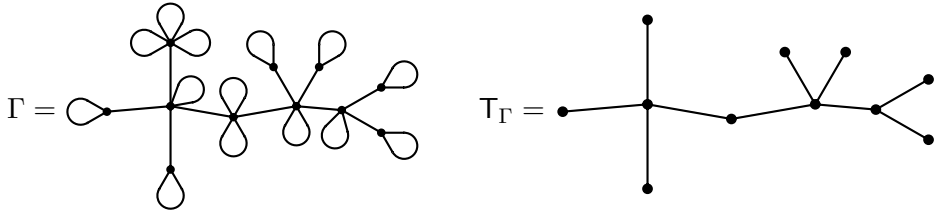
\begin{figure}[ht]
\[
\Gamma=\begin{tikzpicture}[baseline=-.5ex, scale=0.6]
\draw[thick,fill] (0,0) circle (2pt) node (A) {} -- ++(5:1.4) circle(2pt) node(B) {} -- +(0,1.4) circle (2pt) node (C) {}+(0,0) -- +(0,-1.4) circle (2pt) node(D) {} +(0,0) -- ++(-10:1.4) circle(2pt) node(E) {} -- ++(10:1.4) circle(2pt) node(F) {} -- +(120:1) circle(2pt) node(G) {} +(0,0) -- +(60:1) circle (2pt) node(H) {} +(0,0) -- ++(-5:1) circle (2pt) node(I) {} -- +(30:1) circle(2pt) node(J) {} +(0,0) -- +(-30:1) circle (2pt) node(K) {};
\grapeloop[180]{A};
\grapeloop[0]{C}; \grapeloop[90]{C}; \grapeloop[180]{C};
\grapeloop[-90]{D};
\grapeloop[-90]{F}; 
\grapeloop[120]{G};
\grapeloop[60]{H};
\grapeloop[-30]{K};
\grapeloop[30]{J};
\grapeloop[-90]{E}; \grapeloop[90]{E};
\grapeloop[40]{B};
\grapeloop[-105]{I};
\end{tikzpicture}\qquad
\mathsf{T}_\Gamma=\begin{tikzpicture}[baseline=-.5ex, scale=0.8]
\draw[thick,fill] (0,0) circle (2pt) node (A) {} -- ++(5:1.4) circle(2pt) node(B) {} -- +(0,1.4) circle (2pt) node (C) {}+(0,0) -- +(0,-1.4) circle (2pt) node(D) {} +(0,0) -- ++(-10:1.4) circle(2pt) node(E) {} -- ++(10:1.4) circle(2pt) node(F) {} -- +(120:1) circle(2pt) node(G) {} +(0,0) -- +(60:1) circle (2pt) node(H) {} +(0,0) -- ++(-5:1) circle (2pt) node(I) {} -- +(30:1) circle(2pt) node(J) {} +(0,0) -- +(-30:1) circle (2pt) node(K) {};
\end{tikzpicture}
\]
\caption{Examples of a bunch of graphs and its stem}
\label{Figure:Examples}
\end{figure}

A bunch of grapes $\Gamma$ comprises two key elements, a tree $\mathsf{T}_{\Gamma}$ (referred to as the \emph{stem}) and a function $\ell_\Gamma: \mathcal{V}(\mathsf{T}_\Gamma) \to \mathbb{Z}{\geq 0}$ defined on the vertex set of $\mathsf{T}_\Gamma$. The function $\ell_\Gamma$ specifies the number of induced cycles (called \emph{grapes}) attached to each vertex $v \in \mathcal{V}(\mathsf{T}_\Gamma)$.
The set of all bunches of grapes is denoted by $\grapegraph$.
We also define subclasses of $\grapegraph$ 
\begin{align*}
\grapegraph^{\mathsf{large}} &= \{\Gamma\in\grapegraph\mid \exists v,w\in \mathcal{V}(\mathsf{T}_\Gamma),\ (v\neq w)\land (\ell_\Gamma(v)>0)\land (\ell_\Gamma(w)>0)\},\\
\grapegraph^{\mathsf{large}}_{\mathsf{normal}} &= \{\Gamma\in\grapegraph^{\mathsf{large}}\mid \forall v\in \mathcal{V}(\mathsf{T}_\Gamma),\ (\val_{\mathsf{T}_\Gamma}(v)\le2 \Rightarrow \ell_\Gamma(v)\ge 1)\},\\
\grapegraph^{\mathsf{large}}_{\mathsf{rich}} &= \{\Gamma\in\grapegraph^{\mathsf{large}}_{\mathsf{normal}}\mid \forall v\in \mathcal{V}(\mathsf{T}_\Gamma),\ \ell_\Gamma(v)\ge 1\}.
\end{align*}

It is known and checked easily that for any bunch of grapes $\Gamma$ which is non-large, the $2$-braid group over $\Gamma$ is free by Theorem~\ref{scrg} \cite[Theorem~4.8]{KP12}, and its quasi-isometry class can be determined easily.

The first main result of this paper is an algorithmic classification of quasi-isometry classes of $2$-braid groups over bunches of grapes.

\begin{maintheorem}[Theorem~\ref{thm:existenceofAlgorithm}]\label{Mainthm:Algorithm}
Let $\Gamma$ and $\Gamma'$ be bunches of grapes.
Then there exists an algorithm to determine whether $\mathbb{B}_2(\Gamma)$ and $\mathbb{B}_2(\Gamma')$ are quasi-isometric.
\end{maintheorem}

According to \cite{Abrams00}, if a graph $\Gamma$ is suitably subdivided\footnote{For instance, if $n=2$, then any simple graph is suitably subdivided in this sense.}, then the $n$-braid group $\mathbb{B}_n(\Gamma)$ is the fundamental group of the unordered discrete configuration space $UD_n(\Gamma)$ of $n$ points in $\Gamma$, which turns out to be special cube complexes.

Following the preceding paragraph, Main Theorem~\ref{Mainthm:Algorithm} is rephrased as follows: for $\Gamma_1,\Gamma_2\in\grapegraph$, there is an algorithm to determine whether there exists a quasi-isometry between the universal covers $\bar{UD_2(\Gamma_1)}$ and $\bar{UD_2(\Gamma_2)}$ of $UD_2(\Gamma_1)$ and $UD_2(\Gamma_2)$. 
The proof will be done following two common schemes in geometric group theory:
\begin{itemize}
\item In order to show that $\bar{UD_2(\Gamma_1)}$ and $\bar{UD_2(\Gamma_2)}$ are not quasi-isometric, use a quasi-isometry invariant which distinguishes them up to quasi-isometry.
\item If the quasi-isometry invariant does not distinguish $\bar{UD_2(\Gamma_1)}$ and $\bar{UD_2(\Gamma_2)}$ up to quasi-isometry, construct a purported quasi-isometry between them and show that it is indeed a quasi-isometry.
\end{itemize}

\subsubsection{Intersection complexes}
The main quasi-isometry invariant we use is the \emph{intersection complex}, which will be briefly introduced below.

Let $Y$ be a compact special cube complex of dimension $2$ (a \emph{square complex}) and $\bar{Y}$ its universal cover.
A \emph{standard product subcomplex} of $Y$ is defined as a local isometry $\iota:\mathsf{P}_1\times\mathsf{P}_2\to Y$ for two finite graphs $\mathsf{P}_1$ and $\mathsf{P}_2$ without leaves such that the restriction to each factor is injective.
A \emph{standard product subcomplex} of $\bar Y$ is defined as an elevation $\bar{\iota}:\bar{\mathsf{P}_1}\times \bar{\mathsf{P}_2}\to \bar{Y}$ of $\iota$ such that $(\text{covering map})\circ\bar{\iota}=\iota\circ(\text{covering map})$. 
For the standard product subcomplex of either $Y$ or $\bar Y$, $\mathsf{P}_1\times\mathsf{P}_2$ is called its \emph{base}. 
See Definition~\ref{SPS} for the precise definition.

Let $X$ be either $Y$ or $\bar{Y}$.
A standard product subcomplex of $X$ which is maximal under the inclusion relation is called a \emph{maximal product subcomplex}. 
Then we can define the \emph{intersection complex} $\cI(X)$, which encodes the coarse intersection pattern of maximal product subcomplexes of $X$, as follows: 
\begin{itemize}
\item The vertex set of $\cI(X)$ consists of maximal product subcomplexes in $X$, and a set of $(k+1)$ vertices spans $k$-simplices whenever the $(k+1)$ maximal product subcomplexes share standard product subcomplexes. 
\item To each simplex $\triangle$ of $\cI(X)$, the base of the standard product subcomplex corresponding to $\triangle$ is assigned such that if $\triangle$ is a face of $\triangle'$, then the label of $\triangle'$ is factor-wisely contained in that of $\triangle$.
\end{itemize}
A \emph{(semi-)morphism} between intersection complexes is defined as a combinatorial map which preserves this kind of certain inclusion relation between labels (see Definition~\ref{Def:Morphism}).

In particular, $\cI(Y)$ will be called the \emph{reduced intersection complex} of $Y$ and denoted by $\cRI(Y)$. Indeed, there is a group action of $\pi_1(Y)$ on $\cI(\bar{Y})$ such that this action induces a (canonical) quotient morphism $\rho:\cI(\bar{Y})\to\cRI(Y)$ (Theorem~\ref{Thm:TPBCM}).

Then the second author proved the quasi-isometry invariance of $\cI(\bar Y)$ in \cite[Theorem~C]{Oh22}, which is also stated in Theorem~\ref{theorem:IsobetInt}.

\subsubsection{Operations on bunches of grapes}\label{Section:IntroOperations}
We introduce operations on bunches of grapes, called \emph{pruning empty twigs and smoothing twigs}, \emph{picking over-grown grapes}, and \emph{picking over-grown substems} in Definitions~\ref{Def:Elimination}, \ref{definition:picking}, and \ref{Def:Quasi-foldingonSuperrich}. 
These operations naturally induce an isomorphism between the intersection complexes of the $2$-braid groups and thus we follow the second scheme mentioned above to show that the quasi-isometry type of the $2$-braid groups is preserved. See Theorems~\ref{theorem:2-free factor}, \ref{theorem:loop reducing}, and \ref{theorem:quasi folding}.  

Moreover, these operations reduce bunches of grapes to smaller classes such as from $\grapegraph^{\mathsf{large}}$ to $\grapegraph^{\mathsf{large}}_{\mathsf{normal}}$ and to $\grapegraph^{\mathsf{large}}_{\mathsf{rich}}$, and produce eventually a unique bunch of grapes, called the \emph{quasi-minimal representative}. 
\begin{maintheorem}[Theorems~\ref{theorem:unique minimal representative} and \ref{thm:QIminimal}, Algorithm~\ref{alg:quasi-minimal}]\label{maintheorem:existence of quasi-minimal}
Let $\Gamma\in\grapegraph^{\mathsf{large}}$. Then there is a unique quasi-minimal representative $\Gamma_{\mathsf{min}}$ up to isometry, and $\mathbb{B}_2(\Gamma)$ is quasi-isometric to $\mathbb{B}_2(\Gamma_{\mathsf{min}})$.

Moreover, there is a finite-time algorithm producing $\Gamma_{\mathsf{min}}$ from $\Gamma$.
\end{maintheorem}

\subsubsection{The completeness of the invariant}
The intersection complex is not a complete invariant. Specifically, the converse of Theorem~\ref{theorem:IsobetInt} does not hold even when limited to 2-dimensional RAAGs (\cite{Mar20, Hua(c)}); if we further restrict to 2-dimensional RAAGs with finite outer automorphism groups, then the converse does hold \cite[Theorem~5.6]{Oh22}.

By Main Theorem~\ref{maintheorem:existence of quasi-minimal}, we can consider a subclass of $\grapegraph^{\mathsf{large}}_{\mathsf{rich}}$ consisting of quasi-minimal bunches of grapes
\[
\grapegraph^{\mathsf{large}}_{\mathsf{min}}=\{
\Gamma\in\grapegraph^{\mathsf{large}}_{\mathsf{rich}} \mid \Gamma\text{ is quasi-minimal}
\}.
\]
Then the intersection complex $\cI(\bar{UD_2(-)})$ on $\grapegraph^{\mathsf{large}}_{\mathsf{min}}$ is indeed a complete invariant.

\begin{maintheorem}[Proposition~\ref{Prop:FreeFactor} and Theorem~\ref{theorem:Quasi-minimalRepresentative}]\label{MainThm:IsomorphisminducesQI}
Let $\Gamma,\Gamma'\in\grapegraph^{\mathsf{large}}_{\mathsf{min}}$ be quasi-minimal.
If there is an isomorphism between $\cI(\bar{UD_2(\Gamma)})$ and $\cI(\bar{UD_2(\Gamma')})$, then $\Gamma$ and $\Gamma'$ are isometric.
\end{maintheorem}

\subsection{Applications}
There are two applications of our study of $2$-braid groups over bunches of grapes to other groups.

\subsubsection{Graph $2$-braid groups and RAAGs}
The first one is to strengthen the result in \cite{Oh22} about the determination of which graph $2$-braid groups are quasi-isometric to RAAGs or not.
Proposition~5.16 and Corollary~5.18 in \cite{Oh22} say that there are infinitely many graphs with circumference one whose $2$-braid groups are/are not quasi-isometric to RAAGs.
By Main Theorem~\ref{MainThm:IsomorphisminducesQI}, we can enlarge both classes of graph $2$-braid groups.

\begin{maintheorem}[Theorem~\ref{theorem:bunches of grapes related to RAAG}]\label{Mainthm:QItoRAAG}
There are infinitely many graphs with circumference one whose $2$-braid groups are quasi-isometric to RAAGs; this class properly contains the class in \cite[Proposition 5.16]{Oh22}.

There are also infinitely many graphs with circumference one whose $2$-braid groups are not quasi-isometric to RAAGs; this class properly contains the class in \cite[Corollary 5.18]{Oh22}.
\end{maintheorem}

\subsubsection{Tree $4$-braid groups}
Recall that the isomorphic classification of tree $4$-braid groups are completely solved by \cite{Sabalka2009}. That is, two tree $4$-braid groups are isomorphic if and only if the underlying trees are homeomorphic.

The other application of our result is to provide a complete classification of tree $4$-braid groups up to quasi-isometry.

For a tree $\Lambda$, we construct a bunch of grapes $\Gamma(\Lambda)$ \emph{grown from $\Lambda$}, whose stem is given as $\Lambda$ (Definition~\ref{definition:bunch of grapes grown from tree}). Then we have the following result.

\begin{maintheorem}[Theorem~\ref{Thm:B_4quasi-isometrictoB_2} and Corollary~\ref{corollary:algorithm for tree 4-braid groups}]\label{Mainthm:TBG}
Let $\Lambda$ be a tree and $\Gamma=\Gamma(\Lambda)$ be a bunch of grapes grown from $\Lambda$.
Then $\mathbb{B}_4(\Lambda)$ and $\mathbb{B}_2(\Gamma)$ are quasi-isometric.

Moreover, given two trees $\Lambda_1,\Lambda_2$, there exists an algorithm to determine whether $\mathbb{B}_4(\Lambda_1)$ and $\mathbb{B}_4(\Lambda_2)$ are quasi-isometric.
\end{maintheorem}

\subsection{Remarks and questions}
Here, we remark the following:
\begin{enumerate}[wide]
\item For a graph $\Gamma$, let $UP_2(\Gamma)$ be a subcomplex of $UD_2(\Gamma)$ which is the union of all maximal product subcomplexes of $UD_2(\Gamma)$. Our results are indeed for $UP_2(\Gamma)$ (or $\bar{UP_2(\Gamma)}$) as $\bar{UP_2(\Gamma)}$ shares the (almost) same quasi-isometry type of $\bar{UD_2(\Gamma)}$ if $\Gamma$ belongs to $\grapegraph^{\mathsf{large}}_{\mathsf{normal}}$.

\begin{proposition}[Proposition~\ref{Prop:FreeFactor}]
Let $\Gamma\in\grapegraph^{\mathsf{large}}_{\mathsf{normal}}$. Then the subcomplex $UP_2(\Gamma)$ is locally convex in $UD_2(\Gamma)$.
Moreover, the $2$-braid group $\mathbb{B}_2(\Gamma)$ is isomorphic to $\pi_1(UP_2(\Gamma))*\mathbb{F}_n$ for some $n\ge 2$.
\end{proposition}

This proposition was one of the main reasons why $2$-braid groups over bunches of grapes are completely classified up to quasi-isometry in this paper, using the quasi-isometry invariant `intersection complexes' which captures the coarse pattern of maximal standard product subcomplexes.

As $UP_2(\Gamma)$ is always a special square complex itself for any graph $\Gamma$ (Lemma~\ref{lem:SpecialSubcomplexes}), one may ask the local convexity of $UP_2(\Gamma)$ in $UD_2(\Gamma)$.

\begin{question}\label{Qn:LocalConvexity}
For any graph $\Gamma$, is $UP_2(\Gamma)$ always locally convex in $UD_2(\Gamma)$, possibly after subdividing or smoothing $\Gamma$ as described in Remark~\ref{remark:minimal simple graphs}? Alternatively, is $\pi_1(UP_2(\Gamma))$ a subgroup of $\mathbb{B}_2(\Gamma)$?
\end{question}

At the very least, if we can identify a class of graphs that provides an affirmative answer to Question~\ref{Qn:LocalConvexity} and extends $\grapegraph^{\mathsf{large}}_{\mathsf{normal}}$, then we may be able to broaden our results to encompass this class of graphs.

\item Combining our results, we obtain the following equivalences for $\Gamma,\Lambda\in\grapegraph^{\mathsf{large}}$ (Corollary~\ref{corollary:quasi-isometry invariant}): 
\begin{enumerate}[label=(\alph*), leftmargin=4em]
\item\label{item:B_2QI} $\mathbb{B}_2(\Gamma)$ and $\mathbb{B}_2(\Lambda)$ are quasi-isometric;
\item\label{item:QMR} The quasi-minimal representatives $\Gamma_{\mathsf{min}}$ and $\Lambda_{\mathsf{min}}$ of $\Gamma$ and $\Lambda$ are isometric;
\item\label{item:Iisom} $\cI(\bar{UP_2(\Gamma)})$ and $\cI(\bar{UP_2(\Lambda)})$ are isomorphic.
\end{enumerate}

The implication [\ref{item:B_2QI}$\Rightarrow$\ref{item:Iisom}] is proven by a combination of Theorem~\ref{theorem:IsobetInt}, Proposition~\ref{Prop:FreeFactor} and results in \cite{PW02} (Theorem~\ref{PW}).

The idea behind the proof of the implication [\ref{item:QMR}$\Rightarrow$\ref{item:B_2QI}] is as follows. For detailed proofs, see Sections~\ref{Subsection:PickingGrapes} and~\ref{Subsection:Quasi-folding}. 
\begingroup
\begin{itemize}[wide, labelindent=0pt]
\item[\textbf{Step 1.}] Suppose that $\Gamma'$ is obtained from $\Gamma$ by picking over-grown grapes or pruning over-grown substems. Then we have a surjective morphism $\Phi:\cRI(UP_2(\Gamma))\to\cRI(UP_2(\Gamma'))$, naturally induced from the operation. 
%which induces an isomorphism $\bar\Phi:\cI(\bar{UP_2(\Gamma)})\to\cI(\bar{UP_2(\Gamma')})$ such that $\rho'\circ\bar\Phi=\Phi\circ\rho$ where $\rho$ and $\rho'$ are canonical quotient morphisms.
\item[\textbf{Step 2.}] 
For a vertex $\bfv\in\cRI(UP_2(\Gamma))$ ($\bfv'\in\cRI(UP_2(\Gamma'))$, resp.), let $\mathbf{M}_{\bfv}$ ($\mathbf{M}'_{\bfv'}$, resp.) be the domain of the maximal product subcomplex corresponding to $\bfv$ ($\bfv'$, resp.).
As the universal cover $\bar{\mathbf{M}_{\bfv}}$ is quasi-isometric to the universal cover $\bar{\mathbf{M}'_{\Phi(\bfv)}}$ by the construction of $\Phi$, we have finitely many quasi-isometries $\bar\phi_\bfv:\bar{\mathbf{M}_{\bfv}}\to\bar{\mathbf{M}'_{\Phi(\bfv)}}$.
%\item[\textbf{Step 3.}] 
Based on the fact that $\cI(\bar{UP_2(\Gamma)})$ and $\cI(\bar{UP_2(\Gamma')})$ are simply connected and isomorphic (Theorem~\ref{theorem:structureofI} and Lemma~\ref{Lem:IsometricRI}), we can inductively glue (copies of) those quasi-isometries to construct a quasi-isometry between $\bar{UP_2(\Gamma)}$ and $\bar{UP_2(\Gamma')}$.
\end{itemize}
\endgroup
To the best of our knowledge, this is the first attempt to construct a quasi-isometry from an isomorphism between complexes of groups, although the case of trees of groups is examined in \cite{CM17}.

The proof of the implication [\ref{item:Iisom}$\Rightarrow$\ref{item:QMR}] illustrates why it is worthwhile to consider the labels of simplices in intersection complexes, and its underlying idea is as follows.
For each simplex $\triangle$ of $\cRI(UP_2(\Gamma))$, there is a canonical way to give an order on the vertex set of $\triangle$ (see Definition~\ref{def:OrderOnSimplex}).
Moreover, any isomorphism between $\cI(\bar{UP_2(\Gamma)})\to\cI(\bar{UP_2(\Lambda)})$ preserves the order of vertices in each simplex (see Lemma~\ref{lem:OrderOnSimplex}).
By looking at (the complement of) specific vertices preserved by the isomorphism, we deduce that there are specific sub-bunches of grapes $\Gamma_1\subset\Gamma$ and $\Lambda_1\subset\Lambda$ such that $\cI(\bar{UP_2(\Gamma_1)})$ and $\cI(\bar{UP_2(\Lambda_1)})$ are isomorphic.
By carrying out this process until it terminates and then reversing it,  we can demonstrate that $\Gamma$ and $\Lambda$ are isometric. For the details, see Section~\ref{section:QIBetweenConfSpaces}.

\item 
In addition to the above proposition, our results, especially for the proof of the implication [\ref{item:QMR}]$\Rightarrow$[\ref{item:B_2QI}] above, rely on the following fact:
\begin{lemma}[Lemma~\ref{Lem:RIconnected} (modified)]\label{lemma:intersection of maximals}
Let $\{M_1,\dots,M_m\}$ be a finite collection of maximal product subcomplexes of given square complex.
Then $\bigcap_{i=1}^m M_i$ is either empty or a standard product subcomplex.
\end{lemma}

It is worth noticing that this is not the first paper mentioning this fact:
For a tree $\mathsf{T}$ of diameter $\ge 3$, the associated RAAG $A_\mathsf{T}$ is the fundamental group of a special square complex $\widetilde{S}_\mathsf{T}$ (sometimes, called an \emph{extended Salvetti complex}), which is homeomorphic to the Salvetti complex associated to $\mathsf{T}$.
Then $\widetilde{S}_\mathsf{T}$ is not only equal to the union of all its maximal product subcomplexes but also satisfies Lemma~\ref{lemma:intersection of maximals}.

One main result in \cite{BN} is that if $\mathsf{T}_1$ and $\mathsf{T}_2$ are trees of diameter $\ge 3$, then $A_{\mathsf{T}_1}$ and $A_{\mathsf{T}_2}$ are quasi-isometric.
On the way to prove this result, Behrstock and Neumann showed that if $\cI(\widetilde{X}_{\mathsf{T}_1})$ and $\cI(\widetilde{X}_{\mathsf{T}_2})$ are isomorphic, then $A_{\mathsf{T}_1}$ and $A_{\mathsf{T}_2}$ are quasi-isometric, where $\widetilde{X}_{\mathsf{T}_i}$ is the universal cover of $\widetilde{S}_{\mathsf{T}_i}$ for $i=1,2$.

\item Since each of $\widetilde{S}_{\mathsf{T}}$ and $UP_2(\Gamma)$ is equal to the union of all its maximal product subcomplexes and satisfies Lemma~\ref{lemma:intersection of maximals}, it is natural to ask the following question related to the completeness of the quasi-isometry invariant---the intersection complex:
\begin{question}
For $i=1,2$, let $Y_i$ be a compact special square complex such that its maximal product subcomplexes cover $Y_i$ itself and it satisfies Lemma~\ref{lemma:intersection of maximals}. Then is $\pi_1(Y_1)$ quasi-isometric to $\pi_1(Y_2)$ if $\cI(\bar{Y_1})$ is isomorphic to $\cI(\bar{Y_2})$?
\end{question}

\item Generalizing Main Theorem~\ref{Mainthm:QItoRAAG}, one can raise the following two questions:
\begin{question}\label{Qn:QItoRAAG}
Given a bunch of grapes $\Gamma$, does there exist an algorithm to determine whether $\mathbb{B}_2(\Gamma)$ is quasi-isometric to a RAAG or not?
\end{question}

\begin{question}
Given a bunch of grapes $\Gamma$, is $\pi_1({UP_2(\Gamma)})$ quasi-isometric to $A_\Lambda$ for some $\Lambda$ if and only if $\cI(\bar{UP_2(\Gamma)})$ is isomorphic to $\cI(X_\Lambda)$?
\end{question}

The former seems to be solved by looking at the structure of $\cI(\bar{UP_2(\Gamma)})$ if the latter is positive.

\item Before we finish the introduction, we leave two more natural questions considering larger classes of graphs than bunches of grapes.

\begin{question}\label{Qn:QIrigidity}
Let $\Gamma_1$ be a bunch of grapes and $\Gamma_2$ be a graph of circumference at least two. If $\mathbb{B}_2(\Gamma_1)$ and $\mathbb{B}_2(\Gamma_2)$ are quasi-isometric, can we obtain $\Gamma_1$ from $\Gamma_2$ using enlarged graph operations?
\end{question}

If the answer is affirmative, then we may say that the class of $2$-braid groups over graphs with circumference one will be quasi-isometrically rigid up to enlarged graph operations. Moreover, it will be a starting point to find a set of graph operations which preserve the quasi-isometric type of graph $2$-braid groups.

If the answer is negative, then those graph $2$-braid groups will be served as an ingredient to sharpen the quasi-isometry invariant `intersection complexes'.

Related to Question~\ref{Qn:QIrigidity}, we may consider a class of graphs, called \emph{cacti}, which are graphs in which any two cycles meet at most one vertex. The class of cacti is a strictly larger than the class of bunches of grapes.

\begin{question}\label{Qn:cacti}
Let $\Gamma_1$ and $\Gamma_2$ be cacti. Does there exist an algorithm to know whether $\mathbb{B}_2(\Gamma_1)$ and $\mathbb{B}_2(\Gamma_2)$ are quasi-isometric or not?
\end{question}
If we can answer Question~\ref{Qn:cacti}, then we may change the graph in Question~\ref{Qn:QItoRAAG} to any cactus.
\end{enumerate}

\subsection{Organization of the paper}
In Section~\ref{section:prelim}, we introduce the notations and definitions that will be used throughout the paper. Section~\ref{section:WSSCandIC} covers (weakly) special square complexes and their standard product subcomplexes, followed by a review of the theory of (reduced) intersection complexes in Section~\ref{subsection:intersection complexes}. In Section~\ref{section:PW}, we revisit results from \cite{PW02} concerning the relationship between free products of spaces and quasi-isometries.

Section~\ref{section:graph 2-braid groups} begins with background information on the fundamental groups of the discrete $n$-configuration spaces of graphs. In Sections~\ref{section:properties} and~\ref{Section:GOGdecomposition}, we focus specifically on graph $2$-braid groups and certain subcomplexes of the discrete $2$-configuration space of a graph. This discussion will clarify, directly or indirectly, why graphs with circumference at most one are of particular interest when studying graph $2$-braid groups.

In Section~\ref{section:grapes}, after introducing the class of bunches of grapes (which are obtained from graphs with circumference $\leq 1$ by subdividing loops), we study the structure of the (reduced) intersection complexes of (the universal cover of) the discrete $2$-configuration spaces of bunches of grapes.

In Section~\ref{section:Operations}, we study two operations on bunches of grapes which do not change quasi-isometric type of their $2$-braid groups, and in particular, prove the quasi-isometric invariance in Main Theorem~\ref{maintheorem:existence of quasi-minimal}.
Using these operations, in Section~\ref{section:proofofThm}, we finish the proof of Main Theorem~\ref{Mainthm:Algorithm} by proving the existence of quasi-minimal representative in Main Theorem~\ref{maintheorem:existence of quasi-minimal} and Main Theorem~\ref{MainThm:IsomorphisminducesQI}. 

Section~\ref{section:applications} is devoted to applications of the results obtained in the previous sections. In particular, we prove Main Theorems~\ref{Mainthm:QItoRAAG} and~\ref{Mainthm:TBG}.

\subsection*{Acknowledgement}
The first author was supported by Samsung Science and Technology Foundation under Project Number SSTF-BA2022-03.
The second author was supported by the Basque Government grant IT1483-22.

\section{Preliminaries}\label{section:prelim}

In this paper, we only deal with two types of finite dimensional polyhedral complexes; a \emph{cube complex} whose cells are all cubes (\emph{square complex} in dimension 2), and an \emph{almost simplicial complex}
each of whose cells is an isometrically embedded simplex (\emph{simplicial complex} if there are no multi-simplices).
Sometimes, a cube complex can be seen as a metric space equipped with the \emph{length metric} by assuming that each cell is a unit cube in the Euclidean space.
For an almost simplicial complex $Z$, the distance between two vertices $v,v'$ is defined as the smallest integer $k$ such that there is a sequence of vertices $v=v_0,v_1,\dots,v_k=v'$ where $v_i$ and $v_{i+1}$ are adjacent in $Z$.

Each 0-cell (1-cell, resp.) of a polyhedral complex $X$ is called a \emph{vertex} (\emph{edge}, resp.) as usual, and the vertex set (edge set, resp.) of $X$ is denoted by $\mathcal{V}(X)$ ($\mathcal{E}(X)$, resp.).
For a subset $V\subset X$, a subcomplex $Y$ of $X$ is \emph{induced by $V$} if any cell in $X$ whose vertices are contained in $V$ is contained in $Y$; in this case, $Y$ is an \emph{induced} subcomplex.
For an almost simplicial complex $Z$ and its vertex $v$, we denote by $\mathcal{N}_k(v)$ the subcomplex of $Z$ induced by the set of vertices whose distance from $v$ is $\le k$.

Any map $\phi:X\to Y$ between polyhedral complexes is assumed to be \emph{combinatorial}, i.e., the restriction of $\phi$ to the interior of each cell of $X$ is a homeomorphism onto the interior of a cell of $Y$.
If $\phi$ is bijective, then we say that $\phi$ is an \emph{isometry} and $X$ and $Y$ are \emph{isometric}, denoted by $X\cong Y$.

Unless mentioned otherwise, any groups are assumed to be non-trivial and finitely generated. They are mostly denoted by blackboard bold letters including the following cases: a free group is denoted by $\mathbb{F}$, or $\mathbb{F}_n$ when the rank $n$ is specified; $\mathbb{F}_1$ is usually denoted by $\mathbb{Z}$.

\begin{convention}
For a polyhedral complex $Y$, the universal covering map $\bar Y \to Y$ will be denoted by $p$. We may write $p_Y$ to emphasize the polyhedral complex $Y$.
\end{convention}

\subsection{Weakly special cube complexes}\label{section:WSSCandIC}
Let $Y$ be a cube complex. For a vertex $v\in Y$, its \emph{link} $\Lk(v,Y)$ in $Y$ is defined as an almost simplicial complex whose vertices are half edges containing $v$ such that each $k$-simplex $\triangle$ corresponds to a $(k+1)$-cube containing the half edges corresponding to the vertices of $\triangle$.
Then $Y$ is said to be \emph{non-positively curved (NPC)} if $\link(y,Y)$ for each vertex $y\in Y$ is a \emph{flag complex}, a simplicial complex satisfying that pairwise adjacent $(k+1)$-vertices spans a unique $k$-simplex. Additionally, if $Y$ is simply connected, then $Y$ is said to be $\CAT(0)$.

\begin{remark}
When $Y$ is equipped with the length metric, it is locally $\CAT(0)$ as seen in \cite{Grom} (\cite{Lea} for infinite dimensional case), and is $\CAT(0)$ by the Cartan--Hadamard theorem if it is simply connected.
Moreover, the notions of (local) isometry and (locally) isometric embedding above coincide with the usual geometric ones for the length metric on $Y$; see \cite{CW,HW08,Wis12}.
\end{remark}

\begin{Ex}[Graphs]\label{Ex:Graphs}
A \emph{graph} $\Gamma$ is a cube complex of dimension at most $1$; a \emph{tree} is a connected and simply connected graph.
Since the link of any vertex in $\Gamma$ is discrete, $\Gamma$ is NPC and especially $\CAT(0)$ if it is a tree.

If a graph $\Gamma$ is also a simplicial complex, then it is said to be \emph{simple}.
In this case, the link of a vertex $v\in\mathcal{V}(\Gamma)$ in $\Gamma$ can be identified with a discrete subgraph of $\Gamma$, denoted by $\lk_\Gamma(v)$, which is the set of vertices adjacent to $v$.
Then the \emph{star} of $v$ in $\Gamma$ is defined as the subgraph induced by $\lk_\Gamma(v)\cup\{v\}$, denoted by $\st_\Gamma(v)$.
\end{Ex}

\begin{remark}\label{remark:minimal simple graphs}
For any graph $\Gamma$, there are two operations which generate new graphs which are homeomorphic to $\Gamma$: adding bivalent vertices (\emph{subdivision}) and removing bivalent vertices (\emph{smoothing}).
Using these two operations, $\Gamma$ can be promoted to a simple graph; there is a unique one up to isometry, which has the minimal number of vertices. Such a simple graph is called the \emph{minimal simplicial representative} of $\Gamma$ in \cite{AM21}.
\end{remark}

\begin{Ex}[Salvetti complex]
A \emph{right-angled Artin group} (RAAG) is a group admitting a presentation whose relators are commutators of generators. 
The \emph{Salvetti complex} associated to a RAAG is a typical example of NPC cube complex consisting of one vertex, several circles, and various dimensional tori.
For more details about RAAGs and Salvetti complexes, we refer to \cite{CH}.
\end{Ex}

\begin{Ex}[Discrete $2$-configuration spaces of graphs]
For a graph $\Gamma$, the \emph{ordered} and \emph{unordered discrete $2$-configuration spaces} $D_2(\Gamma)$ and $UD_2(\Gamma)$ of $\Gamma$ consist of ordered and unordered pairs of disjoint closed cells, respectively. By Abrams \cite{Abrams00}, it is known that both $D_2(\Gamma)$ and $UD_2(\Gamma)$ are NPC cube complexes; see Theorem~\ref{thm:GBGSpecial}.
\end{Ex}

\begin{assumption}\label{assumption:cube}
Unless mentioned otherwise, every cube complex is assumed to be NPC and connected though its subcomplex is only assumed to be connected. 
\end{assumption}

For two cube complexes $Y$ and $Y'$, a (combinatorial) map $\phi:Y\to Y'$ is called an \emph{immersion} if the induced map $\link(y,Y)\rightarrow \link(\phi(y),Y')$ on links for each vertex $y\in Y$ is combinatorial and injective, and a \emph{local isometry} if the image of every induced link map is additionally an induced subcomplex of the range.
If a local isometry $\phi$ is injective, then it is called a \textit{locally isometric embedding}, and the image $\phi(Y)$ is said to be \emph{locally convex} in $Y'$.
Lastly, if $Y'$ is a $\CAT(0)$ cube complex, instead of the terms `locally isometric embedding' and `locally convex', we use the terms `\emph{isometric embedding}' and `convex', respectively.

\begin{lemma}\label{Lem:ImmersionImpliesLocalisometry}
Let $Y$ be an $n$-dimensional cube complex and let $\Gamma_1,\dots,\Gamma_n$ be 1-dimensional graphs.
Then any immersion $\iota:\Gamma_1\times\dots\times\Gamma_n\to Y$ is a local isometry.
\end{lemma}
\begin{proof}
This lemma is easily derived from the fact that the link of any vertex of $\Gamma_1\times\dots\times\Gamma_n$ is the join of $n$ discrete sets.
\end{proof}

A local isometry $\phi:Y\to Y'$ induces the following two maps:
\begin{itemize}
\item an injective homomorphism $\phi_*:\pi_1(Y,y)\rightarrow\pi_1(Y',\phi(y))$ on fundamental groups, and
\item an isometric embedding $\bar\phi:\bar Y\to \bar{Y'}$, called an \emph{elevation} of $\phi$, on universal covers with $\phi\circ p_Y=p_{Y'}\circ\bar\phi$. 
The image of $\bar\phi$ will be called a \emph{copy} of $\bar{Y}$ or a \emph{lift} of $Y$ in $\bar{Y'}$.
\end{itemize}
In particular, by the Švarc--Milnor lemma, $\phi_*(\pi_1(Y,y))$ is a quasi-isometrically embedded subgroup (or an \emph{undistorted subgroup}) of $\pi_1(Y',\phi(y))$.
For more details about those two maps, we refer to \cite[Chapter II.4]{BH}.

\begin{Ex}[Flats in $\CAT(0)$ cube complexes]
By considering a real line $\mathbb{R}$ as a $1$-dimensional cube complex, $\mathbb{R}^n$ can be thought of as an $n$-dimensional $\CAT(0)$ cube complex. 
For a $\CAT(0)$ cube complex $Y$, the image $F$ of an isometric embedding of $\mathbb{R}^n$ into $Y$ is called an $n$-dimensional \emph{flat} in $Y$.
If $Y$ is $n$-dimensional, then $F$ is called a \emph{top-dimensional} flat.

For instance, if there exists a local isometry $\iota:\mathsf{C}_{i_1}\times\dots\times\mathsf{C}_{i_n}\to Y$ where $Y$ is a cube complex and $\mathsf{C}_i$ is a cycle of length $i$, then the image of an elevation of $\iota$ is an $n$-dimensional flat in the universal cover of $Y$.
\end{Ex}

Two edges $e_1$ and $e_2$ of $Y$ are said to be \emph{parallel} if there is an immersion $e\times [0,n]\to Y$ for some positive integer $n$ such that $e_1$ and $e_2$ are the images of $e\times \{0\}$ and $e\times\{n\}$, respectively, where the interval $[0,n]$ is thought of as a cube complex. 
The \emph{hyperplane} $H$ dual to an edge $e$ of $Y$ is the set of all edges parallel to $e$ and the \emph{carrier} of $H$ is the subcomplex of $Y$ induced by $H$.

Haglund--Wise \cite{HW08} defined a special cube complex as an (NPC) cube complex which does not admit four pathological hyperplanes, and showed that a compact cube complex $Y$ is special if and only if there is a local isometry from $Y$ to some Salvetti complex. 
In the study of the large scale geometry of $\CAT(0)$ cube complexes, Huang \cite{Hua(b)} claimed that discarding only two out of four is enough as below.

\begin{definition}[\cite{HW08, Hua(b)} (Weakly) special cube complex]\label{Def:SpecialCubeComplex}
An (NPC) cube complex $Y$ is \emph{weakly special} if there are no \emph{self-osculate} or \emph{self-intersect} hyperplanes, and is \emph{special} if additionally there are no \emph{one-sided} hyperplanes and pairs of \emph{inter-osculate} hyperplanes.
\end{definition}

\begin{theorem}[\cite{Hua(b)}, Theorem~1.3]\label{FlattoFlat}
Let $Y$ and $Y'$ be two compact weakly special cube complexes with their universal covers $\bar{Y}$ and $\bar{Y'}$, respectively.
If there is a $(\lambda,\varepsilon)$-quasi-isometry $\phi:\bar{Y}\rightarrow \bar{Y'}$, then there exists a constant $C=C(\lambda,\varepsilon)>0$ such that for any top-dimensional flat $F\subset \bar{Y}$, there exists a unique top-dimensional flat $F'\subset \bar{Y'}$ such that $d_H(\phi(F),F')<C$, where $d_H$ denotes the Hausdorff distance.
\end{theorem}

The next lemma illustrates a main difference between two-dimensional NPC (or special) cube complexes and higher-dimensional ones; it may not hold in higher-dimensional ones.

\begin{lemma}\label{lem:SpecialSubcomplexes}
Let $Y$ be a square complex with its subcomplex $Y_0$.
If $Y$ is NPC, special, or weakly special, then $Y_0$ is NPC, special, or weakly special, respectively.
\end{lemma}
\begin{proof}
Suppose that $Y$ is NPC. For any vertex $y$ in $Y_0$, since $\Lk(y,Y)$ is a triangle-free graph without multi-edges, $\Lk(y,Y_0)$ is also a triangle-free graph without multi-edges, and thus $Y_0$ is NPC.

Suppose that $Y$ is special (weakly special, resp.). Since any hyperplane in $Y_0$ is a subset of some hyperplane in $Y$, the hyperplanes in $Y_0$ must satisfy the conditions for $Y_0$ to be special (weakly special, resp.). Combining with the previous paragraph, we conclude that $Y_0$ is special (weakly special, resp.). 
\end{proof}

\subsubsection{Standard product subcomplexes}\label{Sec:SPS}
When the dimension of a weakly special cube complex is two, we can easily define subcomplexes, admitting product structures, which can be used in the study of the large scale geometry of the universal cover of the complex. 
This perspective is written in \cite[Section 2.2]{Oh22} and here we revisit the definition and some properties of such subcomplexes with further explanations. Until the end of this subsection, we fix $Y$ be a compact weakly special square complex.

\begin{definition}[Product subcomplex]\label{Def:ProductSubcomplex}
Let $X$ be either $Y$ or $\bar Y$.
For two $1$-dimensional graphs ${\Gamma}_1$ and $\Gamma_2$, a local isometry $\iota:{\Gamma}_1 \times {\Gamma}_2\to X$ is a \emph{product structure} if $\iota(\Gamma_1\times\{v_2\})$ and $\iota(\{v_1\}\times\Gamma_2)$ are isometric to $\Gamma_1$ and $\Gamma_2$, respectively, for some vertices $v_i\in\Gamma_i$.
The image $K$ of $\iota$ is a \emph{product subcomplex} of $X$ and $\Gamma_1\times\Gamma_2$ is said to be the \emph{domain} of $K$.
\end{definition}

\begin{definition}\label{Def:Pullback and Pushforward}
Given a local isometry $\iota:\Delta_1\times\Delta_2\to Y$ for two finite graphs $\Delta_i$, a \emph{pull-back} $p^*\iota:\Gamma_1\times\Gamma_2\to \bar Y$ of $\iota$ is a local isometry such that
\[
\Gamma_1\cong \bar{\iota(\Delta_1\times\{v_2\})}\quad\text{ and }\quad
\Gamma_2\cong \bar{\iota(\{v_1\}\times\Delta_2)}
\]
for some vertices $v_i\in \Delta_i$ and the image of $p_Y\circ p^*\iota$ is equal to the image of $\iota$. The image of $p^*\iota$ is said to be a \emph{$p$-lift} of the image of $\iota$.

Given a (local) isometry $\bar\iota:\Gamma_1\times\Gamma_2\to \bar Y$ for two (possibly infinite) graphs $\Gamma_i$, a \emph{push-forward} $p_*\bar\iota:\Delta_1\times\Delta_2\to Y$ of $\bar\iota$ is a local isometry such that 
\[ \Delta_1\cong p_Y(\bar\iota(\Gamma_1\times\{\bar v_2\}))\quad\text{ and }\quad\Delta_2\cong p_Y(\bar\iota(\{\bar v_1\}\times\Gamma_2))\]
for some vertices $\bar v_i\in \Gamma_i$ and the image of $p_Y\circ \bar{\iota}$ is equal to the image of $p_*\bar\iota$.
\end{definition}

Note that the `$p$' in the term $p$-lift stands for `product' not the covering map and any $p$-lift is contained in a lift.

The following lemma implies that in the world of compact weakly special square complexes, push-forwards and pull-backs always exist and a pull-back of a local isometry gives us a largest product subcomplex of $\bar Y$ whose image under $p_Y$ is the image of the local isometry.

\begin{lemma}[\cite{Oh22}, Lemmas~2.8 and~2.9]\label{ProjofPS}
The following holds:
\begin{enumerate}
\item\label{push-forward} 
If a subcomplex $\bar{K}\subset\bar Y$ is the image of a local isometry $\bar\iota:\Gamma_1\times \Gamma_2\to\bar{Y}$ for some (possibly infinite) graphs $\Gamma_i$, then $p_Y(\bar K)$ admits a product structure which is a push-forward of $\bar\iota$.
\item\label{pull-back} If a subcomplex $K\subset Y$ is the image of a local isometry $\iota:\Delta_1\times\Delta_2\to Y$ for some finite graphs $\Delta_i$, then there exists a product subcomplex $\bar K\subset \bar Y$ with a product structure $\bar{\iota}':\Gamma'_1\times\Gamma'_2\to \bar Y$, which is a pull-back of $\iota$, such that $K$ admits a product structure which is a push-forward of $\bar{\iota}'$. 
Moreover, $\bar K$ is a largest product subcomplex whose image under $p_Y$ is $K$.
\item Definitions~\ref{Def:ProductSubcomplex} and~\ref{Def:Pullback and Pushforward} do not depend on the choice of $v_i\in\Gamma_i$'s or $\bar v_i\in\Gamma_i$'s.
\end{enumerate}
\end{lemma}

Using the fact that for a graph without leaves both its image under any local isometry and its universal cover have no leaves, we define `standard' product subcomplexes as follows.

\begin{definition}[Standard and maximal product subcomplexes]\label{SPS}
We define the following:
\begin{enumerate}
\item A \emph{standard product subcomplex} $K\subset Y$ is a product subcomplex with a product structure which is a push-forward $p_*\bar\iota$ of a product structure $\bar\iota:\Gamma_1\times \Gamma_2\to\bar{Y}$, where each $\Gamma_i$ is an infinite tree without leaves; such $p_*\bar\iota$ is called a \emph{standard product structure of $K$}.
\item A \emph{standard product subcomplex} $\bar K\subset Y$ is a product subcomplex with a product structure which is a pull-back $p^*\iota$ of a product structure $\iota:{\Delta}_1\times {\Delta}_2\to Y$, where each $\Delta_i$ is a graph without leaves; such $p^*\iota$ is called a \emph{standard product structure of $\bar K$}.
\item A \emph{maximal product subcomplex} of $Y$ or $\bar Y$ is a maximal subcomplex up to inclusion among standard product subcomplexes of $Y$ or $\bar Y$, respectively.
\end{enumerate}
\end{definition}

By Lemma~\ref{ProjofPS}, a standard product subcomplex of $Y$ can be defined as the image of a local isometry $\iota:{\Delta}_1\times {\Delta}_2\to Y$ for two graphs $\Delta_i$'s without leaves.

\begin{remark}\label{Remark:Standard}
The definition of standard product subcomplexes of $Y$ or $\bar Y$ is a generalization of the definition for $2$-dimensional Salvetti complexes \cite{BKS(a)} or the ordered discrete $2$-configuration spaces of a graph \cite{Fer12}.
\end{remark}

For a (standard) product subcomplex $K\subset Y$, let $\iota:\Delta_1\times\Delta_2\to Y$ be a product structure of $K$ and let $\bar K\subset \bar Y$ be a $p$-lift of $K$ which admits a product structure $\bar\iota:\bar{\Delta_1}\times\bar{\Delta_2}\to \bar Y$, an elevation of $\iota$.
Then the domains of $K$ and $\bar K$ are well-defined up to permuting factors since $\iota$ and $\bar\iota$ have a universal property in the following sense. Let $\iota':\Delta_1'\times\Delta_2'\to Y$ and $\bar\iota':\Gamma_1'\times\Gamma_2'\to \bar Y$ be local isometries whose images are $K$ and $\bar K$, respectively.
Then there exist the products 
\[
q_1\times q_2:\Delta_1'\times\Delta_2'\to \Delta_{1}\times \Delta_{2}\quad\text{ and }\quad
f_1\times f_2:\Gamma'_1\times\Gamma'_2\to \bar{\Delta_1}\times\bar{\Delta_2},
\]
of immersions $q_i:\Delta_i'\to \Delta_{i}$ and $f_i:\Gamma'_i\to\bar{\Delta_i}$ (after permuting factors if necessary)
such that $\iota' = \iota\circ (q_1\times q_2)$ and $\bar\iota'=\bar\iota\circ(f_1\times f_2)$, i.e., the following diagram commutes:
%\begin{figure}[ht]
\[
\begin{tikzcd}[row sep=0.7pc]
\bar{\Delta_1}\times\bar{\Delta_2} \arrow[rr, "\bar\iota"]  \arrow[ddd, "p\times p"'] & & \bar Y \arrow[ddd, "p"]\\
& \Gamma'_1\times\Gamma'_2 \arrow[lu, dashed, "\exists f_1\times f_2"', sloped] \arrow[ru, "\bar\iota'"]\\
& \Delta'_1\times\Delta'_2 \arrow[rd, "\iota'"] \arrow[dl, dashed, "\exists q_1\times q_2", sloped]\\
\Delta_1\times\Delta_2 \arrow[rr, "\iota"] & & Y
\end{tikzcd}
\]
%\caption{A commutative diagram related to bases of standard product subcomplexes.}
%\label{figure:base of product subcomplex}
%\end{figure}

By using the product symbol with circle above which implies ``up to permuting factors'', we may denote a product subcomplex $K\subset Y$ and its $p$-lift $\bar K$ by $K=\Delta_1\itimes\Delta_2$ and $\bar K=\bar{\Delta_1}\itimes\bar{\Delta_2}$, such that the following lemma holds.

\begin{lemma}[\cite{Oh22}, Lemma~2.12]\label{InclusionBetweenSPSes}
Let $K$ and $K'$ be standard product subcomplexes of $Y$ (or $\bar Y$) such that $K\subset K'$. Then $K$ and $K'$ are denoted by $\Delta_1\itimes\Delta_2$ and $\Delta_1'\itimes\Delta_2'$, respectively, such that $\Delta_i\subset \Delta'_i$ for $i=1,2$.
\end{lemma}

\begin{remark}\label{remark:denoting}
There are two remarks on the notation for product subcomplexes. Firstly, since $\iota$ may not be an embedding, $K$ can not be written as $\triangle_1\times\triangle_2$.
Secondly, we admit that it has an ambiguity to denote $\bar K$ by $\bar{\Delta_1}\itimes\bar{\Delta_2}$ as there are infinitely many $p$-lifts of $K$. However, as all $p$-lifts come from product subcomplexes of $Y$, if it is clear which $p$-lift of $K=\Delta_1\itimes\Delta_2$ is considered, we denote the $p$-lift $\bar K$ of $K$ by $\bar{\Delta_1}\itimes\bar{\Delta_2}$ emphasizing that the image of $\bar K$ under the universal covering map is $K$.
\end{remark}

In the next subsection, we define the intersection complex of $\bar{Y}$ based on the following rigidity of maximal product subcomplexes.
 
\begin{theorem}[\cite{Oh22}, Lemma~2.14 and Theorem~3.4]\label{MaxtoMax}
Let $\phi:\bar{Y}\rightarrow \bar{Y'}$ be a $(\lambda,\varepsilon)$-quasi-isometry between the universal covers of two compact weakly special square complexes $Y$ and $Y'$.
Then there are constants $A=A(Y)$, $B=B(Y')$ and $D=D(\lambda,\varepsilon)$ such that for any finite collection $\{\bar M_i\}$ of maximal product subcomplexes of $\bar{Y}$ whose intersection $\bar{W}=\bigcap_i \bar{M}_i$ contains a flat, the following holds:
\begin{enumerate}[label=$(\arabic*)$]
\item\label{Item:UniqK} there exists a unique standard product subcomplex $\bar K\subset \bar Y$ with $d_H(\bar{W}, \bar K)<A$;
\item there exists a unique maximal product subcomplex $\bar M_i'\subset \bar{Y'}$ for each $i$ such that $d_H(\bar{M_i}', \phi(\bar M_i))<D$ and the intersection $\bar W' = \bigcap_i \bar M_i'$ also contains a flat;
\item there exists a unique standard product subcomplex $\bar K'\subset \bar{Y'}$ with $d_H(\bar{W}', \bar K')<B$ such that $d_H(\bar K', \phi(\bar K))<D$. 
\end{enumerate}
\end{theorem}

\subsection{(Reduced) Intersection complexes}\label{subsection:intersection complexes}
Associated to $Y$ (or $\bar{Y}$) is a (possibly not connected) almost simplicial complex with labels, called the \emph{intersection complex} originally defined by Fernandes \cite{Fer12} and slightly generalized in \cite{Oh22} as follows:

\begin{definition}\label{IC}
Let $X$ be either $Y$ or $\bar Y$. We define the \textit{intersection complex} of $X$ as an almost simplicial complex $\cI(X)$ with labels as follows:
\begin{enumerate}
\item Vertices in $\cI(X)$ correspond to maximal product subcomplexes of $X$;
\item Suppose that a finite collection $\{M_i\}$ of maximal product subcomplexes of $X$ shares standard product subcomplexes $\{K_j\}$ such that each $K_j$ is a maximal one among standard product subcomplexes contained in $\bigcap_{i} M_i$.
Then the vertices in $\cI(X)$ corresponding to $M_i$'s span $(n-1)$-simplices corresponding to $K_j$'s;
\item Assigned to each simplex $\triangle$ in $\cI(X)$ is the \emph{label} of $\triangle$ which is the domain of (the push-forward of) the standard product structure of the standard product subcomplex $K_\triangle\subset X$ corresponding to $\triangle$;
for each pair of simplices $\triangle_2\subsetneq\triangle_1$ in $X$, if $A_i\itimes B_i$ are the labels of $\triangle_i$, then $A_1\subseteq A_{2}$ and $B_1\subseteq B_{2}$. 
\end{enumerate}
The underlying complex is denoted by $|\cI(X)|$.
\end{definition}

Contrast to $|\cI(Y)|$, by Item \ref{Item:UniqK} in Theorem~\ref{MaxtoMax}, $|\cI(\bar Y)|$ is always a simplicial complex. 

\begin{definition}[(Semi-)morphism]\label{Def:Morphism}
Let $X_i$ be either a compact weakly special square complex or its universal cover for $i=1,2$.
A combinatorial map $\Phi:\cI(X_1)\to\cI(X_2)$ is said to be a \emph{semi-morphism} if the following holds:
For each pair of simplices $\triangle_2\subsetneq\triangle_1$ in $X_1$, suppose that $A_i\itimes B_i$ are the labels of $\triangle_i$ such that $A_1\subseteq A_{2}$ and $B_1\subseteq B_{2}$. 
Then $\Phi(\triangle_i)$ are labelled by $A'_i\itimes B'_i$ with $A'_1\subseteq A'_{2}$ and $B'_1\subseteq B'_{2}$ such that 
\[\text{$A_1\subsetneq A_2$ if and only if $A'_1\subsetneq A'_2$, and $B_1\subsetneq B_2$ if and only if $B'_1\subsetneq B'_2$.}\]
Additionally, if the fundamental groups of the labels of $\triangle$ and $\Phi(\triangle)$ are quasi-isometric for each simplex $\triangle\subset \cI(X_1)$, then $\Phi$ is said to be a \emph{morphism}.

If both $\Phi$ and its inverse map $\Phi^{-1}$ are morphisms (semi-morphisms, resp.), then $\Phi$ is said to be an \emph{isomorphism} (\emph{semi-isomorphism}, resp.).
\end{definition}

As a partial converse of the construction of an intersection complex, we define the set map 
\begin{equation}\label{Eq:g_I}
\mathfrak{g}_\mathcal{I}:\cI(X)\rightarrow 2^X
\end{equation}
as follows:
The interior of each simplex is mapped to the corresponding standard product subcomplex, and for a (closed) subcomplex $A$ of $\cI(X)$, the image of $A$ is the union of the images of (open) simplices in $A$.

\begin{convention}
For a compact special square complex $Y$, $\cI(Y)$ will be called the \emph{reduced intersection complex} of $Y$ and denoted by $\cRI(Y)$. 
\end{convention}

There are two immediate statements which encode facts spread in the previous subsection via the two definitions above.
The first one is a consequence of Theorem~\ref{MaxtoMax} and the other one records the fact that the deck transformation action of $\pi_1(Y)$ on $\bar{Y}$ preserves $p$-lifts. 

\begin{theorem}[\cite{Oh22}, Theorem~C]\label{theorem:IsobetInt}
Let $\phi:\bar{Y}\rightarrow \bar{Y'}$ be given as in Theorem~\ref{MaxtoMax}.
Then there is an isomorphism $\Phi:\cI(\bar{Y}) \rightarrow \cI(\bar{Y'})$ induced from $\phi$.
\end{theorem}

\begin{theorem}[\cite{Oh22}, Theorem~3.15]\label{Thm:TPBCM}
The $\pi_1(Y)$-action on $\bar{Y}$ induces a $\pi_1(Y)$-action on $\cI(\bar{Y})$ by isomorphisms such that the quotient is $\cRI(Y)$, i.e., there is a $\pi_1(Y)$-equivariant morphism $\rho:\cI(\bar{Y})\rightarrow \cRI(Y)$, called the \emph{canonical quotient map}, such that $\mathfrak{g}_\mathcal{RI}\circ \rho = p \circ \mathfrak{g}_\mathcal{I}$, i.e., the following diagram commutes:
\[
\begin{tikzcd}[column sep=4pc, row sep=2pc]
\cI(\bar{Y})\arrow[r,"\mathfrak{g}_{\cI}"]\arrow[d,"\rho"'] & 2^{\bar{Y}}\arrow[d,"p"] \\ \cI(Y) \arrow[r,"\mathfrak{g}_{\cRI}"] & 2^Y
\end{tikzcd}
\]
\end{theorem}

The fact that $\cRI(Y)$ is the quotient of $\cI(\bar {Y})$ by the action of $\pi_1(Y)$ is reminiscent of the definition of a complex of groups.
For a polyhedral complex $X$, if there is a group $G$ acting on $X$ by isometries, then the group action gives rise to a complex of groups $\mathcal{G}(Y)$ over the quotient $Y=X/G$. If $X$ is simply connected, then $\mathcal{G}(Y)$ is said to be \emph{developable} and $X$ is said to be the \emph{development} of $\mathcal{G}(Y)$. Conversely, if $\mathcal{G}(Y)$ is developable, then its development is constructed from the poset of left cosets of the groups assigned to cells of $\mathcal{G}(Y)$ with the inclusion relation
We refer to \cite[Chapter III.$\mathcal{C}$]{BH} for the details about (developable) complexes of groups.

\begin{theorem}[\cite{BH}, Corollary~2.15 in Chapter III.$\mathcal{C}$] \label{Thm:Developable}
A complex of groups $\mathcal{G}(Y)$ is developable if and only if there exists a morphism $\Phi$ from $\mathcal{G}(Y)$ to some group $G$ which is injective on the local groups.
\end{theorem}

\begin{theorem}[\cite{BH}, Theorem~3.13 and Corollary~3.15 in Chapter III.$\mathcal{C}$]\label{Thm:Development} 
The development of a developable complex of groups is simply connected and unique up to isomorphisms.
\end{theorem}

The group assigned to each simplex of $\cRI(Y)$, considered as a complex of groups, is the product of two free groups. In particular, the cardinality of each assigned group is infinite and for any pair of simplices $\triangle\subsetneq\triangle'$ in $\cRI(Y)$ the index of the group assigned to $\triangle'$ in the group assigned to $\triangle$ is either $1$ or infinite. We record the following proposition, which is an easy consequence of the construction of the development, in order to use it later.

\begin{proposition}\label{Prop:IsomorphicI}
Let $Y_1$ and $Y_2$ be two compact weakly special square complexes such that both $\cRI(Y_1)$ and $\cRI(Y_2)$ are connected and developable. 
If there exists a (semi-)isomorphism $\Phi:\cRI(Y_1)\to\cRI(Y_2)$, then there is a (semi-)isomorphism $\bar\Phi$ between the developments of $\cRI({Y_1})$ and $\cRI({Y_2})$ such that $\rho_2\circ\bar\Phi=\Phi\circ\rho_1$, where $\rho_i$ is the quotient map from the development of $\cRI(Y_i)$ to $\cRI(Y_i)$.
\end{proposition}

\begin{remark}\label{Rmk:Developability}
In Section~\ref{section:grapes}, we will show that $\cRI(Y)$ is developable and $\cI(\bar Y)$ is the development of $\cRI(Y)$ if $Y$ is the (unordered) discrete $2$-configuration space $UD_2(\Gamma)$ of the minimal simplicial representative $\Gamma$ of a graph with circumference one. 
In general, however, $\cRI(Y)$ is not developable (see the last paragraph of Section 4.1 in \cite{Oh22}) and $\cI(\bar Y)$ is not even simply connected (\cite[Proposition 4.7]{Oh22}).    
\end{remark}

In the remaining of this subsection, we will see the structure of $\cRI(Y)$ or $\cI(\bar Y)$ directly obtained from $Y$ or $Y_0=\mathfrak{g}_{\cRI}(\cRI(Y))$ which is a weakly special square complex by Lemma~\ref{lem:SpecialSubcomplexes} (unless there is no standard product subcomplex). 

\begin{lemma}\label{lem:UnionofSPS}
Suppose that $Y_0$ is given as above. Then there is a one-to-one correspondence between standard product subcomplexes of $Y_0$ and standard product subcomplexes of $Y$, which induces an isomorphism between $\cRI(Y)$ and $\cRI(Y_0)$.

If $Y_0$ is locally convex in $Y$, then any standard (maximal, resp.) product subcomplex of $\bar{Y_0}$ is a standard (maximal, resp.) product subcomplex of $\bar{Y}$. In particular, there is a canonical injective morphism $\cI(\bar{Y_0})\hookrightarrow \cI(\bar{Y})$.
\end{lemma}
\begin{proof}
By the construction of $Y_0$ and Lemma~\ref{Lem:ImmersionImpliesLocalisometry}, any standard product subcomplex $K\subset Y_0$ with a standard product structure $\iota:\Lambda_1\times\Lambda_2\to Y_0$ can be seen as a standard product subcomplex of $Y$ with a standard product structure $\iota':\Lambda_1\times\Lambda_2\to Y$ (which is the composition of the embedding and $\iota$) and vice-versa.
Thus the first statement holds.

Suppose that $Y_0$ is locally convex in $Y$.
For any subset $Z\subset Y_0$, then a component of $p^{-1}(Z)$ in $\bar{Y}$ intersects $\bar{Y_0}$ if and only if the component belongs to $\bar{Y_0}$.
It follows that any $p$-lift $\bar K$ of a standard product subcomplex of $Y_0$ in $\bar{Y_0}$ is contained in $\bar{Y}$. Moreover, if $\bar K$ is a maximal product subcomplex of $\bar{Y_0}$, then $\bar K$ must be a maximal product subcomplex of $\bar Y$.
Therefore, the second statement also holds.
\end{proof}

It is worth noticing that in general $Y_0$ is not locally convex in $Y$ and thus neither $\bar{Y_0}$ may be isometrically embedded into $\bar{Y}$ nor $\cI(\bar{Y_0})$ may be embedded into $\cI(\bar Y)$.

\begin{Ex}[Extended Salvetti complex]
We start with the barycentric subdivision $X'$ of a pentagon $C$ which is a simplicial complex.
Let $X$ be a cube complex obtained from $X'$ by considering the union of two simplices sharing an edge joining the central vertex to a vertex of $C$ as a $2$-cube.
Let $Y$ be a cube complex obtained from the disjoint union of the products of two loops $a_i$ and $b_i$ with base vertices $x_i$ and $y_i$, respectively, for $i=0,\dots,4$ as follows:
\begin{enumerate}
\item
For each $i$, attach a path $\gamma_i$ of length $2$ connecting $x_i\times y_i\in a_i\times b_i$ to $x_{i+1}\times y_{i+1}\in a_{i+1}\times b_{i+1}$ and then attach the product of a path of length $2$ and an edge along the loop which is the concatenation of $x_i\times b_i$, $\gamma_i$, $a_{i+1}\times y_{i+1}$ and (the inverse of) $\gamma_i$ (the indices are taken modulo $5$); let $Y_0$ be the resulting complex.
\item 
And then attach $X$ to $Y_0$ along the union of $\gamma_i$'s such that the vertices of $C$ are identified with $x_i\times y_i$'s.
\end{enumerate}
Then it is easily seen that $Y_0$ is not locally convex in $Y$ but $\cRI(Y)$ is isomorphic to $\cRI(Y_0)$.
Moreover, $|\cI(\bar{Y_0})|$ is simply connected but $|\cI(\bar Y)|$ is not.
\end{Ex}

Even though $Y$ is assumed to be connected, $\cRI(Y)$ is possibly disconnected, and $\cI(\bar{Y})$ is also possibly disconnected even when $\cRI(Y)$ is connected.
Nonetheless, if $Y$ has a particularly nice structure which we will encounter in Section~\ref{section:grapes}, then both $\cRI(Y)$ and $\cI(\bar{Y})$ are connected.

\begin{lemma}\label{lem:Connected}
Suppose that $Y$ is equal to the union of all standard product subcomplexes of $Y$, i.e., $\mathfrak{g}_{\cRI}(Y)=Y$.
Then not only $\mathfrak{g}_{\cI}(\cI(\bar Y))=\bar Y$ but also both $|\cRI(Y)|$ and $|\cI(\bar{Y})|$ are connected simplicial complexes if the following hold:
\begin{enumerate}
\item Every standard product subcomplex of $Y$ is embedded;
\item The intersection of maximal product subcomplexes of $Y$ is non-empty if and only if it is a standard product subcomplex.
\end{enumerate}
\end{lemma}
\begin{proof}
As discussed below Definition~\ref{IC}, $|\cI(\bar{Y})|$ is always a simplicial complex. 
Moreover, the second condition implies that $|\cRI(Y)|$ is a simplicial complex.

Let $X$ be either $Y$ or $\bar{Y}$ and fix a maximal product subcomplex $M$ of $X$.
By the two assumptions, $X$ is the union of maximal product subcomplexes $M'$ of $X$ such that there exists a sequence $(M=M_1,\dots,M_n= M')$ of maximal product subcomplexes of $X$, where $M_i\cap M_{i+1}$ is a standard product subcomplex of $X$. 
Therefore, $\cI(X)$ is connected and $\mathfrak{g}_{\cI}(\bar Y)=\bar Y$.    
\end{proof}

If $|\cI(\bar{Y_0})|$ is connected, for two maximal product subcomplexes $\bar{M},\bar{M}'$ of $\bar{Y_0}$ there exists a sequence $(\bar{M}=\bar{M}_1,\dots,\bar{M}_n= \bar{M}')$ of maximal product subcomplexes of $\bar{Y_0}$ such that $\bar{M}_i\cap \bar{M}_{i+1}$ is a standard product subcomplex. 
As the direct product of two free groups are \emph{thick of order $0$}, 
$\bar{Y_0}$ is thick of order at most $1$. The precise definition of \emph{thickness} is not necessary for this paper; we refer to \cite{BDM} or \cite{BD14}. Instead, as a consequence of thickness of $\bar{Y_0}$, we can know that $\bar{Y_0}$ is one-ended.

\begin{lemma}\label{lem:One-endedness}
Suppose that $Y$ is equal to the union of all standard product subcomplexes of $Y$. If $|\cI(\bar{Y})|$ is connected, then $\bar{Y}$ is one-ended.
\end{lemma}

\subsection{Free products of graphs}\label{section:PW}
In this subsection, we revisit the work of Papasoglu--Whyte in \cite{PW02}, where the main result is that a non-trivial splitting of a group over finite groups is a quasi-isometry invariant, and the following is its special case. 

\begin{theorem}[\cite{PW02}]\label{PW}
Let $\mathbb{G}_1$ and $\mathbb{G}_2$ be two (finitely generated) groups.
If $\mathbb{G}_1$ and $\mathbb{G}_2$ are quasi-isometric, then $\mathbb{G}_1*\mathbb{F}_m$ and $\mathbb{G}_2*\mathbb{F}_n$ are quasi-isometric for $m,n\geq 1$.  
Moreover, the converse also holds if both $\mathbb{G}_1$ and $\mathbb{G}_2$ are one-ended.
\end{theorem}
\begin{proof}
The first part is shown by \cite[Theorem~0.3]{PW02}. The `moreover' part is shown by \cite[Theorem~0.4]{PW02}.
\end{proof}

In order to prove the above theorem, Papasoglu--Whyte defined the free product of graphs which is served as a geometric model for the free product of groups.
Roughly speaking, the free product of two graphs $\mathsf{G}_1$ and $\mathsf{G}_2$ has the following properties:
\begin{enumerate}
\item It consists of infinitely many copies of $\mathsf{G}_1$ and $\mathsf{G}_2$ and edges joining these copies such that each vertex in a copy of $\mathsf{G}_i$ is joined to a vertex in a copy of $\mathsf{G}_j$ for $\{i,j\}=\{1,2\}$.
\item It has a tree of spaces structure in the sense that if we collapse each copy of $\mathsf{G}_i$ to a point, then the resulting space is a tree.
As a consequence, each vertex of the free product can be thought of a \emph{gate} in the sense that each edge joining two copies separates the free product.
\end{enumerate}
If each $\mathsf{G}_i$ is a Cayley graph of $\mathbb{G}_i$ (with respect to a finite generating set), then the free product of $\mathsf{G}_1$ and $\mathsf{G}_2$ is thought of a geometric model for $\mathbb{G}_1*\mathbb{G}_2$.

In Section~\ref{section:Operations}, we encounter the case that each $\mathsf{G}_i$ is the universal cover $\bar{\Gamma_i}$ of a finite graph $\Gamma_i$, and instead of all the vertices of $\bar{\Gamma_i}$, the preimage of a specific vertex $v_i\in\mathcal{V}(\Gamma_i)$ under the universal covering map is our gates.
After redefining the concept of a free product and introducing terminologies related to quasi-isometries preserving specific structures, we will reinterpret facts presented in \cite{PW02}.
The remaining of this subsection can be skipped on first reading and referred back to when needed in the proofs of these results.

Let $\mathcal{A}=\{(\Gamma_1,v_1),\dots,(\Gamma_m,v_m)\}$ ($m\ge 2$) be a set of pairs of finite graphs $\Gamma_i$ with base vertices $v_i$. 
We define the \emph{free product} of $(\Gamma_1,v_1),\dots,(\Gamma_m,v_m)$ as the graph which is obtained from the disjoint union of ${\Gamma_i}$'s and an $m$-star $\mathsf{S}_{m}$ (see Section~\ref{section:graph 2-braid groups} for the definition of a star) by identifying each leaf of $\mathsf{S}_{m}$ with exactly one $v_i$; $\mathsf{S}_{m}$ will be called the \emph{connecting star} of the free product. 
The free product will be denoted by $*\mathcal{A}$ or (with a slight abuse of notation based on Remark~\ref{Remark:FreeProduct} below) $(\Gamma_1,v_1)*\dots*(\Gamma_m,v_m)$. See the left of Figure~\ref{figure:m-ary vs binary}. 

The universal cover $\bar{*\mathcal{A}}$ of $*\mathcal{A}$ consists of copies of $\bar{\Gamma_i}$'s and $m$-stars such that each $m$-star connects copies of $\bar{\Gamma_1},\dots,\bar{\Gamma_m}$.
For $*\mathcal{A}$, let $V$ be a subset of $\{v_1,\dots,v_m\}$. Then we denote the preimage of $V$ in $\bar{*\mathcal{A}}$ under the universal covering map by $\bar{V}$, which is equipped with the induced metric.
The schematic picture of $\bar{*\mathcal{A}}$ is depicted in Figure~\ref{figure:universal cover}.

\begin{figure}[ht]
\[
\bar{*\mathcal{A}}=
\begin{tikzpicture}[baseline=-.5ex]
\begin{scope}
\draw[fill] (0,0) circle (2pt);
\foreach \i in {1,2,3,4,5} {
\draw[thick, fill] (0,0) -- ({72*\i+18}:1) circle (2pt);
\draw ({72*\i+18}:1) ++({72*\i+18+90}:0.25) -- ++({72*\i+18-90}:0.5);
\begin{scope}[shift=({72*\i+18}:1.7)]
\foreach \k in {1,2,3,4,5} {
\draw ({72*\i+20-90+\k*72}:0.9) node {$\scriptscriptstyle\bar{\Gamma_{\pgfmathparse{int(mod(\i+\k+1,5)+1)}\pgfmathresult}}$};
\foreach \j in {-1,0,1} {
\begin{scope}[shift=({72*\i+18}:-0.7),rotate={\j*5},shift=({72*\i+18}:0.7)]
\begin{scope}[rotate=180,scale=0.7, opacity=0.5]
\if\k0
\fill (0,\k) circle (2pt);
\fi
\draw (0,0) -- ({72*\k+18}:1);
\fill (0,0) -- ({72*\k+18}:1) circle (2pt);
\if\i\k
\else
\draw ({72*\k+18}:1) ++({72*\k+18+90}:0.25) -- ++({72*\k+18-90}:0.5);
\fi
\end{scope}
\end{scope}
}
}
\end{scope}
}
\end{scope}
\end{tikzpicture}
\]
\caption{A part of $\bar{*_{i=1}^5\{(\Gamma_i,v_i)\}}$; at each vertex, the small bar indicates the universal cover $\bar{\Gamma}_i$.}
\label{figure:universal cover}
\end{figure}
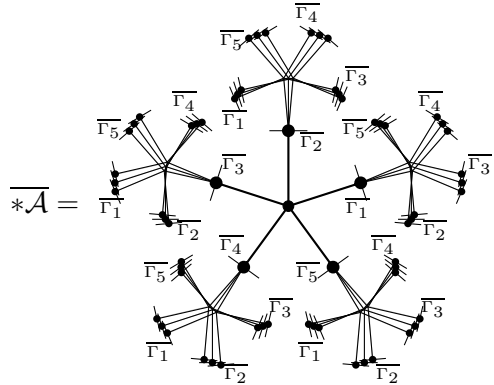

\begin{convention}
By a pair $(X,A)$ of a metric space $X$ with its subspace $A$, we always mean that the inclusion map $A\hookrightarrow X$ is a $(1,C)$-quasi-isometry, where the $C$-neighborhood of $A$ contains $X$. 
And we denote $(\bar X, p^{-1}(A))$ by $\bar{(X, A)}$.
\end{convention}

We caution the reader that the definition of `relative' maps defined below is specifically designed to say that if infinitely many `relative' quasi-isometries are well-glued, then the resulting map is also a relative quasi-isometry (and thus a quasi-isometry).

\begin{definition}[Relative (quasi-)isometry]\label{Def:RelQI}
Suppose that $(X_i,A_i)$ is a pair of a metric space with its subspace with the inclusion map $\iota_i:A_i\hookrightarrow X_i$ and a quasi-inverse $r_i$ of $\iota_i$ for $i=1,2$.
By a \emph{relative $L$-quasi-isometry} $\phi:(X_1, A_1)\to(X_2, A_2)$, we mean a quasi-isometry which is the composition of $\iota_2$, an $L$-bilipschitz map $A_1\to A_2$ and  $r_1$.
By a \emph{relative isometric embedding} $\phi:(X_1, A_1)\to(X_2, A_2)$, we mean an isometric embedding $X_1\to X_2$ whose restriction to $A_1$ is an injection to $A_2$; similarly, we define a \emph{relative isometry}.
\end{definition}

It is worth noticing that a relative $L$-quasi-isometry $\phi:(X_1, A_1)\to(X_2, A_2)$ is an $(L,2C)$-quasi-isometry where $C$ is a constant satisfying $X_i\subseteq \mathcal{N}_C(A_i)$ for $i=1,2$.

Let $\phi:X\to Y$ be either a (quasi-)isometry or a (quasi-)isometric embedding between metric spaces. For a proper subspace $A\subsetneq X$, in general, the restriction $\phi|_A:A\to Y$ is only a (quasi-)isometric embedding due to the lack of (quasi-)surjectivity. 
Unless mentioned otherwise, throughout this paper, we mean a (quasi-)isometry $\phi|_A:A\to B$ onto a certain subspace $B\subset Y$ just by the \emph{restriction} of $\phi$ to $A$ whenever the target space $B$ could be defined obviously, and similarly for relative maps defined above.

\begin{remark}\label{Remark:FreeProduct}
Let $\mathcal{A}=\{(\Gamma_1,v_1),\dots, (\Gamma_m,v_m)\}$ ($m\ge2$) be a set of pairs of finite graphs with base vertices. As $*\mathcal{A}$ is defined, we may define two other graphs from $\mathcal{A}$ as follows. 
One is the graph $\vee_i(\Gamma_i,v_i)$ obtained from the disjoint union of $(\Gamma_i,v_i)$'s by identifying all $v_i$'s.
The other one is constructed inductively: 
Let $\mathsf{Z}_1$ be the free product of $(\Gamma_1,v_1)$ and $(\Gamma_2,v_2)$ and $w_1$ the central vertex of the connecting $2$-star of $\mathsf{Z}_1$. 
Then we inductively define $\mathsf{Z}_k$ as the free product of $(\mathsf{Z}_{k-1},w_{k-1})$ and $(\Gamma_{k+1},v_{k+1})$ and $w_{k}$ as the central vertex of the connecting $2$-star of $\mathsf{Z}_k$ for $k=2,\dots,m-1$ and thus have $(\mathsf{Z}_{m-1},w_{m-1})$.
See Figure~\ref{figure:m-ary vs binary}.

However, there is no difference between $\bar{\mathsf{Z}_{m-1}}$ and $\bar{*\mathcal{A}}$ up to (relative) quasi-isometry. More precisely, for $V=\{v_1,\dots, v_m\}$ in both $\mathsf{Z}_{m-1}$ and $*\mathcal{A}$, there is a relative quasi-isometry $\bar\phi:\bar{(\mathsf{Z}_{m-1},V)}\to\bar{(*\mathcal{A},V)}$
which induces a one-to-one correspondence between copies of $\bar{(\Gamma_i,v_i)}$ in $\bar{\mathsf{Z}_{m-1}}$ and those in $\bar{*\mathcal{A}}$ such that the restriction of $\bar\phi$ to each copy is a relative isometry to the corresponding copy. 
Similarly, there is no difference between $\bar{\vee_i(\Gamma_i,v_i)}$ and $\bar{*\mathcal{A}}$ up to quasi-isometry.
\end{remark}

\begin{figure}[ht]
\[
*_{i=1}^{5}\{(\Gamma_i,v_i)\}=
\begin{tikzpicture}[baseline=-.5ex]
\draw[fill] (0,0) circle (2pt) node[above=1ex, right=-.5ex] {$w$};
\foreach \i in {1,2,3,4,5} {
\draw[fill] (0,0) -- ({72*\i+18}:1) circle (2pt);
\draw ({72*\i+18+20}:0.85) node {$v_{\i}$};
\begin{scope}[rotate={72*\i+18}]
\draw (1,-0.4) rectangle (1.7, 0.4) node[pos=0.5] {$\Gamma_{\i}$};
\end{scope}
}
\end{tikzpicture}
\qquad
\mathsf{Z}_4=
\begin{tikzpicture}[baseline=-.5ex]
\foreach \i in {1,2,3,4,5} {
\begin{scope}[rotate={72*(\i-1)+18}]
\draw (1,-0.4) rectangle (1.7, 0.4) node[pos=0.5] {$\Gamma_{\i}$};
\end{scope}
\draw[fill] (18:1) circle (2pt) node[below=1ex, left=-1ex] {$v_1$} -- coordinate[midway] (m1) ({72+18}:1) circle (2pt) node[below] {$v_2$};
\draw[fill] (m1) circle(2pt) -- coordinate[midway] (m2) ({72*2+18:1}) circle (2pt) node[below=1ex, right=-.5ex] {$v_3$};
\draw[fill] (m2) circle (2pt) -- coordinate[midway] (m3) ({72*3+18:1}) circle (2pt) node[above=.5ex, right=-.5ex] {$v_4$};
\draw[fill] (m3) circle (2pt) -- coordinate[midway] (m4) ({72*4+18:1}) circle (2pt) node[above=0ex] {$v_5$};
\draw[fill] (m4) circle (2pt) node[above] {$w$};
}
\end{tikzpicture}
\]
\caption{Free product of graphs as an $m$-ary operation and as a composition of binary operations.}
\label{figure:m-ary vs binary}
\end{figure}
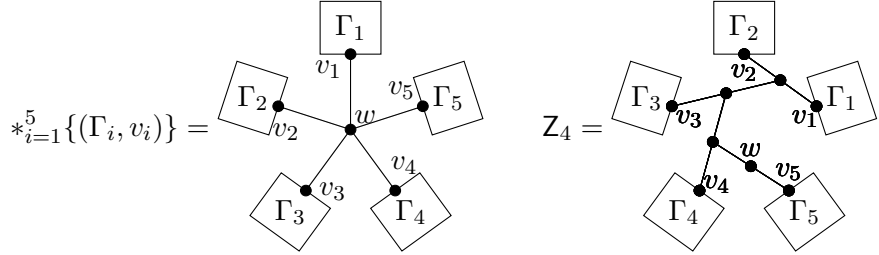

The following two lemmas are obvious and we omit the proof.

\begin{lemma}\label{Lem:RelQI}
Suppose that there are relative $L$-quasi-isometries $\phi_i:(X_i,A_i)\to (Y_i,B_i)$ for $i=0,1,2$ and relative isometric embeddings $(X_0,A_0)\hookrightarrow (X_j,A_j)$ and $(Y_0,B_0)\hookrightarrow (Y_j,Y_j)$ for $j=1,2$.
Then there is a relative $L$-quasi-isometry 
\[\phi_1\cup_{\phi_0}\phi_2:(X_1\cup_{X_0} X_2,A_1\cup_{A_0}A_2)\to(Y_1\cup_{Y_0} Y_2,B_1\cup_{B_0}B_2)\]
if the diagram below commutes:
\[
\begin{tikzcd}[column sep=4em, row sep=2em]
(X_1,A_1)\arrow[d,"\phi_1"]\arrow[hookleftarrow,r,""] &(X_0,A_0)\arrow[hookrightarrow,r," "]\arrow[d,"\phi_0"] & (X_2,A_2)\arrow[d,"\phi_2"]\\
(Y_1,B_1)\arrow[hookleftarrow,r," "]&(Y_0,B_0)\arrow[hookrightarrow,r," "] & (Y_2,B_2)
\end{tikzcd}
\]
\end{lemma}

\begin{lemma}\label{Lem:RelQIProduct}
Suppose that there are relative $L$-quasi-isometries $\phi_i:(X_i,A_i)\to (Y_i,B_i)$ for $i=1,2$.
Then there is a relative $L^2$-quasi-isometry 
\[\phi_1\times\phi_2:(X_1\times X_2,A_1 \times A_2)\to(Y_1\times Y_2,B_1\times B_2).\]
\end{lemma}

Now, we are ready to rewrite results in \cite{PW02} in our language.

\begin{lemma}[\cite{PW02}, Lemma~1.1]\label{Lemma1.1inPW}
Let $(\Gamma,v)$ and $(\Lambda,w)$ be pairs of finite graphs with base vertices and let $\mathsf{Z}$ be a graph satisfying the following:
\begin{itemize}
\item $\mathsf{Z}$ contains a disjoint family of subgraphs $\{\bar{(\Gamma_i,v_i)}\}$ and $\{\bar{(\Lambda_j,w_j)}\}$ for pairs of finite graphs with base vertices $(\Gamma_i,v_i)$ and $(\Lambda_j,w_j)$, whose union contains all vertices of $\mathsf{Z}$.
\item Let $V\subset\mathsf{Z}$ be the union of $p_i^{-1}(v_i)\subset\bar{\Gamma_i}$ and $q_j^{-1}(w_j)\subset\bar{\Lambda_j}$ for all $i$ and $j$, where $p_i:\bar\Gamma_i\to\Gamma_i$ and $q_j:\bar\Lambda_j\to \Lambda_j$ are the universal covering maps.
Then every edge of $\mathsf{Z}$ not in one of the subgraphs connects a vertex in $p_i^{-1}(v_i)$ to a vertex in $q_j^{-1}(w_j)$ for some $i$ and $j$, and there is exactly one such edge at each vertex in $V$.
\item For a fixed $L$, there is a set $\mathcal{P}$ of relative $L$-quasi-isometries $\bar{(\Gamma,v)}\to\bar{(\Gamma_i,v_i)}$ and $\bar{(\Lambda,w)}\to\bar{(\Lambda_j,v_j)}$ for each $i$ and $j$.
\item The graph obtained from $\mathsf{Z}$ by collapsing each of the subgraphs to a point is a tree.
\end{itemize}
Then for $\mathsf{Z}_1=(\Gamma,v)*(\Lambda,w)$ and $V_1=\{v,w\}$, there is a relative $L$-quasi-isometry $\phi:\bar{(\mathsf{Z}_1,V_1)}\to(\mathsf{Z},V)$ such that
\begin{enumerate}
\item $\phi$ induces a one-to-one correspondence between copies of $\bar{(\Gamma,v)}$ ($\bar{(\Lambda,w)}$, resp.) in $\bar{\mathsf{Z}_1}$ and copies of $\{\bar{(\Gamma_i,v_i)}\}$ ($\{\bar{(\Lambda_j,v_j)}\}$, resp.) in $\mathsf{Z}$ such that the restriction of $\phi$ to each copy is a relative quasi-isometry in $\mathcal{P}$.
\item Suppose that a copy $\mathsf{A}$ of $\bar{(\Gamma,v)}$ and a copy $\mathsf{B}$ of $\bar{(\Lambda,w)}$ correspond to a copy $\mathsf{A}'$ of $\bar{(\Gamma_i,v_i)}$ and a copy $\mathsf{B}'$ of $\bar{(\Lambda_j,v_j)}$, respectively.
If $\mathsf{A}$ and $\mathsf{B}$ are joined by a connecting $2$-star, then $\phi$ maps the connecting $2$-star to an edge joining $\mathsf{A}'$ and $\mathsf{B}'$, and \textit{vice versa}.
\end{enumerate}
\end{lemma}

\begin{lemma}[\cite{PW02}, Lemma~1.4]\label{Lemma1.4inPW}
For two pairs of finite graphs with base vertices $(\Gamma,v)$ and $(\Lambda,w)$, let 
\[
(\mathsf{Z}_1, V_1)=((\Gamma, v) * (\Lambda,w), \{v,w\})\text{ and }
(\mathsf{Z}_2, V_2)=(*\{(\Gamma, v),(\Lambda,w),(\Lambda',w')\}, \{v,w,w'\}),
\]
where $(\Lambda',w')$ is a copy of $(\Lambda,w)$.
Then there is a relative quasi-isometry $\bar{(\mathsf{Z}_1, V_1)} \to \bar{(\mathsf{Z}_2, V_2)}$ which induces a one-to-one correspondence between copies of $\bar{(\Gamma,v)}$ ($\bar{(\Lambda,w)}$, resp.) in $\bar{(\mathsf{Z}_1, V_1)}$ and copies of $\bar{(\Gamma,v)}$ ($\bar{(\Lambda,w)}$ or $\bar{(\Lambda',w')},\textrm{ resp.}$) in $\bar{(\mathsf{Z}_2, V_2)}$ such that the restriction to each copy is a relative isometry on the corresponding copy. 
\end{lemma}

For $n\ge 1$, let $\mathsf{R}_n$ be a bouquet of $n$ loops with central vertex $r_n$. Then the role of $(\mathsf{R}_n,r_n)$ in a free product may not be particularly important while studying the large scale geometry of the free product in the following sense.

\begin{lemma}\label{Lem:PWLem5}
For a set of pairs of finite graphs with base vertices $\mathcal{A}=\{(\Gamma_1,v_1),\dots,(\Gamma_n,v_n)\}$ ($n\ge 2$), let 
\[
\mathsf{Z}_1=*\mathcal{A},\quad\mathsf{Z}_2=*\mathcal{A}'\quad\text{and}\quad\mathsf{Z}_3=(\mathsf{Z}_1,c) \vee (\mathsf{R}_m,r_m),
\]
where $\mathcal{A}'=\mathcal{A}\cup\{(\mathsf{R}_m,r_m)\}$ and denote the subset $\{v_1,\dots,v_n\}$ of each $\mathsf{Z}_i$ by $V$.
Then for any pair of $1\le i<j\le 3$, there is a relative quasi-isometry $\phi_{i,j}:\bar{(\mathsf{Z}_i, V)} \to \bar{(\mathsf{Z}_j, V)}$ which induces a one-to-one correspondence between copies of $\bar{(\Gamma_k,v_k)}$ and copies of $\bar{(\Gamma_k,v_k)}$ for each $1\le k\le n$ such that the restriction of $\phi_{i,j}$ to each copy is a relative isometry.
\end{lemma}
\begin{proof}
By \cite[Theorem~2.1]{PW02}, for $i=1$ and $j=2$, there exists a desired relative quasi-isometry $\phi_{1,2}:\bar{(\mathsf{Z}_1,V)}\to\bar{(\mathsf{Z}_2,V)}$.

Note that $\mathsf{Z}_1$ can be naturally seen as a subgraph of either $\mathsf{Z}_2$ and $\mathsf{Z}_3$.
As there is a homotopy equivalence between $\mathsf{Z}_{2}$ and $\mathsf{Z}_{3}$ relative to $\mathsf{Z}_1$, there is a relative quasi-isometry $\phi_{2,3}:\bar{(\mathsf{Z}_{2},V)}\to \bar{(\mathsf{Z}_{3},V)}$ such that the restrictions of $\phi_{2,3}$ and $\phi_{2,3}^{-1}$ to copies of $\bar{(\Gamma_k,v_k)}$ are relative isometries.

Lastly, it is easily seen that $\phi_{2,3}\circ\phi_{1,2}$ is the desired relative quasi-isometry for $\phi_{1,3}$. 
\end{proof}

\begin{lemma}\label{Lem:PWLem1}
Let $(\Gamma_i,v_i)$ be a copy of a pair of a finite graph $(\Gamma,v)$ with base vertex for $i=1,2,3$.
Let 
\[
(\mathsf{Z},V)=({(\Gamma_1, v_1)} * {(\Gamma_2,v_2)},\{v_1,v_2\})\quad
\text{ and } 
\quad
(\mathsf{Z}_n,V_n)=({(\Gamma_3, v_3)} *{(\mathsf{R}_n,r_n)},\{v_3\}).
\]
Then there is a relative quasi-isometry $\phi:\bar{(\mathsf{Z}, \{v_1,v_2\})} \to \bar{(\mathsf{Z}_n, \{v_3\})}$ which induces a one-to-one correspondence between copies of $\bar{(\Gamma_1,v_1)}$ and $\bar{(\Gamma_2,v_2)}$ in $\bar{\mathsf{Z}}$ and copies of $\bar{(\Gamma_3,v_3)}$ in $\bar{\mathsf{Z}_n}$ such that the restriction of $\phi$ to each copy is a relative isometry.
\end{lemma}
\begin{proof}
Let $\mathsf{Z}'_{2}=*\{(\Gamma_3, v_3),(\mathsf{R}_1,r_1),(\mathsf{R}'_1,r'_1)\}$ for a copy $(\mathsf{R}'_1,r'_1)$ of $(\mathsf{R}_1,r_1)$.
By Lemma~\ref{Lemma1.4inPW}, there is a relative quasi-isometry 
$$\phi:\bar{(\mathsf{Z}_1,\{v_3,r_1\})}\to\bar{(\mathsf{Z}'_{2},\{v_3,r_1,r'_1\})}$$
such that the restrictions of $\phi$ and $\phi^{-1}$ to copies of $\bar{(\Gamma,v)}$ are relative isometries and in particular, $\phi$ induces a bijection between the preimage of $v_3$ in $\mathsf{Z}_1$ and the preimage of $v_3$ in $\mathsf{Z}'_2$.
As there is a homotopy equivalence between $\mathsf{Z}_{2}$ and $\mathsf{Z}'_{2}$ relative to $\Gamma_3$, there is a relative quasi-isometry 
$$\varphi:\bar{(\mathsf{Z}'_{2},\{v_3\})}\to \bar{(\mathsf{Z}_{2},\{v_3\})}$$ 
such that the restrictions of $\varphi$ and $\varphi^{-1}$ to copies of $\bar{(\Gamma,v)}$ are relative isometries.
Thus there is a relative quasi-isometry 
$$f_1=\varphi\circ\phi:\bar{(\mathsf{Z}_{1},\{v_3\})}\to\bar{(\mathsf{Z}_{2},\{v_3\})}$$
such that the restrictions of $f_1$ and $f_1^{-1}$ to copies of $\bar{(\Gamma,v)}$ are relative isometries.
By induction, we can also show that there is a relative quasi-isometry $f_n:\bar{(\mathsf{Z}_{n},\{v_3\})}\to\bar{(\mathsf{Z}_{n+1},\{v_3\})}$ such that the restrictions of $f_n$ and $f_n^{-1}$ to copies of $\bar{(\Gamma,v)}$ are relative isometries.

Now, it suffices to prove this lemma for $n=1$.
Let $\mathsf{Z}'=(\Gamma_1,v_1)*(\Gamma_2,v_2)*(\mathsf{R}_1,r_1)$.
By Lemma~\ref{Lemma1.4inPW}, there is a relative quasi-isometry 
$$\phi_1:\bar{(\mathsf{Z}_1,\{v_3,r_1\})}\to\bar{(\mathsf{Z}',\{v_1,v_2,r_1\})}$$
such that the restrictions of $\phi_1$ and $\phi_1^{-1}$ to copies of $\bar{(\Gamma,v)}$ are relative isometries and in particular, $\phi_1$ induces a bijection between the preimage of $v_3$ in $\mathsf{Z}_1$ and the preimage of $\{v_1,v_2\}$ in $\mathsf{Z}'$. 
By Lemma~\ref{Lem:PWLem5}, there is a relative quasi-isometry 
$$\phi_2:\bar{(\mathsf{Z}',\{v_1,v_2\})}\to\bar{(\mathsf{Z},\{v_1,v_2\})}$$
such that the restrictions of $\phi_2$ and $\phi_2^{-1}$ to copies of $\bar{(\Gamma,v)}$ are relative isometries.
Therefore, $\phi_2\circ\phi_1:\bar{(\mathsf{Z}_1,\{v_3\})}\to\bar{(\mathsf{Z},\{v_1,v_2\})}$ is the desired relative quasi-isometry which completes the proof.
\end{proof}

By combining all the results in this subsection, we have the following conclusion.

\begin{lemma}\label{Lemma:PWConclusion}
Consider two finite sets of pairs of finite graphs with base vertices
\begin{align*}
\mathcal{A}=\{(\Gamma_1,v_1),\dots,(\Gamma_n,v_n)\} \text{ and } \mathcal{B}=\{(\Lambda_1,w_1),\dots,(\Lambda_m,w_m)\}\text{ for } n,m\ge 2.
\end{align*}
Let $\mathcal{P}$ be a set of relative quasi-isometries $\bar\varphi_{i,j}:\bar{(\Gamma_i,v_i)}\to\bar{(\Lambda_j,w_j)}$ for some pairs of $i$ and $j$ such that any $i\in I$ or $j\in J$ appears in at least one pair.

Then there is a relative quasi-isometry $\phi:\bar{(*\mathcal{A},A)}\to\bar{(*\mathcal{B},B)}$ 
for $A=\{v_1,\dots,v_n\}\subset\ast\mathcal{A}$ and $B=\{w_1,\dots,w_m\}\subset\ast\mathcal{B}$ such that the restriction of $\phi$ to each copy of $\bar{(\Gamma_i,v_i)}$ is a relative quasi-isometry in $\mathcal{P}$.

The same holds if $(\mathsf{R}_n,r_n)$ is attached to $*\mathcal{A}$ (or $*\mathcal{B}$) by identifying $r_n$ with the central vertex of the connecting star of the free product.
\end{lemma}
\begin{proof}
The first statement holds if we repeatedly apply Lemmas~\ref{Lemma1.1inPW} and \ref{Lemma1.4inPW}.
The second statement hold if we apply Lemma~\ref{Lem:PWLem5}.
\end{proof}

\section{Large scale geometry of general graph 2-braid groups}\label{section:graph 2-braid groups}
This section is devoted to fundamental facts about the large scale geometry of graph $2$-braid groups. Most of these facts will easily be derived from previously known results or graph-of-groups style decompositions of graph $2$-braid groups.
These facts will suggest us several directions for studying the large-scale geometry of graph $2$-braid groups and our specific choice will be initiated from Section~\ref{section:grapes}.

Before we move on to subsections, we first clarify notations and terminologies for graphs which will be used in this paper.

An \emph{$n$-path} or a path on $(n+1)$ vertices (of length $n$) is denoted by $\mathsf{P}_n$. 
An \emph{$n$-cycle} or a cycle of length $n$ is denoted by $\mathsf{C}_n$.
An \emph{$n$-star} or the cone over discrete $n$ vertices is denoted by $\mathsf{S}_n$.
Given a graph $\Gamma$, a \emph{path} (a \emph{cycle}, resp.) in $\Gamma$ is a (possibly non-induced) subgraph of $\Gamma$ which is isomorphic to an $n$-path (an $n$-cycle, resp.) for some $n$.

Let $\Gamma$ be a graph. For a subgraph $\Lambda\subset\Gamma$, the \emph{complement} $\Lambda^c=\Gamma\setminus \Lambda$ is the subgraph of $\Gamma$ induced by $\mathcal{V}(\Lambda^c)=\mathcal{V}(\Gamma)\setminus\mathcal{V}(\Lambda)$.
When $\Lambda$ is a single vertex $v$, each component of $\Lambda^c$ is called a \emph{$\hat{v}$-component} $\mathsf{A}$ of $\Gamma$ and the subgraph of $\Gamma$ induced by the union of a $\hat{v}$-component and $\{v\}$ is called an \emph{extended $\hat{v}$-component} and denoted by $\mathsf{A}^+$.

\subsection{Configuration spaces of graphs}\label{sec:graph configuration spaces}

Let $\Gamma$ be a finite (possibly non-simple) graph with at least one edge.
The \emph{ordered} and \emph{unordered $n$-configuration spaces} of $\Gamma$ are defined as
\begin{align*}
C_n(\Gamma) &= \{(x_1 ,x_2, \dots, x_n)\in \Gamma^n\mid x_i \neq x_j\text{ if }i \neq j \}\ \ \mathrm{and}\\
UC_n(\Gamma) &= \{\{x_1 ,x_2, \dots, x_n\}\subset \Gamma\mid x_i \neq x_j\text{ if }i \neq j \}
\cong C_n(\Gamma)/\mathbb{S}_n,
\end{align*}
where $\mathbb{S}_n$ is the symmetric group of $n$ letters acting on $\Gamma^n$ by permuting coordinates.
The fundamental groups of $UC_n(\Gamma)$ and $C_n(\Gamma)$ are said to be the \emph{$n$-braid group} and the \emph{pure $n$-braid group} over $\Gamma$, denoted by $\mathbb{B}_n(\Gamma)$ and $\mathbb{PB}_n(\Gamma)$, respectively.
The class of the (pure) $n$-braid groups over graphs is called \emph{(pure) graph $n$-braid groups}.
It is worth noticing that a (pure) graph $n$-braid group is determined by $n$ and the homeomorphism type of the base graph, and any (pure) $n$-braid group over a path graph is trivial.

By using a cell structure on $\Gamma$, we can define cubical version of configuration spaces of $\Gamma$ as follows:
The \emph{ordered} and \emph{unordered discrete $n$-configuration spaces} of $\Gamma$ are defined as
\begin{align*}
D_n(\Gamma) &= \{ (\sigma_1, \dots, \sigma_n)\in \Gamma^n\mid {\sigma_i} \cap {\sigma_j} = \emptyset\text{ if }i \neq j\}\ \ \mathrm{and}\\
UD_n(\Gamma)&=\{ \{\sigma_1, \dots, \sigma_n\} \subset \Gamma\mid \sigma_i \cap \sigma_j = \emptyset\text{ if }i \neq j \}\cong D_n(\Gamma)/\mathbb{S}_n,
\end{align*}
where $\sigma_i$ is either a vertex or an edge in $\Gamma$.
Note that $D_n(\Gamma)$ and so $UD_n(\Gamma)$ admit canonical cube complex structures where each element $(\sigma_1,\dots, \sigma_n)$ or $\{\sigma_1,\dots, \sigma_n\}$ defines a $k$-dimensional cube for the number $k$ of edges among $\sigma_i$'s.

\begin{theorem}[\cite{Abrams00, CW, Gen21}]\label{thm:GBGSpecial}
Both $D_n(\Gamma)$ and $UD_n(\Gamma)$ are special cube complexes. 
\end{theorem}

By definition, the cube structure on $UD_n(\Gamma)$ (or $D_n(\Gamma)$) does depend on the cell structure on $\Gamma$. 
However, its homotopy type will eventually stabilize if $\Gamma$ is sufficiently subdivided.

\begin{theorem}[\cite{Abrams00, KKP12, PS12}]\label{Abrams}
The discrete $n$-configuration space $UD_n(\Gamma)$ ($D_n(\Gamma)$, resp.) is a deformation retract of $UC_n(\Gamma)$ ($D_n(\Gamma)$, resp.) for each $n\ge 1$ if and only if $\Gamma$ is \emph{sufficiently subdivided} as follows:
\begin{enumerate}
\item Each path between two non-bivalent vertices in $\Gamma$ contains at least $n-1$ edges;
\item Each cycle in $\Gamma$ contains at least $n+1$ edges.
\end{enumerate}
In particular, $UC_2(\Gamma)$ $(C_2(\Gamma),\ \textrm{resp.})$ is homotopy equivalent to $UD_2(\Gamma)$ $(D_2(\Gamma),\ \textrm{resp.})$ if and only if $\Gamma$ is simple.
\end{theorem}

One may regard $UD_n(-)$ (and thus $UC_n(-)$) as a functor from the category of finite graphs with inclusions to the category of NPC cube complexes with locally isometric embeddings as follows:
\begin{lemma}[\cite{Abrams00}, Theorem~3.13]\label{Lem:embedding}
For any subgraph $\Gamma'$ of $\Gamma$, the inclusion $\Gamma'\hookrightarrow\Gamma$ induces a locally isometric embedding $\iota:UD_n(\Gamma')\hookrightarrow UD_n(\Gamma)$. In particular, $\mathbb{B}_n(\Gamma)$ contains an undistorted subgroup isomorphic to $\mathbb{B}_n(\Gamma')$.
\end{lemma}

\begin{Ex}
As in Figure~\ref{Fig:tripodFig}, $UD_2(\mathsf{S}_3)\cong\mathsf{C}_6$ and $UD_2(\mathsf{C}_3)\cong\mathsf{C}_3$, and therefore $\mathbb{B}_2(\mathsf{S}_3)\cong\mathbb{B}_2(\mathsf{C}_3)\cong\mathbb{Z}$. 
For any simple graph $\Gamma$ containing either a vertex of valency $\geq 3$ or a $3$-cycle, it contains a subgraph $\Gamma'$ isomorphic to $\mathsf{S}_3$ or $\mathsf{C}_3$, and so by Lemma \ref{Lem:embedding}, we have a locally convex subcomplex of $UD_2(\Gamma)$, which is isomorphic to $\mathsf{C}_6$ or $\mathsf{C}_3$, respectively.
\end{Ex}

\begin{figure}[ht]
\[
\setlength{\arraycolsep}{1pc}
\begin{array}{cccc}
\begin{tikzpicture}[baseline=-.5ex]
\tikzstyle{every node}=[draw,circle,fill=black,minimum size=3pt,inner sep=0pt]
  \node[anchor=center] (n1) at (90:1) [label={[label distance=0cm]above:$a$}]{};
  \node[anchor=center] (n2) at (210:1) [label={left:$b$}]{};
  \node[anchor=center] (n3) at (-30:1)  [label={right:$c$}]{};
  \node[anchor=center] (n4) at (0,0) [label={[label distance=0.05cm]left:$v$}]{};
  \foreach \from/\to in {n1/n4,n4/n2,n4/n3}
  \draw (\from) -- (\to);
\end{tikzpicture}&
\begin{tikzpicture}[baseline=-.5ex]
\tikzstyle{every node}=[draw,circle,fill=black,minimum size=3pt,inner sep=0pt]
  \draw (30:1) node (p1) [label={right:$\{a,b\}$}]{};
  \draw (90:1) node (p2) [label={[label distance=-.3cm]above:$\{a,v\}$}]{};
  \draw (150:1) node (p3) [label={left:$\{a,c\}$}]{};
  \draw (210:1) node (p4) [label={left:$\{c,v\}$}]{};
  \draw (270:1) node (p5) [label={[label distance=-.3cm]below:$\{c,b\}$}]{};
  \draw (330:1) node (p6) [label={right:$\{b,v\}$}]{};
  \foreach \from/\to in {p1/p2,p2/p3,p3/p4,p4/p5,p5/p6,p6/p1}
  \draw (\from) -- (\to);
\end{tikzpicture}&
\begin{tikzpicture}[baseline=-.5ex]
\tikzstyle{every node}=[draw,circle,fill=black,minimum size=3pt,inner sep=0pt]
  \node[anchor=center] (n1) at (-30:1) [label={[label distance=0cm]right:$a$}]{};
  \node[anchor=center] (n2) at (90:1) [label={above:$b$}]{};
  \node[anchor=center] (n3) at (210:1)  [label={left:$c$}]{};
  \foreach \from/\to in {n1/n2,n2/n3,n3/n1}
  \draw (\from) -- (\to);
\end{tikzpicture}&
\begin{tikzpicture}[baseline=-.5ex]
\tikzstyle{every node}=[draw,circle,fill=black,minimum size=3pt,inner sep=0pt]
  \draw (-30:1) node (p1) [label={right:$\{a,b\}$}]{};
  \draw (90:1) node (p2) [label={[label distance=-.3cm]above:$\{a,c\}$}]{};
  \draw (210:1) node (p3) [label={left:$\{b,c\}$}]{};
  \foreach \from/\to in {p1/p2,p2/p3,p3/p1}
  \draw (\from) -- (\to);
\end{tikzpicture}\\
\mathsf{S}_3&
UD_2(\mathsf{S}_3)&
\mathsf{C}_3&
UD_2(\mathsf{C}_3)
\end{array}
\]
\caption{Examples of $UD_2(\Gamma)$.}
\label{Fig:tripodFig}
\end{figure}
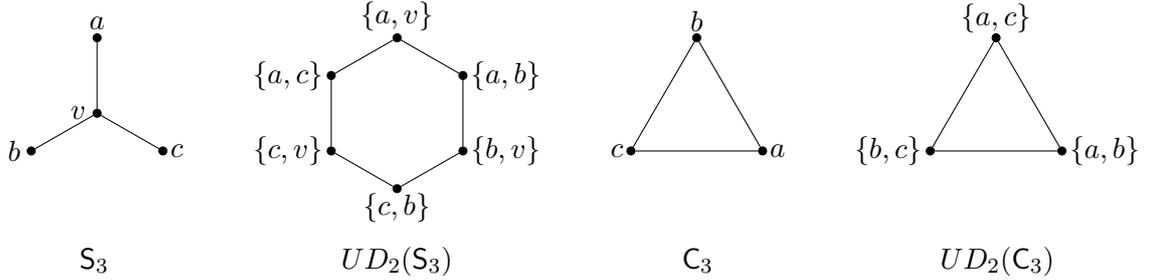

On the other hand, for finite disjoint subgraphs $\Gamma_1, \dots, \Gamma_m$ of a sufficiently subdivided graph $\Gamma$ in the sense of Theorem \ref{Abrams} and nonnegative integers $n_1,\dots, n_m$, there is a locally isometric embedding \[UD_{n_1}(\Gamma_1)\times\dots \times UD_{n_m}(\Gamma_m) \hookrightarrow UD_{n_1+\dots+n_m}(\Gamma).\]
If $n_1=\dots=n_m=1$, then we have a locally isometric embedding $\Gamma_1\times\dots\times\Gamma_m\hookrightarrow UD_m(\Gamma)$, whose image will be denoted by $\Gamma_1\itimes\dots \itimes\Gamma_m$.
In particular, if $m=2$, then the subcomplex $\Gamma_1\itimes\Gamma_2$ is a product subcomplex in $UD_2(\Gamma)$ as defined in Definition~\ref{Def:ProductSubcomplex}.

\begin{proposition}\label{Prop:FreeAbelianSubgroup}
Suppose that $\Gamma$ has disjoint $n$ cycles $\mathsf{C}^1,\dots\mathsf{C}^n$. For $m\ge n$, then $\mathbb{B}_m(\Gamma)$ contains an undistorted subgroup isomorphic to $\mathbb{Z}^n$. In particular, it is not quasi-isometric to a free group.
\end{proposition}
\begin{proof}
With the suitably subdivided assumption, there is a locally isometric embedding \[\mathsf{C}^1\times\dots\times\mathsf{C}^n\times\{v_{n+1}\}\times\dots\times\{v_m\}\hookrightarrow UD_m(\Gamma),\] which means that $\mathbb{B}_m(\Gamma)$ contains an undistorted subgroup isomorphic to $\mathbb{Z}^n$.
Therefore, $\mathbb{B}_m(\Gamma)$ is not quasi-isometric to a free group since there is no quasi-isometric embedding from $\mathbb{Z}^n$ to any free group.
One way to see this is that if there is a quasi-isometric embedding from $\mathbb{Z}^2$ to a free group, then $\mathbb{Z}^n$ must be $\delta$-hyperbolic in the sense of Gromov \cite[Theorem 1.9 in Chapter III.H]{BH}, which is a contradiction.
\end{proof}

\subsection{Graph 2-braid groups}\label{section:properties}
Now, we focus on graph $2$-braid groups.
In the remaining of this section, a graph $\Gamma$ is assumed to be connected (following Assumption \ref{assumption:cube}), simple (based on Theorem~\ref{Abrams}), finite (except the universal cover of such a graph) and not homeomorphic to a path graph.
Later, we additionally assume that $\Gamma$ has no leaves since the existence of leaves only effect on the rank of the free factor of the graph $2$-braid group as follows. 

\begin{lemma}\label{NoHangingEdge}
Let $\Gamma'$ be a graph with a leaf $w$ and $v\in\Gamma'$ be the vertex adjacent to $w$.
Let $\Gamma$ be the complement of $w$ in $\Gamma'$.
Then $\mathbb{B}_2(\Gamma')$ is the free product of $\mathbb{B}_2(\Gamma)$ and the free group $\mathbb{F}_{N}$ of rank $N$ where $N=\val_\Gamma(v)-1$.
\end{lemma}
Recall that $\lk_\Gamma(v)$ is the set of vertices adjacent to $v$ in $\Gamma$ (Example~\ref{Ex:Graphs}).
\begin{proof}
According to whether a point is lying on $e=[v,w]$ or not, we have a decomposition
\begin{equation}\label{eq:Decomposition}
UD_2(\Gamma')=UD_2(\Gamma)\bigcup_{(\Gamma\setminus\{v\})\itimes v}(\Gamma\setminus\{v\})\itimes e\bigcup_{(\Gamma\setminus\{v\})\itimes w} \Gamma\itimes w.      
\end{equation}
Since there is a homotopy equivalence 
\[
(\Gamma\setminus\{v\})\itimes e \bigcup_{(\Gamma\setminus\{v\})\itimes w } \Gamma\itimes w\quad \simeq\quad
(\Gamma\setminus\{v\})\itimes v \bigcup_{\lk_\Gamma(v)\itimes v} \operatorname{Cone}(\lk_\Gamma(v))
\]
relative to $(\Gamma\setminus\{v\})\itimes v$, we have 
\[
\mathbb{B}_2(\Gamma') \cong \mathbb{B}_2(\Gamma)\ast \pi_1(S(\lk_\Gamma(v)))
\cong \mathbb{B}_2(\Gamma)\ast \mathbb{F}_{N},
\]
where $S(\lk_\Gamma(v))$ is the suspension of $\lk_\Gamma(v)$ and $N=\val_\Gamma(v)-1$ as desired.
\end{proof}

Lemma~\ref{NoHangingEdge} and its proof have two important meanings for us: 
\begin{enumerate}
\item
The presence of a separating vertex in $\Gamma$ may induce a decomposition of $UD_2(\Gamma)$ and thus a decomposition of $\mathbb{B}_2(\Gamma)$. This idea will be elaborated on in Section~\ref{Section:GOGdecomposition}.
\item By Theorem \ref{PW}, there is no significant difference between the large scale geometry of $\mathbb{B}_2(\Gamma)$ and $\mathbb{B}_2(\Gamma')$ unless $\mathbb{B}_2(\Gamma)\cong\mathbb{Z}$, i.e, $\Gamma$ is either $\mathsf{S}_3$ or $\mathsf{C}_3$ (Corollary~\ref{cor:Z}).
\end{enumerate}

The main theme of the study of the large scale geometry of graph $2$-braid groups is how graph $2$-braid groups are determined up to quasi-isometry from their defining graphs. 
As a first step, let us see how the (non-)planarity of graphs affects on this theme.

\begin{proposition}\label{Non-planar}
If $\Gamma$ is non-planar, then $\mathbb{B}_2(\Gamma)$ has an undistorted subgroup isomorphic to the fundamental group of a closed hyperbolic surface. 
In particular, $\mathbb{B}_2(\Gamma)$ is not quasi-isometric to a free group.
\end{proposition}
\begin{proof}
By Kuratowski's theorem, a graph is non-planar if and only if it contains a subgraph that is a subdivision of the complete bipartite graph $K_{3,3}$ or the complete graph $K_5$.
Since $UD_2(K_{3,3})$ and $UD_2(K_5)$ are homeomorphic to closed hyperbolic surfaces, by Lemma~\ref{Lem:embedding}, $\mathbb{B}_2(\Gamma)$ has an undistorted subgroup isomorphic to the fundamental group of a closed hyperbolic surface.

Assume that $\mathbb{B}_2(\Gamma)$ is quasi-isometric to a free group.
Since $\mathbb{B}_2(\Gamma)$ contains an undistorted subgroup isomorphic to the fundamental group of a closed hyperbolic surface, there must exist a quasi-isometric embedding from the hyperbolic plane to a free group (equipped with the word metric).
However, this is a contradiction by \cite[Theorem 1]{BK05}. Therefore, the last statement holds.
\end{proof}

It is not known whether there exists a planar $\Gamma$ such that $\mathbb{B}_2(\Gamma)$ has an undistorted subgroup isomorphic to the fundamental group of a closed hyperbolic surface.

If $\Gamma$ is planer, after fixing an embedding $\Gamma\hookrightarrow\mathbb{R}^2$, we define a \emph{boundary cycle} as the boundary of the closure of a bounded component of $\mathbb{R}^2\setminus \Gamma$.
As a corollary of Lemma~\ref{NoHangingEdge}, we can deduce that if $\Gamma$ has no boundary cycles, i.e., $\Gamma$ is a tree, then $\mathbb{B}_2(\Gamma)$ is a free group. The following is a generalization of this fact and Proposition \ref{Prop:FreeAbelianSubgroup} in graph $2$-braid groups.

\begin{theorem}[\cite{KP12}, Theorem~4.8]\label{scrg}
If $\Gamma$ is planar, then $\mathbb{B}_2(\Gamma)$ admits a group presentation whose relators are commutators corresponding to unordered pairs of disjoint boundary cycles.
In particular, if $\Gamma$ has no pair of disjoint boundary cycles, then $\mathbb{B}_2(\Gamma)$ is isomorphic to a free group.
\end{theorem}

\begin{Ex}\label{Ex:FreeGroupCase}
For integers $n,\ell\ge 0$ with $n+\ell\ge 1$, let $\Lambda=\Lambda(n,\ell)$ be a graph obtained from an $n$-star $\mathsf{S}_n$ with central vertex $v$ by attaching $\ell$ copies of $3$-cycles at $v$.
That is, 
\[
\Lambda=\left(\mathsf{S}_n\coprod_{i=1}^\ell\mathsf{C}_3^i\right)\bigg/ v\sim 0^i,
\] 
where $\mathsf{C}_3^i$ is the $i$-th copy of a 3-cycle $\mathsf{C}_3$ and $0^i$ is the designated vertex of $\mathsf{C}_3^i$.
By \cite{KP12}, $\mathbb{B}_2(\Lambda)$ is a free group of rank $N$, where
\[
N(n,\ell)=(n+\ell)(n+2\ell-2)-\binom{n+\ell}{2}+1
=\frac{(n+\ell) ( n+3\ell-3 )}{2} + 1\ge 0.
\]
Note that
\begin{enumerate}
\item $N=0$ if and only if $1\le n\le 2$ and $\ell=0$, namely, $\Gamma$ is a path graph $\mathsf{P}_1$ or $\mathsf{P}_2$.
\item $N=1$ if and only if $(n,\ell)=(3,0)$ or $(0,1)$, namely, $\Gamma$ is either $\mathsf{S}_3$ or $\mathsf{C}_3$.
\end{enumerate}
\end{Ex}

\begin{Ex}\label{example:trees}
For a tree $\mathsf{T}$, by \cite{KP12} or Lemma~\ref{NoHangingEdge}, $\mathbb{B}_2(\mathsf{T})$ is a free group of rank $N$, where
\[
N=\sum_{v\in \mathcal{V}(\mathsf{T})} N(\val_{\mathsf{T}}(v),0)=
\sum_{v\in \mathcal{V}(\mathsf{T})} \binom{\val_{\mathsf{T}}(v)-1}{2}\ge \#\{v\in\mathcal{V}(\mathsf{T})\mid \val_{\mathsf{T}}(v)\ge 3\}.
\]
\end{Ex}

\begin{corollary}\label{cor:Z}
The braid group $\mathbb{B}_2(\Gamma)$ is isomorphic to $\mathbb{Z}$ if and only if $\Gamma$ is homeomorphic to either $\mathsf{S}_3$ or $\mathsf{C}_3$.
\end{corollary}
\begin{proof}
We only need to prove the `only if' direction.

Up to subdivision, we may choose a maximal tree $\mathsf{T}$ of $\Gamma$ such that $\val_{\mathsf{T}}(v)\ge 3$ if and only if $\val_{\Gamma}(v)\ge 3$.
Then by Lemma~\ref{Lem:embedding} and Example~\ref{example:trees}, $\mathsf{T}$ has at most one vertex of valency $\ge 3$ and so does $\Gamma$.
By Example~\ref{Ex:FreeGroupCase}, $\Gamma$ is homeomorphic to either $\mathsf{S}_3$ or $\mathsf{C}_3$.
\end{proof}

In summary, we have the following theorem completely classifying graph $2$-braid groups which are quasi-isometric to free groups.
\begin{theorem}\label{theorem:freegroupcase}
Let $\Gamma$ be a graph. Then $\mathbb{B}_2(\Gamma)$ is quasi-isometric to a free group if and only if $\mathbb{B}_2(\Gamma)$ is isomorphic to a free group if and only if $\Gamma$ is a planar graph without two disjoint (boundary) cycles.
\end{theorem}
\begin{proof}
Suppose that $\mathbb{B}_2(\Gamma)$ is quasi-isometric to a free group.
By Propositions \ref{Prop:FreeAbelianSubgroup} and \ref{Non-planar}, $\Gamma$ must be planar and have no pair of disjoint cycles. And then, Theorem \ref{scrg} implies that $\mathbb{B}_2(\Gamma)$ is a free group. Therefore, the theorem holds. 
\end{proof}

Let us denote by $\mathcal{F}$ the set of graphs which are planar and have no pair of disjoint cycles.
By Theorem~\ref{theorem:freegroupcase}, then it is quasi-isometrically rigid up to the $2$-braid groups in the sense that if a graph $2$-braid group is quasi-isometric to a free group, then its defining graph belongs to $\mathcal{F}$.

\begin{remark}
Indeed, (simple) graphs in $\mathcal{F}$ are completely classified in a combinatorial way by Lov\'asz \cite{Lovasz}, which are graphs satisfying the following: there exist at most three vertices such that any (induced) cycle contains at least one such vertex.
\end{remark}

\subsubsection{The union of maximal product subcomplexes}
By Theorem \ref{Abrams}, the discrete (unordered) configuration space $UD_2(\Gamma)$ is a (weakly) special square complex and therefore $\cI(\bar{UD_2(\Gamma)})$ and $\cRI(UD_2(\Gamma))$ are definable. In order to see their structures inherited from $\Gamma$, let us see how standard product subcomplexes of $UD_2(\Gamma)$ and $\bar{UD_2(\Gamma)}$ look like.

By the paragraph above Proposition~\ref{Prop:FreeAbelianSubgroup}, we know that a pair of disjoint subgraphs $\Gamma_1$ and $\Gamma_2$ induces a product subcomplex $\Gamma_1\itimes\Gamma_2$ of $UD_2(\Gamma)$. Indeed, this is the only way to construct product subcomplexes of $UD_2(\Gamma)$.

\begin{lemma}\label{Lem:SPSinGBG}
A subcomplex $K\subset UD_2(\Gamma)$ is a product subcomplex if and only if there exists a pair of disjoint subgraphs $\Gamma_1$ and $\Gamma_2$ of $\Gamma$ such that $K=\Gamma_1\itimes\Gamma_2$.
In this case, any $p$-lift of $K$ is exactly a component of the preimage $p_{UD_2(\Gamma)}^{-1}(K)$ of $K$.
\end{lemma}
Note that by the definition of $UD_2(\Gamma)$, for an edge $\sigma\itimes\{v\}$ in $UD_2(\Gamma)$, the carrier of the hyperplane dual to $\sigma\itimes\{v\}$ is $\sigma\itimes\mathsf{H}$, where $v$ is a vertex in $\Gamma$, $\sigma$ is an edge of $\Gamma$ and $\mathsf{H}$ is a component of $\Gamma\setminus \sigma$ which contains $v$.
It follows that if the carrier of two hyperplanes $\sigma_1\itimes\mathsf{H}_1$ and $\sigma_2\itimes\mathsf{H}_2$ intersect, then $\sigma_1$ and $\sigma_2$ are disjoint and $\mathsf{H}_i$ contains $\sigma_j$ for $\{i,j\}=\{1,2\}$.

\begin{proof}
We only need to prove the `only if' direction.

Let $K\subset Y$ be a product subcomplex with a product structure $\iota:\Lambda_1\times\Lambda_2\to UD_2(\Gamma)$.
Since $\iota$ is a local isometry, the carriers of hyperplanes are mapped by $\iota$ into the carriers of hyperplanes which are of form $\sigma\itimes\mathsf{H}$ as above.
It follows by the definition of a product structure that for any vertex $y_j\in\Lambda_2$, the restriction of $\iota$ to $\Lambda_1\times y_j$ is an isomorphism onto $\Gamma^j_1\itimes w_j$, where $w_j$ is a vertex in $\Gamma$ and $\Gamma^j_1$ is a subgraph of a component of $\Gamma\setminus w_j$.
Moreover, the restriction of $\iota$ to $\Lambda_1\times e'$ must be an immersion for any edge $e'$ of $\Lambda_2$. It means that $\Gamma^j_1\itimes w_j$ and $\Gamma^{j'}_1\itimes w_{j'}$ are parallel 
and thus $\Gamma^j_1$ and $\Gamma^{j'}_1$ are identical for any $j,j'$.
Hence, there exists a subgraph $\Gamma_1$ of $\Gamma$ such that for any vertex $y\in\Lambda_2$, the restriction of $\iota$ to $\Lambda_1\times y$ is an isomorphism onto $\Gamma_1\itimes w$ for some vertex $w\in\Gamma$.
Similarly, there exists a subgraph $\Gamma_2$ of $\Gamma$ such that for any $x\in\Lambda_1$, the restriction of $\iota$ to $x\times \Lambda_2$ is an isomorphism onto $v\itimes\Gamma_2$ for some $v\in\Gamma$.

Since any hyperplane intersecting $\iota(\Lambda_1\times y)$ intersects a hyperplane intersecting $\iota(x\times \Lambda_2)$ and vice-versa, $\Gamma_1$ and $\Gamma_2$ must be disjoint.
Thus the restriction of $\iota$ to the $1$-skeleton of $\Lambda_1\times\Lambda_2$ is an isomorphism onto the $1$-skeleton of $\Gamma_1\itimes\Gamma_2$. 
Since $\iota$ is a local isometry, the image $K$ of $\iota$ must be equal to $\Gamma_1\itimes\Gamma_2$. 

Indeed, we have shown that $\iota$ is a locally isometric embedding.
Since a $p$-lift $\bar{K}$ of $K$ is the image of an elevation $\bar\iota:\bar{\Lambda_1}\times\bar{\Lambda_2}\to \bar{UD_2(\Gamma)}$ of $\iota$, $\bar K$ must be a component of the preimage of $K$ under $p_{UD_2(\Gamma)}$.
\end{proof}

\begin{corollary}\label{corollary:maximal}
A product subcomplex $K=\Gamma_1\itimes\Gamma_2\subset UD_2(\Gamma)$ is a standard product subcomplex if and only if each of $\Gamma_1$ and $\Gamma_2$ has no leaves.
In this case, $K$ is maximal if and only if there is no pair of two disjoint subgraphs $\Gamma_1'$ and $\Gamma_2'$ without leaves such that $\Gamma_1\subset\Gamma_1'$, $\Gamma_2\subset\Gamma_2'$, and at least one of $\Gamma_i$ is a proper subgraph of $\Gamma_i'$.
\end{corollary}
\begin{proof}
The proof follows from the definition of standard product subcomplexes, Lemma~\ref{Lem:SPSinGBG} and Lemma~\ref{InclusionBetweenSPSes}.
\end{proof}

\begin{definition}\label{Def:UP_2(Gamma)}
We define a subcomplex $UP_2(\Gamma)$ of $UD_2(\Gamma)$ as the union of all maximal product subcomplexes of $UD_2(\Gamma)$.
\end{definition}

It can easily be seen that there is no pair of disjoint cycles of $\Gamma$ if and only if $UP_2(\Gamma)$ is empty. 
If $UP_2(\Gamma)$ is not empty, then it turns out to be a connected subcomplex of $UD_2(\Gamma)$ (thus Assumption~\ref{assumption:cube} still holds for $UP_2(\Gamma)$)  which is special.

\begin{lemma}\label{Lem:UP_2Special}
If $UP_2(\Gamma)$ is non-empty, then it is a connected special square complex.
\end{lemma}
\begin{proof}
It is obvious that $UP_2(\Gamma)$ is a square complex. By Lemma~\ref{lem:SpecialSubcomplexes}, moreover, $UP_2(\Gamma)$ is special. Thus, it only remains to show that $UP_2(\Gamma)$ is connected. In order for this, we will show that any two maximal product subcomplexes $\Gamma_1\itimes\Gamma_2$ and $\Gamma_1'\itimes\Gamma_2'$ intersect using the following observation:
\begin{observation}
Let $\Lambda_1$ and $\Lambda_2$ be two subgraphs of $\Gamma$ without leaves. 
\begin{enumerate}
\item\label{union} 
If they meet, then their union is also a subgraph without leaves. 
\item\label{unionwithpath}
Otherwise, for any path $\mathsf{P}$ joining $\Lambda_1$ to $\Lambda_2$ such that $\mathsf{P}$ meets $\Lambda_i$ only at its endpoint, the union of $\Lambda_1$, $\Lambda_2$ and $\mathsf{P}$ is a subgraph without leaves. 
\end{enumerate} 
\end{observation}

Assume that $\Gamma_1$ meets neither $\Gamma'_1$ nor $\Gamma'_2$.
Since $\Gamma$ is connected, there exists a path $\mathsf{P}$ from $\Gamma_1$ to $\Gamma'_i$ without intersecting $\Gamma'_{j}$ for some $i$ with $j\neq i$, say $i=1$.
By Observation (\ref{unionwithpath}), $\Gamma_1\cup\mathsf{P}\cup\Gamma'_1$ is a graph without leaves which properly contains $\Gamma'_1$. 
It means that $(\Gamma_1\cup\mathsf{P}\cup\Gamma'_1)\itimes\Gamma'_{2}$ is a standard product subcomplex properly containing $\Gamma'_1\itimes\Gamma'_2$, which is a contradiction. 
Thus, $\Gamma_1$ meets at least one of $\Gamma'_1$ and $\Gamma'_2$, and similarly for $\Gamma_2$.

Suppose that $\Gamma_1$ meets both $\Gamma'_1$ and $\Gamma'_2$. Assume that $\Gamma_2$ meets only one of $\Gamma'_1$ and $\Gamma'_2$, say $\Gamma'_2$. Then $\Gamma'_2$ is not properly contained in either $\Gamma_1$ or $\Gamma_2$ since $\Gamma'_2$ meets both $\Gamma_1$ and $\Gamma_2$.
If $\Gamma'_1$ is contained in $\Gamma_1$, by Observation (\ref{union}), $\Gamma'_1\itimes(\Gamma'_2\cup\Gamma_2)$ is a standard product subcomplex which properly contains $\Gamma'_1\itimes\Gamma'_2$, a contradiction.
Otherwise, again by Observation (\ref{union}), $(\Gamma_1\cup\Gamma'_1)\itimes\Gamma_2$ is a standard product subcomplex which properly contains $\Gamma_1\itimes\Gamma_2$, again a contradiction. It follows that $\Gamma_2$ must meet both $\Gamma'_1$ and $\Gamma'_2$, and thus $\Gamma_1\itimes\Gamma_2$ and $\Gamma_1'\itimes\Gamma_2'$ intersect.

Suppose that $\Gamma_1$ meets only one of $\Gamma'_1$ and $\Gamma'_2$, say $\Gamma'_1$. If $\Gamma_2$ does not meet $\Gamma'_2$ as well, then $\Gamma'_1\cup\Gamma_1\cup\Gamma_2$ must coincide with $\Gamma'_1$ since $\Gamma'_1\itimes\Gamma'_2$ is maximal and $\Gamma'_2$ does not meet $\Gamma_1$, $\Gamma_2$ or $\Gamma'_1$. But it is a contradiction since $\Gamma_1\itimes\Gamma_2$ is also maximal. It follows that $\Gamma_2$ meets $\Gamma'_2$ and therefore $\Gamma_1\itimes\Gamma_2$ and $\Gamma_1'\itimes\Gamma_2'$ intersect.
\end{proof}

However, we do not know if $UP_2(\Gamma)$ is always locally convex in $UD_2(\Gamma)$. The point is that though $\Lk(u\itimes v,UP_2(\Gamma)))$ is triangle-free, it may not be an induced subcomplex in $\Lk(u\itimes v,UD_2(\Gamma)))$. In particular, we can not say that $\pi_1(UP_2(\Gamma))$ is a subgroup of $\mathbb{B}_2(\Gamma)$ neither. See Question~\ref{Qn:LocalConvexity}.

\begin{remark}\label{Rmk:LocallyConvex}
In the next section, we will see that if $\Gamma$ is the minimal simplicial representative of a graph with circumference one (which is previously mentioned in Remark~\ref{Rmk:Developability}), then the embedding $UP_2(\Gamma)\hookrightarrow UD_2(\Gamma)$ is a local isometry and thus $\pi_1(UP_2(\Gamma))$ can be considered as a subgroup of $\mathbb{B}_2(\Gamma)$ (Proposition \ref{Prop:FreeFactor}).
\end{remark}

\subsection{Graph-of-groups decomposition}\label{Section:GOGdecomposition}
This subsection is devoted to a graph-of-groups decomposition of the $2$-braid group over a graph having a separating vertex.

Suppose that $\Gamma$ has no leaves. For a vertex $v$ of $\Gamma$, let $\mathcal{C}(\Gamma,v)=\{\Gamma_i\mid i\in I\}$ and $\mathcal{C}^+(\Gamma,v)=\{\Gamma_i^+\mid i\in I\}$
be the sets of $\hat{v}$-components and extended $\hat{v}$-components of $\Gamma$, respectively.
Suppose that $\pi_1(\Gamma_i)$ is non-trivial if and only if $i\in I'\subseteq I$, and define subcomplexes of $UP_2(\Gamma)$ and $UD_2(\Gamma)$,
\[
UP_2(\Gamma, v)=\bigcup_{i\in I'} \Gamma_i\itimes\Gamma_i^c\quad\text{ and }\quad
UD_2(\Gamma, v) = \left(\bigcup_{i\in I'}{UD_2(\Gamma^+_i)}\right)\cup UP_2(\Gamma, v),
\]
respectively.
Note that $|I'|=0$ if and only if $UP_2(\Gamma,v)$ and $UD_2(\Gamma,v)$ are empty sets; in this case, $\Gamma$ must belong to $\mathcal{F}$.
Otherwise, both are special square complexes by Lemma~\ref{lem:SpecialSubcomplexes}.

\begin{lemma}\label{Lem:D'}
Let $\mathbb{B}_2(\Gamma,v)=\pi_1(UD_2(\Gamma,v))$.
Then for some $N\ge 0$, there is an isomorphism
\[
\mathbb{B}_2(\Gamma) \cong \mathbb{B}_2(\Gamma,v) \ast \mathbb{F}_N.
\]
Moreover, $N>0$ if there exists $i\in I\setminus I'$ such that $\pi_1(\Gamma^+_i)$ is non-trivial.
\end{lemma}
\begin{proof}
Let us consider a decomposition of $UD_2(\Gamma)$ as follows:
\begin{align*}
UD_2(\Gamma) &= \left(\bigcup_{\substack{i,j\in I\\ i\neq j}} \Gamma_i\itimes \Gamma_j^+\right)
\cup
\left(\bigcup_{i\in I} UD_2(\Gamma_i^+)\right)
\\
&=
\left(\bigcup_{\substack{i,j\in I\setminus I'\\i\neq j}} \Gamma_{i}\itimes\Gamma_j^+\right)\cup
\underbrace{\left(\bigcup_{j\in I',\ i\in I\setminus I'}\Gamma_{i}\itimes\Gamma_j^+\right)\cup
\underbrace{\left(\bigcup_{i\in I\setminus I'} UD_2(\Gamma_{i}^+)\right)\cup UD_2(\Gamma,v)}_X}_Y.
\end{align*}
In particular, if $I'=I$, then $UD_2(\Gamma,v)=UD_2(\Gamma)$.
For $i\neq j$ and $i'\neq j'$, 
we observe the intersections
\begin{enumerate}
\item\label{enum:intersection1} $UD_2(\Gamma^+_i)\cap UD_2(\Gamma^+_j)=\varnothing$,
\item\label{enum:intersection2} $(\Gamma_i\itimes\Gamma_j^+)\cap UD_2(\Gamma^+_{i'})=\begin{cases}
\Gamma_i\itimes\{v\} & i=i';\\
\varnothing & i\neq i',
\end{cases}$
\item\label{enum:intersection3} $(\Gamma_i\itimes\Gamma_j^+)\cap(\Gamma_{i'}\itimes\Gamma_{j'}^+)=
\begin{cases}
\Gamma_i\itimes\{v\} & i=i', j\neq j';\\
\Gamma_i\itimes\Gamma_j & i=j', j=i';\\
\varnothing & \text{otherwise}.
\end{cases}$
\end{enumerate}
In particular, these are either empty or $\pi_1$-trivial if $i,j\in I\setminus I'$.

If $i\in I\setminus I'$, then $\Gamma_i^+$ is a tree and hence $\mathbb{B}_2(\Gamma_i^+)$ is free or trivial by Theorem~\ref{theorem:freegroupcase} and Example~\ref{Ex:FreeGroupCase}.
It follows by \eqref{enum:intersection1} and \eqref{enum:intersection2} above that for some $N_0\ge 0$, we have
\[
\pi_1(X)=\mathbb{B}_2(\Gamma,v)\ast\left(*_{i\in I\setminus I'}\mathbb{B}_2(\Gamma_{i}^+)\right)\cong \mathbb{B}_2(\Gamma,v)\ast \mathbb{F}_{N_0}.
\]

If $j\in I'$ and $i\in I\setminus I'$, then $(\Gamma_i\itimes\Gamma_j^+)\cap X=(\Gamma_i\itimes \Gamma_j^+)\cap UD_2(\Gamma,v)=\Gamma_i\itimes\Gamma_j$ by \eqref{enum:intersection3} above, and hence we have
\[
Y=\left(\bigcup_{j\in I',\ i\in I\setminus I'} \Gamma_i\itimes E_j\right)\cup X,
\]
where $E_j$ is the union of edges in $\mathcal{E}(\Gamma_j^+)\setminus\mathcal{E}(\Gamma_j)$. (Note that $E_j$ is $\pi_1$-trivial since $\Gamma$ is assumed to be simple.)
It follows by the $\pi_1$-triviality of $\Gamma_i\itimes E_j$ for such a pair of $i$ and $j$ that
$\pi_1(Y)=\pi_1(X)* \mathbb{F}_{N_1}$ for some $N_1\ge 0$.

Finally, if $i,j\in I\setminus I'$ with $i\neq j$ and $i',j'\in I \setminus I'$ with $i'\neq j'$, the intersections of $\Gamma_i\itimes\Gamma_j^+$ with $Y$ and with $\Gamma_{i'}\itimes \Gamma_{j'}^+$ are disjoint unions of $\pi_1$-trivial complexes by \eqref{enum:intersection2} and \eqref{enum:intersection3}. 
Since $\Gamma_i\itimes\Gamma_j$ is $\pi_1$-trivial for such a pair of $i$ and $j$, we therefore have $\mathbb{B}_2(\Gamma)=\pi_1(Y)*\mathbb{F}_{N_2}$ for some $N_2\ge 0$.

If there exists $i\in I\setminus I'$ such that $\pi_1(\Gamma^+_i)$ is non-trivial, then $N_0$ must be non-zero. Therefore, the `moreover' part automatically holds. 
\end{proof}

Since we have inclusions
\[
\begin{tikzcd}
\pi_1(UP_2(\Gamma,v)) & \pi_1(\Gamma_i)\cong\pi_1(\Gamma_i\itimes\{v\})
\arrow[from=l, hookleftarrow, "\alpha_i"]
\arrow[r, hookrightarrow, "\beta_i"] & \mathbb{B}_2(\Gamma_i^+)
\end{tikzcd}
\]
for $i\in I'$ and $(\Gamma_i\itimes\{v\})\cap(\Gamma_{i'}\itimes \{v\})=\emptyset$ for $i,i'\in I'$ with $i\neq i'$, we have
\[
\mathbb{B}_2(\Gamma,v)\cong 
\left(\mathop{\ast}_{i\in I'} \mathbb{B}_2(\Gamma_i^+)\right)\ast
\pi_1(UP_2(\Gamma,v))
\bigg/\sim,
\]
where $\sim$ identifies the images under $\alpha_{i}$ and $\beta_{i}$. 
Hence the embeddings 
\[UP_2(\Gamma,v)\hookrightarrow UD_2(\Gamma,v)\quad\text{and}\quad UD_2(\Gamma^+_i)\hookrightarrow UD_2(\Gamma,v)\text{ for $i\in I'$}\] induce $\pi_1$-injective homomorphisms, and so $\bar{UP_2(\Gamma,v)}$ and $\bar{UD_2(\Gamma_i^+)}$ can be considered as subcomplexes of $\bar{UD_2(\Gamma,v)}$.

Now, there is a \emph{graph-of-groups} decomposition $\mathcal{G}(\mathbb{B}_2(\Gamma,v))$ defined as follows: as $\mathbb{B}_2(\Gamma,v)=\pi_1(UD_2(\Gamma,v))$ is generated by subgroups
\[\mathbb{B}_{1,1}(\Gamma_i,\Gamma_i^c)\cong\pi_1(\Gamma_i)\times\pi_1(\Gamma_i^c)\quad\text{or}\quad\mathbb{B}_2(\Gamma_i^+)=\pi_1(UD_2(\Gamma_i^+))\]
for $i\in I'$, these subgroups will be the set of vertex groups.
Each edge group is the intersection of the vertex groups, which are either
\[\mathbb{B}_{1,1}(\Gamma_i,\Gamma_j)=\mathbb{B}_{1,1}(\Gamma_i,\Gamma_i^c)\cap\mathbb{B}_{1,1}(\Gamma_j,\Gamma_j^c)\quad\text{or}\quad \mathbb{B}_{1,1}(\Gamma_i,\{v\})=
\mathbb{B}_{1,1}(\Gamma_i,\Gamma_i^c)\cap\mathbb{B}_2(\Gamma_i^+)\cong\pi_1(\Gamma_i)\]
for $i\in I'$ and $j\in I'\setminus\{i\}$.
Then the underlying graph is isomorphic to the complete $k$-graph with one edge attached at each vertex where $k=|I'|$; see Figure~\ref{Fig:GraphofGroupsDecomposition} when $k=4$. 

By the construction of $\mathcal{G}(\mathbb{B}_2(\Gamma,v))$, the (sub)graph-of-groups of $\mathcal{G}(\mathbb{B}_2(\Gamma,v))$ corresponding to the complete $k$-graph (presented in blue in Figure~\ref{Fig:GraphofGroupsDecomposition}) is a graph-of-groups decomposition $\mathcal{G}(\pi_1(UP_2(\Gamma,v)))$ of $\pi_1(UP_2(\Gamma,v))$.
Let us denote the Bass-Serre trees associated to $\mathcal{G}(\mathbb{B}_2(\Gamma,v))$ and $\mathcal{G}(\pi_1(UP_2(\Gamma,v)))$ by $\mathcal{T}(\mathbb{B}_2(\Gamma,v))$ and $\mathcal{T}(\pi_1(UP_2(\Gamma,v)))$, respectively.

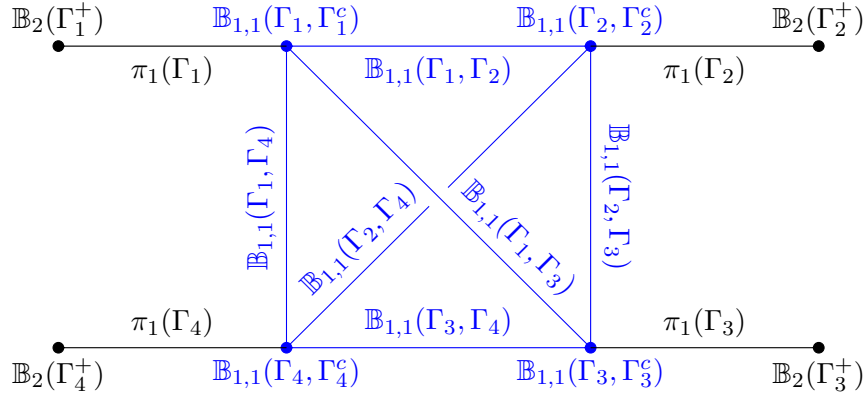
\begin{figure}[ht]
\[
\begin{tikzpicture}[baseline=-.5ex]
\begin{scope}[blue]
\draw (-2,-2) -- node[pos=0.3,above, sloped] {$\mathbb{B}_{1,1}(\Gamma_2,\Gamma_4)$} (2,2);
\draw[white, line width=10] (-2,2) -- (2,-2);
\draw (-2,2) -- node[pos=0.7,above, sloped] {$\mathbb{B}_{1,1}(\Gamma_1,\Gamma_3)$} (2,-2);
\draw[fill] (-2,2) circle (2pt) node[above] {$\mathbb{B}_{1,1}(\Gamma_1,\Gamma_1^c)$} -- node[midway, below] {$\mathbb{B}_{1,1}(\Gamma_1,\Gamma_2)$} (2,2) circle (2pt) node[above] {$\mathbb{B}_{1,1}(\Gamma_2,\Gamma_2^c)$} -- node[midway, above, sloped] {$\mathbb{B}_{1,1}(\Gamma_2, \Gamma_3)$} (2,-2) circle (2pt) node[below] {$\mathbb{B}_{1,1}(\Gamma_3,\Gamma_3^c)$} -- node[midway, above] {$\mathbb{B}_{1,1}(\Gamma_3, \Gamma_4)$} (-2,-2) circle (2pt) node[below] {$\mathbb{B}_{1,1}(\Gamma_4,\Gamma_4^c)$} -- node[midway, above, sloped] {$\mathbb{B}_{1,1}(\Gamma_1,\Gamma_4)$} (-2,2);
\end{scope}
\draw[fill] (-5, 2) circle (2pt) node[above] {$\mathbb{B}_2(\Gamma_1^+)$} -- node[midway,below] {$\pi_1(\Gamma_1)$} (-2,2);
\draw[fill] (5, 2) circle (2pt) node[above] {$\mathbb{B}_2(\Gamma_2^+)$} -- node[midway,below] {$\pi_1(\Gamma_2)$} (2,2);
\draw[fill] (-5, -2) circle (2pt) node[below] {$\mathbb{B}_2(\Gamma_4^+)$} -- node[midway,above] {$\pi_1(\Gamma_4)$} (-2,-2);
\draw[fill] (5, -2) circle (2pt) node[below] {$\mathbb{B}_2(\Gamma_3^+)$} -- node[midway,above] {$\pi_1(\Gamma_3)$} (2,-2);
\end{tikzpicture}
\]
\caption{An example of graph-of-groups decompositions $\mathcal{G}(\mathbb{B}_2(\Gamma,v))$ and $\mathcal{G}(\pi_1(UP_2(\Gamma,v)))$.}
\label{Fig:GraphofGroupsDecomposition}
\end{figure}

If $|I'|\ge 2$, then $\cRI(UP_2(\Gamma,v))$ is non-empty.
In this case, by construction, the graph-of-groups decomposition $\mathcal{G}(\pi_1(UP_2(\Gamma,v)))$ is exactly the same as $\cRI(UP_2(\Gamma,v))$ (considered as a complex of groups based on Theorem~\ref{Thm:TPBCM}) and thus by the theory of Bass-Serre tree, $\cI(\bar{UP_2(\Gamma,v)})$ is isomorphic to the Bass-Serre tree $\mathcal{T}(\pi_1(UP_2(\Gamma,v)))$.
Moreover, infinitely many (disjoint) copies of $\cI(\bar{UP_2(\Gamma,v)})$ sit inside the Bass-Serre tree $\mathcal{T}(\mathbb{B}_2(\Gamma,v))$ associated to $\mathcal{G}(\mathbb{B}_2(\Gamma,v))$.

On the other hand, the Bass-Serre tree $\mathcal{T}(\mathbb{B}_2(\Gamma,v))$ can be inductively constructed as follows:
\begin{enumerate}
\item\label{BSTitem1}
Start with $\cI(\bar{UP_2(\Gamma,v)})$. At each vertex labelled by $\mathbb{B}_{1,1}(\Gamma_i,\Gamma^c_i)$ in $\cI(\bar{UP_2(\Gamma,v)})$, we attach infinitely many edges, labelled by $\mathbb{B}_{1,1}(\Gamma_i,\{v\})$, whose other endpoints are labelled by $\mathbb{B}_2(\Gamma^+_i)$ (the edges correspond to left cosets of $\mathbb{B}_{1,1}(\Gamma_i,\{v\})$ in $\mathbb{B}_2(\Gamma^+_i)$).  
The resulting complex is denoted by $\mathcal{T}_0(\mathbb{B}_2(\Gamma,v))$. 
\item\label{BSTitem2}
At each leaf $\barbfv$ in $\mathcal{T}_k(\mathbb{B}_2(\Gamma,v))$ labelled by $\mathbb{B}_2(\Gamma^+_i)$, we first attach infinitely many edges labelled by $\mathbb{B}_{1,1}(\Gamma_i,\{v\})$ such that the other endpoint of each edge is labelled by $\mathbb{B}_{1,1}(\Gamma_i,\Gamma^c_i)$ (the edges incident to $\barbfv$ correspond to left cosets of $\mathbb{B}_{1,1}(\Gamma_i,\{v\})$ in $\mathbb{B}_2(\Gamma^+_i)$). 
And then we attach a copy of $\cI(\bar{UP_2(\Gamma,v)})$ at each new leaf by identifying it with a vertex labelled by $\mathbb{B}_{1,1}(\Gamma_i,\Gamma^c_i)$ in the copy of $\cI(\bar{UP_2(\Gamma,v)})$.
\item\label{BSTitem3}
For a vertex $\barbfw$ in every copy of $\cI(\bar{UP_2(\Gamma,v)})$ added in \eqref{BSTitem2}, if there is no edge containing $\barbfw$ but not contained in the copy, then we attach infinitely many edges $\bar{\mathbf{E}}$ at $\barbfw$ as in \eqref{BSTitem1}. If $\barbfw$ is labelled by $\mathbb{B}_{1,1}(\Gamma_j,\Gamma^c_j)$, then $\bar{\mathbf{E}}$ is labelled by $\mathbb{B}_{1,1}(\Gamma_j,\{v\})$ and the other endpoint of $\bar{\mathbf{E}}$ is labelled by $\mathbb{B}_2(\Gamma^+_j)$. The resulting complex is denoted by $\mathcal{T}_{k+1}(\mathbb{B}_2(\Gamma,v))$.
\item By performing the processes (\ref{BSTitem2}) and (\ref{BSTitem3}) repeatedly,
we obtain $\mathcal{T}_k(\mathbb{B}_2(\Gamma,v))$ for each $k\ge 0$, whose direct limit is exactly the same as $\mathcal{T}(\mathbb{B}_2(\Gamma,v))$. 
\end{enumerate}

As each vertex and edge in $\cI(\bar{UP_2(\Gamma,v)})$ can be realized as standard product subcomplexes in $\bar{UP_2(\Gamma,v)}$ via the set map $\mathfrak{g}_{\cI}$ (see \eqref{Eq:g_I}), the same holds for $\mathcal{T}(\mathbb{B}_2(\Gamma,v))$. Namely, there is a canonical set map 
\begin{equation}\label{Eq:SetmapforT}
\mathfrak{g}_{\mathcal{T}}:\mathcal{T}(\mathbb{B}_2(\Gamma,v))\to 2^{\bar{UD_2(\Gamma,v)}}
\end{equation}
extending $\mathfrak{g}_{\cI}:\cI(\bar{UP_2(\Gamma,v)})\to 2^{\bar{UP_2(\Gamma,v)}}$ such that 
\begin{itemize}
\item the restriction to each copy of $\cI(\bar{UP_2(\Gamma,v)})$ is $\mathfrak{g}_{\cI}$ postcomposed with an embedding $\bar{UP_2(\Gamma,v)}$ into $\bar{UD_2(\Gamma,v)}$ and 
\item the vertex corresponding to a conjugation of $\mathbb{B}_2(\Gamma^+_i)$ is mapped to the corresponding copy of $\bar{UD_2(\Gamma^+_i)}$ in $\bar{UD_2(\Gamma,v)}$.
\end{itemize}

The upshot of the graph-of-groups decompositions of $\pi_1(UP_2(\Gamma,v))$ and $\mathbb{B}_2(\Gamma,v)$ is that the associated Bass-Serre trees are determined by the loop configurations around a given vertex and the $2$-braid groups $\mathbb{B}_2(\Gamma_i^+)$ over extended $\hat v$-components.
\section{Bunches of grapes and their discrete 2-configuration spaces}\label{section:grapes}
In this section, we first define a family of graphs, called \emph{bunches of grapes} $\Gamma$, 
which contains trees, and then see some properties of $UD_2(\Gamma)$ and $\cRI(UD_2(\Gamma))$.
A speciality of this class of graphs would be that it satisfies two features which make the study of the large scale geometry accessible; Section~\ref{Subsection:BOG} focuses on Remark~\ref{Rmk:LocallyConvex} and its consequences, and Section~\ref{Subsection:IntofBOG} primarily deals with Remark~\ref{Rmk:Developability}.

\subsection{Bunches of grapes}\label{Subsection:BOG}
For a graph $\Lambda$ not necessarily simple, the original definition of the \emph{circumference} of $\Lambda$ is the longest length of cycles in $\Lambda$. Instead of the original one, we will use the following definition `\emph{topological circumference}' in this paper as we usually deal with simple graphs.
\begin{definition}[(Topological) circumference]
Let $\Lambda$ be a (possibly non-simple) graph with at least one vertex of valency $\ge 3$.
We say that $\Lambda$ is \emph{of circumference at most $n-1$} if there is no topological embedding 
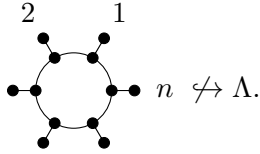
\[
\begin{tikzpicture}[baseline=-.5ex]
\draw (0,0) circle (0.5);
\foreach \i in {0,1,2,3,4,5} {
\draw[fill] ({60*\i}:0.5) circle (2pt) -- ({60*\i}:0.8) circle(2pt);
}
\draw ({60}:1.2) node {$1$};
\draw ({120}:1.2) node {$2$};
\draw ({0}:1.2) node {$n$};
\end{tikzpicture}
\not\hookrightarrow \Lambda.
\]
\end{definition}

\begin{Ex}
A graph of circumference $0$ is a tree, and a graph of circumference $1$ is obtained by attaching cycles to a tree up to smoothing and subdivision.
\end{Ex}

\begin{definition}[Bunches of grapes]\label{Def:BunchesofGrapes}
Let $\mathsf{T}$ be a finite tree and $\loops:\mathcal{V}(\mathsf{T})\to \mathbb{Z}_{\ge 0}$ a function. A \emph{bunch of grapes} $\Gamma=(\mathsf{T},\loops)$ is a graph of circumference at most $1$ obtained by attaching $\loops(v)$ copies of $3$-cycles at each vertex $v$ of $\mathsf{T}$.
We assume that $\Gamma$ is not isomorphic to a path graph.
\end{definition}

\begin{remark}
Any graph of circumference at most $1$ which is not a path is the same as a bunch of grapes $(\mathsf{T}, \loops)$ for some $\mathsf{T}$ and $\loops$ up to smoothing and subdivision on loops.
\end{remark}

We define several terminologies related to a bunch of grapes $\Gamma=(\mathsf{T},\loops)$. 
The tree $\mathsf{T}$ is called a \emph{stem} of $\Gamma$ and each cycle $\mathsf{C}$ attached at a vertex $v\in\mathsf{T}$ is called a \emph{grape}. 
A \emph{substem} $\mathsf{T}'\subset\mathsf{T}$ is a subtree of $\mathsf{T}$ and $\Gamma_{\mathsf{T}'}=(\mathsf{T}', \loops|_{\mathsf{T}'})$ is a bunch of grapes over the substem $\mathsf{T}'$; for convenience's sake, $\Gamma_{\{v\}}$, which is a bouquet of $\loops(v)$ cycles, is denoted by $\Gamma_v$. 
For a subset $V\subset \mathcal{V}(\mathsf{T})$ or a substem $\mathsf{T}'\subset\mathsf{T}$, we define $\loops(V)$ or $\loops(\mathsf{T}')$ as
\[
\loops(V) = \sum_{v\in V}\loops(v),\quad\text{ or }\quad
\loops(\mathsf{T}')=\loops(\mathcal{V}(\mathsf{T}')).
\]
A \emph{twig} of $\Gamma$ (or $\mathsf{T}$) is a path substem $[v_0,\dots, v_n]$ in $\mathsf{T}$ such that $\val_\Gamma(v_i)=2$ if and only if $0<i<n$; it is said to be \emph{empty} if it starts or ends at a leaf of $\Gamma$.
See Figure~\ref{figure:grapes, stem and twigs}; (empty) twigs are indicated by thick (red) lines.

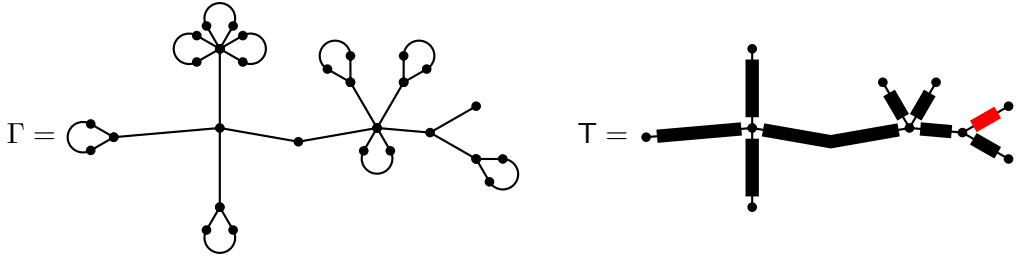
\begin{figure}[ht]
\[
\Gamma=\begin{tikzpicture}[baseline=-.5ex, scale=0.7]
\draw[thick,fill] (0,0) circle (2pt) node (A) {} -- ++(5:2) circle(2pt) node(B) {} -- +(0,1.5) circle (2pt) node (C) {}+(0,0) -- +(0,-1.5) circle (2pt) node(D) {} +(0,0) -- ++(-10:1.5) circle(2pt) node(E) {} -- ++(10:1.5) circle(2pt) node(F) {} -- +(120:1) circle(2pt) node(G) {} +(0,0) -- +(60:1) circle (2pt) node(H) {} +(0,0) -- ++(-5:1) circle (2pt) node(I) {} -- +(30:1) circle(2pt) node(J) {} +(0,0) -- +(-30:1) circle (2pt) node(K) {};
\grape[180]{A};
\grape[0]{C}; \grape[90]{C}; \grape[180]{C};
\grape[-90]{D};
\grape[-90]{F};
\grape[120]{G};
\grape[60]{H};
\grape[-30]{K};
\end{tikzpicture}\qquad
\mathsf{T}=\begin{tikzpicture}[baseline=-.5ex, scale=0.7]
\draw[thick,fill] (0,0) circle (2pt) node (A) {} -- ++(5:2) circle(2pt) node(B) {} -- +(0,1.5) circle (2pt) node (C) {}+(0,0) -- +(0,-1.5) circle (2pt) node(D) {} +(0,0) -- ++(-10:1.5) circle(2pt) node(E) {} -- ++(10:1.5) circle(2pt) node(F) {} -- +(120:1) circle(2pt) node(G) {} +(0,0) -- +(60:1) circle (2pt) node(H) {} +(0,0) -- ++(-5:1) circle (2pt) node(I) {} -- +(30:1) circle(2pt) node(J) {} +(0,0) -- +(-30:1) circle (2pt) node(K) {};
\draw[line width=5, opacity=0.5] (A)--(B)--(C) (B)--(D) (B)--(E.center)--(F) (F)--(G) (F)--(H) (F)--(I) (I)--(K);
\draw[line width=5, opacity=0.5, red] (I)--(J);
\end{tikzpicture}
\]
\caption{A bunch of grapes, its stem and twigs.}
\label{figure:grapes, stem and twigs}
\end{figure}

For a path substem $\mathsf{P}=[v_0,\dots,v_n]$ ($n\ge 1$) of $\mathsf{T}$, \emph{$\mathring{\mathsf{P}}$-components} of $\mathsf{T}$ are defined as components of $\mathsf{T}_{\mathring{\mathsf{P}}}$ containing either $v_0$ and $v_n$, where $\mathsf{T}_{\mathring{\mathsf{P}}}\subset\mathsf{T}$ is the (disconnected) substem induced by $\mathcal{V}(\mathsf{T}_{\mathring{\mathsf{P}}})=\mathcal{V}(\mathsf{T})\setminus\{v_1,\dots,v_{n-1}\}$; there are always two $\mathring{\mathsf{P}}$-components, denoted by $\mathsf{T}_{\mathring{\mathsf{P}},1},\mathsf{T}_{\mathring{\mathsf{P}},2}$.
(Note that if $\mathsf{P}$ is a single edge, then $\mathring{\mathsf{P}}$-components of $\mathsf{T}$ are the complement of the interior of the edge in $\mathsf{T}$, viewed as topological spaces.)
Similarly, \emph{$\mathring{\mathsf{P}}$-components} of $\Gamma$ are defined as $\Gamma_{\mathsf{T}_{\mathring{\mathsf{P}},i}}$, or simply $\Gamma_{\mathring{\mathsf{P}},i}$. 
See Figure~\ref{figure:paths in grapes} for an example.

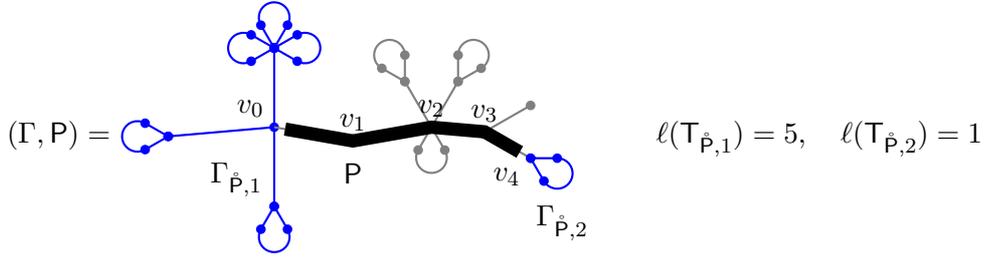
\begin{figure}[ht]
\[
(\Gamma,\mathsf{P})=
\begin{tikzpicture}[baseline=-.5ex, scale=0.7]
\draw[thick,fill,blue] (0,0) circle (2pt) node (A) {} -- ++(5:2) circle(2pt) node(B) {} -- +(0,1.5) circle (2pt) node (C) {}+(0,0) -- +(0,-1.5) circle (2pt) node(D) {};
\draw[thick,fill,gray] (B) +(0,0) -- ++(-10:1.5) circle(2pt) node(E) {} -- ++(10:1.5) circle(2pt) node(F) {} -- +(120:1) circle(2pt) node(G) {} +(0,0) -- +(60:1) circle (2pt) node(H) {} +(0,0) -- ++(-5:1) circle (2pt) node(I) {} -- +(30:1) circle(2pt) node(J) {} +(0,0) -- +(-30:1) circle (2pt) node(K) {};
\draw (B) node[above left] {$v_0$} node[left=3ex, below=2ex] {$\Gamma_{\mathring{\mathsf{P}},1}$};
\draw (E) node[above] {$v_1$} node[below=1ex] {$\mathsf{P}$};
\draw (F) node[above] {$v_2$};
\draw (I) node[above] {$v_3$};
\draw (K) node[below left] {$v_4$} node[right=2.5ex, below=3ex] {$\Gamma_{\mathring{\mathsf{P}},2}$};
\begin{scope}[blue]
\grape[180]{A};
\grape[0]{C}; \grape[90]{C}; \grape[180]{C};
\grape[-90]{D};
\grape[-30]{K};
\end{scope}
\begin{scope}[gray]
\grape[-90]{F};
\grape[120]{G};
\grape[60]{H};
\end{scope}
\draw[line width=5, opacity=0.5] (B)--(E.center)--(F.center)--(I.center)--(K);
\end{tikzpicture}\qquad
\ell(\mathsf{T}_{\mathring{\mathsf{P}},1})=5,\quad
\ell(\mathsf{T}_{\mathring{\mathsf{P}},2})=1
\]
\caption{$\mathring{\mathsf{P}}$-components of a bunch of grapes $\Gamma$.}
\label{figure:paths in grapes}
\end{figure}

We define several adjectives to refer to a specific bunch of grapes as follows:
\begin{enumerate}
\item \emph{large} if there are at least two distinct vertices $v_1, v_2$ of $\mathsf{T}$ with $\loops(v_i)>0$, and \emph{small} otherwise;
\item \emph{normal} if $\val_\Gamma(v)\ge 3$  for all $v\in\mathcal{V}(\mathsf{T})$, and \emph{rich} if $\loops(v)\ge 1$ for all $v\in\mathcal{V}(\mathsf{T})$.
\end{enumerate}
Based on these adjectives, we define the sets
\begin{align*}
\grapegraph&=\{\text{bunches of grapes}\},&
\grapegraph^{\mathsf{large}}&=\{\text{large bunches of grapes}\},\\
\grapegraph_{\mathsf{normal}}&=\{\text{normal bunches of graphs}\},&
\grapegraph_{\mathsf{rich}}&=\{\text{rich bunches of graphs}\},\\
\grapegraph^{\mathsf{large}}_{\mathsf{normal}}&=\grapegraph^{\mathsf{large}}\cap \grapegraph_{\mathsf{normal}},&
\grapegraph^{\mathsf{large}}_{\mathsf{rich}}&=\grapegraph^{\mathsf{large}}\cap \grapegraph_{\mathsf{rich}}.
\end{align*}
Note that for $\grapegraph'$ and $\grapegraph''$ in Section~\ref{Section:IntroOperations}, we have $\grapegraph^{\mathsf{large}}_{\mathsf{normal}}\subsetneq\grapegraph'\subsetneq\grapegraph^{\mathsf{large}}$ and $\grapegraph''=\grapegraph^{\mathsf{large}}_{\mathsf{rich}}$.
See Figure~\ref{figure:normal and rich} for examples.

\begin{figure}[ht]
\[
\begin{tikzpicture}[baseline=-.5ex, scale=0.7]
\draw[thick,fill] (0,0) circle (2pt) node (A) {} -- ++(5:2) circle(2pt) node(B) {} -- +(0,1.5) circle (2pt) node (C) {}+(0,0) -- +(0,-1.5) circle (2pt) node(D) {} +(0,0) -- ++(-10:1.5) circle(2pt) node(E) {} -- ++(10:1.5) circle(2pt) node(F) {} -- +(120:1) circle(2pt) node(G) {} +(0,0) -- +(60:1) circle (2pt) node(H) {} +(0,0) -- ++(-5:1) circle (2pt) node(I) {} -- +(30:1) circle(2pt) node(J) {} +(0,0) -- +(-30:1) circle (2pt) node(K) {};
\grape[180]{A};
\grape[0]{C}; \grape[90]{C}; \grape[180]{C};
\grape[-90]{D};
\grape[-90]{F};
\grape[120]{G};
\grape[60]{H};
\grape[-30]{K};
\begin{scope}[red]
\grape[30]{J};
\grape[-90]{E};
\end{scope}
\end{tikzpicture}\qquad
\begin{tikzpicture}[baseline=-.5ex, scale=0.7]
\draw[thick,fill] (0,0) circle (2pt) node (A) {} -- ++(5:2) circle(2pt) node(B) {} -- +(0,1.5) circle (2pt) node (C) {}+(0,0) -- +(0,-1.5) circle (2pt) node(D) {} +(0,0) -- ++(-10:1.5) circle(2pt) node(E) {} -- ++(10:1.5) circle(2pt) node(F) {} -- +(120:1) circle(2pt) node(G) {} +(0,0) -- +(60:1) circle (2pt) node(H) {} +(0,0) -- ++(-5:1) circle (2pt) node(I) {} -- +(30:1) circle(2pt) node(J) {} +(0,0) -- +(-30:1) circle (2pt) node(K) {};
\grape[180]{A};
\grape[0]{C}; \grape[90]{C}; \grape[180]{C};
\grape[-90]{D};
\grape[-90]{F};
\grape[120]{G};
\grape[60]{H};
\grape[-30]{K};
\begin{scope}[red]
\grape[30]{J};
\grape[-90]{E};
\end{scope}
\begin{scope}[blue]
\grape[40]{B};
\grape[-105]{I};
\end{scope}
\end{tikzpicture}
\]
\caption{Normal and rich bunches of grapes}
\label{figure:normal and rich}
\end{figure}
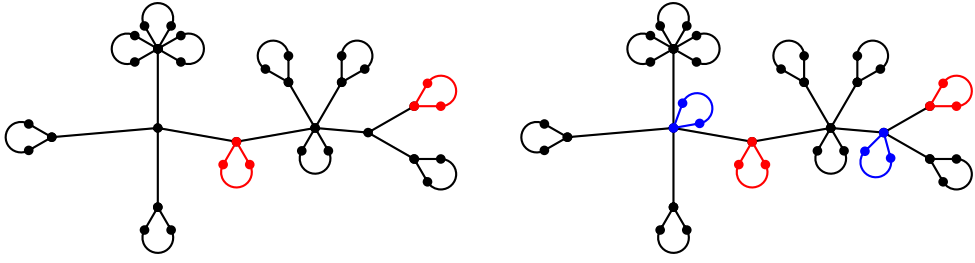

\begin{observation}
Let $\Gamma=(\mathsf{T},\loops)\in\grapegraph$.
\begin{itemize}
\item If $\Gamma\in\grapegraph_{\mathsf{normal}}$, then the set of twigs is the same as the set $\mathcal{E}(\mathsf{T})$ of edges of $\mathsf{T}$. Moreover, for any path substem $\mathsf{P}$ of $\mathsf{T}$, $\mathring{\mathsf{P}}$-components $\Gamma_{\mathring{\mathsf{P}},1}, \Gamma_{\mathring{\mathsf{P}},2}$ of $\Gamma$ belong to $\grapegraph_{\mathsf{normal}}$ and thus $\Gamma_{\mathring{\mathsf{P}},1}\itimes\Gamma_{\mathring{\mathsf{P}},2}$ is a standard product subcomplex of $UD_2(\Gamma)$ by Corollary~\ref{corollary:maximal}.
\item If $\Gamma\in\grapegraph_{\mathsf{rich}}$, then $\Gamma_{\mathsf{T}'}\in\grapegraph_{\mathsf{rich}}$ for any substem $\mathsf{T}'\subset\mathsf{T}$. Hence $\grapegraph_{\mathsf{rich}}$ is closed under taking a restriction over a substem (but $\grapegraph_{\mathsf{normal}}$ is not).
\end{itemize}
\end{observation}

\begin{lemma}\label{lemma:large}
For $\Gamma=(\mathsf{T},\loops)\in\grapegraph$, the following are equivalent.
\begin{enumerate}
\item $\Gamma$ is small,
\item $\Gamma$ is in $\mathcal{F}$ (defined in the paragraph below Theorem~\ref{theorem:freegroupcase}), and 
\item $UP_2(\Gamma)$ (Definition~\ref{Def:UP_2(Gamma)}) is empty.
\end{enumerate}
\end{lemma}
\begin{proof}
The equivalence [(1)$\Leftrightarrow$(2)] and the implication [(2)$\Rightarrow$(3)] are obvious from the definitions of $\mathcal{F}$ and $UP_2(\Gamma)$.

In order to prove the implication [(3)$\Rightarrow$(1)], suppose that $\Gamma$ is large. For $i=1,2$, let $v_i$ be a vertex in $\Gamma$ with $\loops(v_i)\ge 1$ and $\mathsf{C}_i$ a grape at $v_i$ ($v_1$ and $v_2$ are chosen to be distinct). 
By Lemma \ref{Lem:SPSinGBG}, then $\mathsf{C}_1\itimes\mathsf{C}_2$ is a standard product subcomplex, and any standard product subcomplex must be contained a maximal one.
Therefore, $UP_2(\Gamma)$ is non-empty.
\end{proof}

The main goal of this paper to classify the $2$-braid groups over bunches of grapes up to quasi-isometry, and by Theorem~\ref{theorem:freegroupcase}, the $2$-braid group over any small bunch of grapes is not quasi-isometric to the $2$-braid group over any large bunch of grapes. 
\textbf{Hence, in the remaining of this section, we fix and focus on a bunch of grapes $\Gamma=(\mathsf{T},\loops)\in\grapegraph_{\mathsf{normal}}^{\mathsf{large}}$.}
Large but non-normal bunches of grapes will be dealt with in the beginning of Section~\ref{section:Operations}. Large and rich bunches of grapes will step into the spotlight starting from Section~\ref{Subsection:PickingGrapes}.

Our main tool for the goal will be the (reduced) intersection complex of (the universal cover of) the unordered discrete $2$-configuration space over a bunch of grapes.
In order to use this quasi-isometry invariant, our first step will be to know how maximal product subcomplex of $UD_2(\Gamma)$ are distributed inside $UD_2(\Gamma)$. 
This indeed comes from the fact that maximal product subcomplexes of $UD_2(\Gamma)$ can be characterized from (the stem of) $\Gamma$.

\begin{lemma}\label{lem:TwigCorrespondence}
There is a one-to-one correspondence between the set of twigs $t$ of $\Gamma$ and the set of maximal product subcomplexes $M(t)$ of $UD_2(\Gamma)$ (and thus $UP_2(\Gamma)$).
\end{lemma}
\begin{proof}
As standard product subcomplexes of $UP_2(\Gamma)$ are exactly standard product subcomplexes of $UD_2(\Gamma)$ by Lemma~\ref{lem:UnionofSPS}, we will show that the lemma holds for $UD_2(\Gamma)$.

For a twig $t$ of $\Gamma$ which is a path substem of $\mathsf{T}$, let $\mathsf{T}_i=\mathsf{T}_{\mathring{t},i}$ be a $\mathring{t}$-component of $\mathsf{T}$ for $i=1,2$.
Since $t$ is an edge, by Corollary~\ref{corollary:maximal} and normality of $\Gamma$ (and thus $\Gamma_{\mathsf{T}_i}$ is normal), it is obvious that $\Gamma_{\mathsf{T}_1}\itimes\Gamma_{\mathsf{T}_2}$ is a maximal product subcomplex.

Suppose that $\Gamma_1\itimes\Gamma_2\subset UD_2(\Gamma)$ is a maximal product subcomplex. 
By the maximality of $\Gamma_1\itimes\Gamma_2$, if a vertex $v$ in $\mathsf{T}$ is contained in $\Gamma_1$ ($\Gamma_2$, resp.), then $\Gamma_v$ must be contained in $\Gamma_1$ ($\Gamma_2$, resp.); otherwise, $\Gamma_2$ ($\Gamma_1$, resp.) may not be a graph without leaves.
It follows that each $\Gamma_i$ is $\Gamma_{\mathsf{T}_i}$ for a substem $\mathsf{T}_i\subset\mathsf{T}$. 
By Corollary~\ref{corollary:maximal} and normality of $\Gamma$, therefore, there exists a unique twig $t$ such that $\mathsf{T}_1$ and $\mathsf{T}_2$ are $\mathring{t}$-components of $\mathsf{T}$.
\end{proof}

Based on the previous lemma, in this subsection, we first show that $UP_2(\Gamma)$ is locally convex in $UD_2(\Gamma)$, which is an affirmative answer for Question~\ref{Qn:LocalConvexity}. This leads our study of the large scale geometry of $\mathbb{B}_2(\Gamma)$ to the study of the structure of $UP_2(\Gamma)$ and $\bar{UP_2(\Gamma)}$. And then we will see how maximal product subcomplexes of $UP_2(\Gamma)$ or $\bar{UP_2(\Gamma)}$ intersect.

Let us start with the local convexity. 

\begin{proposition}\label{Prop:FreeFactor}
The subcomplex $UP_2(\Gamma)$ is locally convex in $UD_2(\Gamma)$.
Moreover, the braid group $\mathbb{B}_2(\Gamma)$ is isomorphic to $\pi_1(UP_2(\Gamma))*\mathbb{F}_n$ for some $n\ge 2$.
\end{proposition}
\begin{proof}
For any vertex $u\itimes v$ in $UP_2(\Gamma)$, 
$\Lk(u\itimes v, UP_2(\Gamma)))$ can be considered as a subcomplex of $\Lk(u\itimes v,UD_2(\Gamma))$.
For the first statement, it suffices to show that $\Lk(u\itimes v, UP_2(\Gamma)))$ is an induced subcomplex of $\Lk(u\itimes v, UD_2(\Gamma)))$.

Let $l$ be the distance in $\Gamma$ between $u$ and $v$.
If $l\geq 3$, then there exists a maximal product subcomplex $M\subset UP_2(\Gamma)$ which contains $\st_\Gamma(u)\itimes\st_{\Gamma}(v)$ (recall that $\st_\Gamma(u)$ is the subgraph of $\Gamma$ induced by $\lk_\Gamma(u)\cup\{u\}$) and thus both $\Lk(u\itimes v, UP_2(\Gamma))$ and $\Lk(u\itimes v, UD_2(\Gamma))$ are isomorphic to $\operatorname{Join}\bigl(\lk_\Gamma(u), \lk_\Gamma(v)\bigr)$.
If $l=1$, then $u$ and $v$ must be the endpoints of a twig and thus there exists a unique maximal product subcomplex $M\subset UP_2(\Gamma)$ which contains $u\itimes v$. In particular, both $\Lk(u\itimes v, UP_2(\Gamma))$ and $\Lk(u\itimes v, UD_2(\Gamma))$ are isomorphic to $\operatorname{Join}\bigl(\lk_\Gamma(u)\setminus\{v\}, \lk_\Gamma(v)\setminus\{u\}\bigr)$.
Thus the remaining case is when $l=2$.

Let $x$ and $y$ be distinct vertices in $\mathsf{T}$ such that $u\in\Gamma_x$ and $v\in\Gamma_y$. 
If $x$ and $y$ are not adjacent, then $u=x$ and $v=y$, and there are two twigs $t, t'$ between $u$ and $v$ such that $u\itimes v$ is contained only in the maximal product subcomplexes $M=M(t)$ and $M'=M(t')$. Thus, 
\begin{align*}
\Lk(u\itimes v, UP_2(\Gamma))&\cong 
\Lk(u\itimes v, M)\cup
\Lk(u\itimes v, M')\\
&\cong
\operatorname{Join}\bigl(\lk_\Gamma(u)\setminus\{z\},\lk_\Gamma(v)\bigr)\cup
\operatorname{Join}\bigl(\lk_\Gamma(u),\lk_\Gamma(v)\setminus\{z\}\bigr)\\
&\cong \Lk(u\itimes v, UD_2(\Gamma))),
\end{align*}
where $z$ is the vertex adjacent to both $x$ and $y$.

If $x$ and $y$ are adjacent via a twig $t$, then we may assume that $u=x, v\neq y$.
Since $u\itimes v$ is contained in the unique maximal product subcomplex $M=M(t)$, we have
\begin{align*}
\Lk(u\itimes v, UP_2(\Gamma))&\cong \Lk(u\itimes v, M)
\cong \operatorname{Join}\bigl(\lk_\Gamma(u)\setminus\{y\}, \lk_\Gamma(v)\bigr)\\
\Lk(u\itimes v, UD_2(\Gamma)))&\cong
\operatorname{Join}\bigl(\lk_\Gamma(u)\setminus\{y\},\lk_\Gamma(v)\bigr)\cup
\operatorname{Join}\bigl(\lk_\Gamma(u),\lk_\Gamma(v)\setminus\{y\}\bigr),
\end{align*}
where the intersection of the latter two joins is $\operatorname{Join}\bigl(\lk_\Gamma(u)\setminus\{y\},\lk_\Gamma(v)\setminus\{y\}\bigr)$; in this case, $\Lk(u\itimes v, UP_2(\Gamma)))$ is an induced subcomplex of $\Lk(u\itimes v, UD_2(\Gamma)))$.
Therefore $UP_2(\Gamma)$ is locally convex in $UD_2(\Gamma)$.

From the fact that a vertex $u\itimes v$ in $UD_2(\Gamma)$ is not in $UP_2(\Gamma)$ if and only if there exists a vertex $x$ in $\mathsf{T}$ such that both $u$ and $v$ are in $\Gamma_x$, 
we have a decomposition
\[
UD_2(\Gamma)=UP_2(\Gamma)\cup \bigcup_{v\in\mathsf{T},\ \loops(v)\ge 1} UD_2(\Gamma_v\cup\st_{\mathsf{T}}(v)),
\]
such that 
\begin{align*}
UP_2(\Gamma)\cap UD_2(\Gamma_v\cup\st_{\mathsf{T}}(v))&=\lk_{\mathsf{T}}(v)\itimes\Gamma_v,\text{ and }\\
UD_2(\Gamma_v\cup\st_{\mathsf{T}}(v))\cap UD_2(\Gamma_w\cup\st_{\mathsf{T}}(w))&\cong \begin{cases}
v\itimes w & v\text{ and }w \text{ are adjacent};\\
\emptyset & \text{otherwise}.
\end{cases}
\end{align*}
By Example~\ref{Ex:FreeGroupCase}, $\pi_1(UD_2(\Gamma_v\cup\st_{\mathsf{T}}(v)))$ is isomorphic to a non-abelian free group since $\loops(v)\ge 1$.
Moreover, by the decomposition given in \eqref{eq:Decomposition} in the proof of Lemma~\ref{NoHangingEdge}, the inclusion map from a component of $\lk_{\mathsf{T}}(v)\itimes\Gamma_v$ to $UD_2(\Gamma_v\cup\st_{\mathsf{T}}(v))$ induces a non-trivial free factor of $\pi_1(UD_2(\Gamma_v\cup\st_{\mathsf{T}}(v)))$. 
It follows that $\pi_1(UP_2(\Gamma)\cup UD_2(\Gamma_v\cup\st_{\mathsf{T}}(v)))$ is isomorphic to $\pi_1(UP_2(\Gamma))*\mathbb{F}_{n_v}$ for some ${n_v}\geq 1$.
Inductively, we can show that $\mathbb{B}_2(\Gamma)$ is isomorphic to $\pi_1(UP_2(\Gamma))*\mathbb{F}_n$, where $n$ is greater than or equal to the cardinality of $\mathcal{V}(\mathsf{T})$.
\end{proof}

By Theorem~\ref{PW}, we immediately obtain the following fact.

\begin{corollary}\label{cor:QIbetweenGBGs}
For $\Gamma_1, \Gamma_2\in\grapegraph_{\mathsf{normal}}^{\mathsf{large}}$, the $2$-braid groups $\mathbb{B}_2(\Gamma_1)$ and $\mathbb{B}_2(\Gamma_2)$ are quasi-isometric if and only if $\pi_1(UP_2(\Gamma_1))$ and $\pi_1(UP_2(\Gamma_2))$ are quasi-isometric.
\end{corollary}

For twigs $t_i$ of $\Gamma$, let $M(t_i)\subset UP_2(\Gamma)$ be the corresponding maximal product subcomplexes. Then the intersection $M(t_i)\cap M(t_j)$ for any two twigs $t_i,t_j$ is $\Gamma_1\itimes\Gamma_2$, where $\Gamma_1$ ($\Gamma_2$, resp.) is a $\mathring{t_1}$-component ($\mathring{t_2}$-component, resp.) of $\Gamma$ which does not contain $t_2$ ($t_1$, resp.).
Indeed, $\Gamma_1$ and $\Gamma_2$ are $\mathring{\mathsf{P}}$-components of $\Gamma$, where $\mathsf{P}$ is the minimal path substem containing $t_1$ and $t_2$.
This observation can be generalized as follows.

\begin{lemma}\label{Lem:RIconnected}
Let $\{M(t_1),\dots,M(t_m)\}$ be a finite collection of maximal product subcomplexes of $UP_2(\Gamma)$ corresponding to twigs $\{t_1,\dots,t_m\}$ of $\Gamma$.
Then the following are equivalent:
\begin{enumerate}
\item $\bigcap_{i=1}^m M(t_i)$ is non-empty.
\item All $t_i$'s are colinear and $\bigcap_{i=1}^m M(t_i)$ is a standard product subcomplex corresponding to the minimal path substem $\mathsf{P}$ of $\mathsf{T}$ containing all $t_i$'s, namely, $\bigcap_{i=1}^m M(t_i)=\Gamma_{\mathring{\mathsf{P}},1}\itimes\Gamma_{\mathring{\mathsf{P}},2}$, where $\Gamma_{\mathring{\mathsf{P}},i}$'s are $\mathring{\mathsf{P}}$-components of $\Gamma$.
\end{enumerate}
\end{lemma}
\begin{proof}
We only need to show the implication [(1)$\Rightarrow$(2)].

Suppose that $t_i$'s are non-colinear in $\mathsf{T}$, i.e., there exist non-colinear twigs $t_1, t_2$, and $t_3$ (after reordering the indices). For each $i\in\{1,2,3\}$, let $\Gamma_i$ be the $\mathring{t_i}$-component of $\Gamma$, which is away from other $t_j$ for $j\in \{1,2,3\}\setminus\{i\}$.
As observed in the paragraph above this lemma, $M(t_i)\cap M(t_j)$ is the standard product subcomplex $\Gamma_i\itimes\Gamma_j$ for two distinct $i,j\in\{1,2,3\}$. 
Since $\Gamma_2$ and $\Gamma_3$ are disjoint, $\Gamma_1\itimes\Gamma_2$ and $\Gamma_1\itimes\Gamma_3$ are also disjoint and thus, the intersection of $M(t_1)$, $M(t_2)$ and $M(t_3)$ is empty. In particular, $\bigcap_{i=1}^m M(t_i)$ is empty.
\end{proof}

Combining with Lemma~\ref{lem:UnionofSPS}, we also have an analog of Lemma~\ref{Lem:RIconnected} for $\bar{UP_2(\Gamma)}$.

\begin{lemma}\label{lem:Intofmaximal}
Let $\{\bar{M(t_1)},\dots,\bar{M(t_m)}\}$ be a finite collection of pairwise intersecting maximal product subcomplexes of $\bar{UP_2(\Gamma)}$ such that each $\bar{M(t_i)}$ is a $p$-lift of the maximal product subcomplex $M(t_i)\subset UP_2(\Gamma)$ corresponding to a twig $t_i$ of $\Gamma$. 
Then $t_i$'s are colinear. 
Moreover, $\bigcap_{i=1}^m \bar{M(t_i)}$ is a standard product subcomplex which is a $p$-lift of the standard product subcomplex of $UP_2(\Gamma)$ corresponding to the minimal path substem containing all $t_i$'s.
\end{lemma}
Note that a $\CAT(0)$ cube complex satisfies the \emph{Helly} property: if $Y_1,\dots,Y_m$ are convex subcomplexes in a $\CAT(0)$ cube complex such that $Y_i\cap Y_j$ is non-empty for all $1\le i < j\le m$, then $\bigcap_{i=1}^m Y_i$ is non-empty.
\begin{proof}
By Lemma~\ref{Lem:SPSinGBG}, any two $p$-lifts of $M(t_i)$ in $\bar{UP_2(\Gamma)}$ are components of $p^{-1}(M(t_i))$, and in particular, if $t_i=t_j$, then $\bar{M(t_i)}$ and $\bar{M(t_j)}$ must be identical.

Since we have $\bigcap_{i=1}^m \bar{M(t_i)}\neq\varnothing$ by the Helly property, we have
\[
\varnothing\neq p\left(\bigcap_{i=1}^m \bar{M(t_i)}\right) \subseteq \bigcap_{i=1}^m p(\bar{M(t_i)})=\bigcap_{i=1}^m {M(t_i)}.
\]
By Lemma~\ref{Lem:RIconnected}, the intersection $\bigcap_{i=1}^m M(t_i)$ must be the standard product subcomplex $K$ corresponding to the minimal path substem containing all $t_i$'s.
It follows that there exists a unique $p$-lift $\bar{K}$ of $K$ which contains $\bigcap_{i=1}^m \bar{M(t_i)}$. Moreover, since $K$ is contained in each $M(t_i)$, $\bar{K}$ is contained in each $\bar{M(t_i)}$ as well.
Therefore, $\bigcap_{i=1}^m \bar{M(t_i)}$ must be equal to $\bar K$. 
\end{proof}

\subsection{(Reduced) intersection complexes for bunches of grapes}\label{Subsection:IntofBOG}
In this subsection, we elaborate on the (local) structure of $\cRI(UP_2(\Gamma))$ and $\cI(\bar{UP_2(\Gamma)})$.
Especially, we focus on the developability of $\cRI(UP_2(\Gamma))$ and then the order on the vertices of each simplex which is preserved (or reversed) by semi-isomorphisms.

Let us start with the connectivity of $\cRI(UP_2(\Gamma))$ and $\cI(\bar{UP_2(\Gamma)})$. 

\begin{proposition}\label{Prop:connectedsimplicial}
Not only $\mathfrak{g}_{\cI}(\cI(\bar{UP_2(\Gamma)}))$ is equal to $\bar{UP_2(\Gamma)}$ but also both $\cRI(UP_2(\Gamma))$ and $\cI(\bar{UP_2(\Gamma)})$ are connected simplicial complexes.
In particular, $\bar{UP_2(\Gamma)}$ is one-ended.
\end{proposition}
\begin{proof}
By Lemma~\ref{lem:UnionofSPS} and Definition~\ref{Def:UP_2(Gamma)}, $UP_2(\Gamma)$ is equal to the union of all standard product subcomplexes of $UP_2(\Gamma)$.
By Lemma~\ref{Lem:UP_2Special}, moreover, $UP_2(\Gamma)$ is a connected special square complex.
By Lemmas~\ref{Lem:RIconnected} and \ref{lem:Connected}, therefore, the proposition holds.

By Lemma~\ref{lem:One-endedness}, one-endedness of $\bar{UP_2(\Gamma)}$ is shown.
\end{proof}

\begin{Ex}\label{Ex:3-star}
Let $\Lambda=(\mathsf{S}_3,\ell)\in\grapegraph_{\mathsf{normal}}^{\mathsf{large}}$ such that $\mathsf{S}_3$ is an $3$-star with leaves $v_1,v_2,v_3$.
Then $\cRI(UP_2(\Lambda))$ is given as follows:
\[
\cRI(UP_2(\Lambda))=\begin{tikzpicture}[baseline=-.5ex]
\draw[fill] (90:1) circle (2pt) node[above] {$\Lambda_{v_1}\itimes\Lambda^c_{v_1}$};
\draw[fill] (210:1) circle (2pt) node[left] {$\Lambda_{v_2}\itimes\Lambda^c_{v_2}$};
\draw[fill] (330:1) circle (2pt) node[right] {$\Lambda_{v_3}\itimes\Lambda^c_{v_3}$};
\draw (90:1)--(210:1) node[midway, left] {$\Lambda_{v_1}\itimes\Lambda_{v_2}$};
\draw (210:1)--(330:1) node[midway, below] {$\Lambda_{v_2}\itimes\Lambda_{v_3}$};
\draw (330:1)--(90:1) node[midway, right] {$\Lambda_{v_1}\itimes\Lambda_{v_3}$};
\end{tikzpicture}
\]
Note that there is no $2$-simplex in $\cRI(UP_2(\Lambda))$ and $\Lambda_{v_i}\subsetneq\Lambda^c_{v_j}$ for $i\neq j$.
\end{Ex}

As seen in the previous example, $\cRI(UP_2(\Gamma))$ may be neither simply connected nor a flag complex. However, $\cI(\bar{UP_2(\Gamma)})$ is always a simply connected flag complex as follows.

\begin{theorem}\label{theorem:structureofI}
For $\Gamma\in\grapegraph_{\mathsf{normal}}^{\mathsf{large}}$, $\cRI(UP_2(\Gamma))$ is a developable complex of groups and $\cI(\bar{UP_2(\Gamma)})$ is the development of $\cRI(UP_2(\Gamma))$. In particular, $\cI(\bar{UP_2(\Gamma)})$ is a flag complex which is connected and simply connected. 
\end{theorem}
\begin{proof}
It follows from Proposition~\ref{Prop:connectedsimplicial} and Lemma~\ref{lem:Intofmaximal} that $\cI(\bar{UP_2(\Gamma)})$ is a connected flag complex.
By Theorem~\ref{Thm:Development} and Theorem~\ref{Thm:TPBCM}, thus, it only remains to show that $\cRI(UP_2(\Gamma))$ can be considered as a developable complex of groups.

As previously mentioned in Section \ref{subsection:intersection complexes}, if the label of each simplex in $\cRI(UP_2(\Gamma))$ is replaced by its fundamental group, then $\cRI(UP_2(\Gamma))$ can be considered as a complex of groups.
Lemma \ref{Lem:RIconnected} implies that $UP_2(\Gamma)$ is the geometric realization of $\cRI(UP_2(\Gamma))$ such that the fundamental group of $\cRI(UP_2(\Gamma))$ is isomorphic to $\pi_1(UP_2(\Gamma))$. By Theorem~\ref{Thm:Developable}, hence $\cRI(UP_2(\Gamma))$ is developable. 
\end{proof}

Let us see the detailed features of $\cRI(UP_2(\Gamma))$ based on the correspondence between its simplices and path substems of $\mathsf{T}$ given by Lemma~\ref{Lem:RIconnected}. 

\begin{remark}\label{Ex:PathofLengthN}
For a simplex $\triangle=\{\mathbf{v}_0, \dots, \mathbf{v}_{n}\}\subset\cRI(UP_2(\Gamma))$, its label corresponds to the minimal path substem $\mathsf{P}\subset\mathsf{T}$ containing the twigs corresponding to $\mathbf{v}_0,\dots,\mathbf{v}_{n}$.
By Lemma~\ref{Lem:RIconnected}, the fundamental group of the label of $\triangle$ is $\pi_1(\Gamma_{\mathring{\mathsf{P}},1} \times \Gamma_{\mathring{\mathsf{P}},2})\cong \pi_1(\Gamma_{\mathring{\mathsf{P}},1}) \times \pi_1(\Gamma_{\mathring{\mathsf{P}},2})$,
which is isomorphic to $\mathbb{F}_{m_1} \times \mathbb{F}_{m_2}$, where $\Gamma_{\mathring{\mathsf{P}},i}=\Gamma_{\mathsf{T}_{\mathring{\mathsf{P}},i}}$ and $m_i=\ell(\mathsf{T}_{\mathring{\mathsf{P}},i})$ for $i=1,2$.
In particular, the fundamental group is quasi-isometric to $\mathbb{G}$, where
\begin{align*}
\mathbb{G}&=\begin{cases}
\mathbb{Z}\times\mathbb{Z} & \text{if the endpoints } v,w \text{ of } \mathsf{P} \text{ are leaves of }\mathsf{T} \text{ and }\ell(v)=\ell(w)=1;\\
\mathbb{Z}\times\mathbb{F}_2 & \text{if only one endpoint } v \text{ of } \mathsf{P} \text{ is a leaf of }\mathsf{T} \text{ and }\ell(v)=1;\\
\mathbb{F}_2\times\mathbb{F}_2 & \text{otherwise},
\end{cases}
\end{align*}

Let $\triangle'\subset\triangle$ be a sub-simplex and $\mathsf{P}'\subseteq\mathsf{P}$ the minimal path substem of $\mathsf{T}$ containing all twigs corresponding to vertices of $\triangle'$.
Then the labels of $\triangle$ and $\triangle'$ have a common factor if and only if $\mathsf{P}$ and $\mathsf{P}'$ share an endpoint, and in particular, the labels of $\triangle$ and $\triangle'$ coincide if and only if $\mathsf{P} = \mathsf{P}'$.
It means that for a path substem $\mathsf{P}$ of length $(k+1)$, a maximal simplex $\triangle_{\mathsf{P}}\subset\cRI(UP_2(\Gamma))$ whose label is $\Gamma_{\mathring{\mathsf{P}},1}\itimes\Gamma_{\mathring{\mathsf{P}},2}$ is a uniquely determined $k$-simplex.
Thus, such $\triangle_{\mathsf{P}}$ will be said to be \emph{induced from} $\mathsf{P}$. Similarly, a simplex $\bar\triangle\subset\cI(\bar{UP_2(\Gamma)})$ with $\rho(\bar\triangle)=\triangle_{\mathsf{P}}$ will be said to be induced from $\mathsf{P}$, which is consistent with Lemma~\ref{lem:Intofmaximal}.
\end{remark}

With Remark~\ref{Ex:PathofLengthN}, we can see that the stem $\mathsf{T}$ of $\Gamma$ induces the (semi-)isomorphism type of $\cRI(UP_2(\Gamma))$ as $\cRI(UP_2(\Gamma))$ encodes the colinearity of twigs of $\Gamma$. 

\begin{lemma}\label{Lem:IsometricRI}
Let $\Gamma=(\mathsf{T},\ell), \Gamma'=(\mathsf{T},\ell')\in\grapegraph_{\mathsf{normal}}^{\mathsf{large}}$ be two large and normal bunches of grapes over the same stem $\mathsf{T}$.
Then there exists a canonical semi-isomorphism 
$$\Phi:\cRI(UP_2(\Gamma))\to\cRI(UP_2(\Gamma'))$$ which is an isomorphism if and only if $\min\{\ell(v),2\}=\min\{\ell'(v),2\}$ for each leaf $v\in \mathsf{T}$.

Moreover, $\Phi$ induces a (semi-)isomorphism $\bar\Phi:\cI(\bar{UP_2(\Gamma)})\to \cI(\bar{UP_2(\Gamma')})$ such that $\Phi\circ\rho=\rho'\circ\bar\Phi$. 
\end{lemma}
\begin{proof}
Since vertices of $\cRI(UP_2(\Gamma))$ correspond to twigs of $\Gamma$ and a finite collection of vertices spans a simplex if and only if the corresponding twigs are colinear, we have an obvious combinatorial map $\Phi:\cRI(UP_2(\Gamma))\to\cRI(UP_2(\Gamma'))$. By Remark~\ref{Ex:PathofLengthN} and Definition~\ref{Def:Morphism}, $\Phi$ must be a semi-isomorphism and it becomes an isomorphism if and only if $\min\{\ell(v),2\}=\min\{\ell'(v),\}$ for each leaf $v\in \mathsf{T}$.

By Proposition~\ref{Prop:IsomorphicI} and Theorem~\ref{theorem:structureofI}, there is a (semi-)isomorphism $\cI(\bar{UP_2(\Gamma)})\to \cI(\bar{UP_2(\Gamma')})$ which is induced from $\Phi$.
\end{proof}

Indeed, there are two natural orders on twigs in any path substem of $\mathsf{T}$ depending on the choice of the initial vertex of the path substem. Then it is natural to consider the order on vertices of a simplex of $\cRI(UP_2(\Gamma))$ induced from a natural order on the twigs corresponding to the vertices and then ask whether it is preserved by semi-isomorphisms.
In the remaining of this subsection, we will show that this turns out to be true by extracting this order from the information of the labels of the vertices and this fact will be significantly used in Section~\ref{section:QIBetweenConfSpaces}.

For each $2\le k\le \operatorname{diam}(\mathsf{T})$, let $\cRI_{\le k}(UP_2(\Gamma))$ and $\cI_{\le k}(\bar{UP_2(\Gamma)})$ be subcomplexes of $\cRI(UP_2(\Gamma))$ and $\cI(\bar{UP_2(\Gamma)})$, respectively, which are the unions of simplices induced from path substems of length at most $k$. In particular, the dimension of $\cRI_{\le k}(UP_2(\Gamma))$ or $\cI_{\le k}(\bar{UP_2(\Gamma)})$ is $k-1$. Then we have the following strong deformation retractions.

\begin{lemma}\label{lemma:deformation retract}
For each $k\ge 2$, there are strong deformation retractions 
\[r_{\le k}:\cRI(UP_2(\Gamma))\to \cRI_{\le k}(UP_2(\Gamma))\quad\text{and}\quad\bar r_{\le k}:\cI(\bar{UP_2(\Gamma)})\to \cI_{\le k}(\bar{UP_2(\Gamma)})\] 
which fit into the following commutative diagram:
\[
\begin{tikzcd}
\cI(\bar{UP_2(\Gamma)})\ar[r,"\bar r_{\le k}"]\ar[d,"\rho"'] & \cI_{\le k}(\bar{UP_2(\Gamma)})\ar[d,"\rho_{\le k}"]\\
\cRI(UP_2(\Gamma))\ar[r,"r_{\le k}"]& \cRI_{\le k}(UP_2(\Gamma)).
\end{tikzcd}
\]

Moreover, for $\Gamma'\in\grapegraph^{\mathsf{large}}_{\mathsf{normal}}$,
any semi-isomorphism $\psi:\cRI(UP_2(\Gamma))\to\cRI(UP_2(\Gamma'))$ ($\bar\psi:\cI(\bar{UP_2(\Gamma)})\to\cI(\bar{UP_2(\Gamma')})$, resp.) has a restriction
$\psi_{\le k}:\cRI(UP_2(\Gamma))\to\cRI_{\le k}(UP_2(\Gamma'))$ ($\bar\psi_{\le k}:\cI(\bar{UP_2(\Gamma)})\to\cI_{\le k}(\bar{UP_2(\Gamma')})$, resp.)
which commutes with the deformation retracts given above.
\end{lemma}
\begin{proof}
Since $\rho$ preserves labels, the codomain of the restriction $\rho_{k}$ of $\rho$ to $\cI_{\le k}(\bar{UP_2(\Gamma)})$ could be chosen as $\cRI_{\le k}(UP_2(\Gamma))$ for each $k\ge2$. Hence, there is a commutative diagram with rows which are strictly increasing filtrations as follows:
\[\begin{tikzcd}[column sep=2pc, row sep=1pc]
\cI_{\le 2}(\bar{UP_2(\Gamma)})\arrow[r,hookrightarrow]\arrow[d,"\rho_{2}"]&\cI_{\le 3}(\bar{UP_2(\Gamma)})\arrow[r,hookrightarrow]\arrow[d,"\rho_{3}"]&\cdots\arrow[r,hookrightarrow]&
\cI_{\le N}(\bar{UP_2(\Gamma)})\arrow[r,equal]\arrow[d,"\rho_{N}"]&\cI(\bar{UP_2(\Gamma)})\arrow[d,"\rho"]\\
\cRI_{\le 2}(UP_2(\Gamma))\arrow[r,hookrightarrow]&\cRI_{\le 3}(UP_2(\Gamma))\arrow[r,hookrightarrow]&\cdots\arrow[r,hookrightarrow]&
\cRI_{\le N}(UP_2(\Gamma))\arrow[r,equal]&\cRI(UP_2(\Gamma)),
\end{tikzcd}\] 
where $N$ is the diameter of $\mathsf{T}$.

For $k\ge 2$, let $\triangle=\{\bfv_0,\dots,\bfv_k\}\subset\cRI_{\le k+1}(UP_2(\Gamma))$ be the $k$-simplex induced from a path substem $\mathsf{P}^0=[v_0,\dots, v_{k+1}]\subset \mathsf{T}$
such that the twig $[v_i,v_{i+1}]\subset\mathsf{T}$ corresponds to $\bfv_i$.
Let $\triangle_1=\{\bfv_0,\bfv_k\}$ be an edge of $\triangle$.
By Remark~\ref{Ex:PathofLengthN}, any sub-simplex $\triangle'$ with $\triangle_1\subseteq\triangle'\subseteq\triangle$ has the label $\Gamma_{\mathring{\mathsf{P}}^0,1}\itimes \Gamma_{\mathring{\mathsf{P}}^0,2}$.
By the definition of $\cRI_{\le k+1}(UP_2(\Gamma))$, moreover, $\triangle$ is a unique maximal simplex of $\cRI_{\le k+1}(UP_2(\Gamma))$ containing $\triangle_1$.
Hence, one can find a strong deformation retract $r_\triangle$ from $\triangle$ to the union of sub-simplices $\triangle_{k-1,1}$ and $\triangle_{k-1,2}$, which can be extended to a map $\cRI_{\le k+1}(UP_2(\Gamma))\to\cRI_{\le k+1}(UP_2(\Gamma))$ fixing all simplices but $\triangle$.

On the other hand, any sub-simplex of $\triangle$ not containing both $\bfv_0$ and $\bfv_k$ belongs to one of the two $(k-1)$-simplices $\triangle_{k-1,1}=\{\bfv_0,\dots,\bfv_{k-1}\}$ and $\triangle_{k-1,2}=\{\bfv_1,\dots,\bfv_k\}$ corresponding to two path substems $\mathsf{P}^1=[v_0,\dots,v_k]$ and $\mathsf{P}^2=[v_1,\dots,v_{k+1}]$ of length $k$, respectively. Note that $\triangle_{k-1,1}$ and $\triangle_{k-1,2}$ belong to $\cRI_{\le k}(UP_2(\Gamma))$. 

By combining $r_\triangle$ for every $k$-simplex $\triangle$ in $\cRI_{\le k+1}(UP_2(\Gamma))$, we therefore have a strong deformation retract $r_k$ from $\cRI_{\le k+1}(UP_2(\Gamma))$ onto $\cRI_{\le k}(UP_2(\Gamma))$.
Then there is a strong deformation retraction $r_{\le k}:\cRI(UP_2(\Gamma))\to \cRI_{\le k}(UP_2(\Gamma))$, where $r_{\le k}=r_{k}\circ\dots\circ r_{N}$.

Suppose that $\bar{\triangle}=\{\bar\bfv_0,\dots,\bar\bfv_k\}\subset\cI_{\le k+1}(\bar{UP_2(\Gamma)})$ is a $k$-simplex such that $\rho(\bar\triangle)=\triangle$ with $\rho(\bar\bfv_i)=\bfv_i$ for each $0\le i\le k$ and let $\bar{\triangle}_1=\{\barbfv_0,\barbfv_k\}$.
For a simplex $\bar\triangle'$ with $\bar\triangle_1\subseteq \bar\triangle'\subseteq \bar\triangle$, the stabilizers of $\bar\triangle$ and $\bar\triangle'$ under the group action given in Theorem~\ref{Thm:TPBCM} coincide and $\triangle$ is a unique maximal simplex containing $\rho(\bar\triangle')$, which means that $\bar\triangle$ is a unique maximal simplex containing $\bar\triangle'$.
Hence, a strong deformation retract $\bar r_{\bar\triangle}$ from $\bar\triangle$ to the union of two sub-simplices $\bar\triangle_{k-1,1}=\{\bar\bfv_0,\dots,\bar\bfv_{k-1}\}$ and $\bar\triangle_{k-1,2}=\{\bar\bfv_1,\dots,\bar\bfv_{k+1}\}$ is well-defined and extends to a map $\cI_{\le k+1}(\bar{UP_2(\Gamma)})\to \cI_{\le k+1}(\bar{UP_2(\Gamma)})$ fixing all simplices but $\bar\triangle$. 
By combining $\bar r_{\bar\triangle}$ for all $k$-simplices in $\cI_{\le k+1}(\bar{UP_2(\Gamma)})$, we therefore have a strong deformation retract $\bar r_{k}:\cI_{\le k+1}(\bar{UP_2(\Gamma)})\to \cI_{\le k}(\bar{UP_2(\Gamma)})$ satisfying $\rho_{k}\circ \bar r_k = r_k\circ \rho_{k+1}$ since $\rho(\bar r_{\bar\triangle})=r_\triangle$ for each $k$-simplex $\bar\triangle\subset\cI(\bar{UP_2(\Gamma)})$.
Similar to $r_{\le k}$, then we have a strong deformation retraction $\bar r_{\le k}:\cI(\bar{UP_2(\Gamma)})\to \cI_{\le k}(\bar{UP_2(\Gamma)})$, where $\bar r=\bar r_{k}\circ\dots\circ \bar r_{N}$.

Now let $\Gamma'=(\mathsf{T}',\ell')\in\grapegraph_{\mathsf{normal}}^{\mathsf{large}}$ and suppose that there is a (semi-)isomorphism $\psi:\cRI(UP_2(\Gamma))\to \cRI(UP_2(\Gamma'))$.
Since $\cRI(UP_2(\Gamma))$ and $\cRI(UP_2(\Gamma'))$ have the same dimension, the diameter of $\mathsf{T}'$ must be $N$.
Then the restriction $\psi_{\le k}:\cRI_{\le k}(UP_2(\Gamma))\to\cRI_{\le k}(UP_2(\Gamma'))$ is well-defined for $k=N$ (trivially) and $k=N-1$ by the definition of a (semi-)isomorphism. By induction, we can easily see that the restriction is well-defined for any $k\ge 2$.

Note that for $\triangle$ and $\triangle_1$ given above, the images $\psi(\triangle)$ and $\psi(\triangle_1)$ are a $k$-simplex and an edge whose labels coincide by the definition of a (semi-)isomorphism. Thus we have $r'_{k+1}\circ\psi_{\le k+1}= \psi_{\le k}\circ r_{k+1}$, where $r'_{k+1}$ is the the strong deformation retract of $\cRI_{\le k+1}(UP_2(\Gamma'))$.
Similarly, if there is a (semi-)isomorphism $\bar\psi:\cI(\bar{UP_2(\Gamma)})\to\cI(\bar{UP_2(\Gamma')})$, then the assertion holds by the same argument, and we are done.
\end{proof}

Note that the image of a simplex $\triangle\subset\cRI(UP_2(\Gamma))$ induced from a path substem of $\mathsf{T}$ under $r_{\le 2}$ is a path subgraph $\mathcal{P}_\triangle\subset \cRI_{\le 2}(UP_2(\Gamma))$.

\begin{definition}\label{def:OrderOnSimplex}
A \emph{canonical order} on the set $\{\bfv_0,\dots,\bfv_k\}$ of vertices of a simplex $\triangle\subset\cRI(UP_2(\Gamma))$ is given by $\bfv_0<\bfv_1<\dots<\bfv_k$ if the image of $\triangle$ under the strong deformation retraction $r_{\le 2}$ is the path subgraph $\mathcal{P}_\triangle=[\bfv_0,\dots,\bfv_k]\subset\cRI_{\le 2}(UP_2(\Gamma))$, which is unique up to reversing.
Similarly, a \emph{canonical order} on the set of vertices of a simplex $\bar\triangle\subset\cI(\bar{UP_2(\Gamma)})$ is defined using $\bar r_{\le 2}$.
\end{definition}

As previously mentioned, the canonical order on the set of vertices of a simplex $\triangle\subset\cRI(UP_2(\Gamma))$ is the same as the natural order on the set of the twigs corresponding to the vertices up to reversing. The point is that the labels of simplices of $\cRI(UP_2(\Gamma))$ contain pretty much information of (the stem of) $\Gamma$ in the sense that Definition~\ref{def:OrderOnSimplex} is defined without mentioning the stem of $\Gamma$. 

The following lemma summarizes the above discussion and we omit the proof.

\begin{lemma}\label{lem:OrderOnSimplex}
Let $\Gamma, \Gamma'\in\grapegraph_{\mathsf{normal}}^{\mathsf{large}}$.
Then any (semi-)isomorphism $\psi:\cRI(UP_2(\Gamma))\to\cRI(UP_2(\Gamma'))$ or $\bar\psi:\cI(\bar{UP_2(\Gamma)})\to\cI(\bar{UP_2(\Gamma')})$ preserves a canonical order on the vertices of each simplex up to reversing.
\end{lemma}

Before we finish this subsection, as a by-product of canonical order, we give a criterion when $\cRI(UP_2(\Gamma))$ is simply connected.

\begin{proposition}\label{SM-types}
The following are equivalent.
\begin{enumerate}
\item The stem $\mathsf{T}$ is a path graph.
\item The reduced intersection complex $\cRI(UP_2(\Gamma))$ is a simplex.
\item The reduced intersection complex $\cRI(UP_2(\Gamma))$ is simply connected.
\end{enumerate}
\end{proposition}
\begin{proof}
It is obvious that the first implies the second as observed in Remark~\ref{Ex:PathofLengthN}, which implies the third.

Suppose that $\cRI(UP_2(\Gamma))$ is simply connected.
Then so is the subcomplex $\cRI_{\le 2}(UP_2(\Gamma))$ by Lemma~\ref{lemma:deformation retract}.
For each maximal product subcomplex $M(t)\subset UP_2(\Gamma)$, there are a unique twig $t$ in $\mathsf{T}$ and a unique vertex $\bfv_t\in\cRI_{\le 2}(UP_2(\Gamma))$ by Lemma~\ref{lem:TwigCorrespondence} such that $M(t)=\Gamma_{\mathring{t},1}\itimes \Gamma_{\mathring{t},2}$.
Namely, the set of vertices in $\cRI_{\le 2}(UP_2(\Gamma))$ is the same as the set of edges in $\mathsf{T}$, and two distinct vertices $\bfv_1$ and $\bfv_2$ in $\cRI_{\le 2}(UP_2(\Gamma))$ are adjacent if and only if the corresponding distinct twigs $t_1$ and $t_2$ share one vertex in $\mathsf{T}$.

Notice that whenever $\mathsf{T}$ has a vertex of valency $m\ge 3$, there is an embedded complete graph $\mathsf{K}_m$ in $\cRI_{\le 2}(UP_2(\Gamma))$, which is never simply connected.
Hence this contradiction prohibits the existence of a vertex of valency $\ge 3$ in $\mathsf{T}$, which implies that $\mathsf{T}$ is a path graph as desired.
\end{proof}
\section{Operations on large bunches of grapes}\label{section:Operations}
This section is devoted to three operations on $\Gamma=(\mathsf{T},\ell)\in\grapegraph^{\mathsf{large}}$ which are removing irrelevant parts of $\Gamma$ in the sense that these operations turn out to be quasi-isometry invariants for the $2$-braid groups.

\subsection{Pruning empty twigs and smoothing twigs}\label{section:eliminating}
The first operation is easily derived from Lemma~\ref{NoHangingEdge}, which says that attaching an edge to a graph $\Lambda$ by identifying one endpoint of the edge with a vertex in $\Lambda$ only increases the rank of the free factor of the $2$-braid group unless the identified vertex was a leaf of $\Lambda$.

\begin{definition}[Pruning empty twigs and smoothing twigs]\label{Def:Elimination}
Let $\Gamma=(\mathsf{T},\ell)\in\grapegraph^{\mathsf{large}}$ and $t=[v_0,\dots, v_k]$ a twig.
\begin{itemize}
\item Suppose that $t$ is an empty twig with $\val_\Gamma(v_k)=1$. A bunch of grapes $\Gamma'$ is obtained by \emph{pruning an empty twig} $t$ if $\Gamma'=\Gamma_{\mathsf{T}'}$, where $\mathsf{T}'$ is the induced subgraph of $\mathsf{T}$ with $\mathcal{V}(\mathsf{T}')=\mathcal{V}(\mathsf{T})\setminus\{v_1,\dots, v_k\}$.
\item A bunch of grapes $\Gamma'$ is obtained by \emph{smoothing a twig} $t$ if it is obtained by replacing the twig $t$ by an edge joining $v_0$ and $v_k$.
\end{itemize}
\end{definition}

The following lemma is obvious and we omit the proof and instead provide Algorithm~\ref{alg:normal} in Appendix~\ref{appendix:algorithms}.

\begin{lemma}\label{lemma:elimination}
Let $\Gamma\in\grapegraph^{\mathsf{large}}$. Then pruning empty twigs and smoothing twigs terminate in finite time, and the uniquely determined resulting graph $\Gamma_{\mathsf{normal}}$ becomes a large and normal bunch of grapes.

The bunch of grapes $\Gamma_{\mathsf{normal}}$ will be called the \emph{normal representative} of $\Gamma$.
\end{lemma}

\begin{theorem}\label{theorem:2-free factor}
Let $\Gamma\in\grapegraph^{\mathsf{large}}$ and $\Gamma_{\mathsf{normal}}\in\grapegraph^{\mathsf{large}}_{\mathsf{normal}}$ the normal representative of $\Gamma$.
Then $\mathbb{B}_2(\Gamma)$ is quasi-isometric to $\mathbb{B}_2(\Gamma')$.   
\end{theorem}
\begin{proof}
Since smoothing twigs does not change the isomorphism type of the $2$-braid groups over bunches of grapes, without loss of generality, we assume that $\Gamma$ has no twigs of length $\ge 2$.
Let $\Gamma_1$ be the bunch of grapes obtained from $\Gamma$ by only pruning an empty twig $t$.
Then by Lemma~\ref{NoHangingEdge}, $\mathbb{B}_2(\Gamma)\cong\mathbb{B}_2(\Gamma_{1})*\mathbb{F}_{k_{1}}$ for some $k_{1}\geq 1$.

Suppose that $\Gamma'$ is a large and normal bunch of grapes which is obtained from $\Gamma$ by repeatedly pruning empty twigs and smoothing twigs. Namely, based on Algorithm~\ref{alg:normal}, we have bunches of grapes $\Gamma=\Gamma_0,\dots, \Gamma_n=\Gamma'$ such that $\Gamma_{i+1}$ is obtained from $\Gamma_i$ by pruning an empty twig in $\Gamma_i$ and then smoothing twigs in $\Gamma_i$. Then
$\mathbb{B}_2(\Gamma)\cong \mathbb{B}_2(\Gamma')* \mathbb{F}_{k_1}*\dots*\mathbb{F}_{k_n}$.
Since Proposition~\ref{Prop:FreeFactor} implies that $\mathbb{B}_2(\Gamma')\cong\pi_1(UP_2(\Gamma'))*\mathbb{F}_{n}$ for some $n\geq 2$, we conclude that $\mathbb{B}_2(\Gamma)$ and $\mathbb{B}_2(\Gamma')$ are quasi-isometric by Theorem~\ref{PW}.
\end{proof}

In the upcoming two subsections, we use the notation $\mathcal{W}(-)$, a subset of the vertex set $\mathcal{V}(-)$, which is defined for objects related to $\Gamma=(\mathsf{T},\ell)\in\grapegraph^{\mathsf{large}}_{\mathsf{normal}}$ as follows:
First, we let $\mathcal{W}(\Gamma)=\mathcal{V}(\mathsf{T})$ and $\mathcal{W}(UD_2(\Gamma))=\{u\itimes v\mid u,v\in\mathcal{V}(\mathsf{T})\text{ and }u\neq v\}\subset UD_2(\Gamma)$. Then $\mathcal{W}(\bar{UD_2(\Gamma)})\subset \bar{UD_2(\Gamma)}$ is the preimage of $\mathcal{W}(UD_2(\Gamma))$ under the covering map.
For any subcomplex $X$ of $UD_2(\Gamma)$ ($\bar{UD_2(\Gamma)}$, resp.), $\mathcal{W}(X)$ is defined as $\mathcal{W}(UD_2(\Gamma))\cap X$ ($\mathcal{W}(\bar{UD_2(\Gamma)})\cap X$, resp.).
If $X$ is the domain of a standard product subcomplex $K$, then $\mathcal{W}(X)$ is the preimage of $\mathcal{W}(K)$ under the standard product structure.

In this section, whatever $X$ and $Y$ are chosen, there is a constant $C$ such that $X$ and $Y$ are contained in the $C$-neighborhoods of $\mathcal{W}(X)$ and $\mathcal{W}(Y)$, respectively. Thus, by a relative $L$-quasi-isometry $\phi:X\to Y$, we mean that $\phi$ is a relative $L$-quasi-isometry $(X,\mathcal{W}(X))\to (Y,\mathcal{W}(Y))$ in Definition~\ref{Def:RelQI}, which is an $(L,2C)$-quasi-isometry. 

\subsection{Picking over-grown grapes}\label{Subsection:PickingGrapes}
This subsection is devoted to the operation ``picking over-grown grapes'' on $\Gamma=(\mathsf{T},\ell)\in\grapegraph^{\mathsf{large}}_{\mathsf{normal}}$ and show that this operation is a quasi-isometry invariant for the $2$-braid groups, following the theme of Theorem~\ref{theorem:2-free factor}.

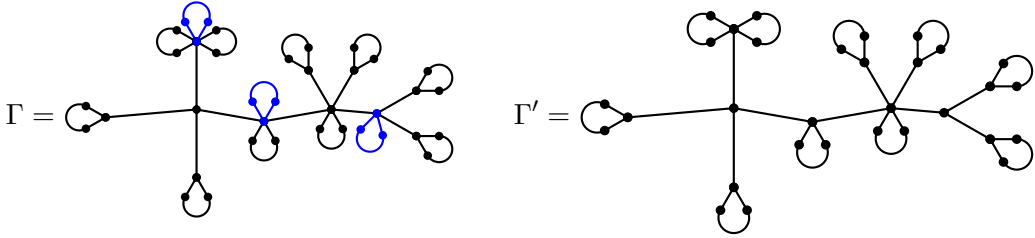
\begin{figure}[ht]
\[
\Gamma=\begin{tikzpicture}[baseline=-.5ex, scale=0.6]
\draw[thick,fill] (0,0) circle (2pt) node (A) {} -- ++(5:2) circle(2pt) node(B) {} -- +(0,1.5) circle (2pt) node (C) {}+(0,0) -- +(0,-1.5) circle (2pt) node(D) {} +(0,0) -- ++(-10:1.5) circle(2pt) node(E) {} -- ++(10:1.5) circle(2pt) node(F) {} -- +(120:1) circle(2pt) node(G) {} +(0,0) -- +(60:1) circle (2pt) node(H) {} +(0,0) -- ++(-5:1) circle (2pt) node(I) {} -- +(30:1) circle(2pt) node(J) {} +(0,0) -- +(-30:1) circle (2pt) node(K) {};
\grape[180]{A};
\grape[0]{C}; \grape[180]{C};
\grape[-90]{D};
\grape[-90]{F};
\grape[120]{G};
\grape[60]{H};
\grape[-30]{K};
\grape[30]{J};
\grape[-90]{E};
\begin{scope}[blue]
\grape[90]{C};
\grape[90]{E};
\grape[-105]{I};
\end{scope}
\end{tikzpicture}
\qquad
\Gamma'=\begin{tikzpicture}[baseline=-.5ex, scale=0.7]
\draw[thick,fill] (0,0) circle (2pt) node (A) {} -- ++(5:2) circle(2pt) node(B) {} -- +(0,1.5) circle (2pt) node (C) {}+(0,0) -- +(0,-1.5) circle (2pt) node(D) {} +(0,0) -- ++(-10:1.5) circle(2pt) node(E) {} -- ++(10:1.5) circle(2pt) node(F) {} -- +(120:1) circle(2pt) node(G) {} +(0,0) -- +(60:1) circle (2pt) node(H) {} +(0,0) -- ++(-5:1) circle (2pt) node(I) {} -- +(30:1) circle(2pt) node(J) {} +(0,0) -- +(-30:1) circle (2pt) node(K) {};
\grape[180]{A};
\grape[0]{C}; \grape[180]{C};
\grape[-90]{D};
\grape[-90]{F};
\grape[120]{G};
\grape[60]{H};
\grape[-30]{K};
\grape[30]{J};
\grape[-90]{E};
\end{tikzpicture}
\]
\caption{A bunch of grapes $\Gamma'$ obtained from $\Gamma$ by picking over-grown grapes.}
\label{figure:over-grown grapes}
\end{figure}

\begin{definition}[Picking over-grown grapes]\label{definition:picking}
Let $\Gamma=(\mathsf{T},\loops)\in \grapegraph^{\mathsf{large}}_{\mathsf{normal}}$. 

We say that a grape $\mathsf{C}$ at $v\in\mathsf{T}$ is \emph{over-grown} if $\loops(v) + \val_{\mathsf{T}}(v)\ge 4$.
We say that a bunch of grapes $\Gamma'$ is obtained from $\Gamma$ by \emph{picking an over-grown grape} $\mathsf{C}$ at $v$ if $\mathsf{C}$ is an over-grown grape and $\Gamma'=\Gamma\setminus(\mathsf{C}\setminus\{v\})$, or $\Gamma'=(\mathsf{T}',\loops')$, where
\begin{align*}
\mathsf{T}' &= \mathsf{T}
\quad\text{and}\quad
\loops'(w) =\begin{cases}
\loops(w) & w\neq v;\\
\loops(v)-1 & w=v.
\end{cases}
\end{align*}
\end{definition}

Note that in the above definition, $\Gamma'$ again belongs to $\grapegraph_{\mathsf{normal}}^{\mathsf{large}}$. 

\begin{theorem}\label{theorem:loop reducing}
Suppose that a bunch of grapes $\Gamma'$ is obtained from $\Gamma\in\grapegraph^{\mathsf{large}}_{\mathsf{normal}}$ by picking an over-grown grape. 
Then $\mathbb{B}_2(\Gamma)$ and $\mathbb{B}_2(\Gamma')$ are quasi-isometric.
\end{theorem}

\begin{Ex}
Let $\Gamma_1$ and $\Gamma_2$ be large and normal bunches of grapes over $\mathsf{S}_3$ as depicted in Figure \ref{RemovingOneLoop}, respectively. 
Since $\bar{UP_2(\Gamma_2)}$ contains maximal product subcomplexes quasi-isometric to $\mathbb{F}_2\times\mathbb{F}_2$ but $\bar{UP_2(\Gamma_1)}$ does not, there is no isomorphism between $\cI(\bar{UP_2(\Gamma_1)})$ and $\cI(\bar{UP_2(\Gamma_2)})$ and therefore $\mathbb{B}_2(\Gamma_1)$ and $\mathbb{B}_2(\Gamma_2)$ are not quasi-isometric by Theorem~\ref{theorem:IsobetInt} and Corollary~\ref{cor:QIbetweenGBGs}.
However, $|\cI(\bar{UP_2(\Gamma_1)})|$ and $|\cI(\bar{UP_2(\Gamma_2)})|$ are both locally (countably) infinite trees of infinite diameter and thus isometric.
\end{Ex}

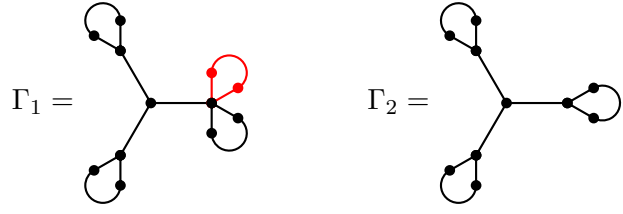
\begin{figure}[ht]
\[
\Gamma_1=\begin{tikzpicture}[baseline=-.5ex, scale=0.8]
\draw[thick,fill] (0,0) circle (2pt) node (A) {} -- ++(0:1) circle(2pt) node(B) {} 
(0,0) -- +(120:1) circle (2pt) node (C) {} 
(0,0) -- +(240:1) circle (2pt) node(D) {};
\begin{scope}[red]
\grape[60]{B};
\end{scope}
\grape[-60]{B};
\grape[120]{C};
\grape[240]{D};
\end{tikzpicture}
\qquad\qquad
\Gamma_2=\begin{tikzpicture}[baseline=-.5ex, scale=0.8]
\draw[thick,fill] (0,0) circle (2pt) node (A) {} -- ++(0:1) circle(2pt) node(B) {} 
(0,0) -- +(120:1) circle (2pt) node (C) {} 
(0,0) -- +(240:1) circle (2pt) node(D) {};
\grape[0]{B};
\grape[120]{C};
\grape[240]{D};
\end{tikzpicture}
\]
\caption{Picking non-over-grown grapes \emph{changes} the quasi-isometry type.}
\label{RemovingOneLoop}
\end{figure}

The process of picking or attaching grapes does not change the stem of the initial bunch of grapes. Moreover, changing the order of picking or attaching grapes at two different vertices does not effect on the resulting graph. Therefore the following lemma is obvious and instead of the proof, we provide Algorithm~\ref{alg:minimal rich} in Appendix~\ref{appendix:algorithms}.

\begin{lemma}\label{lemma:richrepresentative}
Let $\Gamma=(\mathsf{T},\ell)\in\grapegraph^{\mathsf{large}}_{\mathsf{normal}}$.
Then there exists a uniquely determined large and rich bunch of grapes $\Gamma_{\mathsf{rich}}=(\mathsf{T},\loops')\in\grapegraph_{\mathsf{rich}}^{\mathsf{large}}$ such that 
(1) it has no over-grown grapes and (2) it is obtained from $\Gamma$ by repeatedly pruning over-grown grapes or attaching grapes in finite time, i.e., \[\loops'(v) = \begin{cases}
\min\{\loops(v),2\} & \val_\mathsf{T}(v)=1;\\
1 & \val_\mathsf{T}(v)\ge2.
\end{cases}\]

%Then there exists a uniquely determined large and rich bunch of grapes $\Gamma_{\mathsf{rich}}=(\mathsf{T},\loops')\in\grapegraph_{\mathsf{rich}}^{\mathsf{large}}$, called the \emph{rich representative} of $\Gamma$ such that
The bunch of grapes $\Gamma_{\mathsf{rich}}$ will be called the \emph{rich representative} of $\Gamma$.
\end{lemma}

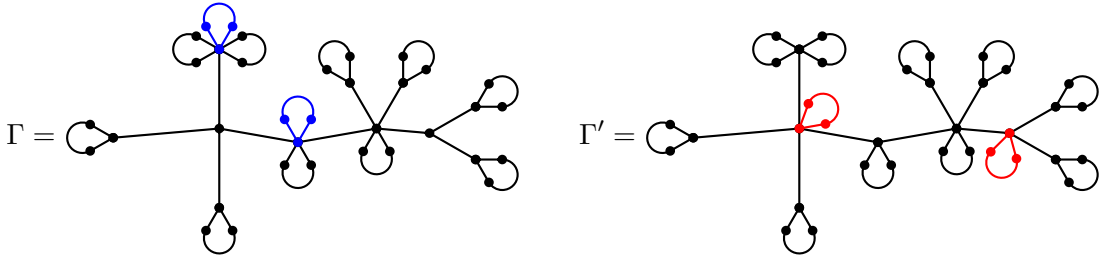
\begin{figure}[ht]
\[\Gamma=\begin{tikzpicture}[baseline=-.5ex, scale=0.7]
\draw[thick,fill] (0,0) circle (2pt) node (A) {} -- ++(5:2) circle(2pt) node(B) {} -- +(0,1.5) circle (2pt) node (C) {}+(0,0) -- +(0,-1.5) circle (2pt) node(D) {} +(0,0) -- ++(-10:1.5) circle(2pt) node(E) {} -- ++(10:1.5) circle(2pt) node(F) {} -- +(120:1) circle(2pt) node(G) {} +(0,0) -- +(60:1) circle (2pt) node(H) {} +(0,0) -- ++(-5:1) circle (2pt) node(I) {} -- +(30:1) circle(2pt) node(J) {} +(0,0) -- +(-30:1) circle (2pt) node(K) {};
\grape[180]{A};
\grape[0]{C}; \grape[180]{C};
\grape[-90]{D};
\grape[-90]{F};
\grape[120]{G};
\grape[60]{H};
\grape[-30]{K};
\grape[30]{J};
\grape[-90]{E};
\begin{scope}[blue]
\grape[90]{C};
\grape[90]{E};
\end{scope}
\end{tikzpicture}
\qquad
\Gamma'=\begin{tikzpicture}[baseline=-.5ex, scale=0.7]
\draw[thick,fill] (0,0) circle (2pt) node (A) {} -- ++(5:2) circle(2pt) node(B) {} -- +(0,1.5) circle (2pt) node (C) {}+(0,0) -- +(0,-1.5) circle (2pt) node(D) {} +(0,0) -- ++(-10:1.5) circle(2pt) node(E) {} -- ++(10:1.5) circle(2pt) node(F) {} -- +(120:1) circle(2pt) node(G) {} +(0,0) -- +(60:1) circle (2pt) node(H) {} +(0,0) -- ++(-5:1) circle (2pt) node(I) {} -- +(30:1) circle(2pt) node(J) {} +(0,0) -- +(-30:1) circle (2pt) node(K) {};
\grape[180]{A};
\grape[0]{C}; \grape[180]{C};
\grape[-90]{D};
\grape[-90]{F};
\grape[120]{G};
\grape[60]{H};
\grape[-30]{K};
\grape[30]{J};
\grape[-90]{E};
\begin{scope}[red]
\grape[40]{B};
\grape[-105]{I};
\end{scope}
\end{tikzpicture}\]
\caption{The rich representative $\Gamma'$ of $\Gamma$.}
\label{figure:minimal rich representative}
\end{figure}

\begin{corollary}\label{cor:minimalrich}
Let $\Gamma\in\grapegraph_{\mathsf{normal}}^{\mathsf{large}}$ and $\Gamma'\in\grapegraph_{\mathsf{rich}}^{\mathsf{large}}$ the rich representative of $\Gamma$.
Then $\mathbb{B}_2(\Gamma)$ and $\mathbb{B}_2(\Gamma')$ are quasi-isometric. 
\end{corollary}
\begin{proof}
Based on Algorithm~\ref{alg:minimal rich}, we have bunches of grapes $\Gamma=\Gamma_0,\dots, \Gamma_n=\Gamma'$ such that $\Gamma_{i+1}$ is obtained from $\Gamma_i$ by picking one grape or attaching one grape (which then becomes an over-grown grape). 
By Theorem~\ref{theorem:loop reducing}, $\mathbb{B}_2(\Gamma_i)$ and $\mathbb{B}_2(\Gamma_{i+1})$ are quasi-isometric, and therefore, the corollary holds.
\end{proof}

Throughout the remaining of this subsection, as we will prove Theorem~\ref{theorem:loop reducing}, we assume that $\Gamma'=(\mathsf{T},\ell')\in\grapegraph_{\mathsf{normal}}^{\mathsf{large}}$ is obtained from $\Gamma=(\mathsf{T},\ell)\in\grapegraph_{\mathsf{normal}}^{\mathsf{large}}$ by picking an over-grown grape at $v\in\mathsf{T}$. 
Additionally, we use the following notations.
\begin{itemize}
\item $\hat v$-components and extended $\hat v$-components of $\mathsf{T}$ are denoted by
$\mathsf{T}_i$ and $\mathsf{T}^+_i$, respectively, such that $\mathsf{T}_i\subset\mathsf{T}^+_i$ for $i=1,\dots,n$. 
\item The complement of $\mathsf{T}_i$ in $\mathsf{T}$ is denoted by $\mathsf{T}^c_i$. 
\item For each $1\le i\le n$, the twig of $\Gamma$ (or $\Gamma'$) joining $\mathsf{T}_i$ and $v$ is denoted by $t_i=[v,v_i]$.
\item The set of twigs of $\mathsf{T}$ is denoted by $\{t_a\mid a\in I\}$.
\end{itemize}

\begin{lemma}\label{lem:Twigs of intersection}
If $K={\Gamma_{\mathsf{T}_x}}\itimes{\Gamma_{\mathsf{T}_y}}$ is the intersection of maximal product subcomplexes of $UP_2(\Gamma)$, then one of the following holds:
\begin{enumerate}
\item\label{item:T_1} Both $\mathsf{T}_x$ and $\mathsf{T}_y$ are contained in $\mathsf{T}_i$ for some $1\le i\le n$.
\item\label{item:T_2} One of $\mathsf{T}_x$ and $\mathsf{T}_y$ is contained in $\mathsf{T}_i$ and the other is contained in $\mathsf{T}_j$ for distinct $1\le i\neq j\le n$.
\item\label{item:T_3} One of $\mathsf{T}_x$ and $\mathsf{T}_y$ is contained in $\mathsf{T}_i$ and the other contains $\mathsf{T}^c_i$ for some $1\le i\le n$.
\end{enumerate}
\end{lemma}
\begin{proof}
If $v$ is a leaf of $\mathsf{T}$, then Item~\eqref{item:T_1} or Item~\eqref{item:T_3} automatically holds.
Otherwise, let $\mathsf{P}$ be the path substem corresponding to $K$ obtained by Lemma~\ref{Lem:RIconnected} and $v'$ the nearest point projection of $v$ onto $\mathsf{P}$. When $v'$ is a leaf of $\mathsf{P}$ (or $v$ belongs to a $\mathring{\mathsf{P}}$-component of $\mathsf{T}$), Item~\eqref{item:T_3} holds. 
When $v'$ is not a leaf of $\mathsf{P}$, Item~\eqref{item:T_2} holds if $v'=v$, and Item~\eqref{item:T_1} holds otherwise.
\end{proof}

Let $\Phi:\cRI(UP_2(\Gamma))\to\cRI(UP_2(\Gamma'))$ be an isomorphism obtained by Lemma~\ref{Lem:IsometricRI}.
We first construct a quasi-isometry between the universal covers of maximal product subcomplexes corresponding to a vertex of $\cRI(UP_2(\Gamma))$ and its image vertex in $\cRI(UP_2(\Gamma'))$ under $\Phi$.
And then we obtain a quasi-isometry $\bar{UP_2(\Gamma)}\to\bar{UP_2(\Gamma')}$ by patching together the previously obtained quasi-isometries while constructing an isomorphism $\bar\Phi:\cI(\bar{UP_2(\Gamma)})\to\cI(\bar{UP_2(\Gamma')})$ such that $\rho'\circ\bar\Phi=\Phi\circ\rho$.
Lastly, Corollary~\ref{cor:QIbetweenGBGs} will complete the proof.

We start with a (relative) quasi-isometry between $\bar{\Gamma}$ and $\bar{\Gamma'}$ which will be an ingredient of the construction of quasi-isometries between maximal product subcomplexes. Recall that $\bar{(\Gamma,V)}=(\bar{\Gamma},\bar{V})$ where $\bar{V}=p^{-1}(V)$. 

\begin{proposition}\label{prop:coherent}
For each $1\le i\le n$, there are a constant $L_i>0$ and a relative $L_i$-quasi-isometry $\bar\varphi_i:\bar{\Gamma}\to \bar{\Gamma'}$ which induces a one-to-one correspondence between copies of $\bar{\Gamma_{\mathsf{T}'}}$ and copies of $\bar{\Gamma'_{\mathsf{T}'}}$ for any substem $\mathsf{T}'\subset\mathsf{T}$ containing $\mathsf{T}^c_i$ or contained in $\mathsf{T}_j$ with any $1\le j\le n$ such that the restriction of $\bar\varphi_i$ to a copy of $\bar{\Gamma_{\mathsf{T}'}}$ is the same as a relative $L_i$-quasi-isometry $\bar\varphi_{i,\mathsf{T}'}:\bar{\Gamma_{\mathsf{T}'}}\to\bar{\Gamma'_{\mathsf{T}'}}$.
\end{proposition}
\begin{proof}
We will only construct $\bar\varphi_1$ since the others can be constructed similarly.
Note that up to smoothing the connecting $2$-stars 
\[\Gamma\cong(\Gamma_{\mathsf{T}_1^c},v)*(\Gamma_{\mathsf{T}_1},v_1)\quad\text{and}\quad\Gamma'\cong(\Gamma'_{\mathsf{T}_1^c},v)*(\Gamma'_{\mathsf{T}_1},v_1).\]

\begin{claim}\label{Claim:InductiononN}
There is a relative $L_1$-quasi-isometry
$\bar\psi_{1}^c:\bar{\Gamma_{\mathsf{T}_1^c}}\to
\bar{\Gamma'_{\mathsf{T}_1^c}}$ such that for $k\ge 2$, $\bar\psi_{1}^c$ induces a one-to-one correspondence between copies of $\bar{\Gamma_{\mathsf{T}_k}}$ and copies of $\bar{\Gamma'_{\mathsf{T}_k}}$ such that the restriction to each copy is the identity map.
\end{claim}

Suppose that the claim holds. Since $\Gamma_{\mathsf{T}_1}$ and $\Gamma'_{\mathsf{T}_1}$ are identical, there is a relative isometry $\bar\psi_1:\bar{\Gamma_{\mathsf{T}_1}}\to \bar{\Gamma'_{\mathsf{T}_1}}$.
By Lemma~\ref{Lemma1.1inPW}, then there is a relative $L_1$-quasi-isometry $\bar\varphi_1:\bar{\Gamma}\to\bar{\Gamma'}$ obtained from $\bar\psi^c_{1}$ and $\bar\psi_1$, which is constructed by inductively patching copies of quasi-isometries $\bar\psi^c_{1}$ and $\bar\psi_1$. 
During the construction, by fixing a base point of $\bar{\Gamma_{\mathsf{T}^c_1}}$ and tracking the base points across all the copies of $\bar{\Gamma_{\mathsf{T}^c_1}}$ in $\bar{\Gamma}$, we may assume that $\bar\varphi_1$ has the following property: For any two copies $\mathsf{A}$ and $\mathsf{B}$ of $\bar{\Gamma_{\mathsf{T}^c_1}}$, there exists an isometry $f$ of $\bar{\Gamma}$ mapping the base point in $\mathsf{A}$ to the base point in $\mathsf{B}$ such that $p_{\Gamma'}\circ\bar\varphi_1=p_{\Gamma'}\circ\bar\varphi_1\circ f$; this is possible since $\bar\psi_1$ is the identity.

Let $\mathsf{T}'\subset\mathsf{T}$ be a substem of $\mathsf{T}$. If $\mathsf{T}'\subset \mathsf{T}_j$ for any $j$, then the restriction of $\bar\varphi_1$ to a copy of $\bar{\Gamma}_{\mathsf{T}'}$ is the restriction of either $\bar\psi^c_{1}$ or $\bar\psi_1$ to a copy of $\bar{\Gamma}_{\mathsf{T}'}$. Hence we are done by Claim~\ref{Claim:InductiononN} or the construction of $\bar\psi_1$.
If $\mathsf{T}'$ contains $\mathsf{T}^c_1$, up to smoothing the connecting $2$-stars, 
\[\Gamma_{\mathsf{T}'}\cong(\Gamma_{\mathsf{T}'\cap\mathsf{T}_1},v_1)*(\Gamma_{\mathsf{T}^c_1},v)\quad\text{and}\quad \Gamma'_{\mathsf{T}'}\cong (\Gamma'_{\mathsf{T}'\cap\mathsf{T}_1},v_1)*(\Gamma'_{\mathsf{T}^c_1},v).\]
Since $\bar\psi_1$ is the identity map, the construction of $\bar\varphi_1$ implies that the restriction of $\bar\varphi_1$ to a copy of $\bar{\Gamma_{\mathsf{T}'}}$ is a relative $L_1$-quasi-isometry 
$\bar\varphi_{1,\mathsf{T}'}:\bar{\Gamma_{\mathsf{T}'}}\to\bar{\Gamma'_{\mathsf{T}'}}$.
By the aforementioned property of $\bar\varphi_1$ above, $\bar\varphi_{1,\mathsf{T}'}$ does not depend on the choice of a copy of $\bar{\Gamma_{\mathsf{T}'}}$ in $\bar{\Gamma}$, and therefore we are done.

Finally, it suffices to prove Claim~\ref{Claim:InductiononN}.
As there is an over-grown grape at $v$, $\loops(v)$ is assumed to be $\ge 3$ if $n=1$, $\ge 2$ if $n=2$, or $\ge 1$ if $n\ge 3$.

\noindent{\textbf{When $n=1$.}} We have $\Gamma_{\mathsf{T}_1^c}=\Gamma_v$ and $\Gamma'_{\mathsf{T}_1^c}=\Gamma'_v$, each of which is a bouquet of at least two cycles. By the result in \cite{P95}, there is a relative quasi-isometry $\bar\psi_v:\bar{(\Gamma_v,v)}\to\bar{(\Gamma_v',v)}$. Since there is no proper substem in $\Gamma_v$, the claim holds vacuously if $\bar\psi^c_1=\bar\psi_v$.

\noindent{\textbf{When $n\ge 2$.}} If $n=2$, then we have up to smoothing the $2$-stars
\[\Gamma_{\mathsf{T}_1^c}\cong(\Gamma_{\mathsf{T}_2},v_2)*(\Gamma_v,v)\quad\text{and}\quad \Gamma'_{\mathsf{T}_1^c}\cong(\Gamma'_{\mathsf{T}_2},v_2)*(\Gamma'_v,v).\]
If $n\ge 3$, then both $\Gamma_{\mathsf{T}_1^c}$ and $\Gamma'_{\mathsf{T}_1^c}$ are obtained from $*_{j\neq 1}(\Gamma_{\mathsf{T}_j},v_j)$ by attaching 3-cycles at the central vertex of the connecting star.
In either case, by Lemma~\ref{Lem:PWLem1}, we have a relative $L'_1$-quasi-isometry
\[\bar\phi:\bar{(\Gamma_{\mathsf{T}^c_1},\{v_j\mid j\neq 1\})}\to\bar{((\Gamma'_{\mathsf{T}^c_1},\{v_j\mid j\neq 1\})}\]
for some $L'_1>0$ such that for each $j\neq 1$, $\bar\phi$ induces a one-to-one correspondence between copies of $\bar{\Gamma_{\mathsf{T}_j}}$ and copies of $\bar{\Gamma'_{\mathsf{T}_j}}$ and the restriction of $\bar\phi$ to each copy of $\bar{\Gamma_{\mathsf{T}_j}}$ is the identity map.
However, $\bar\phi$ is not yet the desired relative quasi-isometry since $p_{\Gamma}^{-1}(v)$ may not be bijectively mapped to $p_{\Gamma'}^{-1}(v)$.

On the other hand, in either $\bar{\Gamma_{\mathsf{T}^c_1}}$ or $\bar{\Gamma'_{\mathsf{T}^c_1}}$, there exists a one-to-one correspondence between the preimage of $v$ and the preimage of $v_2$ as follows:
for each vertex $\bar v$ in the preimage of $v$, there is a unique vertex $\bar{v}_2$ in the preimage of $v_2$ such that $\bar{v}_2$ is adjacent to $\bar v$, and vice-versa.
Consider a map $\bar\psi_{1}^c:\bar{\Gamma_{\mathsf{T}^c_1}}\to\bar{\Gamma'_{\mathsf{T}^c_1}}$ which is obtained from $\bar\phi$ by replacing the image of such $\bar v$ under $\bar\phi$ by a vertex in $p_{\Gamma'}^{-1}(v)\subset\bar{\Gamma'_{\mathsf{T}_1^c}}$ which is adjacent to $\bar\phi(\bar{v}_2)$. 
Since the restriction of $\bar\phi$ to $p^{-1}(v_2)$ is bilipschitz onto $p_{\Gamma'}^{-1}(v_2)$, $\bar\psi_{1}^c$ is the desired relative $L_1$-quasi-isometry, which completes the proof.
\end{proof}

In the above proposition, the following is implicit: If $\mathsf{T}''\subset\mathsf{T}'$ is a substem of $\mathsf{T}$ containing $\mathsf{T}^c_i$ or contained in $\mathsf{T}_j$, then the restriction of $\bar\varphi_{i,\mathsf{T}'}$ to any copy of $\bar{\Gamma_{\mathsf{T}''}}$ is $\bar\varphi_{i,\mathsf{T}''}$.

Consider the sets of the universal covers of labels of vertices of $\cI(\bar{UP_2(\Gamma)})$ and $\cI(\bar{UP_2(\Gamma')})$
\[\{\bar{\mathbf{M}}_a=\bar{\Gamma_{\mathring{t_a},1}}\times\bar{\Gamma_{\mathring{t_a},2}}\mid a\in I\}\quad\text{and}\quad\{\bar{\mathbf{M}}'_a=\bar{\Gamma'_{\mathring{t}_a,1}}\times\bar{\Gamma'_{\mathring{t}_a,2}}\mid a\in I\},\]
respectively. According to Proposition~\ref{prop:coherent} and Lemma~\ref{Lem:RelQIProduct}, we have the set of relative $L^2$-quasi-isometries 
\begin{equation}\label{setofQI}
\{\bar\phi_a:(\bar{\mathbf{M}}_a,\mathcal{W}(\bar{\mathbf{M}}_a))\to(\bar{\mathbf{M}}'_{a},\mathcal{W}(\bar{\mathbf{M}}'_a))\mid a\in I\}
\end{equation}
for $L=\max\{L_1,\dots,L_n\}$ such that if $t_a$ is contained in $\mathsf{T}^+_i$, then $\bar\phi_{a}$ is the product of two relative quasi-isometries each of which is the restriction of $\bar\varphi_i$ defined in Proposition~\ref{prop:coherent} to a copy corresponding to a coordinate of $\bar{\mathbf{M}}_{a}$.

Suppose that there are two distinct maximal product subcomplexes of $\bar{UP_2(\Gamma)}$ with standard product structures 
$\bar\iota_{\barbfv}:\bar{\mathbf{M}}_{a}\to \bar{UP_2(\Gamma)}$ and $
\bar\iota_{\barbfu}:\bar{\mathbf{M}}_{b}\to\bar{UP_2(\Gamma)}$
such that $\bar\iota_{\barbfv}(\bar{\mathbf{M}}_{a})\cap\bar\iota_{\barbfu}(\bar{\mathbf{M}}_{b})=\bar K$ is a standard product subcomplex with the domain $\bar{\mathbf{K}}=\bar{\Gamma_{\mathsf{T}_x}}\times\bar{\Gamma_{\mathsf{T}_y}}$. 
If $t_a$ and $t_b$ are contained in $\mathsf{T}^+_i$, then $\mathsf{T}_x$ and $\mathsf{T}_y$ satisfy Item~\eqref{item:T_1} or Item~\eqref{item:T_3} of Lemma~\ref{lem:Twigs of intersection}. 
If $t_a$ and $t_b$ are contained in $\mathsf{T}^+_i$ and $\mathsf{T}^+_j$, respectively, for distinct $1\le i,j\le n$, then $\mathsf{T}_x$ and $\mathsf{T}_y$ satisfy Item~\eqref{item:T_2} of Lemma~\ref{lem:Twigs of intersection}.
By Proposition~\ref{prop:coherent} and the construction of $\bar\phi_a$ and $\bar\phi_b$, we thus have the commutative diagram represented by solid arrows in Figure~\ref{Figure:the intersection of two maps}, where $\bar{\mathbf{K}}'=\bar{\Gamma'_{\mathsf{T}_x}}\times\bar{\Gamma'_{\mathsf{T}_y}}$ and the four diagonal arrows are canonical (relative) isometric embeddings.
By the paragraph below Proposition~\ref{prop:coherent}, moreover, there is a universal property as follows: Let $\bar{\mathbf{K}}_0=\bar{\Gamma_{\mathsf{T}_p}}\times\bar{\Gamma_{\mathsf{T}_q}}$ be the domain of a standard product subcomplex contained in $\bar{K}$ such that $\mathsf{T}_p$ and $\mathsf{T}_q$ satisfy Lemma~\ref{lem:Twigs of intersection}.
Then there are canonical (relative) isometric embeddings of $\bar{\mathbf{K}}_0$ and $\bar{\mathbf{K}}'_0=\bar{\Gamma'_{\mathsf{T}_p}}\times\bar{\Gamma'_{\mathsf{T}_q}}$ such that the whole diagram in Figure~\ref{Figure:the intersection of two maps} commutes.
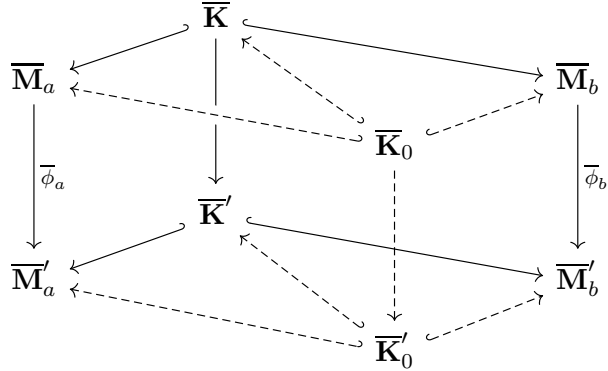
\begin{figure}[ht]
\[
\begin{tikzcd}[column sep=4em, row sep=0.5em]
& \bar{\mathbf{K}}\arrow[hookrightarrow,rrd," "]\arrow[ddd," "]\arrow[dashed,hookleftarrow,rdd," "] && \\
\bar{\mathbf{M}}_a\arrow[ddd,"\bar\phi_{a}"]\arrow[hookleftarrow,ru," "]\arrow[dashed,hookleftarrow,crossing over,rrd," "] &&& \bar{\mathbf{M}}_b\arrow[ddd,"\bar\phi_{b}"]\\
&  &\bar{\mathbf{K}}_0\arrow[dashed,hookrightarrow,ru," "]\arrow[dashed,ddd," "] & \\
& \bar{\mathbf{K}}'\arrow[hookrightarrow,rrd," "]\arrow[dashed,hookleftarrow,rdd," "] && \\
\bar{\mathbf{M}}'_a\arrow[hookleftarrow,ru," "]\arrow[dashed,hookleftarrow,rrd," "]&&& \bar{\mathbf{M}}'_b\\
&  &\bar{\mathbf{K}}'_0\arrow[dashed,hookrightarrow,ru," "] &
\end{tikzcd}
\]
\caption{A universal property of the embedding of the intersection of maximal product subcomplexes.}
\label{Figure:the intersection of two maps}
\end{figure}

For a pair of a vertex $\bar\bfv\in\cI(\bar{UP_2(\Gamma)})$ and a vertex $\bar\bfv'\in\cI(\bar{UP_2(\Gamma')})$ with $\Phi\circ\rho(\bar\bfv)=\rho'(\bar\bfv')$, we have a relative $L^2$-quasi-isometry $\bar\phi_{\barbfv}:\bar M_{\barbfv}\to\bar M'_{\barbfv'}$ such that the diagram in Figure~\ref{Figure:mapbetweenMPS} commutes, where $\bar\iota_{\barbfv}:\bar{\mathbf{M}}_{a}\to\bar{UP_2(\Gamma)}$ ($\bar\iota'_{\barbfv'}:\bar{\mathbf{M}}'_{a}\to\bar{UP_2(\Gamma')}$, resp.) is the standard product structure of $\bar{M}_{\barbfv}$ ($\bar{M}'_{\barbfv'}$, resp.).
If we fix one such relative $L^2$-quasi-isometry, then the following lemma says that we can canonically find a one-to-one correspondence between $\cI^{(0)}(\bar{UP_2(\Gamma)})$ and $\cI^{(0)}(\bar{UP_2(\Gamma')})$, and a relative quasi-isometry between maximal product subcomplexes corresponding to vertices in each correspondence.
\begin{figure}[ht]
\[
\begin{tikzcd}[column sep=4pc, row sep=2pc]
\bar{M}_{\barbfv} (\subset \bar{UP_2(\Gamma)})\arrow[r,"\bar\phi_{\barbfv}"] & \bar{M}'_{\barbfv'} (\subset \bar{UP_2(\Gamma')})\\ \bar{\mathbf{M}}_{a} \arrow[u,"\bar\iota_{\barbfv} "] \arrow[r,"\bar\phi_{a}"] & \bar{\mathbf{M}}'_{a}\arrow[u,"\bar\iota'_{\barbfv'} "]
\end{tikzcd}
\]
\caption{A relative quasi-isometry between maximal product subcomplexes is induced from an element in Set~\eqref{setofQI}.}
\label{Figure:mapbetweenMPS}
\end{figure}

\begin{lemma}\label{ReducingLoops}
Let $\barbfv$, $\barbfv'$ and $\bar\phi_{\barbfv}$ be given as above.
For any vertex $\barbfu\in\mathcal{N}_1(\barbfv)\setminus\{\barbfv\}$, then there are a uniquely determined vertex $\barbfu'\in\cI(\bar{UP_2(\Gamma')})$ adjacent to $\barbfv'$ and a relative $L^2$-quasi-isometry $\bar\phi_{\barbfu}:\bar{M}_{\barbfu}\to \bar{M}_{\barbfu'}$ such that the restrictions of $\bar\phi_{\barbfu}$ and $\bar\phi_{\barbfv}$ to $\bar{M}_{\barbfu}\cap \bar{M}_{\bar\bfv}$ are identical relative $L^2$-quasi-isometries.

Moreover, if vertices $\barbfu_0,\barbfu_1,\dots,\barbfu_k$ in $\mathcal{N}_1(\barbfv)$ form a simplex in $\cI(\bar{UP_2(\Gamma)})$, 
then the restrictions of $\bar\phi_{\barbfu_j}$ to $\bigcap_{j=0}^k\bar{M}_{\barbfu_j}$ for $0\le j\le k$ are identical relative $L^2$-quasi-isometries.
\end{lemma}
\begin{proof}
Let $\bar\iota_{\barbfu}:\bar{\mathbf{M}}_{b}\to\bar{UP_2(\Gamma)}$ be the standard product structure of $\bar{M}_{\barbfu}$.
By Lemma~\ref{lem:Intofmaximal}, $\bar{M}_{\barbfv}\cap\bar{M}_{\barbfu}$ is a standard product subcomplex $\bar K$ with domain $\bar{\mathbf{K}}=\bar{\Gamma_{\mathsf{T}_x}}\times\bar{\Gamma_{\mathsf{T}_y}}$. 
By the commutative diagrams in Figures~\ref{Figure:the intersection of two maps} and \ref{Figure:mapbetweenMPS}, we have the diagram represented by solid arrows in Figure~\ref{Figure:}.
It easily follows that there is a relative isometric embedding $\bar\iota':\bar{\mathbf{M}}'_b\to\bar{UP_2(\Gamma')}$ which is represented by the dashed arrow in Figure~\ref{Figure:} such that the whole diagram commutes.
\begin{figure}[ht]
\[\begin{tikzcd}[column sep=2em, row sep=1em]
& \bar{UP_2(\Gamma)}\arrow[rrr," "']   &&& \bar{UP_2(\Gamma')}  & \\
\bar{\mathbf{M}}_a\arrow[hookleftarrow,rdd," "']\arrow[ru,"\bar\iota_{a} "]\arrow[rrr,"\bar\phi_{a}" near end] &&& \bar{\mathbf{M}}'_a\arrow[hookleftarrow,rdd," "]\arrow[ru,"\bar\iota'_{\barbfv'} "]&&\\
&& \bar{\mathbf{M}}_b\arrow[luu,crossing over,"\bar\iota_{b}" near start]\arrow[ rrr,crossing over,"\bar\phi_b" near end] &&& \bar{\mathbf{M}}'_b\arrow[dashed,luu,"\bar\iota' "']\\
& \bar{\mathbf{K}}\arrow[hookrightarrow,ru," "]\arrow[rrr,"  "] &&& \bar{\mathbf{K}'}\arrow[hookrightarrow,ru," "] &
\end{tikzcd}\]
\caption{The commutative diagram describing how to glue two relative quasi-isometries sharing the domains.}
\label{Figure:}
\end{figure}

Let $\barbfu'\in\cI(\bar{UP_2(\Gamma')})$ be the vertex corresponding to $\bar\iota'(\bar{\mathbf{M}}'_b)$.
It follows that $\barbfu'$ is adjacent to $\barbfv'$ and $\rho'(\barbfu')=\Phi(\rho(\barbfu))$.
Then using the diagram in Figure~\ref{Figure:mapbetweenMPS} (after changing the indices properly), we have a relative $L^2$-quasi-isometry $\bar\phi_{\barbfu}:\bar{M}_{\barbfu}\to\bar{M}_{\barbfu'}$, and the diagram in Figure~\ref{Figure:} implies that $\bar\phi_{\barbfv}$ and $\bar\phi_{\barbfu}$ coincide on $\bar{M}_{\barbfv}\cap\bar{M}_{\barbfu}$.

Suppose that vertices $\barbfu_0,\barbfu_1,\dots,\barbfu_k$ in $\mathcal{N}_1(\barbfv)$ form a simplex in $\cI(\bar{UP_2(\Gamma)})$.
By the universal property described in the commutative diagram in Figure~\ref{Figure:the intersection of two maps}, $\bar\phi_{\barbfu_j}$'s coincide on $\bigcap_{j=0}^{k}\bar{M}_{\barbfu_j}$, which completes the proof.
\end{proof}

Consider the set of relative $L$-quasi-isometries which are quasi-inverses of elements in Set~\eqref{setofQI}. Then the same argument as in Lemma~\ref{ReducingLoops} after switching the role of $\Gamma$ and $\Gamma'$ holds and $\bar\phi_{\barbfu'_i}$ will be a quasi-inverse of $\bar\phi_{\barbfu}$.
By virtue of the uniqueness in Lemma~\ref{ReducingLoops}, thus we have an isomorphism $\mathcal{N}_1(\barbfv)\to\mathcal{N}_1(\barbfv')$ and a relative map 
\[\bar\phi_{1}:\mathfrak{g}_{\cI}(\mathcal{N}_{1}(\barbfv)))\to\mathfrak{g}_{\cI}(\mathcal{N}_{1}(\barbfv')).\]
By Lemma~\ref{Lem:RelQI} and the fact that $\mathcal{N}_1(\barbfv)$ is countably infinite, $\bar\phi_1$ turns out to be a relative $L^2$-quasi-isometry.

\begin{proof}[Proof of Theorem~\ref{theorem:loop reducing}]
For a path $\gamma=(\barbfv=\bar\bfv_0, \dots, \bar\bfv_n)$ in $\cI(\bar{UP_2(\Gamma)})$, by applying Lemma~\ref{ReducingLoops} along consecutive vertices in $\gamma$, we can construct a relative quasi-isometry $\bar\phi_\gamma:\bar M_{\barbfv_n}\to \bar M_{\barbfv_n'}$.
If there is a path $\delta$ in $\cI(\bar{UP_2(\Gamma)})$ from $\barbfv_0$ to $\barbfv_n$ which is homotopic to $\gamma$, since $\cI(\bar{UP_2(\Gamma)})$ is a simply connected flag complex by Theorem~\ref{theorem:structureofI}, we can find a sequence of two kinds of elementary homotopies from $\gamma$ to $\delta$ as follows:
\begin{enumerate}
\item one is to cancel a back-tracking,
\[
(\dots,\bar\bfv_i, \bar\bfv_{i+1}, \bar\bfv_i, \dots)\leftrightarrow(\dots,\bar\bfv_i,\dots)
\]
\item the other is to replace two edges by an edge
\[
(\dots,\bar\bfv_i, \bar\bfv_{i+1}, \bar\bfv_{i+2},\dots)\leftrightarrow(\dots,\bar\bfv_i, \bar\bfv_{i+2},\dots)
\]
when $\bar\bfv_i, \bar\bfv_{i+1}$ and $\bar\bfv_{i+2}$ form a triangle in $\cI(\bar{UP_2(\Gamma)})$.
\end{enumerate}
Then it easily follows from Lemma~\ref{ReducingLoops} that both elementary homotopies do not change the quasi-isometry induced from a path and thus $\bar\phi_\gamma$ and $\bar\phi_\delta$ are identical.

By Lemma~\ref{Lem:IsometricRI}, there is an isomorphism $\cI(\bar{UP_2(\Gamma)})\to\cI(\bar{UP_2(\Gamma')})$ sending $\barbfv$ to $\barbfv'$.
Combining this fact with the previous paragraph, we can inductively construct a relative map
\[\bar\phi_{n+1}:\mathfrak{g}_{\cI}(\mathcal{N}_{n+1}(\barbfv))\to\mathfrak{g}_{\cI}(\mathcal{N}_{n+1}(\barbfv'))\]
which induces an isomorphism $\mathcal{N}_{n+1}(\barbfv)\to\mathcal{N}_{n+1}(\barbfv')$ from the relative map $\bar\phi_{n}$.
By Proposition~\ref{Prop:connectedsimplicial}, $\cI(\bar{UP_2(\Gamma)})$ is connected and $\mathfrak{g}_{\cI}(\cI(\bar{UP_2(\Gamma)}))=\bar{UP_2(\Gamma)}$. 
It follows that the direct limit of $\bar\phi_n$ is a relative map $\bar\phi:\bar{UP_2(\Gamma)}\to\bar{UP_2(\Gamma')}$ which induces an isomorphism $\bar\Phi:\cI(\bar{UP_2(\Gamma)})\to\cI(\bar{UP_2(\Gamma')})$.

Since $\mathcal{N}_{n+1}(\barbfv)\setminus\mathcal{N}_{n}(\barbfv)$ is countably infinite, by Lemma~\ref{Lem:RelQI}, if $\bar\phi_n$ is a relative $L^2$-quasi-isometry, then $\bar\phi_{n+1}$ is also a relative $L^2$-quasi-isometry.
From the fact that $\bar\phi_0=\bar\phi_{\barbfv}$ is a relative $L^2$-quasi-isometry, it follows that $\bar\phi$ is a relative $L^2$-quasi-isometry.
By Corollary~\ref{cor:QIbetweenGBGs}, therefore, $\mathbb{B}_2(\Gamma)$ is quasi-isometric to $\mathbb{B}_2(\Gamma')$.
\end{proof}

\subsection{Pruning over-grown substems}\label{Subsection:Quasi-folding}
In this subsection, we define the last operation, called \emph{pruning over-grown substems} on $\Gamma\in\grapegraph^{\mathsf{large}}_{\mathsf{rich}}$, which is a quasi-isometry invariant for the $2$-braid groups as well. 
This operation is particularly applied to $\Gamma$ which has a kind of symmetry and leads to a more simplified bunch of grapes.

\begin{definition}[Pruning over-grown substems]\label{Def:Quasi-foldingonSuperrich}
Let $\Gamma=(\mathsf{T},\ell)\in\grapegraph^{\mathsf{large}}_{\mathsf{rich}}$ and $v\in\mathcal{V}(\mathsf{T})$ a vertex. 
Consider the set $\mathcal{C}^+(\Gamma,v)$ of extended $\hat v$-components of $\Gamma$ which are not grapes and an equivalence relation on the set defined by the presence of isometry.

Suppose that there exists an equivalence class $\{\Gamma^+_1,\dots,\Gamma^+_m\}$ on $\mathcal{C}^+(\Gamma,v)$ such that $m\ge 3$.
Then each $\Gamma^+_i$ is called an \emph{over-grown stem} at $v$ and the union of the $\Gamma^+_i$'s is called a \emph{scope} of pruning. 
In this case, the complement of $\Gamma_1$ (the $\hat v$-component of $\Gamma$ contained in $\Gamma^+_1$) in $\Gamma$ is said to be obtained from $\Gamma$ by \emph{pruning an over-grown stem} $\Gamma^+_1$ at $v$.

We say that a bunch of grapes $\Gamma'$ is obtained from $\Gamma$ by \emph{pruning over-grown substems} at $v$ if $\Gamma'$ is obtained from $\Gamma$ by repeatedly pruning an over-grown stem from a given scope of pruning at $v$ till the scope is no longer considered as a scope.
\end{definition}

It is worth noticing that in the above definition, $\Gamma'$ again belongs to $\grapegraph^{\mathsf{large}}_{\mathsf{rich}}$.

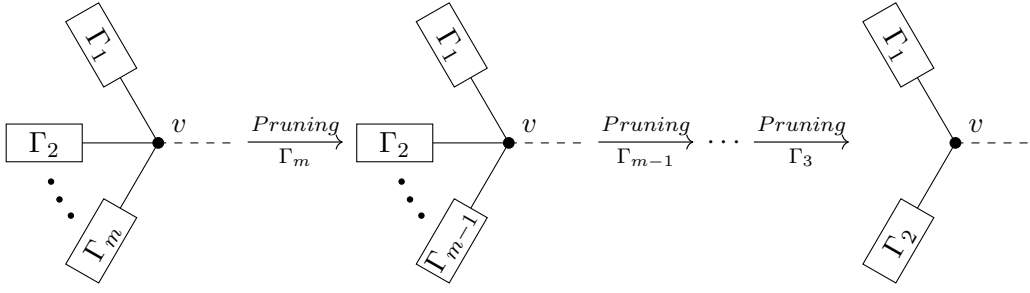
\begin{figure}[ht]
\[
\begin{tikzcd}[ampersand replacement=\&, row sep=2pc, column sep=3pc]
\begin{tikzpicture}[baseline=-.5ex]
\draw[fill] (0,0) circle (2pt) node[above right] {$v$};
\draw[dashed] (0,0) -- (1,0);
\begin{scope}[rotate=120]
\draw (0,0) -- (1,0);
\draw (1,-0.25) rectangle node[rotate={-60},above=-2ex] {$\Gamma_1$} (2,0.25);
\end{scope}
\begin{scope}[rotate=180]
\draw (0,0) -- (1,0);
\draw (1,-0.25) rectangle node[rotate={0},above=-2ex] {$\Gamma_2$} (2,0.25);
\end{scope}
\begin{scope}[rotate=240]
\draw (0,0) -- (1,0);
\draw (1,-0.25) rectangle node[rotate={60},above=-2ex] {$\Gamma_m$} (2,0.25);
\end{scope}
\foreach \i in {1,2,3} {
\draw[fill] ({\i*10+190}:1.5) circle (1pt);
}
\end{tikzpicture}
\arrow[r,"Pruning","\Gamma_m"'] 
\&
\begin{tikzpicture}[baseline=-.5ex]
\draw[fill] (0,0) circle (2pt) node[above right] {$v$};
\draw[dashed] (0,0) -- (1,0);
\begin{scope}[rotate=120]
\draw (0,0) -- (1,0);
\draw (1,-0.25) rectangle node[rotate={-60},above=-2ex] {$\Gamma_1$} (2,0.25);
\end{scope}
\begin{scope}[rotate=180]
\draw (0,0) -- (1,0);
\draw (1,-0.25) rectangle node[rotate={0},above=-2ex] {$\Gamma_2$} (2,0.25);
\end{scope}
\begin{scope}[rotate=240]
\draw (0,0) -- (1,0);
\draw (1,-0.25) rectangle node[rotate={60},above=-2ex] {$\Gamma_{m-1}$} (2,0.25);
\end{scope}
\foreach \i in {1,2,3} {
\draw[fill] ({\i*10+190}:1.5) circle (1pt);
}
\end{tikzpicture}
\arrow[r,"Pruning","\Gamma_{m-1}"'] 
\&
\cdots
\arrow[r,"Pruning","\Gamma_3"'] 
\&
\begin{tikzpicture}[baseline=-.5ex]
\draw[fill] (0,0) circle (2pt) node[above right] {$v$};
\draw[dashed] (0,0) -- (1,0);
\foreach \i in {1,2} {
\begin{scope}[rotate={\i*120}]
\draw (0,0) -- (1,0);
\draw (1,-0.25) rectangle node[rotate={\i*120-180},above=-2ex] {$\Gamma_{\i}$} (2,0.25);
\end{scope}
}
\end{tikzpicture}
\end{tikzcd}
\]
\caption{Pruning over-grown substems; at $v$, there is at least one grape.}
\label{figure:quasi-folding_2}
\end{figure}

As in Theorem~\ref{theorem:2-free factor} or \ref{theorem:loop reducing}, pruning over-grown substems does not change the quasi-isometry type of the $2$-braid group as follows.

\begin{theorem}\label{theorem:quasi folding}
For $\Gamma\in\grapegraph^{\mathsf{large}}_{\mathsf{rich}}$, let $\Lambda$ be a bunch of grapes obtained from $\Gamma$ by pruning an over-grown substem. Then $\mathbb{B}_2(\Gamma)$ is quasi-isometric to $\mathbb{B}_2(\Lambda)$.
\end{theorem}

Throughout the remaining of this subsection, we assume that $\Lambda$ is obtained from $\Gamma$ by pruning an over-grown substem $\Gamma^+_n$ at $v\in\mathsf{T}$ and $\Gamma^+_1,\dots,\Gamma^+_{n-1}$ are the remaining extended $\hat v$-components of $\Gamma$ in $\mathcal{C}^+(\Gamma,v)$. We denote by $\Gamma_i$  the $\hat v$-component of $\Gamma$ contained in $\Gamma^+_i$.
Note that there is an obvious embedding $\iota:\Lambda\hookrightarrow\Gamma$.
Then we let $w=\iota^{-1}(v)$ and $\Lambda_j=\iota^{-1}(\Gamma_j)$ for $j=1,\dots,n-1$.
Additionally, we let $v_i\in\Gamma_i$ be the vertex adjacent to $v$ in in $\Gamma$ for $i=1,\dots,n$, and $w_j=\iota^{-1}(v_j)\in\Lambda_j$ for $j=1,\dots,n-1$.

Here, we use the decompositions of $UD_2(\Gamma)$ (and thus the graph-of-groups decompositions of $\mathbb{B}_2(\Gamma)$) obtained in Section~\ref{Section:GOGdecomposition}. 
Recall that the subcomplex $UP_2(\Gamma,v)$ is the union of $\Gamma_i\itimes\Gamma^c_i$'s and the subcomplex $UD_2(\Gamma,v)$ of $UD_2(\Gamma)$ is the union $UP_2(\Gamma,v)\cup\bigcup_i{UD_2(\Gamma^+_i)}$. 
By Lemma~\ref{Lem:D'} and the fact that there is at least one grape at $v$, we have $\mathbb{B}_2(\Gamma)\cong\mathbb{B}_2(\Gamma,v)*\mathbb{F}_N$, where $\mathbb{B}_2(\Gamma,v)=\pi_1(UD_2(\Gamma,v))$ and $N>0$.
Thus we will only focus on the structures of $UD_2(\Gamma,v)$ and its universal cover.

The intersection $UD_2(\Gamma^+_i)\cap UD_2(\Gamma^+_j)$ is empty for $i\neq j$, and $UP_2(\Gamma,v)\cap UD_2(\Gamma^+_i)=\Gamma_i\itimes\{v\}$.
From this fact, a graph-of-groups decomposition $\mathcal{G}(\mathbb{B}_2(\Gamma,v))$ of $\mathbb{B}_2(\Gamma,v)$ is constructed from $\cRI(UP_2(\Gamma,v))$ by attaching one edge to each vertex such that additional vertices are labelled by $UD_2(\Gamma^+_i)$'s and edges labelled by $\Gamma_i$'s.
Moreover, the Bass-Serre tree $\mathcal{T}(\mathbb{B}_2(\Gamma,v))$ associated to $\mathcal{G}(\mathbb{B}_2(\Gamma,v))$ is the union of copies of $\cI(\bar{UP_2(\Gamma,v)})$ and the additional vertices and edges labelled by lifts of $UD_2(\Gamma_i^+)$ and $\Gamma_i\itimes\{v\}$, respectively, for some $1\le i\le n$.
There is also a canonical set map $\mathfrak{g}_{\mathcal{T}}:\mathcal{T}(\mathbb{B}_2(\Gamma,v))\to 2^{\bar{UD_2(\Gamma,v)}}$ sending (the interiors of) simplices to the corresponding subcomplexes (see \eqref{Eq:SetmapforT}).

Now, we will construct a specific quasi-isometry $\bar{UP_2(\Gamma,v)}\to\bar{UP_2(\Lambda,w)}$ using finitely many quasi-isometries defined below.
For each $(i,j)\in R=\{(i,j)\mid \Gamma^+_i\cong\Lambda^+_j\}$, let $\varphi_{i,j}:{(\Gamma_i,\{v_i\})}\to{(\Lambda_j,\{w_j\})}$ be a relative isometry. Note that $\varphi_{i,j}$ can also be seen as a relative isometry ${\Gamma_i}\to{\Lambda_j}$.
Then there are a relative isometry $\bar\varphi_{i,j}:\bar{(\Gamma_i,\{v_i\})}\to\bar{(\Lambda_j,\{w_j\})}$ which is an elevation of $\varphi_{i,j}$, and the set $\bar Q_R$ of $\bar\varphi_{i,j}$'s.
By Lemma~\ref{Lemma:PWConclusion}, moreover, there is also a relative $L'$-quasi-isometry 
\[f_{i,j}:\bar{(\Gamma^c_i,\mathcal{W}(\Gamma^c_i)\setminus\{v\})}\to\bar{(\Lambda^c_j,\mathcal{W}(\Lambda^c_j)\setminus\{w\})}\] 
with the property that the restriction to a copy of $\bar{(\Gamma_{i'},v_{i'})}$ for any $i'\neq i$ is a relative isometry in $\bar Q_R$ onto a copy of $\bar{(\Lambda_{j'},v_{j'})}$ for some $j'\neq j$.
Then we slightly modify $f_{i,j}$ to a relative $L$-quasi-isometry $\bar\varphi^c_{i,j}:\bar{\Gamma^c_i}\to\bar{\Lambda^c_j}$ which satisfies the property $f_{i,j}$ does and maps the preimage of $v$ to the preimage of $w$. This is possible since any vertex of the preimage of $v$ in $\bar{\Gamma^c_i}$ (the preimage of $w$ in $\bar{\Lambda^c_i}$, resp.) is adjacent to a vertex of the preimage of $v_{i'}$ (the preimage of $w_{j'}$, resp.) for some $i'\neq i$ ($j'\neq j$, resp.). We refer to the construction in the case of $n\ge 2$ in the proof of Proposition~\ref{prop:coherent}.

\begin{lemma}\label{Lem:Generalization1}
There exists a quasi-isometry 
$\bar\phi_v:\bar{UP_2(\Gamma,v)}\to\bar{UP_2(\Lambda,w)}$ such that 
\begin{itemize}
\item a maximal product subcomplex of $\bar{UP_2(\Gamma,v)}$ which is a $p$-lift of $\Gamma_i\itimes\Gamma^c_i$ is mapped to a maximal product subcomplex of $\bar{UP_2(\Lambda,w)}$ which is a $p$-lift of $\Lambda_j\itimes\Lambda^c_j$ for some $j$ satisfying $(i,j)\in R$, and
\item the restriction to the intersection of two maximal product subcomplexes is an isometry.
\end{itemize}
\end{lemma}
\begin{proof}
Let $\bar M, \bar M'\subset\bar{UP_2(\Gamma,v)}$ be two intersecting maximal product subcomplexes with standard product structures $\bar\iota_{\bar M}:\bar{\Gamma_i}\times\bar{\Gamma^c_i}\to\bar{UP_2(\Gamma,v)}$ and
$\bar\iota_{\bar M'}:\bar{\Gamma^c_{i'}}\times\bar{\Gamma_{i'}}\to \bar{UP_2(\Gamma,v)}$ for some $i\neq i'$.
Then their intersection $\bar K=\bar M\cap\bar M'$ admits two identical standard product structures
\begin{align*}
\bar\iota_{\bar M}|_{\bar\iota_{\bar M}^{-1}(\bar K)}\circ (id_i\times\bar\iota_{i'})=
\bar\iota_{\bar M'}|_{\bar\iota_{\bar M'}^{-1}(\bar K)}\circ (\bar\iota_{i}\times id_{i'})
&:\bar{(\Gamma_i,v_i)}\times\bar{(\Gamma_{i'},v_{i'})}\to\bar K\subset\bar{UP_2(\Gamma,v)},
\end{align*}
where $id_{i}$ and $id_{i'}$ are isometries on $\bar{(\Gamma_i,v_i)}$ and $\bar{(\Gamma_{i'},v_{i'})}$, respectively, and $\bar\iota_i:\bar{\Gamma_i}\hookrightarrow\bar{\Gamma^c_{i'}}$ and $\bar\iota_{i'}:\bar{\Gamma_{i'}}\hookrightarrow\bar{\Gamma^c_i}$ are elevations of embeddings.

By the assumption on $\Gamma$ (and $\Lambda$), there exists $\bar{\Lambda_j}$ which is isometric to $\bar{\Gamma_i}$ via a relative isometry $\bar\varphi_{i,j}$ in $\bar Q_R$.
Then there exists a maximal product subcomplex $\bar N \subset \bar{UP_2(\Lambda,w)}$ with standard product structure $\bar\iota_{\bar N}:\bar{\Lambda_j}\times\bar{\Lambda^c_j}\to \bar{UP_2(\Lambda,w)}$.
By Lemma~\ref{Lem:RelQIProduct}, then $\bar\varphi_{i,j}\times\bar\varphi^c_{i,j}$ induces a relative $L^2$-quasi-isometry $\phi_0:\bar M \to \bar N$.
Moreover, there is a unique copy of $\bar{\Lambda_{j'}}$ in $\bar{\Lambda_j^c}$ such that $\bar{\Gamma_{i'}}$ is isometric to $\bar{\Lambda_{j'}}$ via $\bar\varphi_{i',j'}$ and $\phi_0(\bar K)$ is isometric to a standard product subcomplex $\bar K'$ with standard product structure $\bar{\Lambda_j}\times\bar{\Lambda_{j'}}\to\bar{UP_2(\Lambda,w)}$.
Let $\bar N'\subset\bar{UP_2(\Lambda,w)}$ be the maximal product subcomplex containing $\bar K'$ with standard product structure $\bar\iota_{\bar N'}:\bar{\Lambda_{j'}^c}\times\bar{\Lambda_{j'}}\to \bar{UP_2(\Lambda,w)}$.
As before, $\bar\varphi_{i',j'}^c\times\bar\varphi_{i',j'}$ induces a relative $L^2$-quasi-isometry $\phi'_0:\bar{M'}\to\bar{N'}$.

\begin{figure}[ht]
\[
\begin{tikzcd}[column sep=1em, row sep=1em]
& \bar M\cup\bar M' \arrow[dashed, rrr,"\phi_0\cup\phi'_0 "] &&& \bar N\cup\bar N' & \\
\bar{\Gamma_i}\times\bar{\Gamma^c_i}\arrow[hookleftarrow,rdd," id_i\times\bar\iota_{i'}"']\arrow[ru,"\bar\iota_{\bar M} "]\arrow[rrr,"\bar\varphi_{i,j}\times\bar\varphi^c_{i,j}" near end] &&& \bar{\Lambda_j}\times\bar{\Lambda^c_j}\arrow[hookleftarrow,rdd,"id_j\times\bar\iota_{j'}"',near end]\arrow[ru,"\bar\iota_{\bar N}"'] &&\\
&& \bar{\Gamma^c_{i'}}\times\bar{\Gamma_{i'}}\arrow[luu,crossing over,"\bar\iota_{\bar M'}" near start]\arrow[rrr,crossing over,"\bar\varphi^c_{i',j'}\times\bar\varphi_{i',j'}" near end] &&& \bar{\Lambda^c_{j'}}\times\bar{\Lambda_{j'}}\arrow[luu,"\bar\iota_{\bar N'}"']\\
& \bar{\Gamma_i}\times\bar{\Gamma_{i'}}\arrow[hookrightarrow,ru,"\bar\iota_i\times id_{i'}"]\arrow[rrr,"\bar\varphi_{i,j}\times\bar\varphi_{i',j'} "] &&& \bar{\Lambda_j}\times\bar{\Lambda_{j'}}\arrow[hookrightarrow,ru,"\iota_j\times id_{j'}"'] &
\end{tikzcd}
\]
\caption{Gluing of two relative $L^2$-quasi-isometries}
\label{Figure:the intersection of two maps_first}
\end{figure}

Consider the commutative diagram in Figure~\ref{Figure:the intersection of two maps_first} where all maps except for $\phi_0\cup \phi_0'$ are known to be either quasi-isometric embeddings or quasi-isometries.
By Lemma~\ref{Lem:RelQI}, then the map $\phi_0\cup\phi'_0:\bar{M}\cup\bar{M'}\to\bar{N}\cup\bar{N'}$ is a relative $L^2$-quasi-isometry.
Combining this fact with the fact that the intersection of three distinct maximal product subcomplexes in $\bar{UP_2(\Gamma,v)}$ or $\bar{UP_2(\Lambda,w)}$ is empty, we obtain a relative $L^2$-quasi-isometry 
\[
\phi_1:\mathfrak{g}_{\cI}(\mathcal{N}_1(\barbfv_{\bar M}))\to\mathfrak{g}_{\cI}(\mathcal{N}_1(\barbfv'_{\bar N})),
\]
where $\barbfv_{\bar M}$ and $\barbfv'_{\bar N}$ are the vertices of $\cI(\bar{UP_2(\Gamma,v)})$ and $\cI(\bar{UP_2(\Lambda,w)})$ corresponding to $\bar M$ and $\bar N$, respectively. By iterating this process, 
we can construct a relative $L^2$-quasi-isometry 
$\phi_k:\mathfrak{g}_{\cI}(\mathcal{N}_k(\barbfv_{\bar M}))\to\mathfrak{g}_{\cI}(\mathcal{N}_k(\barbfv'_{\bar N}))$, 
whose direct limit is a relative $L^2$-quasi-isometry $\phi:\bar{UP_2(\Gamma,v)}\to\bar{UP_2(\Lambda,w)}$, and this is the desired relative quasi-isometry $\bar\phi_v$.
\end{proof}

Now, we are ready to prove Theorem~\ref{theorem:quasi folding}.

\begin{proof}[Proof of Theorem~\ref{theorem:quasi folding}]
By the paragraph above reviewing $UD_2(\Gamma,v)$, we know that \[\mathbb{B}_2(\Gamma)\cong\mathbb{B}_2(\Gamma,v)*\mathbb{F}_N\quad\text{and}\quad \mathbb{B}_2(\Lambda)\cong\mathbb{B}_2(\Lambda,w)*\mathbb{F}_{M},\]
where $N, M\ge 1$.
By Theorem~\ref{PW}, thus it suffices to prove that $\mathbb{B}_2(\Gamma,v)$ and $\mathbb{B}_2(\Lambda,w)$ are quasi-isometric, or equivalently, $\bar{UD_2(\Gamma,v)}$ and $\bar{UD_2(\Lambda,w)}$ are quasi-isometric.

Now, we fix a copy $\cI_0$ of $\cI(\bar{UP_2(\Gamma,v)})$ in $\mathcal{T}(\mathbb{B}_2(\Gamma,v))$ and call it a $0$-level copy.
Then $1$-level copies of $\cI(\bar{UP_2(\Gamma,v)})$ are defined as copies of $\cI(\bar{UP_2(\Gamma,v)})$ whose distance from the $0$-level copy is $2$.
Inductively, $n$-level copies of $\cI(\bar{UP_2(\Gamma,v)})$ are defined as copies of $\cI(\bar{UP_2(\Gamma,v)})$ whose distance from $(n-1)$-level copies of $\cI(\bar{UP_2(\Gamma,v)})$ is $2$ and which are not $k$-level copies for $k<n$.
Similarly, we fix a copy $\cI'_0$ of $\cI(\bar{UP_2(\Lambda,w)})$ in $\mathcal{T}(\mathbb{B}_2(\Lambda,w))$ and then define $n$-level copies of $\cI(\bar{UP_2(\Lambda,w)})$ in $\mathcal{T}(\mathbb{B}_2(\Lambda,w))$.

Let $\mathcal{L}_k(\mathbb{B}_2(\Gamma,v))$ be the minimal subcomplex of $\mathcal{T}(\mathbb{B}_2(\Gamma,v))$ containing all $m$-level copies for $m\le k$.
Then we will inductively construct an isomorphism and a relative quasi-isometry
\[\bar\Phi_k:\mathcal{L}_k(\mathcal{T}(\mathbb{B}_2(\Gamma,v)))\to\mathcal{L}_k(\mathcal{T}(\mathbb{B}_2(\Lambda,w)))\text{ and }\bar\phi_k:\mathfrak{g}_{\mathcal{T}}(\mathcal{L}_k(\mathbb{B}_2(\Gamma,v)))\to\mathfrak{g}_{\mathcal{T}}(\mathcal{L}_k(\mathbb{B}_2(\Lambda,w))).\]

The case when $k=0$ is shown by declaring $\bar\phi_0$ as the relative $L$-quasi-isometry $\bar\phi_v$ in Lemma~\ref{Lem:Generalization1} and $\bar\Phi_0$ as the isomorphism induced by $\bar\phi_v$.

Let $\barbfv\in\mathcal{T}(\mathbb{B}_2(\Gamma,v))$ be a vertex labelled by a $p$-lift of $UD_2(\Gamma^+_i)$ which is adjacent to the $0$-level copy and let $\bar Z_0=\mathfrak{g}_{\mathcal{T}}(\cI_0)\cap \mathfrak{g}_{\mathcal{T}}(\barbfv)$.
By the construction of $\bar\phi_v$, the restriction of $\bar\phi_0$ to $\bar Z_0$ is an isometry onto a $p$-lift $\bar Z'_0$ of $\Lambda_j\itimes\{w\}$ in $\mathfrak{g}_{\mathcal{T}}(\cI'_0)$ for some $1\le j\le n-1$; the isometry is the restriction of $\bar\varphi_{i,j}\in \bar Q_R$.
Then there exists a $p$-lift $\bar U'$ of $UD_2(\Lambda^+_j)$ in $\bar{UD_2(\Lambda,w)}$ such that $\bar U'\cap\mathfrak{g}_{\mathcal{T}}(\cI'_0)=\bar Z'_0$.
Moreover, there is a relative isometry $\varphi^+_{i,j}:(\Gamma^+_i,\{v\})\to (\Lambda^+_j,\{w\})$ which extends $\varphi_{i,j}$ such that $\varphi^+_{i,j}$ induces a (relative) isometry $\bar{UD_2(\Gamma^+_i)}\to\bar{UD_2(\Lambda^+_j)}$.
Let $\barbfv'$ be the vertex in $\mathcal{T}(\mathbb{B}_2(\Lambda,w))$ corresponding to $\bar U'$, i.e., $\mathfrak{g}_{\mathcal{T}}(\barbfv')=\bar U'$.
Then we can naturally extend $\bar\Phi_0$ and $\bar\phi_0$ to an isomorphism and a relative $L$-quasi-isometry (by Lemma~\ref{Lem:RelQI})
\begin{equation}\label{Eq:AlittlebitExtended}
\cI_0\cup\{\barbfv\}\to\cI_0\cup\{\barbfv'\}\quad\text{and}\quad \bar\phi^+_0:\mathfrak{g}_{\mathcal{T}}(\barbfv) \cup_{\bar Z_0} \bar{UP_2(\Gamma,v)}\to\mathfrak{g}_{\mathcal{T}}(\barbfv') \cup_{\bar Z_0} \bar{UP_2(\Lambda,w)}    
\end{equation}
such that the restriction of $\bar\phi^+_0$ to $\mathfrak{g}_{\mathcal{T}}(\barbfv)$ is an isometry onto $\mathfrak{g}_{\mathcal{T}}(\barbfv')$.
Here, $\cI_0\cup\{\barbfv\}$ and $\cI_0\cup\{\barbfv'\}$ mean their induced subcomplexes.

Let $\barbfv_1\in\mathcal{T}(\mathbb{B}_2(\Gamma,v))$ be a vertex in a $1$-level copy $\cI_1$ of $\cI(\bar{UP_2(\Gamma)})$ adjacent to $\barbfv$.
Then the edge $[\barbfv,\barbfv_1]$ corresponds to a $p$-lift $\bar Z_1$ of $\Gamma_{i}\itimes\{v\}$ in $\mathfrak{g}_{\mathcal{T}}(\barbfv)$ and $\barbfv_1$ corresponds to a maximal product subcomplex which is a $p$-lift of $\Gamma_i\itimes\Gamma^c_i$ in $\bar{UD_2(\Gamma,v)}$.
In $\bar{UD_2(\Lambda,w)}$, there exists a $p$-lift of $\Lambda_j\itimes\{w\}$ which is the image of $\bar Z_1$ under $\bar\phi^+_0$, and thus there is a maximal product subcomplex which is a $p$-lift of $\Lambda_j\itimes\Lambda^c_j$ and corresponds to a vertex $\barbfv'_1\in\mathcal{T}(\mathbb{B}_2(\Lambda,w))$.
Then we can naturally extend maps in \eqref{Eq:AlittlebitExtended} to an isomorphism and a relative $L$-quasi-isometry
\[
\cI_0 \cup \cI_1 \cup \{\barbfv\}\to\cI'_0 \cup \cI'_1 \cup \{\barbfv'\}\text{ and } \mathfrak{g}_{\mathcal{T}}(\cI_0) \cup \mathfrak{g}_{\mathcal{T}}(\barbfv)\cup \mathfrak{g}_{\mathcal{T}}(\cI_1)\to\mathfrak{g}_{\mathcal{T}}(\cI'_0) \cup \mathfrak{g}_{\mathcal{T}}(\barbfv')\cup \mathfrak{g}_{\mathcal{T}}(\cI'_1).
\]
Since this process can be repeated for any $1$-level copies independently, we have an isomorphism and a relative $L$-quasi-isometry 
\[
\bar\Phi_1:\mathcal{L}_1(\mathbb{B}_2(\Gamma,v))\to\mathcal{L}_1(\mathbb{B}_2(\Lambda,w))\quad\text{and}\quad \bar\phi_1:\mathfrak{g}_{\mathcal{T}}(\mathcal{L}_1(\mathbb{B}_2(\Gamma,v)))\to\mathfrak{g}_{\mathcal{T}}(\mathcal{L}_1(\mathbb{B}_2(\Lambda,w))).
\]

From the fact that $\mathcal{T}(\mathbb{B}_2(\Gamma,v))$ and $\mathcal{T}(\mathbb{B}_2(\Lambda,w))$ are the Bass-Serre \emph{trees}, we can inductively define $\bar\Phi_k$ and $\bar\phi_k$.
Since each $\bar\phi_k$ is a relative $L$-quasi-isometry by Lemma~\ref{Lem:RelQI}, so is the direct limit of $\bar\phi_k$.
As in Proposition~\ref{Prop:connectedsimplicial}, it can easily be shown that $\mathcal{T}(\bar{UD_2(\Gamma,v)})$ is connected and $\mathfrak{g}_{\mathcal{T}}(\mathcal{T}(\bar{UD_2(\Gamma,v)}))=\bar{UD_2(\Gamma,v)}$.
It follows that the direct limit of $\bar\phi_k$ is a relative quasi-isometry from $\bar{UD_2(\Gamma,v)}$ to $\bar{UD_2(\Lambda,w)}$, and therefore we are done.
\end{proof}
\section{Quasi-isometric rigidity of 2-braid groups over bunches of grapes}\label{section:proofofThm}

In this section, we show that for any $\Gamma\in\grapegraph^{\mathsf{large}}$, there is a uniquely determined $\Gamma_{\mathsf{min}}\in\grapegraph^{\mathsf{large}}_{\mathsf{rich}}$ which is minimal in the sense that there are no other large and rich bunches of grapes obtained from $\Gamma_{\mathsf{min}}$ by the operations defined in the previous section.
Using this fact, we show that the set of the $2$-braid groups over $\Gamma\in\grapegraph^{\mathsf{large}}$ is quasi-isometrically rigid in the sense that for $\Gamma,\Gamma'\in\grapegraph^{\mathsf{large}}$, two $2$-braid groups $\mathbb{B}_2(\Gamma)$ and $\mathbb{B}_2(\Gamma')$ are quasi-isometric if and only if $\Gamma_{\mathsf{min}}$ and $\Gamma'_{\mathsf{min}}$ are isometric.

\subsection{Quasi-minimal representatives}\label{section:quasi-minimal}
In the previous section, from $\Gamma\in\grapegraph^{\mathsf{large}}$, after obtaining $\Gamma_{\mathsf{normal}}$ by repeatedly picking an empty twig and then smoothing twigs, we obtained the rich represntative $\Gamma_{\mathsf{rich}}$ of $\Gamma_{\mathsf{normal}}$ such that the $2$-braid groups over these three graphs are quasi-isometric to each other.
It was obvious that such $\Gamma_{\mathsf{normal}}$ and thus $\Gamma_{\mathsf{rich}}$ are uniquely determined.

On the other hand, uniqueness of the terminal bunch of grapes under the iterated process of pruning over-grown substems is not clear since it is not sure whether pruning two different over-grown substems could be interchangeable.
However, the following proposition provides an affirmative answer and the uniqueness of the terminal bunch of grapes under this process.

\begin{figure}[ht]
\[
\begin{tikzcd}[ampersand replacement=\&, row sep=1.5pc, column sep=2.5pc]
\Gamma=\begin{tikzpicture}[baseline=-.5ex]
\draw[fill] (0,0) circle (2pt) node[above right] {$v$} (1.5,0) circle (2pt) node[above left] {$v'$};
\draw[dashed] (0,0) -- (1.5,0);
\foreach \i in {120,180,240} {
\begin{scope}[rotate=\i]
\draw (0,0) -- (1,0);
\draw (1,-0.25) rectangle node[rotate={\i-180},above=-2ex] {$\mathsf{X}$} (2,0.25);
\end{scope}
\draw[fill] ({\i/60*10+180}:1.5) circle (1pt);
}
\draw[dashed] (-2.5,-2) rectangle (0.7, 2);
\draw (-1.5,1.5) node[left] {$\Lambda$};
\begin{scope}[xshift=1.5cm]
\foreach \i in {-60,0,60} {
\begin{scope}[rotate=\i]
\draw (0,0) -- (1,0);
\draw (1,-0.25) rectangle node[rotate=\i,above=-2ex] {$\mathsf{Y}$} (2,0.25);
\end{scope}
\draw[fill] ({\i/60*10-30}:1.5) circle (1pt);
}
\draw[dashed] (-.7,-2) rectangle (2.5, 2);
\draw (1.5,1.5) node[right] {$\Lambda'$};
\end{scope}
\end{tikzpicture}
\arrow[r,"Pruning","\Lambda"'] 
\arrow[d,"Pruning"',"\Lambda'"]
\&
\Gamma'=\begin{tikzpicture}[baseline=-.5ex]
\draw[fill] (0,0) circle (2pt) node[above right] {$v$} (1.5,0) circle (2pt) node[above left] {$v'$};
\draw[dashed] (0,0) -- (1.5,0);
\foreach \i in {120,240} {
\begin{scope}[rotate=\i]
\draw (0,0) -- (1,0);
\draw (1,-0.25) rectangle node[rotate={\i-180},above=-2ex] {$\mathsf{X}$} (2,0.25);
\end{scope}
}
\begin{scope}[xshift=1.5cm]
\foreach \i in {-60,0,60} {
\begin{scope}[rotate=\i]
\draw (0,0) -- (1,0);
\draw (1,-0.25) rectangle node[rotate=\i,above=-2ex] {$\mathsf{Y}$} (2,0.25);
\end{scope}
\draw[fill] ({\i/60*10-30}:1.5) circle (1pt);
}
\draw[dashed] (-.5,-2) rectangle (2.5, 2);
\draw (1.5,1.5) node[right] {$\Lambda'$};
\end{scope}
\end{tikzpicture}
\arrow[d,"Pruning"',"\Lambda'"]
\\
\Gamma_2=\begin{tikzpicture}[baseline=-.5ex]
\draw[fill] (0,0) circle (2pt) node[above right] {$v$} (1.5,0) circle (2pt) node[above left] {$v'$};
\draw[dashed] (0,0) -- (1.5,0);
\foreach \i in {120,180,240} {
\begin{scope}[rotate=\i]
\draw (0,0) -- (1,0);
\draw (1,-0.25) rectangle node[rotate={\i-180},above=-2ex] {$\mathsf{X}$} (2,0.25);
\end{scope}
\draw[fill] ({\i/60*10+180}:1.5) circle (1pt);
}
\draw[dashed] (-2.5,-2) rectangle (0.5, 2);
\draw (-1.5,1.5) node[left] {$\Lambda$};
\begin{scope}[xshift=1.5cm]
\foreach \i in {-60,60} {
\begin{scope}[rotate=\i]
\draw (0,0) -- (1,0);
\draw (1,-0.25) rectangle node[rotate=\i,above=-2ex] {$\mathsf{Y}$} (2,0.25);
\end{scope}
}
\end{scope}
\end{tikzpicture}
\arrow[r,"Pruning","\Lambda"']
\&
\Gamma_0=\begin{tikzpicture}[baseline=-.5ex]
\draw[fill] (0,0) circle (2pt) node[above right] {$v$} (1.5,0) circle (2pt) node[above left] {$v'$};
\draw[dashed] (0,0) -- (1.5,0);
\foreach \i in {120,240} {
\begin{scope}[rotate=\i]
\draw (0,0) -- (1,0);
\draw (1,-0.25) rectangle node[rotate={\i-180},above=-2ex] {$\mathsf{X}$} (2,0.25);
\end{scope}
}
\begin{scope}[xshift=1.5cm]
\foreach \i in {-60,60} {
\begin{scope}[rotate=\i]
\draw (0,0) -- (1,0);
\draw (1,-0.25) rectangle node[rotate=\i,above=-2ex] {$\mathsf{Y}$} (2,0.25);
\end{scope}
}
\end{scope}
\end{tikzpicture}
\end{tikzcd}
\]
\caption{Pruning over-grown substems whose scopes are disjoint or intersect at one vertex; at $v$ ($v'$, resp.), there is at least one grape.}
\label{figure:lattice1}
\end{figure}
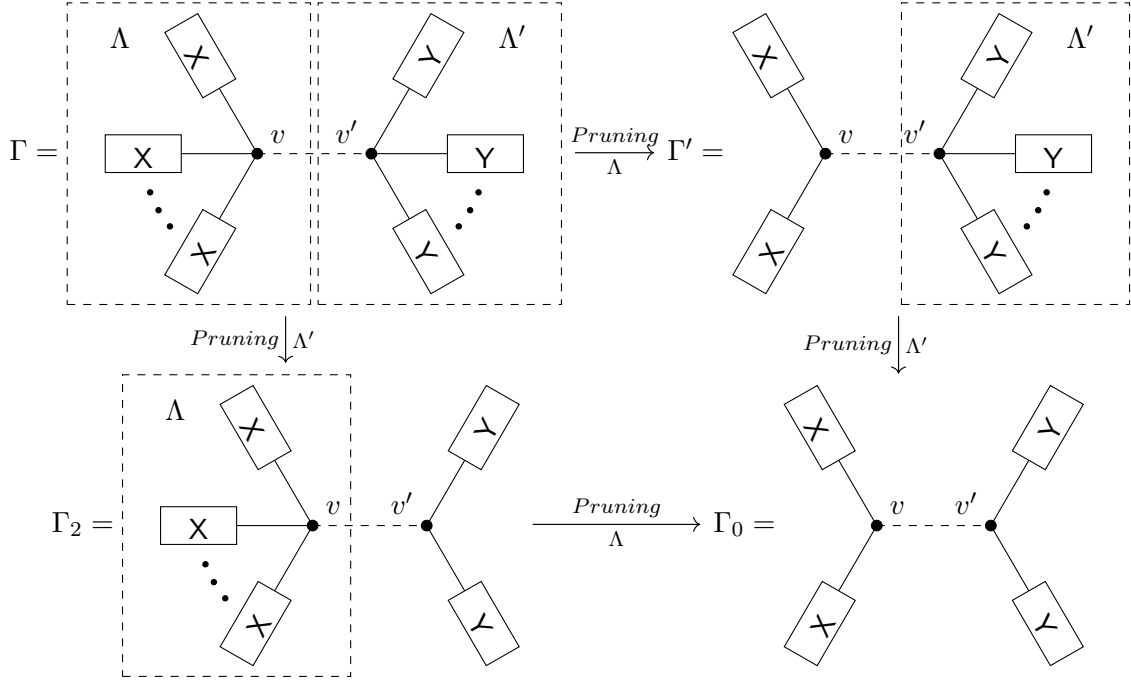

\begin{proposition}\label{proposition:quasi-minimal}
For $\Gamma\in\grapegraph^{\mathsf{large}}_{\mathsf{rich}}$, let $\Gamma'$ and $\Gamma''$ be two bunches of grapes each of which is obtained from $\Gamma$ by pruning over-grown substems repeatedly as many times as possible.
Then $\Gamma'$ and $\Gamma''$ are isometric.
\end{proposition}
\begin{proof}
We first claim the following:
If $\Gamma'$ and $\Gamma''$ are obtained from $\Gamma$ by pruning over-grown substems at $v$ and $v'$, which are not necessarily distinct, then there is a bunch of grapes $\Gamma_0$ obtained from both $\Gamma_1$ and $\Gamma_2$ by (repeatedly) pruning over-grown substems.

If the claim is true, then the set of all bunches of grapes identical to $\Gamma$ up to (repeatedly) pruning over-grown substems has a unique bunch of grapes which can be obtained by (repeatedly) pruning over-grown substems from others. The existence follows easily from the induction argument on the number of vertices of valency $\ge 3$ in the stem.

Now, let us prove the claim. Let $\Lambda$ and $\Lambda'$ be the scopes of pruning at $v$ and $v'$ in $\Gamma$, respectively.
Then there is a trichotomy as follows:
\begin{enumerate}
\item Two scopes $\Lambda$ and $\Lambda'$ are disjoint;
\item Two scopes $\Lambda$ and $\Lambda'$ intersect exactly at $v=v'$;
\item One contains the other, say $\Lambda'\subsetneq \Lambda$.
\end{enumerate}

For the first two cases, it is easy to see that by pruning over-grown substems from both $\Gamma$ and $\Gamma'$, we have a bunch of grapes $\Gamma_0$ as desired. See Figure~\ref{figure:lattice1}.

Suppose that $\Lambda'\subsetneq\Lambda$. Then there are distinct copies  $\Lambda'^{(0)}=\Lambda',\Lambda'^{(1)},\cdots,\Lambda'^{(k)}$ of $\Lambda'$ in $\Lambda$, where $k+1$ is the number of (extended) $\hat v$-components in $\Lambda$ and the vertices corresponding to $v'$ will be denoted by $v'^{(0)}=v',v'^{(1)},\cdots,v'^{(k)}$.
Then in $\Gamma'$, we can prune over-grown substems at $v'^{(0)}, v'^{(1)},\cdots,v'^{(k)}$ with scopes $\Lambda'^{(0)},\Lambda'^{(1)},\cdots,\Lambda'^{(k)}$, respectively, to obtain $\Gamma''$, where it is still possible to prune over-grown substems at $v$. Hence we obtain a grape $\Gamma_0$, which can be obtained from $\Gamma'$ by repeatedly pruning over-grown substems at the vertices corresponding to $v'$. See Figure~\ref{figure:lattice2}.
\end{proof}

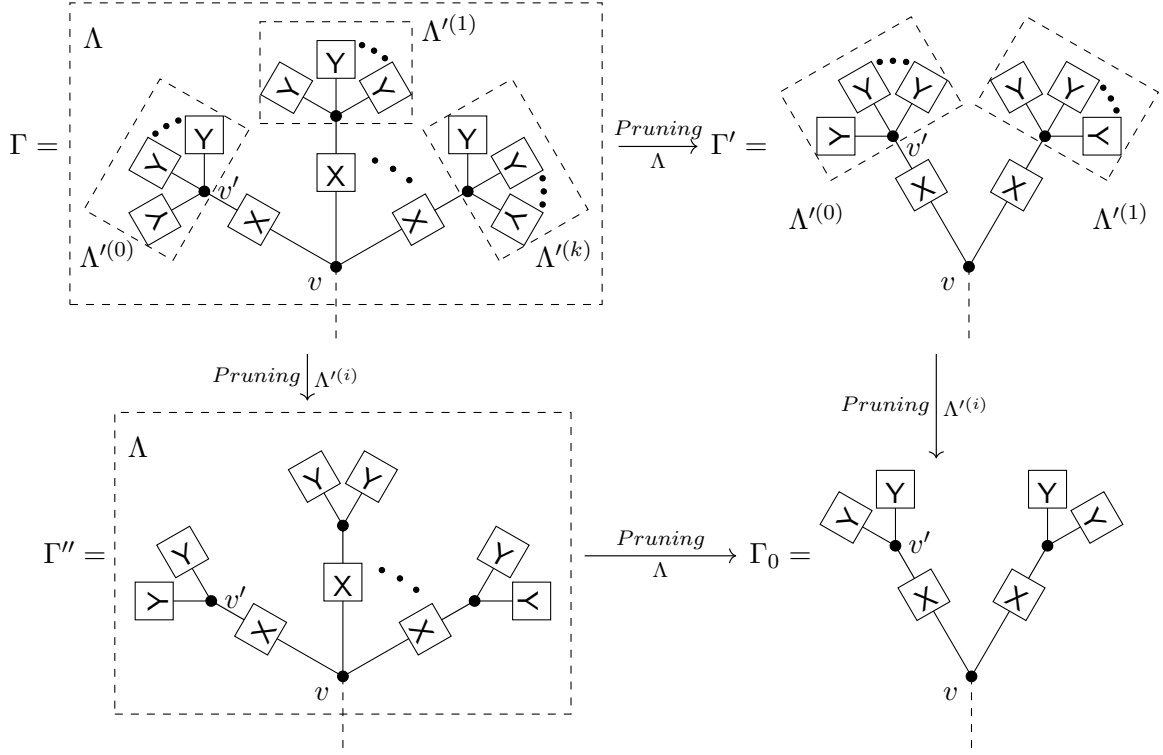
\begin{figure}[ht]
\[
\begin{tikzcd}[ampersand replacement=\&,row sep=1.5pc,column sep=2.5pc]
\Gamma=\begin{tikzpicture}[baseline=1.5cm, rotate=90]
\draw[fill] (0,0) circle (2pt) node[below left] {$v$};
\draw[dashed] (0,0) -- (-1,0);
\foreach \i in {-60,0,60} {
\begin{scope}[rotate=\i]
\draw (0,0) -- (1,0);
\draw (1,-0.25) rectangle node[rotate=\i,above=-2ex] {$\mathsf{X}$} (1.5,0.25);
\draw[fill] (1.5,0) -- (2,0) circle (2pt);
\begin{scope}[xshift=2cm]
\foreach \j in {-60,0,60} {
\begin{scope}[rotate=\j]
\draw (0,0) -- (0.5,0);
\draw (0.5,-0.25) rectangle node[rotate={\i+\j},above=-2ex] {$\mathsf{Y}$} (1,0.25);
\end{scope}
\draw[fill] ({\j/60*10-30}:1) circle (1pt);
}
\end{scope}
\draw[dashed] (1.9,-1) rectangle (3.25,1);
\end{scope}
\draw[fill] ({\i/60*10-30}:1.5) circle (1pt);
}
\draw[dashed] (-.5,-3.5) rectangle (3.5, 3.5);
\draw (3,3.5) node[right] {$\Lambda$};
\draw (0,-3) node {$\Lambda'^{(k)}$};
\draw (3,-1.5) node {$\Lambda'^{(1)}$};
\draw (0,3) node {$\Lambda'^{(0)}$};
\draw (60:2) node[right] {$v'$};
\end{tikzpicture}
\arrow[r,"Pruning","\Lambda"'] 
\arrow[d,"Pruning"',"\Lambda'^{(i)}"]
\&
\Gamma'=\begin{tikzpicture}[baseline=1.5cm, rotate=90]
\draw[fill] (0,0) circle (2pt) node[below left] {$v$};
\draw[dashed] (0,0) -- (-1,0);
\foreach \i in {-30,30} {
\begin{scope}[rotate=\i]
\draw (0,0) -- (1,0);
\draw (1,-0.25) rectangle node[rotate=\i,above=-2ex] {$\mathsf{X}$} (1.5,0.25);
\draw[fill] (1.5,0) -- (2,0) circle (2pt);
\begin{scope}[xshift=2cm]
\foreach \j in {-60,0,60} {
\begin{scope}[rotate=\j]
\draw (0,0) -- (0.5,0);
\draw (0.5,-0.25) rectangle node[rotate={\i+\j},above=-2ex] {$\mathsf{Y}$} (1,0.25);
\end{scope}
\draw[fill] ({\j/60*10-30}:1) circle (1pt);
}
\end{scope}
\draw[dashed] (1.9,-1) rectangle (3.25,1);
\end{scope}
}
\draw (0.5,-2) node {$\Lambda'^{(1)}$};
\draw (0.5,2) node {$\Lambda'^{(0)}$};
\draw (30:2) node[below=1ex, right] {$v'$};
\end{tikzpicture}
\arrow[d,"Pruning"',"\Lambda'^{(i)}"]
\\
\Gamma''=\begin{tikzpicture}[baseline=1.5cm, rotate=90]
\draw[fill] (0,0) circle (2pt) node[below left] {$v$};
\draw[dashed] (0,0) -- (-1,0);
\foreach \i in {-60,0,60} {
\begin{scope}[rotate=\i]
\draw (0,0) -- (1,0);
\draw (1,-0.25) rectangle node[rotate=\i,above=-2ex] {$\mathsf{X}$} (1.5,0.25);
\draw[fill] (1.5,0) -- (2,0) circle (2pt);
\begin{scope}[xshift=2cm]
\foreach \j in {-30,30} {
\begin{scope}[rotate=\j]
\draw (0,0) -- (0.5,0);
\draw (0.5,-0.25) rectangle node[rotate={\i+\j},above=-2ex] {$\mathsf{Y}$} (1,0.25);
\end{scope}
}
\end{scope}
\end{scope}
\draw[fill] ({\i/60*10-30}:1.5) circle (1pt);
}
\draw[dashed] (-.5,-3) rectangle (3.5, 3);
\draw (3,3) node[right] {$\Lambda$};
\draw (60:2) node[right] {$v'$};
\end{tikzpicture}
\arrow[r,"Pruning","\Lambda"']
\&
\Gamma_0=\begin{tikzpicture}[baseline=1.5cm, rotate=90]
\draw[fill] (0,0) circle (2pt) node[below left] {$v$};
\draw[dashed] (0,0) -- (-1,0);
\foreach \i in {-30,30} {
\begin{scope}[rotate=\i]
\draw (0,0) -- (1,0);
\draw (1,-0.25) rectangle node[rotate=\i,above=-2ex] {$\mathsf{X}$} (1.5,0.25);
\draw[fill] (1.5,0) -- (2,0) circle (2pt);
\begin{scope}[xshift=2cm]
\foreach \j in {-30,30} {
\begin{scope}[rotate=\j]
\draw (0,0) -- (0.5,0);
\draw (0.5,-0.25) rectangle node[rotate={\i+\j},above=-2ex] {$\mathsf{Y}$} (1,0.25);
\end{scope}
}
\end{scope}
\end{scope}
}
\draw (30:2) node[right] {$v'$};
\end{tikzpicture}
\end{tikzcd}
\]
\caption{Pruning over-grown substems whose scopes are nested}
\label{figure:lattice2}
\end{figure}

\begin{definition}
For $\Gamma\in\grapegraph^{\mathsf{large}}$, the \emph{quasi-minimal representative} $\Gamma_{\mathsf{min}}$ of $\Gamma$ is a bunch of grapes in $\grapegraph^{\mathsf{large}}_{\mathsf{rich}}$ obtained as follows:
Let $\Gamma_{\mathsf{normal}}$ be the normal representative of $\Gamma$ as in Lemma~\ref{lemma:elimination} and let $\Gamma_{\mathsf{rich}}$ be the rich representative of $\Gamma_{\mathsf{normal}}$ as in Lemma~\ref{lemma:richrepresentative}.
Then $\Gamma_{\mathsf{min}}$ is a bunch of grapes obtained from $\Gamma_{\mathsf{rich}}$ by repeatedly pruning over-grown substems until the process terminates.

We say $\Gamma$ is \emph{quasi-minimal} if $\Gamma=\Gamma_{\mathsf{min}}$, and denote the set of quasi-minimal bunches of grapes by $\grapegraph^{\mathsf{large}}_{\mathsf{min}}$.
\end{definition}

\begin{theorem}\label{theorem:unique minimal representative}
Let $\Gamma\in\grapegraph^{\mathsf{large}}$. Then there is a unique quasi-minimal representative $\Gamma_{\mathsf{min}}$ up to isometry.
\end{theorem}
\begin{proof}
This follows from Lemma~\ref{lemma:elimination}, Corollary~\ref{cor:minimalrich} and Proposition~\ref{proposition:quasi-minimal}.
\end{proof}

In Appendix~\ref{appendix:algorithms}, we provide Algorithm~\ref{alg:quasi-minimal} to obtain the quasi-minimal representative $\Gamma_{\mathsf{min}}$ of $\Gamma\in\grapegraph^{\mathsf{large}}$.

\begin{theorem}\label{thm:QIminimal}
For each $\Gamma\in\grapegraph^{\mathsf{large}}$, $\mathbb{B}_2(\Gamma)$ is quasi-isometric to $\mathbb{B}_2(\Gamma_{\mathsf{min}})$.
\end{theorem}
\begin{proof}
This is a consequence of Theorems~\ref{theorem:2-free factor}, \ref{theorem:loop reducing}, and \ref{theorem:quasi folding} and Algorithm~\ref{alg:quasi-minimal}.
\end{proof}

Therefore we have the following commutative diagram.
\[
\begin{tikzcd}[column sep=14pc, row sep=2pc]
\grapegraph^{\mathsf{large}}\big/\cong
\ar[r,->>,"\mathbb{B}_2(-)"]
\ar[rd,->>, end anchor=north west, "\pi_1(UP_2(-))",sloped]
\ar[rdd,->>, pos=0.35, end anchor=north west, "\cI(\bar{UP_2(-)})", sloped]
\ar[dd,->>, sloped, yshift=1ex, "(-)_{\mathsf{min}}"] 
&
\{\mathbb{B}_2(\Gamma)\mid \Gamma\in \grapegraph^{\mathsf{large}}_{\mathsf{min}}\}\big/\simeq_\text{q.i.}
\ar[d, "\cong"',"\text{Cor.~\ref{cor:QIbetweenGBGs}}"]\\
& \{\pi_1(UP_2(\Gamma))\mid \Gamma\in \grapegraph^{\mathsf{large}}_{\mathsf{min}}\}\big/\simeq_\text{q.i.}
\ar[d,->>,"\text{Thm.~\ref{theorem:IsobetInt}}"] \\
\grapegraph^{\mathsf{large}}_{\mathsf{min}}\big/\cong
\ar[ruu,->>, pos=0.35, sloped, end anchor=south west, "\mathbb{B}_2(-)"'] 
\ar[ru,->>, end anchor=south west, "\pi_1(UP_2(-))"',sloped]
\ar[r,->>, "\cI(\bar{UP_2(-)})"']
\ar[uu, sloped, yshift=1ex, hook]
& \{\cI(\bar{UP_2(\Gamma)})\mid \Gamma\in\grapegraph^{\mathsf{large}}_{\mathsf{min}}\}\big/\cong
\end{tikzcd}
\]

Our goal of this section is to show that there is a one-to-one correspondence between $\grapegraph^{\mathsf{large}}_{\mathsf{min}}$ and $\{\cI(\bar{UP_2(\Gamma)})\mid \Gamma\in\grapegraph^{\mathsf{large}}_{\mathsf{min}}\}$ up to isomorphism, which implies the injectivity of the bottom map $\cI(\bar{UP_2(-)})$ and therefore not only the vertical maps on the right but also all maps from $\grapegraph^{\mathsf{large}}_{\mathsf{min}}\big/\cong$ are bijective in the above diagram.

\begin{theorem}\label{theorem:Quasi-minimalRepresentative}
Let $\Gamma,\Gamma'\in\grapegraph^{\mathsf{large}}_{\mathsf{min}}$ be quasi-minimal.
If there is an isomorphism between $\cI(\bar{UP_2(\Gamma)})$ and $\cI(\bar{UP_2(\Gamma')})$, then $\Gamma$ and $\Gamma'$ are isometric.
\end{theorem}

We defer the proof of Theorem~\ref{theorem:Quasi-minimalRepresentative} to Section~\ref{section:QIBetweenConfSpaces} while we consider its consequences. 

\begin{corollary}\label{corollary:quasi-isometry invariant}
For $\Gamma,\Gamma'\in\grapegraph^{\mathsf{large}}$, $\mathbb{B}_2(\Gamma)$ and $\mathbb{B}_2(\Gamma')$ are quasi-isometric if and only if $\cI(\bar{UP_2(\Gamma)})$ and $\cI(\bar{UP_2(\Gamma')})$ are isomorphic if and only if the quasi-minimal representatives $\Gamma_{\mathsf{min}}$ and $\Gamma'_{\mathsf{min}}$ are isometric.
\end{corollary}

By Example~\ref{Ex:FreeGroupCase} and Lemma~\ref{NoHangingEdge}, if $\Gamma=(\mathsf{T},\ell)\in\grapegraph$ is small, then it is the free group of rank $\sum_{v\in \mathcal{V}(\mathsf{T})} N(\val_{\mathsf{T}}(v),\loops(v))$, where $N(n,\ell)= \frac{(n+\ell) ( n+3\ell-3 )}{2} + 1$.
Combining this fact with Lemma~\ref{lemma:large} and Corollary~\ref{corollary:quasi-isometry invariant}, we have the following theorem and provide Algorithm~\ref{alg:quasi-isom} in Appendix~\ref{appendix:algorithms}.

\begin{theorem}\label{thm:existenceofAlgorithm}
Let $\Gamma$ and $\Gamma'$ be bunches of grapes.
Then there exists an algorithm to determine whether $\mathbb{B}_2(\Gamma)$ and $\mathbb{B}_2(\Gamma')$ are quasi-isometric by looking at $\Gamma$ and $\Gamma'$.
\end{theorem}

\subsection{Proof of Theorem~\ref{theorem:Quasi-minimalRepresentative}}\label{section:QIBetweenConfSpaces}
Throughout this subsection, we assume that $\Gamma=(\mathsf{T},\ell)$ and $\Gamma'=(\mathsf{T}',\ell')$ are in $\grapegraph^{\mathsf{large}}_{\mathsf{min}}$ but both have \emph{two grapes} at each vertex which is not a leaf of the stems, and there is an isomorphism $\bar\Phi:\cI(\bar{UP_2(\Gamma)})\to\cI(\bar{UP_2(\Gamma')})$.
Note that adding grapes at a vertex of valency at least $2$ does not change the isomorphism type of $\cI(\bar{UP_2(-)})$ by Theorem~\ref{theorem:loop reducing} and incur additional process of pruning over-grown substems.

For simplicity, we denote  $\cRI({UP_2(\Gamma)})$, $\cI(\bar{UP_2(\Gamma)})$, $\cRI({UP_2(\Gamma')})$ and $\cI(\bar{UP_2(\Gamma')})$ by $\cRI$, $\cI$, $\cRI'$ and $\cI'$, respectively. 
Since $\cI$ and $\cI'$ have the same dimension, which is exactly the same as $\operatorname{diam}(\mathsf{T})-1$, both $\mathsf{T}$ and $\mathsf{T}'$ have the same diameter $D=\operatorname{diam}(\mathsf{T})$ which is at least $1$ since $\Gamma$ and $\Gamma'$ are large.

Note that any maximal simplex of $\cI$ is induced from a path substem whose endpoints are leaves of $\mathsf{T}$ (see Remark~\ref{Ex:PathofLengthN}).
Combining this fact with the fact that $\bar\Phi$ sends maximal simplices to maximal simplices, we can easily obtain the following proposition. 

\begin{proposition}\label{prop:SimpleCase}
If there is no maximal simplex in $\cI$ of dimension $<(D-1)$, then Theorem~\ref{theorem:Quasi-minimalRepresentative} holds.
\end{proposition}
\begin{proof}
If $D$ is odd, then both $\mathsf{T}$ and $\mathsf{T}'$ are paths of length $D$, and thus, by Remark~\ref{Ex:PathofLengthN} and the quasi-minimality assumption, $\Gamma$ and $\Gamma'$ must be isometric.

If $D$ is even, then each of $\mathsf{T}$ and $\mathsf{T}'$ must be the union of $m$ rays of length $D/2$ issuing from the same vertex for some $2\le m\le 4$.
Then we have the following observation which is preserved by the isomorphism $\bar\Phi$.
\begin{itemize}
\item $\cI$ has no vertex whose label has the fundamental group having $\mathbb{Z}$ as a direct factor if and only if there is no leaf of $\mathsf{T}$ which has one grape.
\item There is a maximal simplex in $\cI$ whose label has the fundamental group $\mathbb{Z}\times\mathbb{Z}$ if and only if there are at least two leaves of $\mathsf{T}$ having one grape.
\item Every vertex in $\cI$ has a label whose fundamental group has $\mathbb{Z}$ as a direct factor if and only if there is no leaf of $\mathsf{T}$ which has two grapes.
\item There is a maximal simplex in $\cI$ whose label has the fundamental group not having $\mathbb{Z}$ as a direct factor if and only if there are at least two leaves of $\mathsf{T}$ having two grapes.
\end{itemize}
Combining this observation with the quasi-minimality assumption, it can be easily seen that $\Gamma$ and $\Gamma'$ are isometric.
\end{proof}
 
The fact we used in the proof is that $\bar \Phi$ sends maximal simplices to maximal simplices preserving the quasi-isometric type of the fundamental groups of their labels and each maximal simplex is induced a path substem of length $D$. In order to resolve the general case, we delve into what $\bar \Phi$ preserves, relate it to a sub-bunch of grapes, and then conclude that $\Gamma$ and $\Gamma'$ are isometric by the quasi-minimality assumption.

Before we move on, let us see how the embedding of a disjoint union of sub-bunches of grapes into a given bunch of grapes characterize the induced maps on $\cRI(UP_2(-))$ and $\cI(\bar{UP_2(-)})$.
Let $\mathsf{S}$ be the disjoint union of substems $\mathsf{T}_1,\cdots,\mathsf{T}_n$ of $\mathsf{T}$, where each $\mathsf{T}_i$ has diameter at least $1$.
We define $\Gamma_{\mathsf{S}}$ as the disjoint union of $\Gamma_{\mathsf{T}_i}$'s for $i=1,\dots,n$.
Then we define $UP_2(\Gamma_{\mathsf{S}})$ as the disjoint union of $UP_2(\Gamma_{\mathsf{T}_i})$'s, and $\cRI(UP_2(\Gamma_{\mathsf{S}}))$ as the disjoint union of $\cRI(UP_2(\Gamma_{\mathsf{T}_i}))$'s.
Then the canonical embedding $\underline{\iota}:\Gamma_{\mathsf{S}}\to\Gamma$ induces a map on $\cRI(UP_2(-))$ as follows.

\begin{lemma}\label{Lem:EmbeddedRI}
The embedding $\underline{\iota}$ given above induces an injective morphism \[\iota:\cRI(UP_2(\Gamma_{\mathsf{S}}))\to\cRI\] such that for each component $\mathcal{A}$ of $\rho^{-1}(\operatorname{Image}(\iota))\subset \cI$, there is an isomorphism $\bar\iota_{\mathcal{A}}$ from $\cI(\bar{UP_2(\Gamma_{\mathsf{T}_i})})$ to $\mathcal{A}$ for some $1\le i\le n$ such that $\rho\circ\bar\iota_{\mathcal{A}}=\iota\circ\rho$ on $\cI(\bar{UP_2(\Gamma_{\mathsf{T}_i})})$.
\end{lemma}
Such a component $\mathcal{A}$ will be said to be a \emph{block} (with respect to $\mathsf{S}$) \emph{induced from} $\mathsf{T}_i$.
\begin{proof}
Since $\underline{\iota}$ maps not only twigs to twigs but also a finite collection of colinear twigs in $\mathsf{T}_i$ to such a collection in $\mathsf{T}$, by mimicking the proof of Lemma~\ref{Lem:IsometricRI}, we can obtain an injective morphism $\iota:\cRI(UP_2(\Gamma_{\mathsf{S}}))\to\cRI$.
We note that it can easily be checked that for each $i$, $U_i=\mathfrak{g}_{\cRI}(\iota(\cRI(\Gamma_{\mathsf{T}_i})))$ is a special square complex and the embedding $f_i:U_i\hookrightarrow UP_2(\Gamma)$ is a local isometry. By Lemma~\ref{Lem:IsometricRI}, moreover, $\iota|_{\cRI(UP_2(\Gamma_{\mathsf{T}_i}))}$ can be seen as an isomorphism from $\cRI(UP(\Gamma_{\mathsf{T}_i}))$ to $\cRI(U_i)$ by identifying $\cRI(U_i)$ with $\iota(\cRI(UP_2(\Gamma_{\mathsf{T}_i})))$.

Let $\mathcal{A}\subset\cI$ be a component of $\rho^{-1}(\operatorname{Image}(\iota))$.
Indeed, $\mathfrak{g}_{\cI}(\mathcal{A})\subset\bar{UP_2(\Gamma)}$ is equal to the image of an elevation of $f_i$ for some $i$, and thus $\cI(\bar{U_i})$ can be identified with $\mathcal{A}$.
By Lemma~\ref{Lem:IsometricRI}, therefore, there is an isomorphism $\bar\iota_{\mathcal{A}}:\cI(\bar{UP_2(\Gamma_{\mathsf{T}_i})})\to\cI(\bar{U_i})$ such that $\rho\circ\bar\iota_{\mathcal{A}}=\iota\circ\rho$ on $\cI(\bar{UP_2(\Gamma(\mathsf{T}_i))})$, which completes the proof.
\end{proof}

Recall in Lemma~\ref{lemma:deformation retract} that the isomorphism $\bar\Phi:\cI\to\cI'$ induces an isomorphism $\bar\Phi_{\le 2}:\cI_{\le 2}\to\cI_{\le 2}'$, where $\cI_{\le 2}$ and $\cI'_{\le 2}$ are subcomplexes of $\cI$ and $\cI'$, respectively, such that their any simplex is induced from a path substem of length at most $2$.
In summary, we have the following diagram: 
\[
\begin{tikzcd}[column sep=4em, row sep=0.3em]
\cI \arrow[rr,"\bar\Phi"]\arrow[rd,"\bar r_{\le 2}" ']\arrow[dd,"\rho "] & & \cI'\arrow[dd, "\rho' " near end]\arrow[rd,"\bar r'_{\le 2} "]& \\
& \cI_{\le 2}\arrow[dd,"\rho_{\le 2}"]\arrow[rr,crossing over,"\bar \Phi_{\le 2} " near start] && \cI'_{\le 2}\arrow[dd,"\rho'_{\le 2}"]\\
\cRI\arrow[rd,"r_{\le 2} "] & & \cRI'\arrow[rd,"r'_{\le 2} "]\\
& \cRI_{\le 2} && \cRI'_{\le 2}
\end{tikzcd}
\]
where $r_{\le 2}$, $r'_{\le 2}$, $\bar r_{\le 2}$ and $\bar r'_{\le 2}$ are the deformation retractions given in Lemma~\ref{lemma:deformation retract}, and $\rho_{\le 2}$ and $\rho'_{\le 2}$ are the restrictions.

Since $\underline{\iota}$ maps path substems to path substems, the morphism $\iota$ in Lemma~\ref{Lem:EmbeddedRI} commutes with the deformation retraction $\cRI(-)\to\cRI_{\le 2}(-)$ defined in Lemma~\ref{lemma:deformation retract} by construction.
Moreover, there is a one-to-one correspondence, induced by the deformation retraction $r_{\le 2}:\cRI\to\cRI_{\le 2}$, between components of $\operatorname{Image}(\iota)$ and components of $\operatorname{Image}(\iota_{\le 2})$, since $\operatorname{Image}(\iota)$ is a union of simplices of $\cRI$ induced from path substems contained in $\mathsf{S}$ and the images of such simplices under $r_{\le 2}$ are connected path graphs in $\cRI_{\le 2}$.
Similarly, there is a one-to-one correspondence between components of $\rho^{-1}(\operatorname{Image}(\iota))$ in $\cI$ and components of $\rho_{\le 2}^{-1}(\operatorname{Image}(\iota_{\le 2}))$ in $\cI_{\le 2}$. 

\begin{corollary}\label{corollary:components and retraction}
Let $\iota$ be the injective semi-morphism given in Lemma~\ref{Lem:EmbeddedRI}.
Then there is a one-to-one correspondence between components of $\rho^{-1}(\operatorname{Image}(\iota))$ and components of $\rho_{\le 2}^{-1}(\operatorname{Image}(\iota_{\le2}))$, given as follows:
\[
\begin{tikzcd}[row sep=0.5pc]
\pi_0\left(
\rho^{-1}(\operatorname{Image}(\iota))
\right) \arrow[r, leftrightarrow] &
\pi_0\left(
\rho_{\le 2}^{-1}(\operatorname{Image}(\iota_{\le2}))
\right)\\
\mathcal{A} \arrow[r, mapsto, yshift=0.2em, "\text{Deformation retract}"] & \mathcal{A}_{\le2}
\arrow[l, mapsto, yshift=-.2em, "\text{Induced subcomplex}"]
\end{tikzcd}
\]
\end{corollary}

\begin{remark}\label{Remark:HowtoapplyCorollary}
The way to use Corollary~\ref{corollary:components and retraction} will be the following.
Suppose that there are sets $\mathcal{V}$ and $\mathcal{V}'$ of vertices in $\cRI$ and $\cRI'$, respectively, such that $\bar\Phi$ maps $\rho^{-1}(\mathcal{V})$ to $\rho'^{-1}(\mathcal{V})$ bijectively.
Consider the union $\mathsf{S}$ of twigs of $\Gamma$ which are not contained in the twigs corresponding to vertices in $\mathcal{V}$ and the injective semi-morphisms $\iota:\cRI(\Gamma_{\mathsf{S}})\to\cRI$ obtained by Lemma~\ref{Lem:EmbeddedRI}.
Then we have $\rho_{\le 2}^{-1}(\operatorname{Image}(\iota_{\le 2}))=\cI_{\le 2}\setminus\rho^{-1}(\mathcal{V})$.
From $\mathcal{V}'$, similarly, $\mathsf{S}'$ and $\iota':\cRI'(\Gamma'_{\mathsf{S}'})\to\cRI'$ are defined and thus $\rho'^{-1}_{\le 2}(\operatorname{Image}(\iota'_{\le 2}))=\cI'_{\le 2}\setminus\rho'^{-1}(\mathcal{V}')$.
As $\bar\Phi$ induces an isomorphism $\bar\Phi_{\le 2}:\cI_{\le 2}\to\cI'_{\le 2}$ and $\bar\Phi(\rho^{-1}(\mathcal{V}))=\rho^{-1}(\mathcal{V}')$, Corollary~\ref{corollary:components and retraction} implies that the restriction of $\bar\Phi$ to any block (with respect to $\mathsf{S}$) in $\cI$ is an isomorphism onto a block (with respect to $\mathsf{S}'$) in $\cI'$.
\end{remark}

One main characteristic of $\bar\Phi$ that can be derived from $\cI$ and $\cI'$, apart from the fact that $\bar\Phi$ is an isomorphism, is that $\bar\Phi$ preserves a canonical order (up to reversing) on the vertex set of each simplex, which is defined in Definition~\ref{def:OrderOnSimplex}.
Using this fact, we will find classes of vertices in $\cI$ and $\cI'$ such that $\bar\Phi$ preserves these classes, and thus we can apply Remark~\ref{Remark:HowtoapplyCorollary}.

Now, we decompose $\mathsf{T}$ as a union of a kind of spheres as follows. 
Let $\mathcal{S}_{\mathsf{T}}(1)=\mathcal{B}_{\mathsf{T}}(1)$ be the intersection of the union of path substems of length $D$ and the star of the central vertex in $\mathsf{T}$ if $D$ is even, or the twig containing the center of $\mathsf{T}$ if $D$ is odd.
For $k\ge 2$, let $\mathcal{S}_{\mathsf{T}}(k)$ be the set of twigs of $\mathsf{T}$ which intersect $\mathcal{B}_{\mathsf{T}}(k-1)$ at its leaves but are not contained in $\mathcal{B}_{\mathsf{T}}(k-1)$, and let $\mathcal{B}_{\mathsf{T}}(k)=\mathcal{B}_{\mathsf{T}}(k-1)\cup \mathcal{S}_{\mathsf{T}}(k)$ and $\mathcal{B}_{\mathsf{T}}=\mathcal{B}_{\mathsf{T}}([(D+1)/2])$.
When $D$ is even, $\mathsf{T}(0)$ is defined as the union of extended $\hat{v}$-components of $\mathsf{T}$ not containing rays of length $D/2$ issuing from $v$, i.e., the closure of the complement of $\mathcal{B}_{\mathsf{T}}$ in $\mathsf{T}$, viewed as topological spaces. See Figure~\ref{figure:Sphere decomposition}.

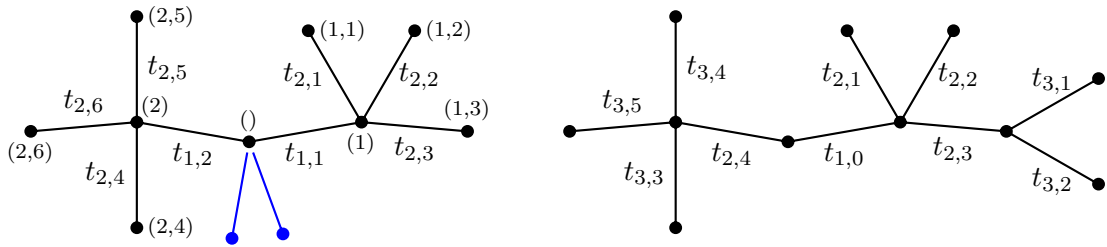
\begin{figure}[ht]
\[
\begin{tikzpicture}[baseline=-.5ex, scale=1]
\draw[thick, fill] 
  (0,0) circle (2pt) node (A) {} node[below] {$_{(2,6)}$}
  -- ++(5:1.4) circle (2pt) node (B) {} node[midway, above] {$t_{2,6}$}
  node[xshift=0.25cm, yshift=0.2cm] {$_{(2)}$} % Adjusted label position
  -- +(0,1.4) circle (2pt) node (C) {} node[midway, right] {$t_{2,5}$}
  node[right] {$_{(2,5)}$}
  +(0,0) -- +(0,-1.4) circle (2pt) node (D) {} node[midway, left] {$t_{2,4}$}
  node[right] {$_{(2,4)}$}
  +(0,0) -- ++(-10:1.5) circle (2pt) node (E) {} node[midway, below] {$t_{1,2}$} node[above] {$_{(\empty)}$} 
  -- ++(10:1.5) circle (2pt) node (F) {} node[midway, below] {$t_{1,1}$}
  node[below] {$_{(1)}$}
  -- +(120:1.4) circle (2pt) node (G) {} node[midway, left] {$t_{2,1}$}
  node[right] {$_{(1,1)}$}
  +(0,0) -- +(60:1.4) circle (2pt) node (H) {} node[midway, right] {$t_{2,2}$}
  node[right] {$_{(1,2)}$}
  +(0,0) -- ++(-5:1.4) circle (2pt) node (I) {} node[midway, below] {$t_{2,3}$}
  node[above] {$_{(1,3)}$};

\draw[thick, fill, blue]
  (E) -- ++(260:1.3) circle (2pt) [fill=blue] node {};
\draw[thick, fill, blue]
  (E) -- ++(290:1.3) circle (2pt) [fill=blue] node {};
\end{tikzpicture}
\qquad
\begin{tikzpicture}[baseline=-.5ex, scale=1]
\draw[thick, fill] 
  (0,0) circle (2pt) node (A) {} 
  -- ++(5:1.4) circle (2pt) node (B) {} node[midway, above] {$t_{3,5}$}
  -- +(0,1.4) circle (2pt) node (C) {} node[midway, right] {$t_{3,4}$}
  +(0,0) -- +(0,-1.4) circle (2pt) node (D) {} node[midway, left] {$t_{3,3}$}
  +(0,0) -- ++(-10:1.5) circle (2pt) node (E) {} node[midway, below] {$t_{2,4}$}
  -- ++(10:1.5) circle (2pt) node (F) {} node[midway, below] {$t_{1,0}$}
  -- +(120:1.4) circle (2pt) node (G) {} node[midway, left] {$t_{2,1}$}
  +(0,0) -- +(60:1.4) circle (2pt) node (H) {} node[midway, right] {$t_{2,2}$}
  +(0,0) -- ++(-5:1.4) circle (2pt) node (I) {} node[midway, below] {$t_{2,3}$}
  -- +(30:1.4) circle (2pt) node (J) {} node[midway, above] {$t_{3,1}$}
  +(0,0) -- +(-30:1.4) circle (2pt) node (K) {} node[midway, below] {$t_{3,2}$};
\end{tikzpicture}
\]
\caption{Examples of sphere decompositions; the left stem is when $D$ is odd and the right stem is when $D$ is even. For each stem, $\mathcal{S}_{\mathsf{T}}(i)$ consists of twigs $t_{i,j}$, and the unlabelled blue edges in the left stem belong to $\mathsf{T}(0)$.}
\label{figure:Sphere decomposition}
\end{figure}

Let $\bar{\mathcal{V}}_{\Gamma}(k)$ be the set of vertices in $\cI(\bar{UP_2(\Gamma)})$ induced from twigs in $\mathcal{S}_{\mathsf{T}}(k)$.
Then the following lemma says $\bar{\mathcal{V}}_{\Gamma}(k)$ is the desired class.

\begin{lemma}\label{lem:PreservingVertices}
For each $k\ge 1$, $\bar\Phi$ induces a bijection between $\bar{\mathcal{V}}_{\Gamma}(k)$ and $\bar{\mathcal{V}}_{\Gamma'}(k)$.
\end{lemma}
\begin{proof}
Let $\bar\triangle\subset\cI$ be a $(D-1)$-simplex.
If $D$ is odd, then there is a vertex in $\bar\triangle$ induced from a unique twig in $\mathcal{S}_{\mathsf{T}}(1)$ which is exactly the middle vertex with respect to a canonical order on $\mathcal{V}(\bar\triangle)$. 
If $D$ is even, then there are two vertices in $\bar\triangle$ induced from two twigs in $\mathcal{S}_{\mathsf{T}}(1)$ and they are exactly the middle vertices with respect to a canonical order on $\mathcal{V}(\bar\triangle)$. 
Conversely, if a vertex $\barbfv\in\cI$ is in $\bar{\mathcal{V}}_{\Gamma}(1)$, then it must be a unique middle vertex (if $D$ is odd) or one of two middle vertices (if $D$ is even) in any $(D-1)$-simplex containing $\barbfv$.
By Lemma~\ref{lem:OrderOnSimplex}, thus $\bar\Phi$ induces a bijection between $\bar{\mathcal{V}}_{\Gamma}(1)$ and $\bar{\mathcal{V}}_{\Gamma'}(1)$.

Inductively, we have that a vertex $\barbfv\in\cI$ is in $\bar{\mathcal{V}}_{\Gamma}(k)$ if and only if it is contained in a maximal simplex $\bar\triangle$ which contains one vertex (if $D$ is odd) or two vertices (if $D$ is even) in $\bar{\mathcal{V}}_{\Gamma}(1)$ such that it is a predecessor or a successor of $\mathcal{V}(\bar\triangle)\cap\bigcup_{i=1}^{k-1}\bar{\mathcal{V}}_{\Gamma}(i)$ in $\mathcal{V}(\bar\triangle)$.
By Lemma~\ref{lem:OrderOnSimplex} again, $\bar\Phi$ induces a bijection between $\bar{\mathcal{V}}_{\Gamma}(k)$ and $\bar{\mathcal{V}}_{\Gamma'}(k)$ for each $k\ge 2$, which completes the proof.
\end{proof}

\begin{convention}
We denote by a $0$-tuple the center vertex of $\mathcal{B}_{\mathsf{T}}$ if $D$ is even, or an empty set if $D$ is odd. For $k\ge 1$, leaves of $\mathcal{B}_{\mathsf{T}}(k)$ are denoted by distinct $k$-tuples such that if a leaf $\mathbbm{v}_1$ of $\mathcal{B}_{\mathsf{T}}(k)$ is adjacent to a $(k-1)$-tuple $\mathbbm{v}$, then $\mathbbm{v}$ is the prefix of $\mathbbm{v}_1$. For instance, if $\mathbbm{v}_1=(a_1,\dots,a_{k-1},a_k)$, then $\mathbbm{v}=(a_1,\dots,a_{k-1})$; see the labels of vertices of the left stem in Figure~\ref{figure:Sphere decomposition}.

For a $k$-tuple $\mathbbm{v}\in\mathcal{B}_{\mathsf{T}}$, let ${\mathsf{T}}_{\mathbbm{v}}$ be the union of extended $\hat{\mathbbm{v}}$-components of $\mathsf{T}$ except one containing $\mathcal{B}_{\mathsf{T}}(k)$ if $k\ge 1$ (see Figure~\ref{figure:(non)ray branch}), or the whole stem $\mathsf{T}$ if $k=0$. 
\end{convention}

\begin{lemma}\label{Lem:ComponentPreserving}
Suppose that $\mathbbm{v}\in\mathcal{B}_{\mathsf{T}}$ is a $k$-tuple for $k\ge 0$.
Then there exists a $k$-tuple $\mathbbm{w}\in\mathcal{B}_{\mathsf{T}'}$ such that 
$\bar\Phi$ induces an isomorphism \[\bar\Phi_{\mathbbm{v}}:\cI(\bar{UP_2(\Gamma_{\mathsf{T}_{\mathbbm{v}}})})\to\cI(\bar{UP_2(\Gamma'_{\mathsf{T}'_{\mathbbm{w}}})})\] which sends vertices induced from twigs in $\mathsf{T}_{\mathbbm{v}}\cap\mathcal{S}_\mathsf{T}(k+1)$ to vertices induced from twigs in $\mathsf{T}'_{\mathbbm{w}}\cap\mathcal{S}_{\mathsf{T}'}(k+1)$.

Moreover, if $\mathbbm{v}_1$ is a $(k+1)$-tuple which has $\mathbbm{v}$ as a prefix, then $\bar\Phi_{\mathbbm{v}}$ induces an isomorphism $\cI(\bar{UP_2(\Gamma_{\mathsf{T}_{\mathbbm{v}_1}})})\to\cI(\bar{UP_2(\Gamma'_{\mathsf{T}'_{\mathbbm{w}_1}})})$ for a $(k+1)$-tuple $\mathbbm{w}_1$ having $\mathbbm{w}$ as a prefix. 

If $\mathsf{T}(0)$ exists, so does $\mathsf{T}'(0)$ and there is an isomorphism $\cI(\bar{UP_2(\Gamma_{\mathsf{T}(0)})})\to\cI(\bar{UP_2(\Gamma'_{\mathsf{T}'(0)})})$ induced by $\bar\Phi$.
\end{lemma}
\begin{proof}
If $k=0$, then $\mathsf{T}_{\mathbbm{v}}=\mathsf{T}$ and $\mathsf{T}'_{\mathbbm{w}}=\mathsf{T}'$, and thus, the first statement holds since $\bar\Phi$ induces a bijection between $\bar{\mathcal{V}}_{\mathsf{T}}(1)$ and $\bar{\mathcal{V}}_{\mathsf{T}'}(1)$ by Lemma~\ref{lem:PreservingVertices}.
By Remark~\ref{Remark:HowtoapplyCorollary}, moreover, the restriction of $\bar\Phi$ to a block (with respect to $\mathsf{S}$) in $\cI$ is an isomorphism onto a block (with respect to $\mathsf{S}'$) in $\cI'$ where $\mathsf{S}$ is the union of twigs of $\Gamma$ not in $\mathcal{B}_{\mathsf{T}}(1)$ and $\mathsf{S}'$ is the union of twigs of $\Gamma'$ not in $\mathcal{B}_{\mathsf{T}'}(1)$.

We define a specific simplex of $\cI$ whose feature is preserved by $\bar\Phi$ as follows.
For any maximal $(D-1)$-simplex $\bar\triangle$ in $\cI$, its sub-simplex spanned by vertices all less than or all larger than vertices in $\bar\triangle\cap\bar{\mathcal{V}}(1)$ will be called a \emph{partial simplex}.
Then it is easily seen that partial simplices are mapped by $\bar\Phi$ to partial simplices. 

If a block $\mathcal{A}$ in $\cI$ has no partial simplices, by Lemma~\ref{Lem:EmbeddedRI} and the definition of partial simplices, it must be induced from $\mathsf{T}(0)$. 
Since $\bar\Phi(\mathcal{A})$ also has no partial simplices, it follows that $\mathsf{T}'(0)$ exists and $\bar\Phi(\mathcal{A})$ is a block induced from $\mathsf{T}'(0)$. This completes the proof of the third statement of the lemma.

If a block $\mathcal{A}$ in $\cI$ has a partial simplex, then it must be induced from $\mathsf{T}_{\mathbbm{v}}$ for some $1$-tuple $\mathbbm{v}\in\mathcal{B}_{\mathsf{T}}$.
It follows that $\bar\Phi(\mathcal{A})$ is a block induced from $\mathsf{T}'_{\mathbbm{w}}$ for some $1$-tuple $\mathbbm{w}\in\mathcal{B}_{\mathsf{T}'}$, and thus the second statement of the lemma holds for $k=0$.

Let $\mathbbm{v}$ be a $1$-tuple in $\mathcal{B}_{\mathsf{T}}$. By the previous paragraph, there exists a $1$-tuple $\mathbbm{w}$ in $\mathcal{B}_{\mathsf{T}'}$ such that $\bar\Phi$ induces an isomorphism between $\cI(\bar{UP_2(\Gamma_{\mathsf{T}_{\mathbbm{v}}})})$ and $\cI(\bar{UP_2(\Gamma'_{\mathsf{T}'_{\mathbbm{w}}})})$.
Since $\bar\Phi$ induces a bijection between $\bar{\mathcal{V}}_{\Gamma}(2)$ and $\bar{\mathcal{V}}_{\Gamma'}(2)$ by Lemma~\ref{lem:PreservingVertices}, we can do the same process above in order to show that the second statement of the lemma holds for $k=1$ by defining the union of twigs in $\mathsf{T}_{\mathbbm{v}}$ not contained in $\mathcal{S}_{\mathsf{T}}(2)$ and the union of twigs in $\mathsf{T}'_{\mathbbm{w}}$ not contained in $\mathcal{S}_{\mathsf{T}'}(2)$.
The only difference is that we do not have to define partial simplices for this case.
By induction, we can show that the second statement holds for any $k$ and this completes the proof.
\end{proof}

\begin{remark}
One may think that we can finish the proof of Theorem~\ref{theorem:Quasi-minimalRepresentative} by inductively applying Lemma~\ref{Lem:ComponentPreserving}. More precisely, we wish that for a $1$-tuple $\mathbbm{v}\in\mathcal{B}_{\mathsf{T}}$, there was a $1$-tuple $\mathbbm{w}\in\mathcal{B}_{\mathsf{T}'}$ such that an induction argument would imply that there was an isometry from $\Gamma_{\mathsf{T}_{\mathbbm{v}}}$ to $\Gamma'_{\mathsf{T}'_{\mathbbm{w}}}$ sending $\mathbbm{v}$ to $\mathbbm{w}$.
However, this is not true since we only know that the quasi-minimal representatives of $\Gamma_{\mathsf{T}_{\mathbbm{v}}}$ and $\Gamma'_{\mathsf{T}'_{\mathbbm{w}}}$ are isometric and $\Gamma_{\mathsf{T}_{\mathbbm{v}}}$ does not have to be quasi-minimal.
\end{remark}

For a $k$-tuple $\mathbbm{v}\in\mathcal{B}_{\mathsf{T}}$ with $k\ge 1$, an extended $\hat{\mathbbm{v}}$-component of $\mathsf{T}_{\mathbbm{v}}$ will be called a \emph{ray substem} if it is a ray issuing from $\mathbbm{v}$, and a \emph{non-ray substem} otherwise.

\begin{figure}[ht]
\begin{tikzpicture}[baseline=-.5ex, scale=1]
\draw[thick, fill, dashed] 
  (0,0) circle (2pt) node (A) {}
  -- ++(5:1.4) circle (2pt) node (B) {} node[above right] {$\mathbbm{v}_0$} 
  -- +(0,1.4) circle (2pt) node (C) {} 
  +(0,0) -- +(0,-1.4) circle (2pt) node (D) {};

\draw[thick, fill] 
  (B) -- ++(-10:1.5) circle (2pt) node (E) {} node[above] {$\mathbbm{v}_1$};

% Red edge between E and F
\draw[thick, fill, draw=red]
  (E) -- ++(20:1.5) circle (2pt) [fill=red, draw=red] node (F) {} node[below] {$\mathbbm{v}_2$};

\draw[thick, fill, draw=red]
  (F) -- +(120:1.4) circle (2pt) [fill=red] node (G) {} 
  +(0,0) -- +(60:1.4) circle (2pt) [fill=red] node (H) {} 
  +(0,0) -- ++(-5:1.4) circle (2pt) [fill=red] node (I) {};

\draw[thick, fill, draw=red]
  (I) -- +(0:1.4) circle (2pt) [fill=red] node (L) {};
  
% New edges attached to node E
\draw[thick, fill, draw=blue]
  (E) -- ++(-30:1.5) circle (2pt) [fill=blue]  node (J) {};
\draw[thick, fill, draw=blue]
  (E) -- ++(-70:1.5) circle (2pt) [fill=blue]  node (K) {};

% Horizontal edge from node J
\draw[thick, fill, draw=blue]
  (J) -- ++(0:1.5) circle (2pt) [fill=blue]  node (L) {};
\end{tikzpicture}
\caption{An example of a non-ray substem of $\mathsf{T}_{\mathbbm{v}_0}$ where $\mathbbm{v}_{0}$ is a prefix of $\mathbbm{v}_{1}$. The blue subgraph is the union of ray substems of $\mathsf{T}_{\mathbbm{v}_1}$ and the red subgraph is a non-ray substem of $\mathsf{T}_{\mathbbm{v}_1}$.}
\label{figure:(non)ray branch}
\end{figure}
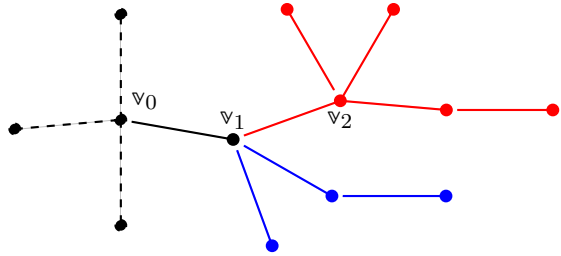

\begin{lemma}\label{Lem:UnionofRays}
For a $k$-tuple $\mathbbm{v}\in\mathsf{T}$ with $k\ge 1$, suppose that $\mathbbm{w}\in\mathsf{T}'$ is a $k$-tuple obtained in Lemma~\ref{Lem:ComponentPreserving} such that there is an isomorphism $\bar\Phi_{\mathbbm{v}}:\cI(\bar{UP_2(\Gamma_{\mathsf{T}_{\mathbbm{v}}})})\to\cI(\bar{UP_2(\Gamma'_{\mathsf{T}'_{\mathbbm{w}}})})$.
Let $\mathsf{U}\subset\mathsf{T}_{\mathbbm{v}}$ be the union of ray substems of $\mathsf{T}_{\mathbbm{v}}$.
Then for the union $\mathsf{U}'$ of ray substems of $\mathsf{T}'_{\mathbbm{w}}$, there exists an isometry from $\Gamma_{\mathsf{U}}$ to $\Gamma'_{\mathsf{U}'}$ sending $\mathbbm{v}$ to $\mathbbm{w}$.
\end{lemma}
\begin{proof}
For any maximal simplex in $\cI(\bar{UP_2(\Gamma_{\mathsf{T}_{\mathbbm{v}}})})$ ($\cI(\bar{UP_2(\Gamma'_{\mathsf{T}'_{\mathbbm{w}}})})$, resp.) containing vertices induced from ${\mathcal{S}}_{\mathsf{T}}(k+1)$ (${\mathcal{S}}_{\mathsf{T}'}(k+1)$, resp.), its sub-simplex spanned by vertices all less than or all larger than the vertices induced from ${\mathcal{S}}_{\mathsf{T}}(k+1)$ (${\mathcal{S}}_{\mathsf{T}'}(k+1)$, resp.) will be called a \emph{partial simplex}. (This definition is similar to `partial simplices' defined in the proof of Lemma~\ref{Lem:ComponentPreserving} for $\cI$ and $\cI'$, but a little bit different since there is no restriction on the dimension of maximal simplices in this case.)
It is easily seen that $\bar\Phi_{\mathbbm{v}}$ maps partial simplices to partial simplices by Lemmas~\ref{lem:OrderOnSimplex} and~\ref{Lem:ComponentPreserving}.

Let $\mathsf{S}$ be the union of $\mathsf{T}_{\mathbbm{v}_i}$'s for $(k+1)$-tuples $\mathbbm{v}_i$ in $\mathsf{T}_{\mathbbm{v}}$ and $\mathsf{S}'$ the union of $\mathsf{T}'_{\mathsf{w}_j}$'s for $(k+1)$-tuples $\mathbbm{w}_j$ in $\mathsf{T}'_{\mathbbm{w}}$.
By Remark~\ref{Remark:HowtoapplyCorollary} and Lemma~\ref{Lem:ComponentPreserving}, it follows that $\bar\Phi_{\mathbbm{v}}$ maps blocks (with respect to $\mathsf{S}$) to blocks (with respect to $\mathsf{S}'$).

In order to finish the proof, we need the following observation, which is about the relation between ray substems of $\mathsf{T}_{\mathbbm{v}}$ and partial simplices or blocks in $\cI(\bar{UP_2(\Gamma(\mathsf{T}_{\mathbbm{v}}))})$:
\begingroup
\renewcommand{\theenumi}{\roman{enumi}}
\begin{enumerate}
\item\label{one ray} There is a ray substem of length $n\ge 2$ if and only if there is a block which is the union of partial simplices of dimension $(n-2)$.
There is a ray substem of length $1$ if and only if there is a maximal simplex which has at most one partial simplex.
\item\label{item:TwoRays}
There are at least two ray substems of length $n\ge 2$ if and only if there is a maximal simplex of dimension $(2n-1)$ such that it intersects two distinct blocks each of which is the union of partial simplices of dimension $(n-2)$.
There are at least two rays substems of length $1$ if and only if there is a maximal simplex which does not have any partial simplices.
\item
There are three (four, resp.) ray substems of length $n\ge 2$ if and only if the labels of the maximal simplices given in \eqref{item:TwoRays} have the fundamental groups quasi-isometric to two (three, resp.) of $\mathbb{Z}\times\mathbb{Z}$, $\mathbb{Z}\times\mathbb{F}$ and $\mathbb{F}_2\times\mathbb{F}_2$; this is due to the quasi-minimality of $\Gamma$.
There are three (four, resp.) ray substems of length $1$ if and only if there are three maximal simplices which does not have partial simplices such that their labels have the fundamental groups quasi-isometric to two (three, resp.) of $\mathbb{Z}\times\mathbb{Z}$, $\mathbb{Z}\times\mathbb{F}$ and $\mathbb{F}_2\times\mathbb{F}_2$. 
\item Since $\Gamma$ and $\Gamma'$ are assumed to be quasi-minimal, there are no more than four ray substems of the same length.
\end{enumerate}
Note that this observation is invariant under $\bar\Phi_{\mathbbm{v}}$ since $\bar\Phi_{\mathbbm{v}}$ is an isomorphism and both partial simplices and blocks are invariant under $\bar\Phi_{\mathbbm{v}}$.
\endgroup

By looking at the quasi-isometric types of the fundamental groups of the labels appeared in the above observation, we can also know that the numbers of ray substems of the same length with the same number of grapes at leaves are equal.
For instance, in \eqref{one ray}, the labels of the partial simplices in $\cI(\bar{UP_2(\Gamma_{\mathsf{T}_{\mathbbm{v}}})})$ of dimension $(n-2)$ have the fundamental groups quasi-isometric to $\mathbb{Z}\times\mathbb{F}_2$ if and only if the leaf of $\mathsf{T}_{\mathbbm{v}}$ in the ray substem has one grape.
Therefore, we can conclude that there exist an isometry from $\Gamma_{\mathsf{U}}$ to $\Gamma'_{\mathsf{U}'}$ sending $\mathbbm{v}$ to $\mathbbm{w}$.
\end{proof}

\begin{proof}[Proof of Theorem~\ref{theorem:Quasi-minimalRepresentative}]
If $D\le 2$, then Proposition~\ref{prop:SimpleCase} implies that the theorem holds. Thus it suffices to only consider the case when $D\ge 3$.

For a $1$-tuple $\mathbbm{v}\in\mathcal{B}_{\mathsf{T}}$, let $\mathbbm{w}$ be a $1$-tuple in $\mathcal{B}_{\mathsf{T}'}$ which is obtained in Lemma~\ref{Lem:ComponentPreserving}.
By Lemma~\ref{Lem:UnionofRays}, if there is a non-ray substem in $\mathsf{T}_{\mathbbm{v}}$, then $\mathsf{T}'_{\mathbbm{w}}$ also has a non-ray substem.
In order to determine the number of non-ray substems, we need the following observation.
\begin{observation}
Let $\mathbbm{v}\in\mathcal{B}_{\mathsf{T}}$ be a $k$-tuple and $\mathbbm{w}$ be a $k$-tuple in $\mathcal{B}_{\mathsf{T}'}$ which is obtained in Lemma~\ref{Lem:ComponentPreserving}.
For a block $\mathcal{A}_1$ induced from $\mathsf{T}_{\mathbbm{v}_1}$ for some $(k+1)$-tuple $\mathbbm{v}_1$ having $\mathbbm{v}$ as a prefix, if a block $\mathcal{A}_2$ is adjacent to $\mathcal{A}_1$ in $\cI(\bar{UP_2(\Gamma_{\mathsf{T}_{\mathbbm{v}}})})$, then it must be induced from $\mathsf{T}_{\mathbbm{v}_2}$ for another $(k+1)$-tuple $\mathbbm{v}_2$ having $\mathbbm{v}$ as a prefix. 
Thus, $\bar\Phi(\mathcal{A}_1)$ and $\bar\Phi(\mathcal{A}_2)$ are blocks induced from $\mathsf{T}'_{\mathbbm{w}_1}$ and $\mathsf{T}'_{\mathbbm{w}_2}$ for two distinct $(k+1)$-tuples $\mathbbm{w}_1$ and $\mathbbm{w}_2$ having $\mathbbm{w}$ as their prefix, respectively.
\end{observation}

\begin{claim}\label{ClaimInSection6}
There is an isometry between 
$\Gamma_{\mathsf{T}_{\mathbbm{v}}}$ and $\Gamma_{\mathsf{T}'_{\mathbbm{w}}}$ sending $\mathbbm{v}$ to $\mathbbm{w}$.
\end{claim}

\begin{proof}[Proof of Claim~\ref{ClaimInSection6}]
For any $k$-tuple $\mathbbm{u}\in\mathcal{B}_\mathsf{T}$, we say $\mathbbm{u}$ has Property $(*)$ if there is a non-ray substem in $\mathsf{T}_{\mathbbm{u}}$.

If $\mathsf{T}_{\mathbbm{v}}$ has no vertex having Property $(*)$, by Lemma~\ref{Lem:UnionofRays}, the claim holds.

If $\mathsf{T}_{\mathbbm{v}}$ has only one vertex having Property $(*)$, then the vertex must be $\mathbbm{v}$ and 
for any $2$-tuple $\mathbbm{v}_1\in\mathsf{T}_{\mathbbm{v}}$, $\mathsf{T}_{\mathbbm{v}_1}$ is the union of ray substems.
By Lemma~\ref{Lem:UnionofRays}, there exists a $2$-tuple $\mathbbm{w}_1$ in $\mathsf{T}'_{\mathbbm{w}}$ such that there is an isometry between 
%the minimal rich representatives of 
$\Gamma_{\mathsf{T}_{\mathbbm{v}_1}}$ and $\Gamma'_{\mathsf{T}'_{\mathbbm{w}_1}}$ sending $\mathbbm{v}_1$ to $\mathbbm{w}_1$. 
By the above observation and the assumption that $\Gamma$ and $\Gamma'$ are quasi-minimal, we deduce that the claim holds in this case.

Suppose that $\mathsf{T}_{\mathbbm{v}}$ has at least two vertices having Property $(*)$.
Then one of them, say $\mathbbm{v}_1$, must be $\mathbbm{v}$.
Let $\mathbbm{v}_2$ be a $2$-tuple in $\mathsf{T}_{\mathbbm{v}}$.
Since $\mathsf{T}_{\mathbbm{v}_2}$ has less vertices having Property $(*)$ than $\mathsf{T}_{\mathbbm{v}}$, by induction, there must be a $2$-tuple $\mathbbm{w}_2$ in $\mathsf{T}'_{\mathbbm{w}}$ such that there is an isometry between 
$\Gamma_{\mathsf{T}_{\mathbbm{v}_2}}$ and $\Gamma'_{\mathsf{T}'_{\mathbbm{w}_2}}$ sending $\mathbbm{v}_2$ to $\mathbbm{w}_2$. 
By the above observation, we deduce that $\Gamma_{\mathsf{T}_{\mathbbm{v}}}$ and $\Gamma'_{\mathsf{T}'_{\mathbbm{w}}}$ must be isometric.
This completes the proof of the claim.
\end{proof}

Let $\mathbbm{v}_i$ be a $1$-tuple in $\mathcal{B}_{\mathsf{T}}$.
By Claim~\ref{ClaimInSection6} and the above observation, there exists a $1$-tuple $\mathbbm{w}_j$ such that there is an isometry between 
$\Gamma_{\mathsf{T}_{\mathbbm{v}_i}}$ and $\Gamma'_{\mathsf{T}'_{\mathbbm{w}_j}}$ sending $\mathbbm{v}_i$ to $\mathbbm{w}_j$.
Therefore, by the quasi-minimality assumption and the above observation, we can conclude that $\Gamma_{\mathcal{B}_\mathsf{T}}$ and $\Gamma'_{\mathcal{B}_{\mathsf{T}'}}$ are isometric.

The remaining case is when $\mathsf{T}(0)$ exists. In this case, by Lemma~\ref{Lem:ComponentPreserving}, $\mathsf{T}'(0)$ also exists and there is an isomorphism between $\cI(\bar{UP_2(\Gamma_{\mathsf{T}(0)})})$ and $\cI(\bar{UP_2(\Gamma'_{\mathsf{T}'(0)})})$ induced by $\bar\Phi$.

Now, define $\mathcal{S}_{\mathsf{T}(0)}(1)$ as the set of twigs in $\mathsf{T}(0)$ incident to $v$.
Then a vertex $\barbfv\in\cI$ induced from a twig in $\mathcal{S}_{\mathsf{T}(0)}(1)$ can be characterized as follows: there is a maximal simplex of $\cI$ which contains only one vertex $\barbfu$ in $\bar{\mathcal{V}}_{\mathsf{T}}(1)$ such that $\barbfv$ is (1) a predecessor or a successor of $\barbfu$ and (2) not in $\bar{\mathcal{V}}_{\mathsf{T}}(2)$.
By Lemma~\ref{lem:PreservingVertices}, then $\bar\Phi(\barbfv)$ is contained in a maximal simplex of $\cI'$ which contains only one vertex in $\bar{\mathcal{V}}_{\mathsf{T}'}(1)$, which is $\bar\Phi(\barbfu)$, such that $\bar\Phi(\barbfv)$ is (1) a predecessor or a successor of $\bar\Phi(\barbfu)$ and (2) not in $\bar{\mathcal{V}}_{\mathsf{T}'}(2)$.
It follows that $\bar\Phi$ induces a bijection between vertices induced from $\mathcal{S}_{\mathsf{T}(0)}(1)$ and vertices induced from $\mathcal{S}_{\mathsf{T}'(0)}(1)$.

After inductively defining $\mathcal{S}_{\mathsf{T}(0)}(k)$ and $\mathcal{S}_{\mathsf{T}'(0)}(k)$, we can apply all the lemmas, the observation in this subsection and Claim~\ref{ClaimInSection6} to show that there is an isometry between 
$\Gamma_{\mathsf{T}(0)}$ and $\Gamma'_{\mathsf{T}'(0)}$ such that the $0$-tuple of $\mathsf{T}$ is mapped to the $0$-tuple of $\mathsf{T}'$.
As $\Gamma$ and $\Gamma'$ are assumed to be quasi-minimal, therefore $\Gamma$ and $\Gamma'$ must be isometric.
\end{proof}

\section{Applications}\label{section:applications}
In this section, we will introduce two applications of our study on the 2-braid groups over bunches of grapes. One is to determine whether the $2$-braid group of a bunch of grapes is quasi-isometric to a right-angled Artin group, and the other is the quasi-isometry classification of the $4$-braid group over trees.

\subsection{Quasi-isometric to RAAGs}
Recall that the right-angled Artin group (RAAG) $A_\Lambda$ associated to a (possibly disconnected) simple graph $\Lambda$ is the group admitting the following group presentation:
\[
\langle \mathcal{V}(\Lambda)\mid [v_i,v_j]=1\quad \forall\{v_i,v_j\}\in \mathcal{E}(\Lambda)\rangle.
\]
In \cite{Oh22}, it was shown that there are two families of the $2$-braid groups over bunches of grapes, which are quasi-isometric and not quasi-isometric to RAAGs, respectively, as below.
\begin{itemize}
\item
For example, the $2$-braid group over a bunch of grapes such that its stem is a star graph $\mathsf{S}_n$ for some $n\ge 3$ and every leaf of the stem has exactly one grape (see Figure~\ref{fig:qitoRAAG}) is quasi-isometric to a RAAG $A_\Lambda*\mathbb{Z}$ for a tree $\Lambda$ of diameter at least $3$ \cite[Proposition 5.16]{Oh22}.
\item
For non-example, consider a bunch of grapes $\Gamma$, whose stem is a tree with two vertices of valency $3$ and four leaves that looks like an \emph{affine Dynkin diagram} $\tilde{\mathsf{D}}_n$ with $(n+1)$ vertices for some $n\ge 5$ (see Figure~\ref{fig:notqitoRAAG}).
If every vertex of valency $\le 2$ in the stem has exactly one grape, then $\mathbb{B}(\Gamma)$ is not quasi-isometric to any RAAG \cite[Corollary 5.18]{Oh22}.
\end{itemize}
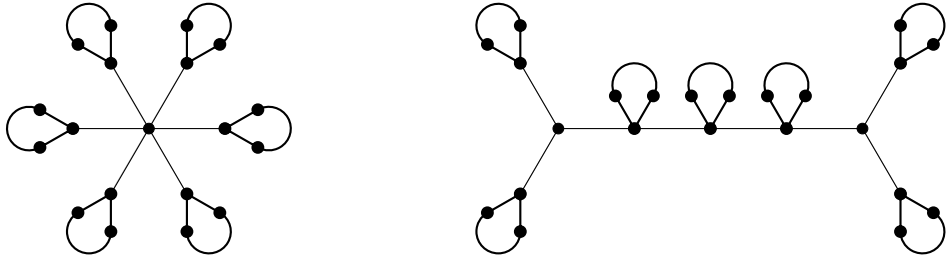
\begin{figure}[ht]
\subfigure[A bunch of grapes over a star graph $\mathsf{S}_n$]{
\begin{tikzpicture}[baseline=-.5ex]
\useasboundingbox (-3,-2) rectangle (3,2);
\draw[fill] (0,0) circle (2pt);
\foreach \i in {0,1,2,3,4,5} {
\draw[fill] (0,0) -- ({\i*60}:1) circle (2pt) node (A) {};
\grape[{\i*60}]{A};
}
\end{tikzpicture}
\label{fig:qitoRAAG}
}
\subfigure[A bunch of grapes over a affine Dynkin diagram $\tilde{\mathsf{D}}_n$]{
\begin{tikzpicture}[baseline=-.5ex]
\useasboundingbox (-2,-2) rectangle (6,2);
\draw[fill] (0,0) circle (2pt) -- (120:1) circle (2pt) node (A) {};
\draw[fill] (0,0) -- (240:1) circle (2pt) node (B) {};
\draw[fill] (0,0) -- (1,0) circle (2pt) node (C) {} -- (2,0) circle (2pt) node (D) {} -- (3,0) circle (2pt) node (E) {} -- (4,0) circle (2pt);
\draw[fill] (4,0) -- +(60:1) circle (2pt) node (F) {} +(0,0) -- +(-60:1) circle (2pt) node (G) {};
\grape[120]{A};
\grape[240]{B};
\grape[90]{C};
\grape[90]{D};
\grape[90]{E};
\grape[60]{F};
\grape[-60]{G};
\end{tikzpicture}
\label{fig:notqitoRAAG}
}
\caption{An example and a non-example of bunches of grapes, whose $2$-braid group is quasi-isometric to a RAAG}
\end{figure}
In this subsection, we will slightly extend both classes of the $2$-braid groups over bunches of grapes.
Throughout this subsection, let $\Gamma=(\mathsf{T},\ell)\in\grapegraph$.

Let us first see the case when $\mathbb{B}_2(\Gamma)$ is quasi-isometric to a RAAG.
If $\mathsf{T}$ is an $n$-star such that each leaf has one grape, then the quasi-minimal representative of $\Gamma$ has a stem which has length $2$.
This case can easily be generalized as follows.

\begin{lemma}\label{lem:isomorphictoRAAG}
Let $\Gamma=(\mathsf{P},\loops)$ be a bunch of grapes over the path graph $\mathsf{P}$.
Then $\mathbb{B}_2(\Gamma)$ is isomorphic to a RAAG.
\end{lemma}
\begin{proof}
If $\Gamma$ is not large, then $\Gamma$ is homeomorphic to a graph in Example~\ref{Ex:FreeGroupCase} and thus $\mathbb{B}_2(\Gamma)$ is isomorphic to a free group.
Otherwise, after smoothing twigs, by Lemma~\ref{NoHangingEdge}, we assume that $\Gamma$ is rich.
By Proposition~\ref{Prop:FreeFactor}, it suffices to show that $\pi_1(UP_2(\Gamma))$ is isomorphic to a RAAG.
  
Let $\mathsf{P}=[v_0,\dots,v_n]$ and denote the induced subgraph $[v_i,\dots, v_j]\subset\mathsf{P}$ for $0\le i\le j\le n$ by $[i,j]$.
By Lemma~\ref{lem:TwigCorrespondence}, there are $n$ maximal product subcomplexes of $UP_2(\Gamma)$:
\[
K_i=\Gamma_{[0,i]}\interior{\times} \Gamma_{[i+1,n]}\text{ for } i=0,\dots,n-1.
\]
Let $U_i=K_0\cup\dots\cup K_i$ and so $U_0=K_0$ and $U_{n-1}=UP_2(\Gamma)$.

Suppose that $\ell(v_i)=1$ for $0\leq i\leq n$.
Then $\pi_1(K_i)\cong \langle a_0,\dots,a_i\rangle\times\langle b_{i+1},\dots,b_n\rangle$, and in particular $\pi_1(K_0)\cong \langle a_0\rangle\times \langle b_1,\dots, b_n\rangle \cong A_{\mathsf{S}_n}$.
By induction on $i$ and the Seifert-van Kampen Theorem, one can easily prove that 
\begin{align*}
\pi_1(U_i)=\langle a_0,\dots,a_i,b_1,\dots,b_{n}\mid [a_j,b_k]=1\quad\forall j<k\rangle\cong A_{\mathsf{S}_n\cup\dots\cup\mathsf{S}_{n-i}},
\end{align*}
where the central vertex and leaves of $\mathsf{S}_{n-j}$ are given by $a_j$ and $b_{j+1},\dots, b_n$, respectively.
Therefore $\pi_1(UP_2(\Gamma))=\pi_1(U_{n-1})\cong A_{\mathsf{S}_n\cup\dots\cup\mathsf{S}_{1}}$, and so we are done.

The proofs when $\ell(v_i)\ge 2$ for some $i$'s are essentially the same as above and we omit the detail.
\end{proof}

\begin{proposition}\label{prop:QItoRAAG}
Let $\Gamma$ be a large bunch of grapes.
If the quasi-minimal representative of $\Gamma$ has a path stem, then $\mathbb{B}_2(\Gamma)$ is quasi-isometric to a RAAG.
\end{proposition}
\begin{proof}
This is done by a combination of Theorem~\ref{thm:QIminimal} and Lemma~\ref{lem:isomorphictoRAAG}
\end{proof}

Now, let us see the case when $\mathbb{B}_2(\Gamma)$ is not quasi-isometric to any RAAG. Before we move on, we need some preliminaries about RAAGs starting with well-known facts about ends of RAAGs.
\begin{itemize}
\item $A_\Lambda$ is one-ended if and only if $\Lambda$ is connected and has at least two vertices.
\item $A_\Lambda$ is two-ended if and only if $\Lambda$ is a single vertex.
\item $A_\Lambda$ is infinite-ended if and only if $\Lambda$ is disconnected. In this case, $A_\Lambda\cong A_{\Lambda_1}*\dots*A_{\Lambda_n}$ where $\Lambda_1,\dots,\Lambda_n$ are all the components of $\Lambda$.
\end{itemize}

For any (possibly disconnected) simple graph $\Lambda$, the RAAG $A_\Lambda$ can be considered as the fundamental group of the \emph{Salvetti complex} $S_\Lambda$, whose universal cover will be denoted by $X_\Lambda$.
The dimension of $S_\Lambda$ or $A_\Lambda$ is equal to $\max\{n\mid \mathsf{K}_n\hookrightarrow\Lambda\}$, where $\mathsf{K}_n$ is a complete graph with $n$ vertices, and therefore $S_\Lambda$ is a special \emph{square} complex if and only if $\Lambda$ is triangle-free and has at least one edge. 
In this case, we can define both $\cRI(S_\Lambda)$ and $\cI(X_\Lambda)$ using the fact that each standard product subcomplex of $S_\Lambda$ corresponds to an induced subgraph of $\Lambda$ which admits a join decomposition. 
See \cite{CH} and \cite[\S4.1]{Oh22} for more about RAAGs and the structures of $\cRI(S_\Lambda)$ and $\cI(X_\Lambda)$, respectively.

\begin{lemma}\label{lem:CopiesofR}
Let $\Lambda$ be a triangle-free connected simple graph with at least two vertices.
Then $\cI(X_\Lambda)$ is the union of copies $\mathcal{R}_x$ of $\cRI(S_\Lambda)$ where $x$ corresponds to a vertex of $X_\Lambda$.
\end{lemma}
\begin{proof}
For each vertex $x\in X_\Lambda$, let $\mathcal{R}_x$ be the induced subcomplex of $\cI(X_\Lambda)$ spanned by the vertices corresponding to the maximal product subcomplexes of $X_\Lambda$ containing $x$.
Then it is obvious that $\cI(X_\Lambda)$ is covered by $\mathcal{R}_x$'s for all $x\in \mathcal{V}(X_\Lambda)$.
Since $S_\Lambda$ has one vertex, the restriction of the canonical quotient map $\rho:\cI(X_\Lambda)\to\cRI(S_\Lambda)$ to $\mathcal{R}_x$ is an isomorphism.
Therefore, each $\mathcal{R}_x$ can be seen as a copy of $\cRI(S_\Lambda)$.
\end{proof}

We remark that $\mathcal{R}_x$ and $\mathcal{R}_y$ are identical if and only if any maximal product subcomplex of $X_\Lambda$ containing $x$ also contains $y$ and \textit{vice versa}.

\begin{lemma}\label{lem:flag}
Let $\Lambda$ be given as in Lemma~\ref{lem:CopiesofR}.
If $\cI(X_\Lambda)$ is a flag complex, then so is $\cRI(S_\Lambda)$. 
\end{lemma}
\begin{proof}
By Propositions 4.3 and 4.6 in \cite{Oh22}, it is shown that both $\cI(X_\Lambda)$ and $\cRI(S_\Lambda)$ are simplicial complexes.

Assume that $\cRI(S_\Lambda)$ is not a flag complex, i.e., there exists a complete subgraph with vertices $\bfv_1,\dots,\bfv_n$ ($n\ge 3$) in $\cRI(S_\Lambda)$ which does not span a simplex.
Consider the preimage of the complete subgraph in $\cRI(S_\Lambda)$ under the restriction of $\rho$ to $\mathcal{R}_x$ for any $x\in X_\Lambda$ obtained in Lemma~\ref{lem:CopiesofR}.
Since $\cI(X_\Lambda)$ is assumed to be a flag complex, the intersection of the maximal product subcomplexes of $X_\Lambda$ corresponding to the vertices of the preimage contains a standard product subcomplex.
It follows that there is a standard product subcomplex contained in the intersection of the maximal product subcomplexes of $S_\Lambda$ corresponding to $\bfv_1,\dots,\bfv_n$.
By the definition of $\cRI(S_\Lambda)$, it is a contradiction, and therefore, $\cRI(S_\Lambda)$ is a flag complex.
\end{proof}

An affine Dynkin diagram $\tilde{\mathsf{D}}_n$ is obtained from two path graphs $\mathsf{P}_2^i$ of length $2$ and a path graph $\mathsf{P}_{n-4}$ of length $(n-4)$ by identifying each central vertex of $\mathsf{P}_2^i$ with each leaf of $\mathsf{P}_{n-4}$. 
If $\mathsf{T}$ is an affine Dynkin diagram of diameter $\ge 3$ and $\Gamma$ is normal, then $\cRI(UP_2(\Gamma))$ has at least two maximal edges such that any path joining these two edges must pass through a simplex of dimension $>1$. Indeed, this is a sufficient condition for $\cI(\bar{UP_2(\Gamma)})$ to be not isomorphic to $\cI(X_\Lambda)$ for any simple graph $\Lambda$.

\begin{proposition}\label{prop:NotQItoRAAG}
Let $\Gamma=(\mathsf{T},\ell)$ be a bunch of grapes. 
Suppose that there are four leaves in $\mathsf{T}$ such that the minimal substem $\mathsf{T}_0\subset\mathsf{T}$ containing these four leaves is isometric to $\tilde{\mathsf{D}}_n$ for some $n\ge 5$ and $\Gamma_{\mathsf{T}_0}$ is normal. 
Then $\mathbb{B}_2(\Gamma)$ is not quasi-isometric to any RAAG.
\end{proposition}
\begin{proof}
Having such a minimal substem $\mathsf{T}_0$ isometric to $\tilde{\mathsf{D}}_n$ is preserved by any operations on bunches of grapes defined in Section~\ref{section:Operations}.
By Theorem~\ref{thm:QIminimal}, thus, we can assume that $\Gamma$ is quasi-minimal.

Assume that there exists a RAAG is quasi-isometric to $\mathbb{B}_2(\Gamma)$.
By Proposition~\ref{Prop:FreeFactor}, $\mathbb{B}_2(\Gamma)$ is isomorphic to $\pi_1(UP_2(\Gamma))*\mathbb{F}_n$ for some $n\ge 2$ such that $\pi_1(UP_2(\Gamma))$ is one-ended by Proposition~\ref{Prop:connectedsimplicial}.
It follows from \cite[Theorem~0.4]{PW02} that the defining graph of the RAAG must have a component $\Lambda$ such that $A_\Lambda$ is quasi-isometric to $\pi_1(UP_2(\Gamma))$. 
By \cite[Theorem 1-3]{Hua(b)}, moreover, $\Lambda$ must be 2-dimensional, i.e. it is triangle-free and has at least two vertices.
By Theorem~\ref{theorem:IsobetInt}, thus, we assume for contradiction that there is an isomorphism $\Phi:\cI(\bar{UP_2(\Gamma)})\to\cI(X_\Lambda)$.

Let $v,w,x$, and $y$ be four leaves in $\mathsf{T}$ corresponding to leaves in $\mathsf{T}_0$, and let $\mathsf{P}^0, \mathsf{P}^1, \mathsf{P}^2$, and $\mathsf{P}^3$ be the path substems joining $v$ to $w$, $w$ to $x$, $x$ to $y$, and $y$ to $v$, respectively.
Let $\triangle_i$ be the maximal simplex in $\cRI(UP_2(\Gamma))$ corresponding to $\mathsf{P}^i$ for each $i$.
We may assume that both $\mathsf{P}^0$ and $\mathsf{P}^2$ are of length $2$ and both $\mathsf{P}^1$ and $\mathsf{P}^3$ are of length $n-2\ge 3$.
Then both $\triangle_0$ and $\triangle_2$ are one-dimensional, and both $\triangle_1$ and $\triangle_3$ are $(n-3)$-dimensional.

Consider $\bigcup_{i=0}^3\triangle_i$ in $\cRI(UP_2(\Gamma))$.
Then it forms a non-trivial loop in $\cRI(UP_2(\Gamma))$ since both $\triangle_0$ and $\triangle_2$ are maximal edges.
On the other hand, by Theorem~\ref{theorem:structureofI}, $\cI(\bar{UP_2(\Gamma)})$ is simply connected.
It follows that there is a bi-infinite sequence of distinct maximal simplices in $\cI(\bar{UP_2(\Gamma)})$
\[
\dots,\bar\triangle_{-1},\bar\triangle_0,\bar\triangle_1,\bar\triangle_2,\bar\triangle_3,\dots
\]
such that $\rho(\bar\triangle_{n})=\triangle_i$ and $\bar\triangle_{i}\cap\bar\triangle_{i}=\empty$ if $|i-j|\ge 2$.
Note that $\bar\triangle_{2n}$ for any integer $n$ satisfies the following property: there is no path joining two endpoints of $\bar\triangle_{2n}$ without passing through $\bar\triangle_{2n}$.

Note that since $\Phi$ is an isomorphism and $\cI(\bar{UP_2(\Gamma)})$ is a flag complex by Theorem~\ref{theorem:structureofI}, so is $\cI(X_\Lambda)$, and thus by Lemma~\ref{lem:flag}, so is $\cRI(S_\Lambda)$.
Moreover, $\Phi(\triangle_i)$ and $\Phi(\triangle_{i+1})$ have different dimensions for any $i$.
Combining these two facts with the fact that the canonical quotient map $\rho:\cI(X_\Lambda)\to\cRI(S_\Lambda)$ is a combinatorial map, it follows that the restriction of $\rho$ to $\Phi(\bigcup_{j=0}^2\bar\triangle_{i+j})$ is injective for any $i$.

Suppose that $k\ge 3$ is the minimal integer such that the restriction of $\rho$ to $\Phi(\bigcup_{j=0}^k\bar\triangle_{i+j})$ is not injective for some $i$.
By the assumption on $k$, $\rho(\Phi(\bar\triangle_{i+k}))$ must be incident to $\rho(\Phi(\bar\triangle_{i}))$ and $\rho(\Phi(\bar\triangle_{i+k-1}))$. However, it means that $\Phi(\bar\triangle_{2k})$ for $i\le 2k\le i+k-1$ does not have the property that there is no path joining two endpoints of $\Phi(\bar\triangle_{2k})$ without passing through $\Phi(\bar\triangle_{2k})$, which is a contradiction.
It follows that the restriction of $\rho$ to $\Phi(\bigcup_{i=0}^k\bar\triangle_{i})$ is injective for any integer $k>0$ and any $i$ but it is also a contradiction since $\cRI(S_\Lambda)$ is finite.
Therefore, there is no RAAG quasi-isometric to $\mathbb{B}_2(\Gamma)$.
\end{proof}

In summary, we have the following theorem.
\begin{theorem}\label{theorem:bunches of grapes related to RAAG}
There are infinitely many graphs with circumference one whose $2$-braid groups are quasi-isometric to RAAGs; this class properly contains the class in \cite[Proposition 5.16]{Oh22}.

There are also infinitely many graphs with circumference one whose $2$-braid groups are not quasi-isometric to RAAGs; this class properly contains the class in \cite[Corollary 5.18]{Oh22}.
\end{theorem}

\subsection{4-braid groups over trees}
In this subsection, we apply the methods in Section~\ref{section:proofofThm} to the 4-braid groups over trees. 

\begin{definition}[Bunches of grapes grown from trees]\label{definition:bunch of grapes grown from tree}
Given a tree $\Lambda$, we construct a bunch of grapes $\Gamma(\Lambda)=(\Lambda,\loops)$ over the stem $\Lambda$ such that $\loops(v)=\binom{\val_\Lambda(v)-1}2$ for each vertex $v\in\mathcal{V}(\Lambda)$. 
We say that $\Gamma(\Lambda)$ is \emph{grown from} $\Lambda$, or $\Lambda$ is obtained from $\Gamma(\Lambda)$ by \emph{crushing grapes}.
See Figure~\ref{figure:tree and grapes} for example.
\end{definition}

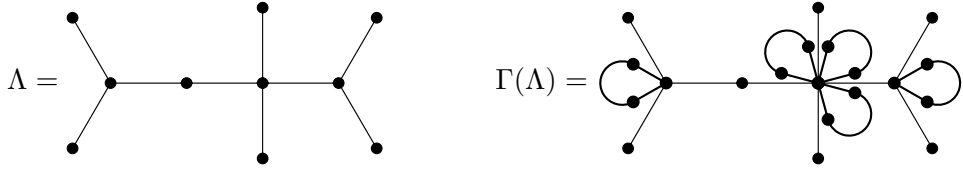
\begin{figure}[ht]
\begin{align*}
\Lambda&=
\begin{tikzpicture}[baseline=-.5ex]
\draw[fill] (120:1) circle (2pt) -- (0,0) circle (2pt) -- (240:1) circle (2pt);
\draw[fill] (0,0) -- (1,0) circle (2pt) -- (2,0) circle (2pt) -- +(90:1) circle (2pt) +(0,0) -- +(-90:1) circle (2pt) +(0,0) -- +(1,0) circle (2pt);
\draw[fill] (3,0) -- +(60:1) circle (2pt) +(0,0) -- +(-60:1) circle (2pt);
\end{tikzpicture}&
\Gamma(\Lambda)&=
\begin{tikzpicture}[baseline=-.5ex]
\draw[fill] (120:1) circle (2pt) -- (0,0) node (A) {} circle (2pt) -- (240:1) circle (2pt);
\draw[fill] (0,0) -- (1,0) circle (2pt) -- (2,0) node (B) {} circle (2pt) -- +(90:1) circle (2pt) +(0,0) -- +(-90:1) circle (2pt) +(0,0) -- +(1,0) circle (2pt);
\draw[fill] (3,0) node (C) {} -- +(60:1) circle (2pt) +(0,0) -- +(-60:1) circle (2pt);
\grape[180]{A};
\grape[45]{B};
\grape[-45]{B};
\grape[135]{B};
\grape[0]{C};
\end{tikzpicture}
\end{align*}
\caption{A tree $\Lambda$ and a bunch of grapes $\Gamma(\Lambda)$ grown from $\Lambda$}
\label{figure:tree and grapes}
\end{figure}

Then we have the following observation:
\begin{enumerate}
\item $\Gamma(\Lambda)$ contains $\Lambda$ and the sets of essential vertices of $\Lambda$ and $\Gamma(\Lambda)$ are identical.
\item Any grape in $\Gamma(\Lambda)$ adjacent to $v$ could be identified with an embedding of the star $3$-graph into $\Lambda$ near $v$, which defines a natural isomorphism
\[
\mathbb{B}_1(\Gamma(\Lambda))\cong \mathbb{B}_2(\Lambda)\cong \mathbb{F}_N,
\]
where $N=\sum_{v\in\mathcal{V}(\Lambda)}\binom{\val_\Lambda(v)-1}2$.
\item For each subtree $\Lambda'\subset\Lambda$, $\Gamma(\Lambda')$ is a sub-bunch of grapes of $\Gamma(\Lambda)$.
\item For each $n\ge 3$ and $\ell=\binom{n-1}2$, both $\mathbb{B}_4(\mathsf{S}_n)$ and $\mathbb{B}_2(\Gamma(\mathsf{S}_n))$ are free groups of rank 
\[
\binom{n+2}{3}(n-2)-\binom{n+2}{4}+1 \quad\text{ and }\quad
\frac{(n+\ell)(n+3\ell-3)}{2}+1,
\]
respectively, which are at least $\ell$.
\item For every edge $e\in\Lambda$, there are edge stabilization maps described in \cite{ADCK2020}
\[
e_*:\mathbb{B}_1(\Gamma(\Lambda))\to\mathbb{B}_2(\Gamma(\Lambda))
\quad\text{ and }\quad
e_*^2:\mathbb{B}_2(\Lambda)\to \mathbb{B}_4(\Lambda),
\]
which are injective if $e$ is an edge adjacent to a leaf.
\end{enumerate}

\begin{proposition}
Let $\Lambda$ be a tree and $\Gamma=\Gamma(\Lambda)$ be a small bunch of grapes grown from $\Lambda$.
For each edge $e\in\Lambda$ adjacent to a leaf, there is a triple $(\mathbb{F},\mathbb{F}', f)$ of free groups $\mathbb{F}$, $\mathbb{F}'$ and an isomorphism $f:\mathbb{B}_2(\Gamma(\Lambda))\to \mathbb{B}_4(\Lambda)$ depending on $e$, which make the following commutative:
\[
\begin{tikzcd}[baseline=-.5ex]
\mathbb{B}_2(\Lambda)\arrow[r,"e_*^2"]\arrow[d,"\cong"'] & \mathbb{B}_4(\Lambda)\arrow[r, hookrightarrow] & \mathbb{B}_4(\Lambda)*\mathbb{F}\arrow[d,"\exists f"]\\
\mathbb{B}_1(\Gamma)\arrow[r,"e_*"] & \mathbb{B}_2(\Gamma)\arrow[r, hookrightarrow] & \mathbb{B}_2(\Gamma)*\mathbb{F}'
\end{tikzcd}
\]

Moreover, if $\Lambda$ is a star graph, then the images of $e_*^2$ and $e_*$ are free factors of their ranges.
\end{proposition}
\begin{proof}
We will use the induction on the number $N$ of essential vertices of $\Lambda$ or $\Gamma$.

Suppose that $N=1$. By the discrete Morse theory and the definition of the stabilization map, the images under $e_*^2$ and $e_*$ of critical cells in $UD_2(\Lambda)$ and $UD_1(\Gamma)$ could be realized as critical cells in $UD_4(\Lambda)$ and $UD_2(\Gamma)$, respectively.
Since both $\mathbb{B}_4(\Lambda)$ and $\mathbb{B}_2(\Gamma)$ are free since $\Lambda$ is a star graph, both $e_*^2(\mathbb{B}_2(\Lambda))$ and $e_*(\mathbb{B}_1(\Gamma))$ are free factors of $\mathbb{B}_4(\Lambda)$ and $\mathbb{B}_2(\Gamma)$, respectively.

Since both $e_*^2$ and $e_*$ are injective as seen in the observation above, we may write
\[
\mathbb{B}_4(\Lambda)\cong \mathbb{B}_2(\Lambda) * \mathbb{F}'\quad\text{ and }\quad
\mathbb{B}_2(\Gamma)\cong \mathbb{B}_1(\Gamma) * \mathbb{F},
\]
for some free groups $\mathbb{F}$ and $\mathbb{F}'$, and we are done by defining $f_\Lambda$ as the composition of isomorphisms
\[
f_\Lambda:\mathbb{B}_4(\Lambda) *\mathbb{F} \to \mathbb{B}_2(\Lambda)*\mathbb{F}'*\mathbb{F} \to
\mathbb{B}_1(\Gamma)*\mathbb{F}*\mathbb{F}' \to \mathbb{B}_2(\Gamma)*\mathbb{F}'.
\]

Now assume that $N\ge 2$. 
By subdividing edges of $\Lambda$ if necessary, we can pick a bivalent vertex $v$ of $\Lambda$, which separate $\Lambda$ into two parts $\Lambda_1$ and $\Lambda_2$ by cutting $e$ as seen in the picture below
\[
\begin{tikzcd}[column sep=4pc]
\Lambda=\begin{tikzpicture}[baseline=-.5ex]
\draw[fill] (0,0) circle (2pt);
\draw (-1.5,-0.5) rectangle (-0.5,0.5);
\draw (-0.5,0) -- node[midway, above] {$e_1$} (0,0) node[below] {$v$} -- (0.5,0) node[midway,above] {$e_2$};
\draw (1.5,0.5) rectangle (0.5, -0.5);
\end{tikzpicture}
\ar[r,"cutting"] &
\Lambda_1=\begin{tikzpicture}[baseline=-.5ex]
\draw (-1.5,-0.5) rectangle (-0.5,0.5);
\draw[fill] (-0.5,0) -- node[midway,above] {$e_1$} (0,0) circle(2pt) node[right] {$v$};
\end{tikzpicture}\quad
\begin{tikzpicture}[baseline=-.5ex]
\draw[fill] (0,0) circle (2pt) node[left] {$v$} -- node[midway,above] {$e_2$} (0.5,0);
\draw (1.5,0.5) rectangle (0.5, -0.5);
\end{tikzpicture}
=\Lambda_2
\end{tikzcd}
\]
so that both $\Lambda_1$ and $\Lambda_2$ has at least one essential vertices. 
Then we have triples $(\mathbb{F}, \mathbb{F}', f_1)$ and $(\mathbb{G}, \mathbb{G}', f_2)$ by the induction hypothesis.

We denote two bunches of grapes grown from $\Lambda_1$ and $\Lambda_2$ by $\Gamma_1$ and $\Gamma_2$.
Since $\Gamma$ is obtained by identifying two leaves of $\Gamma_1$ and $\Gamma_2$, 
we decompose the configuration space $UC_4(\Lambda)$ according to whether the vertex $v$ is occupied or not by a point in $UC_4(\Lambda)$ as follows: 
for $\mathring{\Lambda}_i=\Lambda_i\setminus\{v\}$,
\[
\setlength\arraycolsep{0.1em}
\begin{array}{ccccccc}
UC_4(\Lambda)&=& UC_4(\Lambda_1)&\displaystyle{\bigcup_{UC_3(\mathring{\Lambda}_1)\times\{*\}}} &\left(UC_3(\Lambda_1)\times UC_1(\Lambda_2)\right) \\
 & & &\displaystyle{\bigcup_{UC_2(\mathring{\Lambda}_1)\times\{*\}\times UC_1(\mathring{\Lambda}_2)}}&
\left(UC_2(\mathring{\Lambda}_1)\times UC_2(\mathring{\Lambda}_2)\right) \\
 & & &\displaystyle{\bigcup_{UC_1(\mathring{\Lambda}_1)\times\{*\}\times UC_2(\mathring{\Lambda}_2)}} &\left(UC_1(\Lambda_1)\times UC_3(\Lambda_2)\right)
&\displaystyle{\bigcup_{\{*\}\times UC_3(\mathring{\Lambda}_2)}} &UC_4(\Lambda_2),
\end{array}
\]
which yields an iterated amalgamated free product decomposition of $\mathbb{B}_4(\Lambda)$
\begin{align*}
\mathbb{B}_4(\Lambda)&\cong \mathbb{B}_4(\Lambda_1)*_{\mathbb{B}_3(\Lambda_1)\times\mathbb{B}_0(\Lambda_2)}\left(\mathbb{B}_3(\Lambda_1)\times \mathbb{B}_1(\Lambda_2)\right) *_{\mathbb{B}_2(\Lambda_1)\times\mathbb{B}_1(\Lambda_2)}
\left(\mathbb{B}_2(\Lambda_1)\times \mathbb{B}_2(\Lambda_2)\right)\\
&\mathrel{\hphantom{=}}
*_{\mathbb{B}_1(\Lambda_1)\times\mathbb{B}_2(\Lambda_2)} \left(\mathbb{B}_1(\Lambda_1)\times \mathbb{B}_3(\Lambda_2)\right) *_{\mathbb{B}_0(\Lambda_1)\times\mathbb{B}_3(\Lambda_2)} \mathbb{B}_4(\Lambda_2).
\end{align*}
Here, two images of $\mathbb{B}_i(\Lambda_1)\times \mathbb{B}_{3-i}(\Lambda_2)$ under $(e_1)_*$ and $(e_2)_*$ are identified for each $0\le i\le 3$.
Since $\mathbb{B}_i(\Lambda_1)$ and $\mathbb{B}_i(\Lambda_2)$ are trivial for $i\le 1$ and both $\mathbb{B}_3(\Lambda_1)$ and $\mathbb{B}_3(\Lambda_2)$ are subgroups of $\mathbb{B}_4(\Lambda_1)$ and $\mathbb{B}_4(\Lambda_2)$, respectively, as observed before, we have
\begin{align*}
\mathbb{B}_4(\Lambda)&\cong \mathbb{B}_4(\Lambda_1)*_{\mathbb{B}_2(\Lambda_1)}
\left(\mathbb{B}_2(\Lambda_1)\times \mathbb{B}_2(\Lambda_2)\right)*_{\mathbb{B}_2(\Lambda_2)} \mathbb{B}_4(\Lambda_2),
\end{align*}
where $\mathbb{B}_2(\Lambda_i)$ is identified with the subgroup of $\mathbb{B}_4(\Lambda_i)$ via $(e_i)_*^2$ for each $i=1,2$.

On the other hand, the decomposition on $UC_2(\Gamma)$ gives us the following:
\begin{align*}
\mathbb{B}_2(\Gamma)&\cong
\mathbb{B}_2(\Gamma_1)*_{\mathbb{B}_1(\Gamma_1)\times\mathbb{B}_0(\Gamma_2)}\left(\mathbb{B}_1(\Gamma_1)\times\mathbb{B}_1(\Gamma_2)\right)*_{\mathbb{B}_0(\Gamma_1)\times\mathbb{B}_1(\Gamma_2)}
\mathbb{B}_2(\Gamma_2)\\
&\cong\mathbb{B}_2(\Gamma_1)*_{\mathbb{B}_1(\Gamma_1)}\left(\mathbb{B}_1(\Gamma_1)\times\mathbb{B}_1(\Gamma_2)\right)*_{\mathbb{B}_1(\Gamma_2)}\mathbb{B}_2(\Gamma_2)
\end{align*}
since $\mathbb{B}_0(\Gamma_1)$ and $\mathbb{B}_0(\Gamma_2)$ are trivial and $\mathbb{B}_2(\Gamma_1)\cong\mathbb{B}_1(\Gamma_1)*\mathbb{F}$.

Moreover, the commutative diagrams for $\Lambda_1$ and $\Lambda_2$ imply that
\begin{align*}
\mathbb{B}_4(\Lambda)*\mathbb{F}*\mathbb{G}&\cong (\mathbb{F}*\mathbb{B}_4(\Lambda_1))*_{\mathbb{B}_2(\Lambda_1)}
\left(\mathbb{B}_2(\Lambda_1)\times \mathbb{B}_2(\Lambda_2)\right)*_{\mathbb{B}_2(\Lambda_2)} (\mathbb{B}_4(\Lambda_2)*\mathbb{G})\\
&\cong
(\mathbb{F}'*\mathbb{B}_2(\Gamma_1))*_{\mathbb{B}_1(\Gamma_1)}
\left(\mathbb{B}_1(\Gamma_1)\times \mathbb{B}_1(\Gamma_2)\right)*_{\mathbb{B}_1(\Gamma_2)} (\mathbb{B}_2(\Gamma_2)*\mathbb{G}')\\
&\cong
\mathbb{B}_2(\Gamma)*\mathbb{F}'*\mathbb{G}'.
\end{align*}
Hence we have an isomorphism $f:\mathbb{B}_4(\Lambda)*\mathbb{F}*\mathbb{G} \cong \mathbb{B}_2(\Gamma)*\mathbb{F}'*\mathbb{G}'$.

Finally, we have decompositions of $\mathbb{B}_2(\Lambda)$ and $\mathbb{B}_1(\Gamma)$ as follows:
\begin{align*}
\mathbb{B}_2(\Lambda)&\cong
\mathbb{B}_2(\Lambda_1) *_{\mathbb{B}_1(\Lambda_1)\times\mathbb{B}_0(\Lambda_1)} 
\left(\mathbb{B}_1(\Lambda_1)\times\mathbb{B}_1(\Lambda_2)\right) *_{\mathbb{B}_1(\Lambda_1)\times\mathbb{B}_0(\Lambda_1)} 
\mathbb{B}_2(\Lambda_2)\\
&\cong \mathbb{B}_2(\Lambda_1) * \mathbb{B}_2(\Lambda_2)\\
\mathbb{B}_1(\Gamma)&\cong \mathbb{B}_1(\Gamma_1)*\mathbb{B}_1(\Gamma_2),
\end{align*}
which have isomorphic factors.

Now let $e$ be an edge adjacent to any leaf of $\Lambda$, say $*$.
We may assume that $e$ is an edge of $\Lambda_1$.
Then there is a unique edge path $\gamma$ from $*$ to $v$, and we define $\Lambda_2'=\Lambda_2\cup\gamma$, which is homeomorphic to $\Lambda_2$.
This homeomorphism $h:\Lambda_2\to\Lambda_2'$ induces an equivariant isomorphism so that $(e_*)^2 h_* = h_* (e_2)_*^2 : \mathbb{B}_2(\Lambda_2) \to \mathbb{B}_4(\Lambda_2')$.
Similarly, if we define $\Gamma'=\Gamma(\Lambda_2')$, then there is an induced homeomorphism $h:\Gamma_2\to\Gamma_2'$ such that $e_* h_*= h_* (e_2)_*:\mathbb{B}_1(\Gamma_2)\to\mathbb{B}_2(\Gamma_2')$.

However, the map $e_*$ is indeed a conjugate of $(e_2)_*$ by a path between two basepoints of $UC_2(\Lambda_1)$ or $UC_1(\Gamma_1)$, one of which has an additional point near the leaf $*$ and the other has an additional point near the vertex $v$.
Roughly speaking, $e^2(\mathbf{x})$ can be obtained by \emph{dragging} two points closest to $v$ in $e_2^2(\mathbf{x})$ to the leaf $*$ for each $\mathbf{x}\in UC_2(\Lambda_1)$.

More precisely, we have the commutative diagram in Figure~\ref{figure:changing basepoints}, where $x_0\in UC_1(\Gamma)$ and $\mathbf{x}_0\in UC_2(\Lambda)$ are basepoints and $f_e$ is the composition of $f$ and two isomorphisms coming from dragging isomorphisms.
Then the outer hexagon is the desired commutative diagram.
\end{proof}

\begin{figure}[ht]
\[
\begin{tikzcd}[row sep=1pc]
& \mathbb{B}_4(\Lambda, e^2(\mathbf{x}_0)) \arrow[r, hookrightarrow] \arrow[dd, "dragging", "\cong"'] & \mathbb{B}_4(\Lambda, e^2(\mathbf{x}_0)) * \mathbb{F}*\mathbb{G} \arrow[dd, "\cong"'] \arrow[dddddd, bend left=90, "f_e", "\cong"'] \\
\mathbb{B}_2(\Lambda, \mathbf{x}_0) \arrow[dddd, "\cong"] \arrow[ru, "e_*^2"] \arrow[rd, "(e_2)_*^2"'] \\
& \mathbb{B}_4(\Lambda, e_2^2(\mathbf{x}_0)) \arrow[r, hookrightarrow] & \mathbb{B}_4(\Lambda, e_2^2(\mathbf{x}_0)) * \mathbb{F} * \mathbb{G} \arrow[dd, "\cong"', "f"]\\ \\
& \mathbb{B}_2(\Gamma, e_2(x_0)) \arrow[r, hookrightarrow] \arrow[dd, "dragging", "\cong"'] & \mathbb{B}_2(\Gamma, e_2(x_0)) * \mathbb{F}'*\mathbb{G}'\arrow[dd, "\cong"']\\
\mathbb{B}_1(\Gamma,x_0) \arrow[ru, "(e_2)_*"] \arrow[rd, "e_*"']\\
& \mathbb{B}_2(\Gamma, e(x_0)) \arrow[r, hookrightarrow] & \mathbb{B}_2(\Gamma, e(x_0)) * \mathbb{F}'*\mathbb{G}'
\end{tikzcd}
\]
\caption{Changing basepoints}
\label{figure:changing basepoints}
\end{figure}
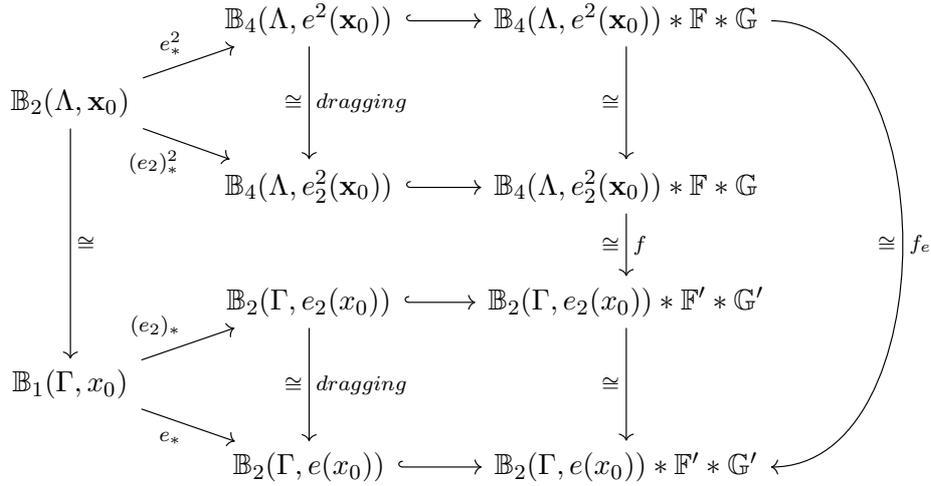

As a direct consequence of the previous proposition, we have the following theorem.

\begin{theorem}\label{Thm:B_4quasi-isometrictoB_2}
Let $\Lambda$ be a tree and $\Gamma=\Gamma(\Lambda)$ be a bunch of grapes grown from $\Lambda$.
Then the braid groups $\mathbb{B}_4(\Lambda)$ and $\mathbb{B}_2(\Gamma)$ are isomorphic up to free groups. Namely, 
\[
\mathbb{B}_4(\Lambda)*\mathbb{F}\cong \mathbb{B}_2(\Gamma)*\mathbb{F}'
\]
for some free groups $\mathbb{F}$ and $\mathbb{F}'$.

In particular, $\mathbb{B}_4(\Lambda)$ and $\mathbb{B}_2(\Gamma)$ are quasi-isometric.
\end{theorem}

By the aid of the above theorem, we have the following corollaries:
\begin{corollary}\label{Cor:B_4QIiffB_2QI}
Let $\Lambda_1$ and $\Lambda_2$ be trees, and $\Gamma_1$ and $\Gamma_2$ be two bunches of graphs grown from $\Lambda_1$ and $\Lambda_2$, respectively.
Then $\mathbb{B}_4(\Lambda_1)$ and $\mathbb{B}_4(\Lambda_2)$ are quasi-isometric if and only if so are $\mathbb{B}_2(\Gamma_1)$ and $\mathbb{B}_2(\Gamma_2)$.
\end{corollary}

\begin{corollary}\label{corollary:algorithm for tree 4-braid groups}
Let $\Lambda_1$ and $\Lambda_2$ be trees.
Then there exists an algorithm to determine whether $\mathbb{B}_4(\Lambda_1)$ and $\mathbb{B}_4(\Lambda_2)$ are quasi-isometric.
\end{corollary}
\begin{proof}
The proof is a combination of Theorem~\ref{thm:existenceofAlgorithm} and Corollary~\ref{Cor:B_4QIiffB_2QI}.
\end{proof}

\appendix
\section{Algorithms}\label{appendix:algorithms}

\begin{algorithm}[ht]
\caption{An algorithm producing the normal representative}\label{alg:normal}
\begin{algorithmic}
\REQUIRE A large bunch of grapes $\Gamma=(\mathsf{T},\loops)$
\ENSURE The normal representative $\Gamma_{\mathsf{normal}}=(\mathsf{T}_{\mathsf{normal}},\loops_{\mathsf{normal}})$
\STATE $\mathsf{T}_{\mathsf{normal}}\gets \mathsf{T}$, $\loops_\mathsf{normal}\gets \loops$
\FORALL{$t=[v_0,\dots, v_k]\subset\mathsf{T}_{\mathsf{normal}}$ is an empty twig with $\loops_{\mathsf{normal}}(v_k)=0$}
	\STATE $\mathsf{T}_{\mathsf{normal}}\gets \mathsf{T}_{\mathsf{normal}}\setminus \{v_1,\dots, v_k\}$
\ENDFOR
\FORALL{$t=[v_0,\dots, v_k]\subset \mathsf{T}_{\mathsf{normal}}$ is a twig of length at least $2$}
	\STATE $\mathsf{T}_{\mathsf{normal}}\gets$ the replacement of $t$ by an edge $[v_0,v_k]$
\ENDFOR
\RETURN{$(\mathsf{T}_{\mathsf{normal}},\loops_{\mathsf{normal}})$}
\end{algorithmic}
\end{algorithm}

\begin{algorithm}[ht]
\caption{An algorithm producing the minimal rich representative}\label{alg:minimal rich}
\begin{algorithmic}
\REQUIRE A large bunch of grapes $\Gamma=(\mathsf{T},\loops)$
\ENSURE The minimal rich representative $\Gamma_{\mathsf{rich}}=(\mathsf{T}_{\mathsf{rich}},\loops_{\mathsf{rich}})$
\STATE $(\mathsf{T}_{\mathsf{rich}},\loops_{\mathsf{rich}})\gets$ the normal representative of $\Gamma$
\FORALL{$v\in\mathsf{T}_{\mathsf{rich}}$}
	\IF{$\val_{\mathsf{T}_{\mathsf{rich}}}(v)\ge 2$}
		\STATE $\loops_{\mathsf{rich}}(v)\gets 1$
	\ELSE
		\STATE $\loops_{\mathsf{rich}}(v)\gets \min\{2,\loops_{\mathsf{rich}}(v)\}$
	\ENDIF
\ENDFOR
\RETURN{$(\mathsf{T}_{\mathsf{rich}},\loops_{\mathsf{rich}})$}
\end{algorithmic}
\end{algorithm}

\begin{algorithm}[ht]
\caption{An algorithm producing the quasi-minimal representative}\label{alg:quasi-minimal}
\begin{algorithmic}
\REQUIRE A large bunch of grapes $\Gamma=(\mathsf{T},\loops)$
\ENSURE The quasi-minimal representative $\Gamma_{\mathsf{rich}}=(\mathsf{T}_{\mathsf{min}},\loops_{\mathsf{min}})$
\STATE $\mathsf{T}_{\mathsf{min}}\gets \mathsf{T}_{\mathsf{rich}}$, $\loops_\mathsf{min}\gets \loops_{\mathsf{rich}}$
\STATE $d\gets \operatorname{diam}(\mathsf{T}_{\mathsf{min}})$
\IF{$d$ is even}
	\STATE $r\gets \frac d2$
\ELSE
	\STATE $r\gets \frac{d+1}2$
\ENDIF
\STATE $c\gets$ the central vertex or edge in $\mathsf{T}$
\FOR{$i=r-1$ to $0$}
	\STATE $\mathcal{V}_i\gets\{v\in\mathsf{T}_{\mathsf{min}}\mid d(c,v)=i\}$
	\FORALL{$v\in \mathcal{V}_i$}
	\STATE $\mathcal{C}\gets$ $\hat v$-components not containing $c$
	\IF{$\exists$ at least $3$ identical extended $\hat v$-components $\Gamma^+_{v,1},\dots,\Gamma^+_{v,m}$ in $\mathcal{C}$}
		\FOR{$j=3$ to $m$}
			\STATE $\mathsf{T}_{v,i}\gets$(the stem of $\Gamma^+_{v,j}$)$\setminus \{v\}$
			\STATE $\mathsf{T}_{\mathsf{min}}\gets \mathsf{T}_{\mathsf{min}}\setminus \mathsf{T}_{v,3}\setminus\cdots\setminus \mathsf{T}_{v,m}$
		\ENDFOR
	\ENDIF
	\ENDFOR
\ENDFOR
\RETURN{$(\mathsf{T}_{\mathsf{min}},\loops_{\mathsf{min}})$}
\end{algorithmic}
\end{algorithm}

\begin{algorithm}[ht]
\caption{An algorithm that determines whether $\mathbb{B}_2(\Gamma)$ and $\mathbb{B}_2(\Gamma')$ are quasi-isometric or not}\label{alg:quasi-isom}
\begin{algorithmic}
\REQUIRE Two bunches of grapes $\Gamma=(\mathsf{T},\loops)$ and $\Gamma'=(\mathsf{T}',\loops')$
\ENSURE Determine whether $\mathbb{B}_2(\Gamma)$ and $\mathbb{B}_2(\Gamma')$ are quasi-isometric or not.
\STATE $V\gets \#\{v\in\mathcal{V}(\mathsf{T})\mid \loops(v)\ge 1\}$
\STATE $V'\gets \#\{v'\in\mathcal{V}(\mathsf{T}')\mid \loops'(v')\ge 1\}$
\IF{$V\le 1$ and $V'>1$ or $V>1$ and $V'\le1$}
	\RETURN \FALSE
\ELSIF{$V\le 1$ and $V'\le 1$}
	\STATE $N(n,\ell)\gets \frac{(n+\ell) ( n+3\ell-3 )}{2} + 1$
	\STATE $N_1\gets\sum_{v\in \mathcal{V}(\mathsf{T})} N(\val_{\mathsf{T}}(v),\loops(v))$
	\STATE $N_1'\gets\sum_{v'\in \mathcal{V}(\mathsf{T'})} N(\val_{\mathsf{T}'}(v'),\loops'(v'))$
	\IF{$\min\{N_1,2\}=\min\{N_1',2\}$}
		\RETURN \TRUE
	\ELSE
		\RETURN \FALSE
	\ENDIF
\ENDIF
\STATE $\Gamma_{\mathsf{min}}\gets$ the minimal quasi-representative of $\Gamma$
\STATE $\Gamma'_{\mathsf{min}}\gets$ the minimal quasi-representative of $\Gamma'$
\IF{$\Gamma_{\mathsf{min}}$ and $\Gamma_{\mathsf{min}}'$ are isometric}
	\RETURN \TRUE
\ELSE
	\RETURN \FALSE
\ENDIF
\end{algorithmic}
\end{algorithm}

\bibliographystyle{alpha} 
\bibliography{Ref.bib}
\end{document}